\documentclass[11pt,a4paper,reqno]{amsart}
\pdfoutput=1
\usepackage{times}
\usepackage[margin=1.0in,tmargin=0.9in,bmargin=0.80in]{geometry}
\usepackage[dvipsnames]{xcolor}
\usepackage[english]{babel} 

\definecolor{bred}{rgb}{0.8,0,0}

\usepackage{subfigure}
\usepackage{amsmath,amssymb,graphicx,epsfig}
\usepackage{mathtools}
\usepackage{enumerate,amsthm,epstopdf}
\usepackage{enumitem}
\setlist[enumerate]{label={\upshape(\roman*)}}
\usepackage[flushleft]{threeparttable}
\usepackage{multirow}
\usepackage{longtable}
\usepackage[utf8]{inputenc} 
\usepackage[T1]{fontenc}    
\usepackage{bbm}
\usepackage{xargs}
\usepackage{latexsym}
\usepackage{wasysym}
\usepackage{float}
\usepackage{hyperref}       
\usepackage{url}            
\usepackage{booktabs}       
\usepackage{amsfonts}       
\usepackage{nicefrac}       
\usepackage{microtype}      
\usepackage{cleveref}
\usepackage[textsize=footnotesize]{todonotes}
\hypersetup{colorlinks,linkcolor={blue},citecolor={bred},urlcolor={blue}}
\usepackage[numbers]{natbib}

\def\rmd{\mathrm{d}}

\newcommand{\N}{\mathbb{N}}

\newcommand{\R}{\mathbb{R}}
\newcommand{\1}{\mathbbm{1}}

\renewcommand{\P}{\mathbb{P}}

\newcommand{\lfrf}[1]{\left\lfloor #1 \right\rfloor}
\newcommand{\lcrc}[1]{\left\lceil #1 \right\rceil}

\DeclareMathOperator{\E}{\mathbb{E}}

\DeclareMathOperator{\Tr}{Tr}

\newtheorem{theorem}{Theorem}[section]
\newtheorem{proposition}[theorem]{Proposition}
\newtheorem{lemma}[theorem]{Lemma}
\newtheorem{corollary}[theorem]{Corollary}
\newtheorem{remark}[theorem]{Remark}
\newtheorem{definition}[theorem]{Definition}

\newtheorem{assumption}{Assumption}

\newcommandx{\CPE}[3][1=]{{\mathbb E}^{#1}\left[\left. #2 \, \right| #3 \right]} 

\allowdisplaybreaks

\begin{document}

\title[Non-asymptotic estimates for TUSLA with applications to ReLU neural networks]{Non-asymptotic estimates for TUSLA algorithm \\for non-convex learning with applications to neural networks with ReLU activation function}

\author[D.-Y. Lim]{Dong-Young Lim}
\author[A. Neufeld]{Ariel Neufeld}
\author[S. Sabanis]{Sotirios Sabanis}
\author[Y. Zhang]{Ying Zhang}

\address{Department of Industrial Engineering, Ulsan National Institute of Science and Technology (UNIST),  112 Enginnering Building, 301-14, Ulsan, South Korea}
\email{dlim@unist.ac.kr}

\address{Division of Mathematical Sciences, Nanyang Technological University, 21 Nanyang Link, 637371 Singapore}
\email{ariel.neufeld@ntu.edu.sg}

\address{School of Mathematics, The University of Edinburgh, James Clerk Maxwell Building, Peter Guthrie Tait Rd, Edinburgh EH9 3FD, UK \& The Alan Turing Institute, 2QR, 96 Euston Rd, London NW1 2DB, UK \& National Technical University of Athens, Athens, 15780, Greece}
\email{s.sabanis@ed.ac.uk}

\address{Financial Technology Thrust, Society Hub, The Hong Kong University of Science and Technology (Guangzhou), No.\ 1 Du Xue Rd, Nansha District, Guangzhou, China \& Division of Mathematical Sciences, Nanyang Technological University, 21 Nanyang Link, 637371 Singapore}
\email{ying.zhang@ntu.edu.sg}

\date{}
\thanks{
Financial supports by The Alan Turing Institute, London under the EPSRC grant EP/N510129/1, by the Nanyang Assistant Professorship Grant (NAP Grant) \textit{Machine Learning based Algorithms in Finance and Insurance}, by the European Union’s Horizon 2020 research and innovation programme under the Marie Skłodowska-Curie grant agreement No 801215, and by the University of Edinburgh Data-Driven Innovation programme, part of the Edinburgh and South East Scotland City Region Deal are gratefully acknowledged.}
\keywords{Non-convex optimization, non-asymptotic estimates, artificial neural networks, ReLU activation function, taming technique, super-linearly growing coefficients, discontinuous stochastic gradient}

\begin{abstract}
We consider non-convex stochastic optimization problems where the objective functions have super-linearly growing and discontinuous stochastic gradients. In such a setting, we provide a non-asymptotic analysis for the tamed unadjusted stochastic Langevin algorithm (TUSLA) introduced in Lovas et al.\ (2020).  In particular, we establish non-asymptotic error bounds for the TUSLA algorithm in Wasserstein-1 and Wasserstein-2 distances. The latter result enables us to further derive non-asymptotic estimates for the expected excess risk. To illustrate the applicability of the main results, we consider an example from transfer learning with ReLU neural networks, which represents a key paradigm in machine learning. Numerical experiments are presented for the aforementioned example which support our theoretical findings. Hence, in this setting, we demonstrate both theoretically and numerically that the TUSLA algorithm can solve the optimization problem involving neural networks with ReLU activation function. Besides, we provide simulation results for synthetic examples where popular algorithms, e.g. ADAM, AMSGrad, RMSProp, and (vanilla) stochastic gradient descent (SGD) algorithm, may fail to find the minimizer of the objective functions due to the super-linear growth and the discontinuity of the corresponding stochastic gradient, while the TUSLA algorithm converges rapidly to the optimal solution.
Moreover, we provide an empirical comparison of the performance of TUSLA with popular stochastic optimizers on real-world datasets, as well as investigate the effect of the key hyperparameters of TUSLA on its performance.
\end{abstract}
\maketitle

\section{Introduction}
In this paper, we focus on non-convex stochastic optimization problems. More precisely, for positive integers $d$ and $m$, let $U: \R^d \times \R^m \rightarrow \R$ be a measurable function and let $X$ be a given $\R^m$-valued random variable. We assume that $\E[|U(\theta, X)|]<\infty$ for all $\theta \in \R^d$, and define $u: \R^d \rightarrow \R$ by $u(\theta) := \E[U(\theta, X)]$, $\theta \in \R^d$. We then consider the following optimization problem:
\begin{equation}\label{opproblem}
\text{minimize} \quad \R^d \ni \theta \mapsto u(\theta) := \E[U(\theta, X)].
\end{equation}
Our aim is to generate an estimator $\hat{\theta}$ such that the expected excess risk given by
\begin{equation}\label{eer}
\E[u(\hat{\theta})] - \inf_{\theta \in \R^d} u(\theta)
\end{equation}
is minimized. The optimization problem \eqref{opproblem} is closely linked to the problem of sampling from a target distribution $\pi_{\beta}(\rmd \theta) \wasypropto \exp(-\beta u(\theta))\rmd \theta$ with $\beta>0$, see \cite{dalalyan}, \cite{pmlr-v65-dalalyan17a}. This is due to the fact that $\pi_{\beta}$ concentrates around the minimizers of $u$ for sufficiently large $\beta$, see \cite{hwang}. It is well-known that, under mild conditions, the (overdamped) Langevin stochastic differential equation (SDE) given by
\begin{equation} \label{sdeintro}
Z_0 = \theta_0, \quad \mathrm{d} Z_t=-h\left(Z_t\right) \mathrm{d} t+ \sqrt{2\beta^{-1}} \mathrm{d} B_t, \quad t \geq 0,
\end{equation}
where $\theta_0$ is an $\R^d$-valued random variable, $h:= \nabla u$, $\beta>0$ is the so-called inverse temperature parameter, and $(B_t)_{t \geq 0}$ is a $d$-dimensional Brownian motion, admits  $\pi_\beta$ as its unique invariant measure. To sample from the target distribution $\pi_{\beta}$, one approach is to consider the stochastic gradient Langevin dynamics (SGLD) algorithm introduced in \cite{wt}, which is given by
\[
\theta^{\text{SGLD}}_0 :=\theta_0, \quad \theta^{\text{SGLD}}_{n+1}=\theta^{\text{SGLD}}_n-\lambda H(\theta^{\text{SGLD}}_n,X_{n+1})+ \sqrt{2\lambda\beta^{-1}} \xi_{n+1},\quad  n\in\N_0,
\]
where $\lambda>0$ is the stepsize, $(X_n)_{n \in \N_0}$ is an i.i.d. sequence of random variables, $H: \R^d \times \R^m \rightarrow \R^d$ is a measurable function satisfying $h(\theta) = \E[H(\theta, X_0)]$ for each $\theta \in \R^d$, $\beta>0$, and $\{\xi_n\}_{n\ge 1}$ is a sequence of independent standard $d$-dimensional Gaussian random variables.

The SGLD algorithm, which can be viewed as the Euler discretization of \eqref{sdeintro} with inexact gradient, has been extensively studied in literature. Under the conditions that the objective function $u$ is strongly convex and the (stochastic) gradient of $u$ is Lipschitz continuous,  i.e., there exist positive constants $M$, $L_{Lip}$ such that for all $\theta, \theta' \in \R^d$, $x, x' \in \R^m$,
\begin{align*}
\begin{cases}
\langle H(\theta, x)-H(\theta', x),  \theta-\theta' \rangle \geq M|\theta-\theta'|^2 &{\text{(Strong convexity)}},\\
|H(\theta, x) -H(\theta', x')| \leq L_{Lip}(|\theta-\theta'|+|x-x'|) &{\text{(Lipschitz continuity)}},
\end{cases}
\end{align*}
or similarly (but slightly weaker conditions) that there exist positive constants $\overline{M}$, $\overline{L}_{Lip}$ such that for all $\theta, \theta' \in \R^d$,
\begin{align*}
\begin{cases}
\langle h(\theta)-h(\theta'),  \theta-\theta'\rangle \geq \overline{M}|\theta-\theta'|^2 &{\text{(Strong convexity)}},\\
|h(\theta) -h(\theta')| \leq \overline{L}_{Lip}|\theta-\theta'| &{\text{(Lipschitz continuity)}},
\end{cases}
\end{align*}
\cite{convex}, \cite{ppbdm}, and \cite{dk} provide non-asymptotic error bounds in Wasserstein-2 distance between the SGLD algorithm and the target distribution $\pi_{\beta}$. In particular, the results in \cite{convex} are obtained in the case of dependent data stream $(X_n)_{n \in \N_0}$. Recent research focuses on the relaxation of the strong convexity condition of $u$. In \cite{raginsky} and \cite{xu}, a dissipativity condition is considered under which non-asymptotic estimates are obtained for the SGLD algorithm in Wasserstein-2 distance. By using a contraction result developed in \cite{eberle2019quantitative}, \cite{nonconvex} improves significantly the aforementioned convergence results in \cite{raginsky} and \cite{xu} even without assuming the independence of the data stream. Moreover, in \cite{sgldloc}, a local dissipativity condition is proposed, and non-asymptotic bounds are obtained in Wasserstein distances following a similar approach as in \cite{nonconvex}. Furthermore, \cite{berkeley} provide convergence results by assuming a convexity at infinity condition of $u$ based on the contraction property established in \cite{eberleold}.

The aforementioned results in both convex and non-convex case are obtained under a global Lipschitz continuity condition (in $\theta$) of the stochastic gradient $H$. However, popular applications in machine learning, especially those with the use of artificial neural networks (ANNs), typically have highly nonlinear objective functions, which results in super-linearly\footnote{We refer to functions $f: \R^k \rightarrow \R^j$, for $k,j \in \N$, to be super-linearly growing if $\sup_{\theta \in \R^k} \frac{|f(\theta)|}{1+|\theta|} = \infty$.} growing stochastic gradients. It has been shown in \cite{hutzenthaler2011} that the Euler scheme with super-linearly growing coefficients is unstable in the sense that the absolute moments of the Euler approximations could diverge to infinity at finite time point. As many (stochastic) gradient descent methods can be viewed as Euler discretizations of SDE \eqref{sdeintro}, their application to super-linearly growing stochastic gradient is problematic, which is confirmed by the numerical experiments in \cite{lovas2020taming} for the SGLD algorithm. To cope with this problem, \cite{lovas2020taming} considers the use of a taming technique, see, e.g., \cite{hutzenthaler2012}, \cite{eulerscheme}, \cite{SabanisAoAP}, \cite{tula}, and a tamed unadjusted stochastic Langevin algorithm (TUSLA) is proposed, which is  given by
\[
\theta^{\lambda}_0 :=\theta_0, \quad \theta^{\lambda}_{n+1}=\theta^{\lambda}_n-\lambda H_\lambda(\theta^{\lambda}_n,X_{n+1})+ \sqrt{2\lambda\beta^{-1}} \xi_{n+1},\quad  n\in\N_0,
\]
where for all $\theta \in \R^d, x \in \R^m$,
\[
H_\lambda(\theta,x):=\frac{H(\theta,x)}{1+\sqrt{\lambda} |\theta|^{2r}}
\]
with $\lambda>0$ and $r>0$. Non-asymptotic analysis of the TUSLA algorithm is provided in \cite{lovas2020taming} in the case of locally Lipschitz continuous $H$, and non-asymptotic results are obtained in Wasserstein-1 and Wasserstein-2 distances with the rate of convergence equal to $1/2$ and $1/4$, respectively. However, in the case of super-linearly growing and discontinuous $H$, theoretical guarantees for the TUSLA algorithm have not been established in the existing literature. Hence, the results established in \cite{lovas2020taming} cannot be applied to optimization problems involving neural networks with ReLU activation function.

To address the issue of $H$ being discontinuous, one line of research considers certain continuity in average conditions. Under such a type of condition, \cite{fort2016} and \cite{4} provide an almost sure convergence result and a strong $L^1$ convergence result, respectively, for the stochastic gradient descent (SGD) algorithm. Another line of research focuses on the application of proximal operators. In \cite{durmus2019analysis}, the Stochastic Proximal Gradient Langevin Dynamics (SPGLD) algorithm is proposed, and a non-asymptotic error bound between the Kullback-Leibler divergence from the target distribution to the averaged distribution associated with the SPGLD algorithm is obtained under the condition that the potential of the target distribution is convex (but no strong convexity condition is imposed). Furthermore, proximal operators can also be used for the design of algorithms involving discontinuous gradient $h$. In \cite{durmus2018efficient}, the Moreau-Yosida Unadjusted Langevin Algorithm (MYULA) is proposed by using proximal operators and Moreau-Yosida envelopes, and a non-asymptotic error bound in total variation distance is obtained under a convexity condition. In addition, \cite{luu2021sampling} proposes proximal type algorithms to sample from distributions that are not necessarily smooth nor log-concave, which can be applied to regression problems with non-smooth penalties. There, by using Moreau-Yosida envelopes, a convergence result in mean square of the proposed algorithm to the smoothed target distribution is obtained, but without specifying relevant constants.  It is worth noting that the aforementioned results in \cite{fort2016}, \cite{4}, \cite{durmus2019analysis}, and \cite{durmus2018efficient} 
are established 
in the case where the (stochastic) gradients are growing (at most) linearly, and hence cannot be applied to optimization problems involving ReLU neural networks.

As an application of \eqref{opproblem}, we are interested in optimization problems in transfer learning with ReLU neural networks, see, e.g. \cite{Goodfellow2016} and references therein. One concrete example\footnote{The following example is presented in dimension one for the illustrative purpose. We refer to Section \ref{fixedexample} for the multidimensional setting.} would be to obtain the best nonlinear mean-square estimator by solving the (regularized) minimization problem:
\begin{equation}\label{probfixedintro}
\min_{\theta \in \R^2} u(\theta) := \min_{\theta \in \R^2} \left(\E[(Y - \mathfrak{N} (\theta, Z))^2]+\frac{\eta}{2(r+1)}|\theta|^{2(r+1)}\right),
\end{equation}
where $\theta \in \R^2$ is the parameter to be optimized, $Z$ is the $\R$-valued input random variable, $Y$ is the $\R$-valued target random variable, $\eta, r>0$, and $\mathfrak{N}: \R^2 \times \R \rightarrow \R$ takes the form:
\begin{equation}\label{nnexpintro}
\mathfrak{N} (\theta,z) :=   W_1 \sigma_1\left(cz+b_0 \right)
\end{equation}
with $z \in \R$ the input data, $c  \in \R $ the fixed (pre-trained) input weight, $b_0  \in \R $ the bias parameter, $W_1  \in \R $ the weight parameter, $\sigma_1$ the ReLU activation function given by $\sigma_1(\nu) = \max\{0, \nu\}$, $\nu \in \R$, and $\theta = (W_1, b_0)$. One observes that the stochastic gradient $H$ of the problem \eqref{probfixedintro}-\eqref{nnexpintro} is super-linearly growing and discontinuous. Thus, the theoretical results for the TUSLA algorithm obtained in \cite{lovas2020taming} cannot be applied. To extend the applicability of the TUSLA algorithm to, e.g., optimization problems involving neural networks with ReLU activation function, 
we consider the case where $H$ is super-linearly growing and discontinuous. More precisely, we assume that $H$ takes the form $H := F+G$, where $G:  \R^d \times \R^m \rightarrow \R^d$ and $F:  \R^d \times \R^m \rightarrow \R^d$. The function $F$ is assumed to be locally Lipschitz continuous, and satisfy a certain convexity at infinity condition, while $G$ is assumed to satisfy a ``continuity in average'' condition. The precise formulations of the assumptions are provided in Assumption \ref{AI}-\ref{AC} in Section \ref{assumption}. For further discussions on the assumptions, we refer to the corresponding remarks in Section \ref{assumption}. 
Under these conditions, non-asymptotic estimates in Wasserstein distances are established in Theorem \ref{mainw1} and Corollary \ref{mainw2}, while a non-asymptotic error bound for the expected excess risk \eqref{eer} is established in Theorem \ref{mainop}, which provides a theoretical guarantee for the TUSLA algorithm to converge to a global minimizer. Detailed proofs are presented in Section \ref{po}. To illustrate the applicability of the main results, we consider an example in transfer learning with the use of ReLU neural networks in Section \ref{fixedexample}, which can be viewed as a multidimensional version of \eqref{probfixedintro}-\eqref{nnexpintro}. It is shown that the stochastic gradient of the problem satisfies Assumption \ref{AI}- \ref{AC}, and numerical experiments support our theoretical findings. Hence, in this setting, we show both theoretically and numerically that the TUSLA algorithm can solve the optimization problem involving neural networks with ReLU activation function. Moreover, we present synthetic examples in Section \ref{atexample} to demonstrate that widely-used machine learning algorithms, e.g. ADAM, AMSGrad, RMSProp, and (vanilla) SGD, may fail to find the minimizer of the corresponding objective function, which is due to the super-linear growth of the stochastic gradient. In contrast, the TUSLA algorithm converges rapidly to the optimal solution. Furthermore, we provide in Section \ref{sub:real_app} an empirical comparison of the performance of TUSLA with popular stochastic optimizers on real-world datasets, whereas in Section \ref{sub:eff_params}, we investigate the effect of the key hyperparameters of TUSLA on its performance. The proofs of the results in Section \ref{app}, i.e. Proposition \ref{propfixed}, Corollary \ref{corofixed}, and Proposition \ref{propat}, are provided in Section \ref{appproof}. 

We conclude this section by introducing some notation. Let $(\Omega,\mathcal{F},P)$ be a probability space. We denote by $\E[Z]$  the expectation of a random variable $Z$. For $1\leq p<\infty$, $L^p$ is used to denote the usual space of $p$-integrable real-valued random variables. Fix integers $d, m \geq 1$. For an $\R^d$-valued random variable $Z$, its law on $\mathcal{B}(\R^d)$, i.e. the Borel sigma-algebra of $\R^d$, is denoted by $\mathcal{L}(Z)$. For a positive real number $a$, we denote by $\lfrf{a}$ its integer part, and define $\lcrc{a} := \lfrf{a}+1 $. The Euclidean scalar product is denoted
by $\langle \cdot,\cdot\rangle$, with $|\cdot|$ standing for the corresponding Euclidean norm (where the dimension of the space may vary depending on the context). For any integer $q \geq 1$, let $\mathcal{P}(\R^q)$ denote the set of probability measures on $\mathcal{B}(\R^q)$. For $\mu\in\mathcal{P}(\R^d)$ and for a $\mu$-integrable function $f:\R^d\to\R$, the notation $\mu(f):=\int_{\R^d} f(\theta)\mu(\rmd \theta)$ is used. For $\mu,\nu\in\mathcal{P}(\R^d)$, let $\mathcal{C}(\mu,\nu)$ denote the set of probability measures $\zeta$ on $\mathcal{B}(\R^{2d})$ such that its respective marginals are $\mu,\nu$. For two Borel probability measures $\mu$ and $\nu$ defined on $\R^d$ with finite $p$-th moments, the Wasserstein distance of order $p \geq 1$ is defined as
\begin{equation} \label{eq:definition-W-p}
{W}_p(\mu,\nu):=
\left(\inf_{\zeta\in\mathcal{C}(\mu,\nu)}\int_{\R^d}\int_{\R^d}|\theta-\theta'|^p\zeta(\rmd \theta \rmd \theta')\right)^{1/p}.
\end{equation}

\section{Assumptions and main results}
Let $U: \R^d \times \R^m \rightarrow \R$ be a measurable function. We assume that $\E[|U(\theta, X)|]<\infty$ for all $\theta \in \R^d$, where $X$ is a given $\R^m$-valued random variable with probability law $\mathcal{L}(X)$. Assume that $u: \R^d \rightarrow \R$ defined by $u(\theta) := \E[U(\theta, X)]$, $\theta \in \R^d$, is a continuously differentiable function, and denote by $h:=\nabla u$ its gradient. Furthermore, define
\begin{equation}\label{pibetaexp}
\pi_{\beta}(A) := \frac{\int_A e^{-\beta u(\theta)} \, \rmd \theta}{\int_{\R^d} e^{-\beta u(\theta)} \, \rmd \theta}, \quad A \in \mathcal{B}(\R^d),
\end{equation}
where we assume $\int_{\R^d} e^{-\beta u(\theta)} \, \rmd \theta <\infty$.

Denote by $(\mathcal{G}_n)_{n\in\N_0}$ a given filtration representing the flow of past information, and denote by $\mathcal{G}_{\infty} := \sigma(\bigcup_{n \in \N_0} \mathcal{G}_n)$. Moreover, let $(X_n)_{n\in\N_0}$ be a $(\mathcal{G}_n)$-adapted process such that $(X_n)_{n\in\N_0}$ is a sequence of i.i.d. $\R^m$-valued random variables with probability law $\mathcal{L}(X)$. Let $(\xi_{n})_{n\in\N_0}$ be  a sequence of independent standard $d$-dimensional Gaussian random variables.  It is assumed throughout the paper that the $\R^d$-valued random variable $\theta_0$ (initial condition), $\mathcal{G}_{\infty}$, and $(\xi_{n})_{n\in\N_0}$ are independent.

The tamed unadjusted stochastic Langevin algorithm (TUSLA) is given by
\begin{equation}\label{tusla}
\theta^{\lambda}_0 :=\theta_0, \quad \theta^{\lambda}_{n+1}=\theta^{\lambda}_n-\lambda H_\lambda(\theta^{\lambda}_n,X_{n+1})+ \sqrt{2\lambda\beta^{-1}} \xi_{n+1},\quad  n\in\N_0,
\end{equation}
where $\lambda>0$ is the stepsize, and $\beta>0$ is the inverse temperature parameter. In addition, for all $\theta \in \R^d, x \in \R^m$, let
\begin{equation} \label{tamedH}
H_\lambda(\theta,x):=\frac{H(\theta,x)}{1+\sqrt{\lambda} |\theta|^{2r}},
\end{equation}
where $H: \R^d \times \R^m \rightarrow \R^d$ takes the following form: for all $\theta \in \R^d, x \in \R^m$,
\begin{equation}\label{expressionH}
H(\theta,x) := G(\theta,x)+F(\theta,x),
\end{equation}
where $G:  \R^d \times \R^m \rightarrow \R^d$ and $F:  \R^d \times \R^m \rightarrow \R^d$ are measurable functions.
\subsection{Assumptions} \label{assumption}
In this section, we present the conditions required to obtain the main results. Let $q \in [1, \infty), r \in [q/2, \infty) \cap \N, \rho \in [1, \infty)$ be fixed. The following assumptions are stated.

We first impose conditions on the initial value $\theta_0$ and the data process $(X_n)_{n \in \N_0}$. In addition, it is assumed that $H(\theta, x)$ is an unbiased estimate of $h(\theta)$ for all $\theta \in \R^d, x \in \R^m$.
\begin{assumption} \label{AI}
The initial condition $\theta_0$ has a finite $4(2r+1)$-th moment, i.e., $\E[|\theta_0|^{4(2r+1)}]<\infty$. The process $(X_n)_{n \in \N_0}$ has a finite $4(2r+1)\rho$-th moment, i.e. $\E[|X_0|^{4(2r+1)\rho}]<\infty$. Furthermore, we have that $h(\theta) = \E[H(\theta, X_0)]$, for all $\theta \in \R^d$.
\end{assumption}
Recall the expression of $H$ presented in \eqref{expressionH}. In the second assumption below, we impose a ``continuity in average'' condition on $G$, which is weaker than a (locally) Lipschitz continuity condition. This concept is proposed in \cite[Eqn. (6)]{4}, and a similar continuity condition can be found in \cite[{\textbf{H4}}]{fort2016}. Moreover, we assume the function $G$ satisfies a polynomial growth condition.
\begin{assumption} \label{AG}
There exists a constant $L_G>0$ such that, for all $\theta, \theta' \in \R^d$,
\[
\E[|G(\theta, X_0)-G(\theta', X_0)|] \leq L_G(1+|\theta|+|\theta'|)^{q-1}|\theta-\theta'|.
\]
In addition, there exists a constant $K_G >1$, such that for all $\theta \in \R^d, x \in \R^m$,
\[
|G(\theta,x) | \leq K_G(1+|x|)^{\rho}(1+|\theta|)^q.
\]
\end{assumption}
\begin{remark} One observes that Assumption \ref{AG} is slightly weaker than the conditional Lipschitz continuity (CLC) property in \cite[Eqn. (6)]{4}, as we consider i.i.d. data stream while \cite{4} considers dependent data stream. In addition, one may refer to \cite[Remark 2.4]{4} for the comments on the differences between \cite[Eqn. (6)]{4} and its similar condition \cite[{\textbf{H4}}]{fort2016}.

Furthermore, consider $G = (G^1, \dots, G^d): \R^d \times \R^m \rightarrow \R^d$ with $G^l$, $l = 1, \dots, d$, taking the following form:
\begin{equation} \label{functionofformclc}
G^l(\theta,x) := \sum_{j=1}^N g_j^l(\theta, x)\1_{ \{\langle c^l(\theta),x \rangle \in I_j^l (\theta)\} }, \qquad \theta \in \R^d, x \in \R^m,
\end{equation}
where $N \in \N$, $c^l=(c_1^l, \dots, c_m^l): \R^d \rightarrow \R^m$ with $c_k^l: \R^d \rightarrow \R$ being Lipschitz continuous, where $g_j^l:\R^d \times \R^m \rightarrow \R$ are jointly local Lipschitz continuous functions, and where the intervals $I_j^l(\theta)$ take the form $(-\infty, \bar{i}^l_j(\theta))$, $(\underline{i}^l_j(\theta), \infty)$, or $(\underline{i}^l_j(\theta), \bar{i}^l_j(\theta))$ with $\underline{i}^l_j, \bar{i}^l_j: \R^d \rightarrow \R$ being Lipschitz continuous functions. Let $X=(X^1, \dots, X^m)$ be an $\R^m$-valued continuous random variable. For any $k = 1, \dots, m$, denote by $f_{X^k|X_{-k}}:  \R \to [0,\infty)$ the density function of $X^k$ given $X^1, \dots, X^{k-1}$, $X^{k+1}, \dots, X^m$. Let $f_{X^k|X_{-k}}$, $k = 1, \dots, m$, be continuous and bounded, and let $|x^k|^2f_{X^k|X_{-k}}(x^k|x_{-k})$ be bounded for any $x = (x_1, \dots, x_m) \in \R^m$, $k = 1, \dots, m$. Then, Assumption \ref{AG} is satisfied in the following cases:
\begin{enumerate}
\item The functions $g_j^l:\R^d \times \R^m \rightarrow \R$ are jointly Lipschitz continuous, and for each $x \in \R^m$, the functions $g_j^l(\cdot, x)$ are bounded. Moreover, $c_k^l(\theta) = 1$, or $c_k^l(\theta) \in (0,1)$, for all $\theta \in \R^d, k= 1, \dots, m$. Real-world applications satisfying the aforementioned form include quantile estimation, vector quantization (Kohonen algorithm), and CVaR minimization, see \cite[Section 5]{4} and \cite[Section 5.2]{sglddiscont} for detailed proofs.
\item The functions $g_j^l:\R^d \times \R^m \rightarrow \R$ are locally Lipschitz continuous, and $c^l(\theta) = c^* \in \R^m \setminus \{0\}$, for all $\theta \in \R^d$. We also refer to the optimization problem involving ReLU neural networks introduced in Section \ref{fixedexample}.
\end{enumerate}
\end{remark}
Similarly, in the following assumption, we assume that the function $F$ satisfies a (joint) local Lipschitz condition and a certain growth condition.
\begin{assumption} \label{AF}
There exists a constant $L_F>0$ such that, for all $\theta, \theta' \in \R^d, x,x' \in \R^m$,
\[
|F(\theta, x)-F(\theta', x')| \leq L_F(1+|x|+|x'|)^{\rho-1}(1+|\theta|+|\theta'|)^{2r}(|\theta-\theta'|+|x-x'|).
\]
Furthermore, there exists a constant $K_F>0$ such that for all $\theta \in \R^d, x \in \R^m$,
\[
|F(\theta,x) | \leq K_F(1+|x|)^{\rho}(1+|\theta|^{2r+1}).
\]
\end{assumption}
\begin{remark} One notes that, in Assumption \ref{AF}, we assume separately a growth condition of $F$, even though a similar condition can be deduced from the local Lipschitzness of $F$. The reason is that we aim to optimise the restriction on the stepsize, i.e. $\lambda_{p, \max}$ with $p \in \N$ given in \eqref{stepsizemax}, which is proportional to the reciprocal of (a power of) $K_F$. For example, consider $F(\theta,x) = c\theta|\theta|^{2l}+x$ for all $\theta \in \R^d, x \in \R^m$, where $c>0, l\geq q/2$. For any $\theta, \theta' \in \R^d, x, x' \in \R^m$, by using
\[
\left||\theta|^{2l} - |\theta'|^{2l}\right|\leq 2l(|\theta|+|\theta'|)^{2l-1}|\theta-\theta'|,
\]
one obtains the following:
\[
|F(\theta,x)-F(\theta',x')| \leq (1+c)(1+2l)(1+|\theta|+|\theta'|)^{2l}(|\theta-\theta'|+|x-x'|).
\]
Here we see that Assumption \ref{AF} is satisfied with $L_F = (1+c)(1+2l), \rho = 1, r = l$, and moreover, it further implies
\[
|F(\theta,x)|\leq K_F(1+|x|)(1+|\theta|^{2l+1}),
\]
where $K_F = 2^{2l}(1+c)(1+2l)+|F(0,0)|$. However, by using directly the expression of $F$, one obtains
\[
|F(\theta,x)|\leq K_F (1+|x|)(1+|\theta|^{2l+1}),
\]
where $K_F = 1+c$.
\end{remark}
\begin{remark}\label{growthHh} By Assumption \ref{AI}, \ref{AG} and \ref{AF}, one notes that $\E[G(\theta, X_0)]$, and $\E[F(\theta, X_0)]$ are well defined. Moreover, one obtains for all, $\theta \in \R^d, x \in \R^m$,
\[
|H(\theta, x)| \leq K_H(1+|x|)^{\rho}(1+|\theta|^{2r+1}),
\]
where $K_H := 2^{2r}K_G+ K_F$. Furthermore, by  Assumption \ref{AI}, \ref{AG} and \ref{AF}, it follows that $h$ is locally Lipschitz continuous, i.e. there exists a constant $L_h>0$ such that for all $\theta, \theta'  \in \R^d$,
\[
|h(\theta) - h(\theta')| \leq L_h(1+|\theta|+|\theta'|)^{2r}|\theta-\theta'|,
\]
where $L_h := L_G + L_F\E[(1+2|X_0|)^{\rho-1}]+1$.
\end{remark}
In the next assumption, a (local) convexity at infinity condition is imposed on $F$.
\begin{assumption} \label{AC}
There exist measurable functions $A:\R^m\to\R^{d\times d}, B:\R^m\to\R^{d\times d}$, and $0\leq \bar{r}<2r$ such that the following holds:
\begin{enumerate}[leftmargin=*]
\item For any $ x \in \R^m, y \in \R^d$,
\[
\langle y, A(x) y\rangle \geq 0, \quad \langle y, B(x) y\rangle \geq 0.
\]
\item For all $\theta, \theta' \in \R^d$ and $x\in\R^m$,
\begin{align}\label{cailoclip}
\begin{split}
\langle  \theta- \theta' , F(\theta,x)- F(\theta',x)\rangle
&\geq \langle\theta- \theta', A(x) (\theta- \theta')\rangle(|\theta|^{2r} +|\theta'|^{2r}) \\
&\quad -\langle\theta- \theta', B(x) (\theta- \theta')\rangle(|\theta|^{\bar{r}} +|\theta'|^{\bar{r}}).
\end{split}
\end{align}
\item The smallest eigenvalue of $\E[A(X_0)]$ is a positive real number $a$, and the largest eigenvalue of $\E[B(X_0)]$ is a nonnegative real number $b$.
\end{enumerate}
\end{assumption}
\begin{remark} To understand Assumption \ref{AC}, we first consider the following condition:
\begin{equation}\label{convlip}
\langle  \theta- \theta' , F(\theta,x)- F(\theta',x)\rangle \geq \langle\theta- \theta', A(x) (\theta- \theta')\rangle(|\theta|^{2r} +|\theta'|^{2r}).
\end{equation}
In the case that $r = 0$, the function $F$ is globally Lipschitz continuous according to Assumption \ref{AF}, 
and the condition \eqref{convlip} becomes a local convexity condition which is the same as \cite[Assumption 3.9]{convex}. It is a ``local'' condition in the sense that \eqref{convlip} depends on the data stream $x$. One may refer to \cite[Assumption 3]{sgldloc} and \cite[Assumption 4]{sglddiscont} for local dissipativity conditions. When $r > 0$ (in particular, $r\geq 1/2$ considered in our setting), the function $F$ is locally Lipschitz continuous and the condition \eqref{convlip} is nothing else than an equivalent (local) convexity condition for a super-linearly growing function $F$.

The condition \eqref{cailoclip} presented in Assumption \ref{AC} is weaker than \eqref{convlip} since, for any $r\geq 0, 0\leq \bar{r}<2r$,  \eqref{convlip} implies \eqref{cailoclip}. For the illustrative purpose, consider a simple example $F(\theta, x) = \theta^3 - \theta$ for all $\theta \in \R, x \in \R$. Here, it is clear that $F(\theta, x)$ does not satisfy the condition \eqref{convlip}, however, Assumption \ref{AC} holds with $A(x)= 1/2, B(x)= 1, r =1 , \bar{r} =0$, i.e.
\[
(\theta- \theta')(F(\theta, x) -F(\theta', x) )\geq(|\theta|^2+|\theta'|^2)|\theta- \theta'|^2/2 - |\theta- \theta'|^2.
\]
Moreover, for $\theta, \theta' \geq \sqrt{2}$, it follows that
\[
(\theta- \theta')(F(\theta, x) -F(\theta', x) ) \geq |\theta - \theta'|^2.
\]
\end{remark}
We note that Assumption \ref{AC} can be satisfied for a wide class of functions, for example, the regularization term in (regularized) optimization problems. The proof of the following statement can be found in Appendix \ref{proofACexample}.
\begin{remark}\label{ACexample} One example of $F$ satisfying Assumption \ref{AC} is given by $F(\theta, x) = \eta \theta|\theta|^{2l}, \theta \in \R^d, x \in \R^m$, with $\eta \in (0,1), l \geq q/2$, which can be viewed as the gradient of the regularization term in regularized optimization problems, see \eqref{probfixed} in Section \ref{fixedexample}. More precisely, in this case, $A(x) =\eta I_d/2, B(x) = 0, r = l, \bar{r} = 0, a= \eta/2, b = 0$. 
\end{remark}
Under Assumption \ref{AI}, \ref{AG}, \ref{AF}, \ref{AC}, one can obtain dissipativity conditions for $F$ and $h$. The explicit statement is presented in the following remark with the proof given in Appendix \ref{proofADFh}.
\begin{remark}\label{ADFh} By Assumption \ref{AI}, \ref{AG}, \ref{AF}, \ref{AC} and the expression of $H$ given in \eqref{expressionH}, one obtains, for all $\theta \in \R^d$,
\begin{equation}\label{sldissipative}
\langle \theta, \E[F(\theta, X_0)]\rangle \geq a_F|\theta|^{2r+2} - b_F,
\end{equation}
where $a_F := a/2$ and $b_F := (a/2+b)R_F^{\bar{r}+2}+K_F^2\E[(1+|X_0|)^{2\rho}]/{2a}$ with
\[
R_F := \max\{(4b/a)^{1/(2r-\bar{r})}, 2^{1/(2r)}\}.
\]
Moreover, it follows that $h$ satisfies the following inequality: for all $\theta \in \R^d$,
\begin{equation}\label{sdedissipative}
\langle \theta, h(\theta)\rangle \geq a_h|\theta|^2 - b_h,
\end{equation}
where $a_h := 2^qK_G\E[(1+|X_0|)^{\rho}]$, $b_h := 3(2^{q+1}K_G\E[(1+|X_0|)^{\rho}]/\min{\{1, a_F\}})^{q+2}+b_F$. One notes that by \eqref{sdedissipative}, $u$ has a minimum $\theta^* \in \R^d$ due to \cite[Eqn. (25), (26)]{JARNER2000341} and \cite[Theorem 2.32]{beck2014introduction}.
\end{remark}
We can further obtain an one-sided Lipschitz continuity condition on $h$, which is stated in the remark below. The proof of the statement is provided in Appendix \ref{proofACh}.
\begin{remark} \label{ACh} By Assumption \ref{AI}, \ref{AG}, \ref{AF}, and \ref{AC} and the expression of $H$ given in \eqref{expressionH}, we have, for all $\theta, \theta' \in \R^d$, that
\[
\langle \theta-\theta',h(\theta)-h(\theta')\rangle\geq -L_R |\theta-\theta'|^2,
\]
where $L_R :=  L_h(1+2R)^{2r}>0$ with $R :=\max\{1, (3^{q-1}L_G/a)^{1/(2r-q+1)}, (2b/a)^{1/(2r - \bar{r})} \}$. 
\end{remark}
\subsection{Main results}
For any $p\in {\N}$, we denote
\begin{equation}\label{stepsizemax}
\lambda_{p, \max} := \min\left\{1, \tfrac{\min\{(a_F/K_F)^2, (a_F/K_F)^{2/(2p-1)}\}}{9\binom{2p}{p}^2K_F^2(\E\left[(1+|X_0|)^{2p\rho  }\right])^2}, \tfrac{1}{a_F}, \tfrac{1}{4a_F^2}\right\}, \qquad \lambda_{\max}:=\lambda_{4r+2, \max}.
\end{equation}
\begin{remark}\label{lambdapmax} One notes that $\lambda_{p, \max}$ given in \eqref{stepsizemax} decreases as $p$ increases.
\end{remark}
One may refer to Appendix \ref{prooflambdapmax} for the detailed proof of the above statement. Then, under the assumptions presented in Section \ref{assumption}, the following non-asymptotic upper bound in Wasserstein-1 distance can be obtained.
\begin{theorem}\label{mainw1}
Let Assumption \ref{AI}, \ref{AG}, \ref{AF}, and \ref{AC} hold.  Then, for any $\beta>0$, there exist constants $C_0, C_1,C_2>0$ such that, for any $0<\lambda\leq \lambda_{\max}$ with $\lambda_{\max}$ given in \eqref{stepsizemax} and $n \in \N_0$, 
\[
W_1(\mathcal{L}(\theta^{\lambda}_n),\pi_{\beta}) \leq C_1 e^{-C_0 \lambda n}(\E[|\theta_0|^{4(2r+1)}]+1) +C_2\sqrt{\lambda},
\]
where $C_0, C_1,C_2$ are given explicitly in \eqref{mainw1const}.
\end{theorem}
The result below provides a non-asymptotic estimate in Wasserstein-2 distance between the law of the algorithm \eqref{tusla} and $\pi_{\beta}$.
\begin{corollary}\label{mainw2}
Let Assumption \ref{AI}, \ref{AG}, \ref{AF}, and \ref{AC} hold.  Then, for any $\beta>0$, there exist constants $C_3, C_4,C_5>0$ such that, for any $0<\lambda\leq \lambda_{\max}$ with $\lambda_{\max}$ given in \eqref{stepsizemax} and $n \in \N_0$,
\[
W_2(\mathcal{L}(\theta^{\lambda}_n),\pi_{\beta})\leq C_4 e^{-C_3 \lambda n}(\E[|\theta_0|^{4(2r+1)}]+1)^{1/2} +C_5\lambda^{1/4},
\]
where $C_3, C_4,C_5$ are given explicitly in \eqref{mainw2const}.
\end{corollary}
Let $\hat{\theta} =  \theta_n^{\lambda}$, where $ \theta_n^{\lambda}$ is the $n$-th iteration of the TUSLA algorithm given in \eqref{tusla}. Then, by using Corollary \ref{mainw2} and by applying a similar splitting approach as suggested in \cite[Eqn. (1.5)]{raginsky}, one can obtain an upper estimate for the expected excess risk of the minimization problem \eqref{opproblem} given by $\E[u( \theta_n^{\lambda})] - \inf_{\theta \in \R^d} u(\theta)$. The statement is provided below.
\begin{theorem}\label{mainop}
Let Assumption \ref{AI}, \ref{AG}, \ref{AF}, and \ref{AC} hold.  Then, for any $\beta>0$, there exist constants $C_6, C_7,C_8, C_9>0$ such that, for any $0<\lambda\leq \lambda_{\max}$ with $\lambda_{\max}$ given in \eqref{stepsizemax} and $n \in \N_0$,
\[
\E[u( \theta_n^{\lambda})] - u^* \leq C_7 e^{-C_6\lambda n}+C_8\lambda^{1/4}+C_9/\beta,
\]
where $u^*:= \inf_{\theta \in \R^d} u(\theta)$, $C_6, C_7,C_8$ are given explicitly in \eqref{eerp1const} while $C_9$ is given in \eqref{eerp2const}.
\end{theorem}
The proofs of the main results can be found in Section \ref{ca}.

\begin{remark}
One notes that the constants $C_1, C_2, C_4, C_5, C_7, C_8$ have exponential dependence on the dimension (as shown in Table \ref{table1fullexp} and \ref{table2keydep}) \textbf{only} due to the contraction result \cite[Theorem 2.2]{eberle2019quantitative} (see also Proposition \ref{contr}). 
In particular, if one could improve the aforementioned result and remove the exponential dependence on the dimension of its constants, then all our constants would have at most polynomial dependence on the dimension. Moreover, it has been shown in the simulation results in Section \ref{app} that the TUSLA algorithm converges rapidly to the corresponding optimal solutions. The explicit expressions of all constants in the main theorems are also provided in Table \ref{table1fullexp} and \ref{table2keydep}.
\end{remark}
\section{Applications}\label{app}
In this section, we apply our theoretical results to various settings. First, in Section \ref{fixedexample}, we present an example with the use of ANNs. In particular, we consider a single-hidden-layer feed-forward neural network (1LFN) with ReLU activation function where its input weight matrix is fixed which could be either obtained from a pre-trained model or randomly generated. Then, in Section \ref{atexample}, we consider a synthetic example where the TUSLA algorithm \eqref{tusla} outperforms state of the art optimizers when the stochastic gradient fails to be Lipschitz continuous. We show in Proposition \ref{propfixed} and \ref{propat} that both examples satisfy Assumption \ref{AI}-\ref{AC}, hence Theorem \ref{mainop} provides theoretical guarantees for the TUSLA algorithm to find the optimal solutions. Simulation results are provided for both examples which support our theoretical findings. Section~\ref{sub:real_app} provides an empirical comparison of the performance of TUSLA with popular stochastic optimizers on real-world datasets such as the concrete compressive strength dataset \citep{data:concrete} for regression and Fashion MNIST \citep{xiao2017/online} for image classification. Lastly, in Section~\ref{sub:eff_params}, we investigate the effect of key hyperparameters $\beta$, $r$, $\lambda$, as well as $\eta$ (for, e.g., regularized optimization problems with target functions of the form \eqref{probfixed}), on the performance of TUSLA. Python code for all the experiments in this paper is avilable at \url{https://github.com/DongyoungLim/TUSLA_RELU}.

\subsection{Feed-forward neural network with fixed input weights}\label{fixedexample}
ANNs with fixed (pre-trained) parameters in the first layer are used in transfer learning and multi-task learning to reduce the computational cost, see, e.g. \cite{Goodfellow2016} and references therein. In this section, we consider a 1LFN with fixed input weights in the context of transfer learning. More precisely, let $d_1, m_1, m_2 \in \N $,  and let $\mathfrak{N} =(\mathfrak{N}^1, \dots, \mathfrak{N}^{m_2}): \R^d \times \R^{m_1} \rightarrow \R^{m_2}$ be the 1LFN with its $i$-th element given by
\begin{equation}\label{nnexp}
\mathfrak{N}^{i}(\theta,z) :=  \sum_{j = 1}^{d_1}W_1^{ij}\sigma_1\left(\langle c^{j\cdot}, z \rangle+b_0^j\right),
\end{equation}
where $z= (z^1, \dots, z^{m_1}) \in \R^{m_1}$ is the input vector, $c=(c^{jk}) \in \R^{d_1 \times m_1}$ is the fixed weight matrix, $b_0=(b_0^1, \dots, b_0^{d_1}) \in \R^{d_1}$ is the bias parameter, $W_1=(W_1^{ij}) \in \R^{m_2 \times d_1}$ is the weight parameter, and $\sigma_1:\R \rightarrow \R$ is the ReLU activation function, i.e., $\sigma_1(\nu) := \max\{0,\nu\}$. Denote by $[W_1]$ the vector of all elements in $W_1$, then
\[
\theta = ([W_1], b_0) \in \R^d
\]
with $d := d_1(1+m_2)$. Moreover, denote by $c_F$ the Frobenius norm of the fixed weight matrix $c$. We assume that at least one element in each row of $c \in \R^{d_1 \times m_1}$ is nonzero, i.e., for each $J=1, \dots, d_1$, there exists $K=1, \dots, m_1$ such that $c^{JK} \neq 0$.

\paragraph{\textbf{Optimization problem}} We consider an $m$-dimensional random variable $X = (Y,Z)$ with $Y=(Y^1, \dots, Y^{m_2}) \in \R^{m_2}$ and $Z=(Z^1, \dots, Z^{m_1}) \in \R^{m_1}$, where $m:=m_1+m_2$. We aim to obtain the best nonlinear mean-square estimator by solving the following (regularized) minimization problem:
\begin{equation}\label{probfixed}
\text{minimize} \quad \R^d \ni \theta \mapsto u(\theta) :=  \E[(Y - \mathfrak{N} (\theta, Z))^2]+\frac{\eta}{2(r+1)}|\theta|^{2(r+1)},
\end{equation}
where $\eta>0$, $r \in [q/2, \infty) \cap \N$ with $q \geq 1$ are given explicitly in Proposition \ref{propfixed}.

\begin{proposition}\label{propfixed} Let $u$ be defined in \eqref{probfixed}. Let $X = (Y, Z)$ be a continuously distributed random variable with probability law $\mathcal{L}(X)$. For any $I = 1, \dots, m_2, K = 1, \dots, m_1$, let
\[
f_{Z^K|Z^1, \dots, Z^{K-1}, Z^{K+1}, \dots, Z^{m_1},  Y^I}: \R \to [0,\infty)
\]
be the density function of $Z^K$ given $Z^1, \dots, Z^{K-1}$, $Z^{K+1}, \dots, Z^{m_1}, Y^I$. For any $I = 1, \dots, m_2$, $K = 1, \dots, m_1$, assume that there exist constants $C_{Z^K}, \bar{C}_{Z^K}>0$, such that for any $z = (z^1, \dots, z^{m_1})  \in \R^{m_1}, y^I \in \R$,
\begin{align}\label{fixedmmtbd}
\begin{split}
f_{Z^K|Z^1, \dots, Z^{K-1}, Z^{K+1} ,\dots, Z^{m_1}, Y^I}(z^K|z^1, \dots,z^{K-1}, z^{K+1}, \dots, z^{m_1}, y^I)&\leq C_{Z^K}, \\
|z^K|^2f_{Z^K|Z^1, \dots, Z^{K-1}, Z^{K+1} ,\dots, Z^{m_1}, Y^I}(z^K|z^1, \dots,z^{K-1}, z^{K+1}, \dots, z^{m_1}, y^I) &\leq \bar{C}_{Z^K}.
\end{split}
\end{align}
Moreover, let $(X_n)_{n\in\N_0} = (Y_n, Z_n)_{n \in \N_0}$ be a sequence of i.i.d. random variables with probability law $\mathcal{L}(X)$, and assume that  $\E[|\theta_0|^{20}+|X_0|^{40}]<\infty$. Furthermore, let $H: \R^d \times \R^m \rightarrow \R^d$ be the stochastic gradient of $u$ which satisfies $H(\theta, x) := G(\theta, x)+ F(\theta, x)$ for all $\theta \in \R^d$ and all $x \in \R^m$ with $x = (y,z)$, $y=(y^1, \dots, y^{m_2}) \in \R^{m_2}, z \in \R^{m_1}$, where the functions $F$ and $G$ are given by
\begin{align}\label{fixedfgexp1}
\begin{split}
F(\theta,x) 	:= \eta \theta|\theta|^{2r}, \quad G(\theta,x) 	:=\left(G_{W_1^{11}}(\theta,x), \dots, G_{W_1^{m_2d_1}}(\theta,x),  G_{b_0^1}(\theta,x), \dots, G_{b_0^{d_1}}(\theta,x) \right),
\end{split}
\end{align}
where for $I = 1, \dots, m_2, J = 1, \dots, d_1$,
\begin{align}\label{fixedfgexp2}
\begin{split}
G_{W_1^{IJ}}(\theta,x) 	&:= -2(y^I - \mathfrak{N}^I(\theta, z))\sigma_1\left(\langle c^{J\cdot}, z \rangle+b_0^J\right),\\
G_{b_0^J}(\theta,x) 		&:=-2\sum_{i = 1}^{m_2}(y^i - \mathfrak{N}^i(\theta, z)) W_1^{iJ}\1_{A_J} (z)
\end{split}
\end{align}
with
\[
A_J := \{z\in \R^{m_1}|\langle c^{J\cdot}, z \rangle+b_0^J\geq 0\}.
\]
Fix $q = 4, r = 2, \rho = 2$. Then, the following hold:
\begin{enumerate}[leftmargin=*]
\item The function $u$ is continuously differentiable, and Assumption \ref{AI} holds.
\item Assumption \ref{AG} is satisfied with
\begin{align*}
L_G &= 55m_2d_1^2(1+c_F)^2  \left(1+C_{Z, \max}+ \bar{C}_{Z, \max}\right)  \E\left[ (1+|X_0|)^2\right],\\
K_G &=  8m_2d_1^2 (1+c_F)^2,
\end{align*}
where
\[
C_{Z, \max} := \max_{J\in\{ 1, \dots, d_1\}}\{C_{Z^{\nu_J}}/c^{J\nu_J}\}, \quad \bar{C}_{Z, \max} :=\max_{J\in\{ 1, \dots, d_1\}}\{\bar{C}_{Z^{\nu_J}}/c^{J\nu_J}\},
\]
with $\nu_J := \min\{K\in \{1, \dots, m_1\}| c^{JK}\neq 0\}$ for each $J= 1, \dots, d_1$.
\item Assumption \ref{AF} is satisfied with $L_F =5\eta, K_F = \eta$.
\item Assumption \ref{AC} holds with $A(x) = \eta I_d/2, B(x) = 0 , \bar{r} = 0, a = \eta/2, b  =0 $.
\end{enumerate}
\end{proposition}
\begin{proof}
See Section \ref{proofpropfixed}.
\end{proof}

\begin{corollary}\label{corofixed} Let $y: \R^{m_1} \to \R^{m_2}$ be a Borel measurable function such that, for any $z \in \R^{m_1}$, $|y(z)| \leq c_y(1+|z|^{q_y})$ with $c_y\geq 0, q_y \geq 1$. Moreover, let $Z$ be an $m_1$-dimensional continuously distributed random variable with probability law $\mathcal{L}(Z)$, let $Y$ be an $m_2$-dimensional random variable defined by $Y = y(Z)$, and let $X = (Y, Z)$. Furthermore, for any $K = 1, \dots, m_1$, let $f_{Z^K|Z^1, \dots, Z^{K-1}, Z^{K+1}, \dots, Z^{m_1}}$ be the density function of $Z^K$ given $Z^1, \dots, Z^{K-1}$, $Z^{K+1}, \dots, Z^{m_1}$. For any $K = 1, \dots, m_1$, assume that there exist constants $C_{Z^K}, \bar{C}_{Z^K}>0$, such that for any $z = (z^1, \dots, z^{m_1}) \in \R^{m_1}$,
\begin{align}\label{remarkfixedmmtbd}
\begin{split}
f_{Z^K|Z^1, \dots, Z^{K-1}, Z^{K+1} ,\dots, Z^{m_1} }(z^K|z^1, \dots,z^{K-1}, z^{K+1}, \dots, z^{m_1} )&\leq C_{Z^K}, \\
|z^K|^{2 q_y}f_{Z^K|Z^1, \dots, Z^{K-1}, Z^{K+1} ,\dots, Z^{m_1} }(z^K|z^1, \dots,z^{K-1}, z^{K+1}, \dots, z^{m_1} ) &\leq \bar{C}_{Z^K}.
\end{split}
\end{align}
In addition, let $(X_n)_{n\in\N_0} =  (Y_n, Z_n)_{n\in\N_0}$ be a sequence of i.i.d. random variables with probability law $\mathcal{L}(X)$. Assume that $\E[ |X_0|^{40q_y}]<\infty$. Then, the results in Proposition \ref{propfixed} hold for $u$ defined in \eqref{probfixed} and for $H$ given in \eqref{fixedfgexp1} and \eqref{fixedfgexp2} but with $\rho = 2q_y$ and
\[
C_{Z, \max} := \max_{J\in\{ 1, \dots, d_1\}}\left\{c_y^2\left(2C_{Z^{\nu_J}}+2^{q_y}  \bar{C}_{Z^{\nu_J}}+2^{q_y}  C_{Z^{\nu_J}}\right)/c^{J\nu_J}\right\}.
\]
\end{corollary}
\begin{proof}
See Section \ref{proofremarkfixed}.
\end{proof}

\paragraph{\textbf{Simulation result}}
Denote by $z = (z^1,z^2) \in \R^2$ with $z^1,z^2 \in \R$. We aim to approximate the function $y(z) =  |2z^1+2z^2-1.5|^3$ on $[0, 1]\times [0, 1]$ using the 1LFN given in \eqref{nnexp} with $d_1 = 15, m_1 = 2, m_2 = 1$. In order to obtain the fixed input weight matrix $c\in \R^{d_1\times m_1}$ in \eqref{nnexp}, we consider the following two methods:
\begin{enumerate}[leftmargin=*]
\item \label{item:tl} In this approach, we obtain $c$ using transfer learning. More precisely, we first train a two-hidden-layer feed-forward neural network (2LFN) to approximate a function $\tilde{y}$ that is similar to the target function $y$. Once the aforementioned 2LFN is fully trained, we obtain the trained parameters involved in the 2LFN. Denote by $\widetilde W_0^*$ the trained input matrix of the 2LFN. Then, when approximating the target function $y$ using 1LFN \eqref{nnexp}, we set $c:=\widetilde W_0^*$. 

To illustrate the aforementioned procedures of transfer learning, we provide below a concrete example. Consider the 2LFN $\widetilde{\mathfrak{N}}:\R^{\widetilde{d}} \times \R^{m_1} \rightarrow \R$ given by
\begin{equation}\label{tlfnexp}
\widetilde{\mathfrak{N}} (\widetilde{\theta},\widetilde{z}) = \sum_{j = 1}^{d_2}\widetilde W_2^{1j}\sigma_2\left(\sum_{k = 1}^{d_1}\left[\widetilde W_1^{jk}\sigma_1\left(\langle \widetilde W_0^{k\cdot}, \widetilde{z} \rangle +\widetilde b_0^k\right)\right]+\widetilde b_1^j\right)
\end{equation}
where $\widetilde{z} = (\widetilde{z}^1, \widetilde{z}^2) \in \R^{m_1}$, $\widetilde W_0\in \R^{d_1 \times m_1}$, $\widetilde b_0\in \R^{d_1}$, $\widetilde b_1\in \R^{d_2}$, $\widetilde W_1\in \R^{d_2\times d_1}$, $\widetilde W_2 \in \R^{1\times d_2}$ with $ d_2 = 15$, $\sigma_1$ is the ReLU activation function, $\sigma_2$ is the $\tanh$ activation function, and where the parameter $\widetilde{\theta} = ([\widetilde W_0], [\widetilde W_1], [\widetilde W_2], \widetilde b_0, \widetilde b_1) \in \R^{\widetilde{d}} $ with $\widetilde{d}=d_1(m_1+1)+d_2(d_1+2)$. We aim to use the 2LFN \eqref{tlfnexp} to approximate the function $\widetilde{y}(\widetilde{z}) =  -|\widetilde{z}^1+2\widetilde{z}^2-1|^2$ on $[0, 1]\times [0, 1]$, and  to obtain the best nonlinear mean-square estimator by solving the optimization problem:
\[
\text{minimize} \quad \R^{\widetilde{d}} \ni \widetilde{\theta} \mapsto \widetilde{u}(\widetilde{\theta} ) :=  \E\left[\left(\widetilde{Y} - \widetilde{\mathfrak{N} }(\widetilde{\theta} , \widetilde{Z})\right)^2\right]+\frac{\eta}{8}\left|\widetilde{\theta} \right|^{8}.
\]
Let $\widetilde{Z} =(\widetilde{Z}^1,\widetilde{Z}^2)\in \R^2$. For the simulation, let $\widetilde{Z}^1, \widetilde{Z}^2 \sim Uni(0,1)$ be independent, and we generate 10000 independent samples $(\widetilde{y}_n, \widetilde{z}_n)_{n = 1}^{10000}$ with $\widetilde{z}_n = (\widetilde{z}^1_n, \widetilde{z}^2_n) $ and $\widetilde{y}_n = \widetilde{y}(\widetilde{z}_n) $ for each $n$. Moreover, we set
\begin{equation}\label{tlfnhyperprmter}
\lambda = 0.5, \quad \eta=10^{-6}, \quad \beta=10^{10}
\end{equation}
with initial value $\widetilde{\theta}_0$ obtained using Xavier initialization \cite{glorot2010understanding}, which is the default setting in Pytorch. 
Figure~\ref{fig:first_nn_loss} shows the training error for the 2LFN \eqref{tlfnexp} with TUSLA. Figure~\ref{fig:first_nn_plot} plots the true function and the fitted curve. After training the 2LFN \eqref{tlfnexp}, we obtain the trained parameters denoted by $\widetilde{\theta}^* = ([\widetilde W_0^*], [\widetilde W_1^*], [\widetilde W_2^*], \widetilde b_0^*, \widetilde b_1^*)$. Then, we set
$
c := \widetilde W_0^*
$ 
when approximating $y$ using the 1LFN~\eqref{nnexp}.
\item Alternatively, one may consider using a randomly generated input weight matrix $c$. For example, we generate each element in $c$ by using a standard uniform distribution. One notes that it has been proved in \cite[Corollary 3]{gonon2020approximation}, \cite[Theorem 5.1 and Corollary 5.4]{cuchiero2020deep}, and \cite[Proposition 4.8]{neufeld2022chaotic} that 1LFNs with randomly generated input weight matrix $c$ and input bias vector $b_0$ possess a certain form of universal approximation property. One may refer to \cite{gonon2020approximation}, \cite{cuchiero2020deep}, and \cite{neufeld2022chaotic} for detailed discussions on neural networks with randomly generated input weights.
\end{enumerate}

\begin{figure}[t]
  \begin{minipage}{0.49\textwidth}
     \centering
     \includegraphics[width=.95\linewidth]{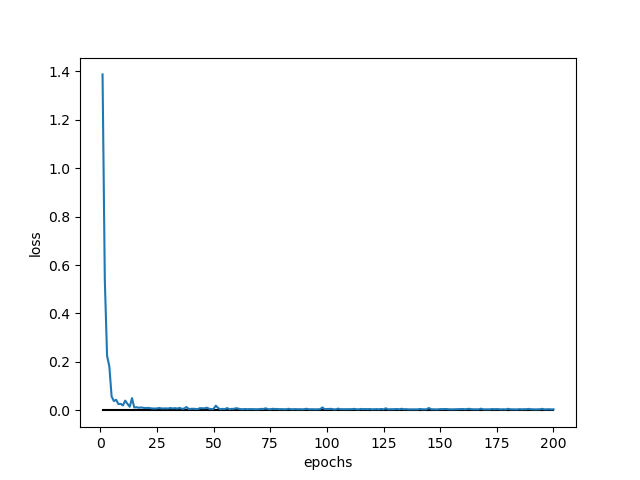}
     \caption{Training loss curve for the 2LFN}
     \label{fig:first_nn_loss}
  \end{minipage}\hfill
  \begin{minipage}{0.49\textwidth}
     \centering
     \includegraphics[width=.95\linewidth]{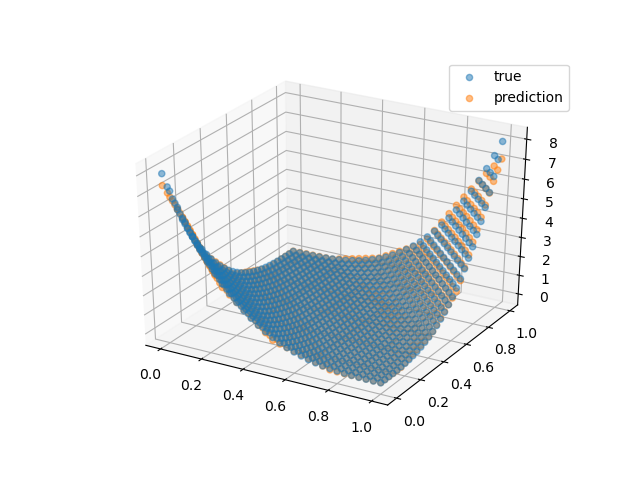}
     \caption{True and estimated value from the 2LFN}
     \label{fig:first_nn_plot}
  \end{minipage}
\end{figure}

Here, we aim to present the simulation results for the optimization problem \eqref{probfixed} with 1LFN \eqref{nnexp} in the context of transfer learning described in \ref{item:tl}, thus, we set the fixed input matrix $c := \widetilde W_0^*$. Let $X = (Y, Z) \in \R^3$ with $Y \in \R$ and $Z =(Z^1, Z^2)\in \R^2$, and let $Z^1, Z^2 \sim Uni(0,1)$ be independent. We generate 10,000 independent samples $(x_n)_{n = 1}^{10000} =(y_n, z_n)_{n = 1}^{10000}$ with $z_n = (z^1_n,z^2_n)$ and $y_n = y(z_n)$ for each $n$. Furthermore, we set the hyperparameters to be the same as in \eqref{tlfnhyperprmter} with $\theta_0$ obtained using Xavier initialization. One notes that Assumption \ref{AI}-\ref{AC} for our main results hold in this setting due to Corollary \ref{corofixed}. Figure~\ref{fig:second_nn_loss} shows the training loss curve for the 1LFN \eqref{nnexp}. Also, Figure~\ref{fig:second_nn_plot} displays the true function and the estimated values computed from the 1LFN \eqref{nnexp}. These results indicate that TUSLA can be successfully used for solving minimization problems involving neural networks with discontinuous activation functions like ReLU.

\begin{figure}[t]
  \begin{minipage}{0.49\textwidth}
     \centering
     \includegraphics[width=.95\linewidth]{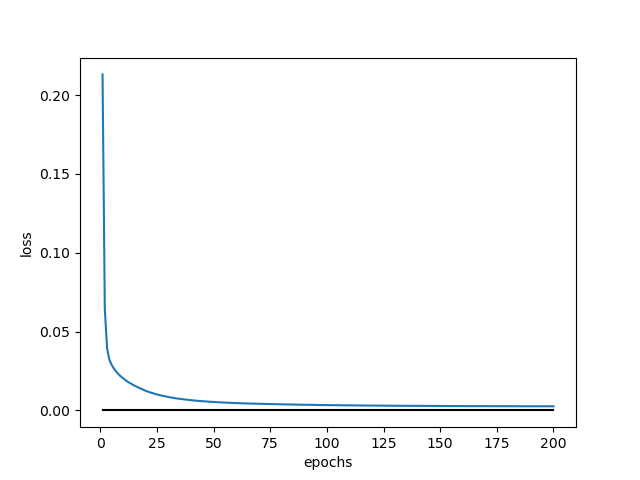}
     \caption{Training loss curve for the 1LFN}
     \label{fig:second_nn_loss}
  \end{minipage}\hfill
  \begin{minipage}{0.49\textwidth}
     \centering
     \includegraphics[width=.95\linewidth]{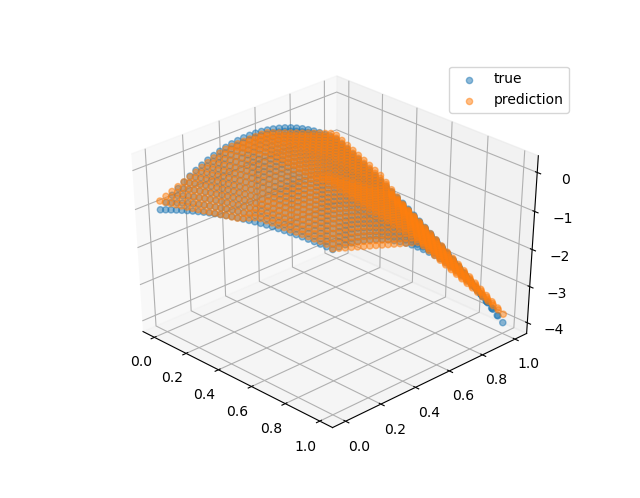}
     \caption{True and estimated value from the 1LFN}
     \label{fig:second_nn_plot}
  \end{minipage}
\end{figure}

\subsection{Artificial example}\label{atexample}
\paragraph{\textbf{Optimization problem}} In this example, we set $d = m = 1$. Consider the following optimization problem:
\begin{equation}\label{atobj}
\text{minimize} \quad \R \ni \theta \mapsto u(\theta) :=   \E[U(\theta, X )],
\end{equation}
where $U: \R \times \R \rightarrow \R$ is defined by
\begin{equation}\label{atobjexp}
U(\theta, x) =
\begin{cases}
a_1\theta^2 \1_{\{x \leq \theta\}}+a_2\theta^2\1_{\{x > \theta\}}+\theta^{30}, \quad |\theta| \leq 1,\\
(a_3|\theta|+a_4) \1_{\{x \leq  \theta\}}+(a_5|\theta|+a_6) \1_{\{x > \theta\}}+\theta^{30}, \quad |\theta| > 1
\end{cases}
\end{equation}
with $a_3, a_4, a_5, a_6 \in \R$ satisfying
\begin{equation}\label{atobjas}
a_3= 2a_1, \quad a_4 = -a_1 , \quad a_5 = 2a_2, \quad a_6 = -a_2
\end{equation}
for any fixed $a_1, a_2\in \R$.
\begin{proposition}\label{propat}
Let $u$ be defined in \eqref{atobj} - \eqref{atobjas}. Let $X$ be a continuously distributed random variable with probability law $\mathcal{L}(X)$, and denote by $f_{X}$ its density function. Assume $f_{X}$ is Lipschitz continuous with Lipschitz constant $L_{X}$, and let $f_{X}$ be upper bounded by the constant $c_{X}$. Moreover, let $(X_n)_{n\in\N_0}$ be a sequence of i.i.d. random variables with probability law $\mathcal{L}(X)$, and assume that $\E[|\theta_0|^{116}+|X_0|^{116}]<\infty$. Furthermore, let $H: \R \times \R \rightarrow \R$ be the stochastic gradient of $u$ that satisfies $H(\theta, x): =F(\theta, x)+G(\theta, x)$ for all $\theta, x \in \R$, where
\begin{equation}\label{atsgexp}
F(\theta, x) =30\theta^{29}, \quad G(\theta, x) =
\begin{cases}
2a_2\theta +2 (a_1-a_2)  \theta \1_{\{x \leq  \theta\}} + (a_1-a_2)\theta^2f_{X}(\theta), &\quad |\theta| \leq 1,\\
2(a_2+ (a_1-a_2) \1_{\{x \leq  \theta\}})(\1_{\{\theta>1\}}-\1_{\{\theta<-1\}}) &\\
+ (a_1-a_2) (2|\theta|-1)f_{X}(\theta), &\quad |\theta| > 1.
\end{cases}
\end{equation}
Fix $q = 3, r = 14, \rho = 1$. Then, the following hold:
\begin{enumerate}[leftmargin=*]
\item The function $u$ is continuously differentiable. Moreover, Assumption \ref{AI} holds.
\item Assumption \ref{AG} is satisfied with
\[
L_G = (4+5c_{X}+2L_{X})(1+|a_1|+|a_2|), \quad K_G =  (4+2c_{X})(1+|a_1|+|a_2|).
\]
\item Assumption \ref{AF} is satisfied with $L_F =870, K_F = 30$.
\item Assumption \ref{AC} holds with $A(x) =15I_d, B(x) = 0, \bar{r} = 0, a = 15, b = 0$.
\end{enumerate}
\end{proposition}
\begin{proof}
See Section \ref{proofpropat}.
\end{proof}
\paragraph{\textbf{Simulation results}} We provide two examples which demonstrate the non-convergence issue of the existing stochastic optimization methods including ADAM, AMSGrad, RMSProp, and (vanilla) SGD when the stochastic gradient $H$ fails to satisfy the global Lipschitz condition, which is commonly assumed in the literature. For illustrative purposes, we consider the optimization problem with super-linearly growing gradients. Our numerical results show that TUSLA can successfully deal with these synthetic examples for both input distributions with bounded and unbounded support, respectively.

\paragraph{\textbf{\textit{Simulation 1}}} Set $a_1 = 2, a_2 = 1$. Then, one observes that by using \eqref{atobj} - \eqref{atobjas}, the optimal solution of $u$ is attained at $\theta=0$ since $U(\theta, x) \geq 0$ for all $\theta, x \in \R$ and $U(0, x) = 0$ for all $x \in \R$. We first present the simulation results for the optimization problem \eqref{atobj} - \eqref{atobjas} with input data $X \sim Beta(2,2)$ and $\theta_0=4$. We solve the optimization problem using TUSLA, ADAM, AMSGrad, RMSProp, and (vanilla) SGD. For ADAM and AMSGrad, we set $\epsilon=10^{-8}$, $\beta_1= 0.9$, $\beta_2=0.999$, and $0.001$ as the stepsize, which are suggested in their papers \cite{chen:19} and \cite{kingma:15}. For RMSProp, we use the default settings in Pytorch, which are $0.01$ for the stepsize, and $\alpha=0.99$. We run TUSLA with $\lambda=0.001$ and $\beta = 10^{10}$.  Figure~\ref{fig:discon1} shows that TUSLA finds the optimal solution after about 500 iterations whereas the other algorithms fail to converge to the true solution even after 1,000 iterations. We also highlight that the vanilla SGD instantly blows up in the presence of higher-order gradients. In addition, Figure~\ref{fig:discon1-2} shows that the same problematic behaviors are consistently observed with larger step sizes for ADAM, AMSGrad and RMSprop, implying that the non-convergence issue cannot be simply resolved by adjusting the learning rate.


\begin{figure}[t]
   \begin{minipage}{0.49\textwidth}
     \centering
     \includegraphics[width=.95\linewidth]{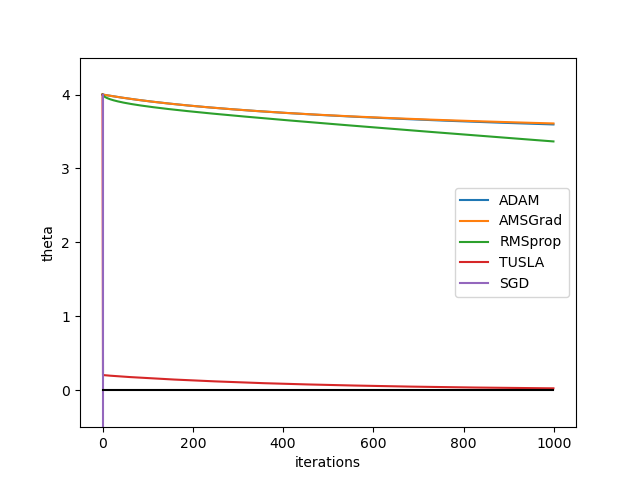}
     \caption{Initial value $\theta_0 = 4.0$}
     \label{fig:discon1}
   \end{minipage}\hfill
   \begin{minipage}{0.49\textwidth}
     \centering
     \includegraphics[width=.95\linewidth]{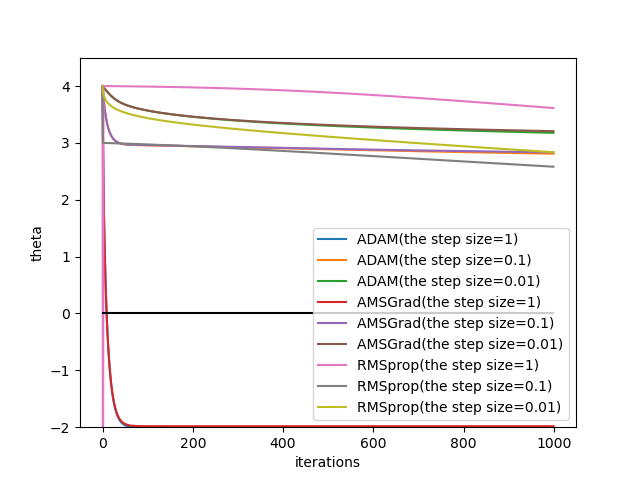}
     \caption{Different learning rates for ADAM, AMSGrad and RMSprop}
     \label{fig:discon1-2}
   \end{minipage}
\end{figure}

\paragraph{\textbf{\textit{Simulation 2}}} Set $a_1 = 2, a_2 = 1$. We then present the simulation results for the optimization problem \eqref{atobj} - \eqref{atobjas} with $X \sim  N(0,1)$ and $\theta_0 = 5$. In this setting, we employ TUSLA, ADAM, AMSGrad, RMSProp, and (vanilla) SGD to find the optimal solution $\theta=0$. Hyper-parameters for these algorithms are the same as in Simulation 1. Figure~\ref{fig:discon2} depicts that TUSLA reaches the optimal point already after about 200 iterations. On the contrary, ADAM, AMSGrad and RMSprop do not effectively work. Figure \ref{fig:discon2-2} further illustrates the non-convergence issue of ADAM, AMSGrad and RMSprop with different learning rates. Moreover, it is worth noting that TUSLA is convergent to the true solution even with larger step sizes, implying that the stability of TUSLA is superior to the existing adaptive optimization methods in the presence of super-linearly growing gradients.

\begin{figure}[t]
   \begin{minipage}{0.49\textwidth}
     \centering
     \includegraphics[width=.95\linewidth]{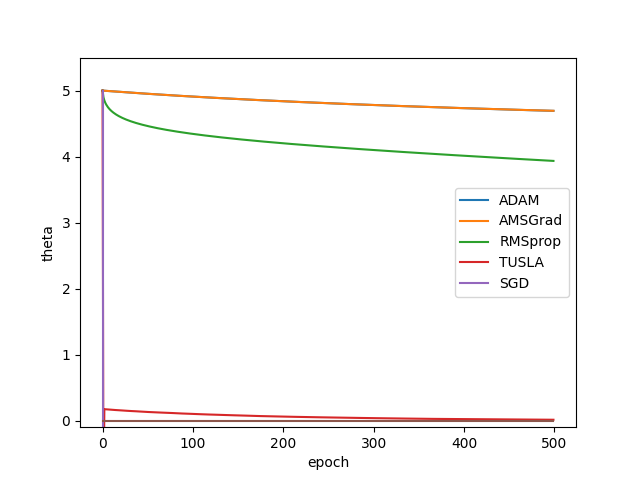}
     \caption{Initial value $\theta_0 = 5.0$}
     \label{fig:discon2}
   \end{minipage}\hfill
   \begin{minipage}{0.49\textwidth}
     \centering
     \includegraphics[width=.95\linewidth]{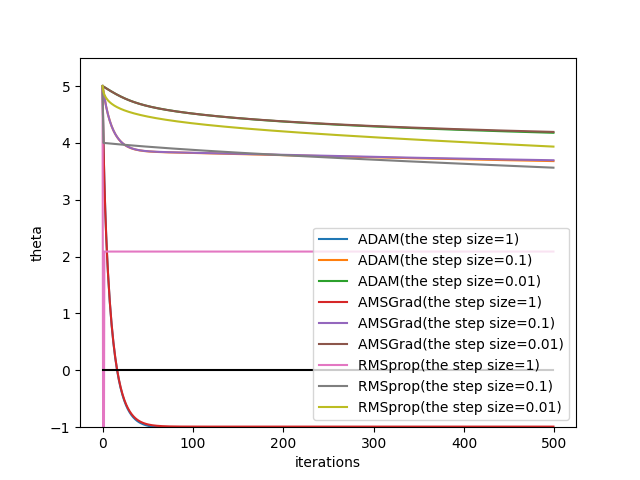}
     \caption{Different learning rates for ADAM, AMSGrad and RMSprop}
     \label{fig:discon2-2}
   \end{minipage}
\end{figure}

\subsection{Real-world applications}\label{sub:real_app}
This subsection presents two real-world applications, image classification on Fashion MNIST \cite{xiao2017/online} and (nonlinear) regression on the concrete compressive strength dataset \cite{data:concrete}, to demonstrate the empirical performance of TUSLA in comparison with other popular stochastic optimization algorithms. For the experiments, we solve the following (regularized) optimization problem:
\begin{equation}\label{optimization_real}
\text{minimize} \quad \R^d \ni \theta \mapsto u(\theta) :=  \E[\ell(Y, \mathfrak{N} (\theta, Z))]+\frac{\eta}{2(r+1)}|\theta|^{2(r+1)} ,
\end{equation}
where $\ell: \R^{m_2} \times \R^{m_2} \rightarrow \R$ is a loss function with $m_2 \in \N$, $\theta \in \R^d$ is the parameter to be optimized, $Z$ is the $\R^{m_1}$-valued input random variable with $m_1 \in \N$, $Y$ is the $\R^{m_2}$-valued target random variable, and $\mathfrak{N}: \R^d \times \R^{m_1} \rightarrow \R^{m_2}$ is a neural network which will be specified later. Note that $d$ and $\ell$ depends on the structure of the neural network and the task of interest.

\subsubsection{UCI regression data}\label{subsec:uci}
We test the performance of TUSLA on the concrete compressive strength dataset of \cite{data:concrete}, which is publicly available at the UCI machine learning repository \footnote{https://archive.ics.uci.edu/ml/datasets.php}. The dataset consists of 1,030 samples where each sample has 9 different attributes, e.g., age, and one target variable: the concrete compressive strength. We aim to find the best estimator that predicts the concrete compressive strength $Y\in \R$ given the input variable $Z \in \R^9$ by solving the optimization problem \eqref{optimization_real} with the squared loss function $\ell(u, v) = |u-v|^2$ for $u,v\in \R$, $m_1=9$, $m_2=1$, and $\eta=0$.

In this example, we consider a 1LFN, which is defined as
\begin{equation}\label{eq:1LFN}
\mathfrak{N}(\theta, z) :=   W_2\sigma_1(W_1 z + b_1)+b_2,
 \quad\quad {\text{(1LFN)}}
\end{equation}
where $W_1 \in \R^{d_1 \times m_1}$, $b_1 \in \R^{d_1}$, $W_2 \in \R^{m_2 \times d_1}$, $b_2\in \R^{m_2}$, $\theta = (W_1, W_2, b_1, b_2) \in \R^d$ with $d=d_1(m_1+m_2+1)+m_2$, $\sigma_1$ is the ReLU activation function and $d_1$ is the number of neurons. Here, $d_1$ is set to be $50$ so that $d$ is 551.

We randomly select 10\% of samples as test set for the evaluation of the trained models and employ TUSLA, RMSprop, ADAM, and AMSGrad to solve the regression problem. For ADAM and AMSGrad, we search the optimal learning rate between $\{0.01, 0.001\}$ and set $\epsilon=10^{-8}$, $\beta_1=0.9$, $\beta_2=0.999$. For RMSprop, the learning rate is chosen from $\{0.01, 0.001\}$ where $\beta=0.99$ and $\epsilon=10^{-8}$ are fixed. For TUSLA, we use $\lambda=0.5$, $r=0.5$, and $\beta=10^{12}$ throughout the experiment. We train the models for $5,000$ epochs with $256$ batch size. Each experiment is performed three times to compute the mean and standard deviation of test loss generated by each optimizer. As shown in Table~\ref{tab:best_test}, TUSLA achieves the lowest  mean-squared error (MSE) in comparison with ADAM, AMSGrad, and RMSprop.


\subsubsection{Fashion MNIST}  We conduct image classification on Fashion MNIST dataset \cite{xiao2017/online} consisting of a training set of 60,000 images and a test set of 10,000 images \footnote{The Fashion MNIST data set can be downloaded at ``
https://github.com/zalandoresearch/fashion-mnist'' .}. Each sample of the dataset $(z_i)_{i = 1}^{60,000}$ is a $28\times 28$ pixel image, i.e., $z_i \in \R^{784}$, and is assigned to one of 10 different labels $l_i \in \{0, 1, \dots, 9\}$ describing T-shirt (0), Trouser (1), Pullover (2), Dress (3), Coat (4), Sandal (5), Shirt (6), Sneaker (7), Bag (8), and Ankle boot (9). Then, the label variables are converted to vectors such that $y_i = [y_{i,0}, y_{i,1}, \dots, y_{i,9}]^\top \in \R^{10}$ with $y_{i, j}=\1_{j=l_i}, j=0, 1, \cdots, 9, i=1,\dots, 60,000.$ \footnote{For example, the target variables for Trouser and Bag are $[0, 1, 0, 0, 0, 0, 0, 0, 0, 0]^\top$ and $[0, 0, 0, 0, 0, 0, 0, 0, 1, 0]^\top$, respectively.}

For image classification, we consider both the 1LFN \eqref{eq:1LFN} with 50 neurons as well as a 2LFN with 50 neurons on each hidden layer, which is defined by
\begin{equation}\label{eq:2LFN}
\mathfrak{N}(\theta, z) :=   W_5\sigma_1\bigg(W_4\sigma_1(W_3 z + b_3)+b_4\bigg)+b_5,
 \quad\quad {\text{(2LFN)}}
\end{equation}
where $\theta = (W_3, W_4, W_5, b_3, b_4, b_5)$, $W_3 \in \R^{50 \times 784}$, $W_4 \in \R^{50 \times 50}$, $W_5 \in \R^{10 \times 50}$,  $b_3 \in \R^{50}$, $b_4 \in \R^{50}$, $b_5 \in \R^{10}$ and $\sigma_1$ is the ReLU activation function. Therefore, we have $d=39,760$ for the 1LFN \eqref{eq:1LFN} and $d=42,310$ for the 2LFN \eqref{eq:2LFN}. Furthermore, the cross entropy loss is used, which is given by $\ell(u, v) = -\sum_{i=1}^{10} u_i \log v_i$ for $u = [u_1, u_2, \cdots, u_{10}]^\top \in \R^{10}, v= [v_1, v_2, \cdots, v_{10}]^\top \in \R^{10}$. $\eta$ is fixed to $10^{-5}$ for all experiments. The models are trained for 200 epochs with 128 batch size.

The hyperparameters for TUSLA are set as follows: $\lambda=0.5$, $r=0.5$, and $\beta=10^{12}$. We apply the same hyperparameters used in Subsection~\ref{subsec:uci} to tune ADAM, AMSgrad, and RMSprop optimizers. Also, we decay the initial learning rate by $10$ after 150 epochs.

Table~\ref{tab:best_test} shows that the performance of TUSLA is slightly better than that of ADAM, AMSgrad, and RMSprop in terms of test accuracy. Also, it is worth noting that TUSLA produces a very stable learning curve compared to other optimizers as shown in Figure~\ref{fig:fmnist_slfn} and \ref{fig:fmnist_tlfn}, confirming the effectiveness of the taming technique.

\begin{figure}[t]
   \begin{minipage}{0.49\textwidth}
     \centering
     \includegraphics[width=.95\linewidth]{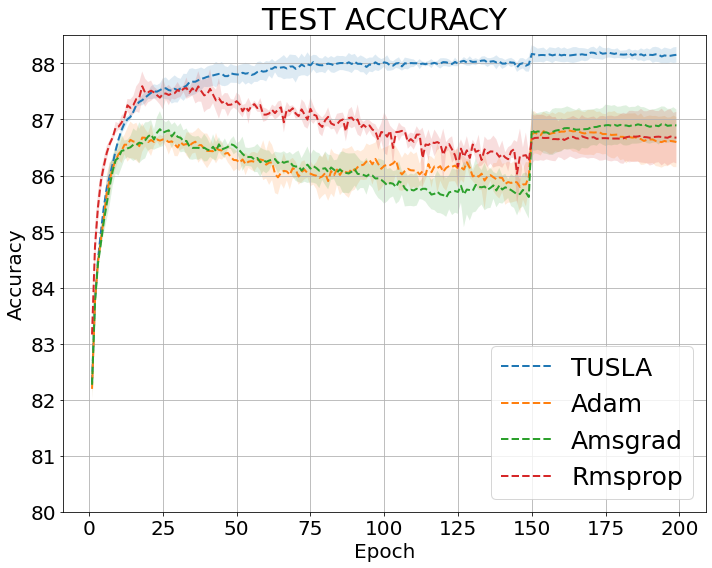}
     \caption{Test accuracy curve for 1LFN}
     \label{fig:fmnist_slfn}
   \end{minipage}\hfill
   \begin{minipage}{0.49\textwidth}
     \centering
     \includegraphics[width=.95\linewidth]{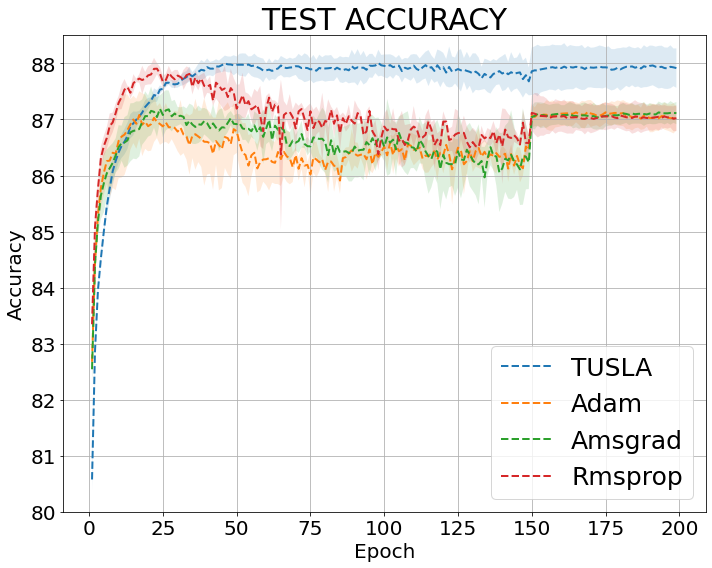}t
     \caption{Test accuracy curve for 2LFN}
     \label{fig:fmnist_tlfn}
   \end{minipage}
\end{figure}

\begin{table}[h]
\centering
\caption{Best metric score evaluated on the test set for concrete compressive strength (Concrete) and Fashion MNIST datasets. We report average and standard deviation of the best metric score on the test set from three repetitive experiments with different random seeds. The numbers in parenthesises indicate the standard deviations.}
\begin{tabular}{c | c  c  c }
 \hline
 Dataset & Concrete & Fashion MNIST & Fashion MNIST  \\
Model & 1LFN & 1LFN & 2LFN   \\
Metric & MSE & Accuracy & Accuracy  \\ \hline
TUSLA & $\textbf{0.3386 } (\textbf{0.0467})$  & $\textbf{88.22 } (\textbf{0.09})$ & $\textbf{88.19 } (\textbf{0.13})$ \\
ADAM & $0.3861$ $(0.0315)$ & $87.03$ $(0.03)$ & $87.32$ $(0.23)$  \\
AMSgrad & $0.4165$ $(0.0010)$ & $87.13$ $(0.22)$ & $87.40$ $(0.09)$ \\
RMSprop & $0.3850$ $(0.0397)$ & $87.70$ $(0.22)$ & $87.99$ $(0.07)$ \\ \hline
\end{tabular}
\label{tab:best_test}
\end{table}

\subsection{Effect of $\beta$, $r$, $\lambda$, and $\eta$ on the performance of TUSLA}\label{sub:eff_params}

In this subsection, we perform a sensitivity analysis to investigate the effect of the hyperparameters, $\beta$, $r$, $\eta$, and $\lambda$ on the performance of TUSLA. We test experiments with the 1LFN on the Fashion MNIST dataset. We train the models for 200 epochs with a batch size of 128 under different hyperparameter settings for the experiment with $\beta$, $r$, and $\eta$, while we train the models for 2,000 epochs with a batch size of 128 for the experiment with $\lambda$.

The inverse temperature $\beta > 0$ is a key feature of Langevin based algorithms, which helps the algorithm to escape from local minima or saddle points. There is a trade-off between a large and small inverse temperature $\beta$. Intuitively, a large inverse temperature generates the solutions that explore the local geometry, so called \textit{the exploitation}. On the other hand, a small inverse temperature allows for the solutions to jump drastically, leading to \textit{the exploration}. To leverage the effects of $\beta$, simulated annealing methods for $\beta$ is widely applied in sampling and optimization. In our experiments, we fix $\beta$ as a constant during the training. Table~\ref{tab:beta} shows that, for fixed other parameters $\lambda=0.5, r=0.5$, and $\eta=10^{-5}$, TUSLA achieves the highest accuracy when $\beta$ is large, namely $10^{8}\sim 10^{12}$. This is consistent with the cold posterior effect in Bayesian deep learning which states that the model performance is improved when a large inverse temperature $\beta$ is chosen, see \cite{aitc20} and \cite{wenzel20}.

\begin{table}[h]
\centering
\caption{Test accuracy for Fashion MNIST datasets with different $\beta$. Other hyperparameters are $\lambda=0.5$, $r=0.5$, and $\eta=10^{-5}$. }
\begin{tabular}{c | c  c  c c  c }
 \hline
 $\beta$ & $10^4$ & $10^6$ & $10^8$ & $10^{10}$ & $10^{12}$  \\ \hline
test accuracy & 68.32 & 86.71 & 88.48 & 88.63 & 88.33 \\ \hline
\end{tabular}
\label{tab:beta}
\end{table}

The hyperparameter $r\geq 0.5$ controls the intensity of the taming function of TUSLA. We conduct experiments with $\lambda =0.5$, $\beta=10^{12}$, and different $r\in \{0.5, 1, 2, 3\}$ and summarize the results in Table~\ref{tab:r}. It turns out that the choice of an appropriate $r$ is a crucial factor for the performance of TUSLA. It is encouraged to gradually increase $r$, as a large $r$ can excessively suppress the gradient part in the formula of TUSLA.

\begin{table}[h]
\centering
\caption{Test accuracy for Fashion MNIST datasets with different $r$. Other hyperparameters are $\lambda=0.5$, $\beta=10^{12}$, and $\eta=10^{-5}$.}
\begin{tabular}{c | c  c  c c  }
 \hline
 $r$ & 0.5 & 1 & 2 & 3 \\ \hline
test accuracy & 88.33 & 87.27 & 82.72 & 79.28 \\ \hline
\end{tabular}
\label{tab:r}
\end{table}

Next, we report the performance of TUSLA with different  $\lambda \in \{0.5, 0.1, 0.05, 0.01, 0.005, 0.001\}$ where $r=0.5$, $\beta=10^{12}$, and $\eta=10^{-5}$ are fixed. For the experiment, we use 2,000 epochs to ensure all the models are fully trained. Then, we report the best test accuracy and the epoch at which the best performance is attained. As shown in Table~\ref{tab:lambda}, it is observed that while the highest test accuracy is attained with $\lambda=0.1$, the model with $\lambda=0.5$ reaches its best accuracy the fastest.

\begin{table}[h]
\centering
\caption{Test accuracy for Fashion MNIST datasets with different $\lambda$. Other hyperparameters are $r=0.5$, $\beta=10^{12}$, and $\eta=10^{-5}$.}
\begin{tabular}{c | c  c  c c c c}
 \hline
 $\lambda$ & $0.5$  & $0.1$ & $0.05$ & $0.01$ & $0.005$ & $0.001$ \\ \hline
test accuracy & 88.69 & 88.81 & 88.66 & 88.24 & 87.96 & 86.83  \\ \hline
best epoch & 584 & 906 & 614 & 621 & 1153 & 1481 \\ \hline
\end{tabular}
\label{tab:lambda}
\end{table}

Lastly, we investigate the impact of $\eta$, which controls the magnitude of the regularization term $|\theta|^{2(r+1)}$, on test accuracy. When the regularized term is incorporated in optimization problems, i.e., $\eta>0$ in \eqref{optimization_real}, overfitting can be reduced by forcing the neural network to have smaller values of its parameters which leads to a simpler model. On the other hand, the deviation between the regularized and original objective functions could lead to degrading the performance of the model. To balance this trade-off, one needs to find an appropriate $\eta$ numerically. Table~\ref{tab:eta} shows the test accuracy of TUSLA with different values of $\eta$ varying from $10^{-5}$ to $10^{-1}$. The other parameters are fixed as follows: $\lambda=0.5$, $r=0.5$, and $\beta=10^{12}$. We observe that TUSLA generates the highest test accuracy when $\eta$ is $10^{-4}$.

\begin{table}[h]
\centering
\caption{Test accuracy for Fashion MNIST datasets with different $\eta$. Other hyperparameters are $\lambda=0.5$, $r=0.5$, and $\beta=10^{12}$.}
\begin{tabular}{c | c  c  c c c }
 \hline
 $\eta$ & $10^{-5}$ & $10^{-4}$ & $10^{-3}$ & $10^{-2}$ & $10^{-1}$ \\ \hline
test accuracy & 88.33 & 88.51 & 86.04 & 80.84 & 71.84 \\ \hline
\end{tabular}
\label{tab:eta}
\end{table}

\section{Proof overview of the main results}\label{po}
In this section, we present the main ideas of establishing Theorem \ref{mainw1}, Corollary \ref{mainw2}, and Theorem \ref{mainop}. We first introduce several auxiliary processes which are key for the analysis of the convergence results. Then, suitable Lyapunov functions are defined, and explicit moment bounds are obtained for the auxiliary processes. Finally, we provide detailed explanations of the proofs of the main results. All proofs of the intermediate results are postponed to Appendix \ref{proofpo}.
\subsection{Auxiliary processes}
Consider the $\R^d$-valued Langevin SDE $(Z_t)_{ t\geq 0}$ given by
\begin{equation} \label{sde}
\mathrm{d} Z_t=-h\left(Z_t\right) \mathrm{d} t+ \sqrt{2\beta^{-1}} \mathrm{d} B_t,
\end{equation}
with $Z_0 := \theta_0$, where $(B_t)_{t \geq 0}$ is a standard $d$-dimensional Brownian motion on $(\Omega,\mathcal{F},P)$. Denote by $(\mathcal{F}_t)_{t\geq 0}$ the completed natural filtration of $(B_t)_{t \geq 0}$, which is assumed to be independent of $\mathcal{G}_{\infty} \vee \sigma(\theta_0)$.

For each $\lambda>0$, denote by $Z^{\lambda}_t := Z_{\lambda t},  t\geq 0$,  the time-changed Langevin SDE given by
\begin{equation} \label{tcsde}
\rmd Z^{\lambda}_t = -\lambda h(Z^{\lambda}_t)\, \rmd t +\sqrt{2\lambda\beta^{-1}}\,\rmd B^{\lambda}_t,
\end{equation}
with the initial condition $Z^{\lambda}_0 := \theta_0$, where $ B^{\lambda}_t :=B_{\lambda t}/\sqrt{\lambda}, t \geq 0$. Note that $( B^{\lambda}_t )_{t\geq 0}$ is a $d$-dimensional standard Brownian motion. For each $\lambda >0$, denote by $(\mathcal{F}^{\lambda}_t)_{t \geq 0}$ the completed natural filtration of $( B^{\lambda}_t )_{t\geq 0}$ with $\mathcal{F}^{\lambda}_t := \mathcal{F}_{\lambda t}$ for each $t \geq 0$, which is also independent of $\mathcal{G}_{\infty} \vee \sigma(\theta_0)$.

Then, define the continuous-time interpolation of the TUSLA algorithm \eqref{tusla}, denoted by $(\bar{\theta}^{\lambda}_t)_{ t\geq 0}$, as
\begin{equation}\label{tuslaproc}
\rmd \bar{\theta}^{\lambda}_t=-\lambda H_{\lambda}(\bar{\theta}^{\lambda}_{\lfrf{t}},X_{\lcrc{t}})\, \rmd t
+ \sqrt{2\lambda\beta^{-1}} \rmd B^{\lambda}_t
\end{equation}
with the initial condition $\bar{\theta}^{\lambda}_0:=\theta_0$. One notes that the law of the interpolated process coincides with the law of the TUSLA algorithm \eqref{tusla} at grid-points, i.e. $\mathcal{L}(\bar{\theta}^{\lambda}_n)=\mathcal{L}(\theta_n^{\lambda})$, for each $n\in\N_0$.

Moreover, denote by $\zeta^{s,v, \lambda}_t, t \geq s$, a continuous-time process defined by the SDE:
\[
\zeta^{s,v, \lambda}_s := v \in \R^d, \quad \rmd\zeta^{s,v, \lambda}_t = -\lambda h(\zeta^{s,v, \lambda}_t)\, \rmd t +\sqrt{2\lambda\beta^{-1}}\,\rmd B^{\lambda}_t.
\]
\begin{definition}\label{auxzeta} For each fixed $\lambda >0$ and $n \in \N_0$, define $\bar{\zeta}^{\lambda, n}_t := \zeta^{nT,\bar{\theta}^{\lambda}_{nT}, \lambda}_t$, $t \geq nT$, where $T\equiv T(\lambda) : = \lfrf{1/\lambda}.$
\end{definition}
\subsection{Preliminary estimates}\label{me}
For each $p\in [2, \infty)\cap {\N}$, define the Lyapunov function $V_p(\theta) := (1+|\theta|^2)^{p/2}$, for all $\theta \in \R^d$, and similarly, define $\mathrm{v}_p(w) := (1+w^2)^{p/2}$ for any real $w \geq 0$. Denote by $\mathcal{P}_{V_p}(\R^d)$ the set of probability measures $\mu \in \mathcal{P}(\R^d)$ satisfying $\int_{\R^d} V_p(\theta)\, \mu(\rmd \theta) <\infty$. Note that $V_p$ is twice continuously differentiable, and
\begin{equation}\label{contrassumption}
\sup_{\theta \in \R^d}\frac{|\nabla V_p(\theta)|}{ V_p(\theta)} <\infty, \quad \lim_{|\theta|\to\infty}\frac{\nabla V_p(\theta)}{V_p(\theta)}=0.
\end{equation}

It is well-known that, under Assumption \ref{AI}, \ref{AG}, \ref{AF}, and \ref{AC}, and by Remark \ref{ADFh}, \ref{ACh}, the Langevin SDE \eqref{tcsde} has a unique solution adapted to $\mathcal{F}_t \vee \sigma(\theta_0)$, $ t\geq 0$, see, e.g. \cite[Theorem 1]{krylovsolmontone}. In addition, for each $p\in {\N}$, the $2p$-th moment of SDE \eqref{tcsde} is finite, see Lemma \ref{tcsdemmtbd}. Moreover, by using the same arguments as in the proof of  \cite[Proposition 1-(ii)]{DM16} together with Lemma \ref{tcsdemmtbd}, it follows that the $2p$-th moment of $\pi_{\beta}$ is finite.

In the following lemma, we provide moment estimates for $(\bar{\theta}^{\lambda}_t)_{t\geq 0}$ defined in \eqref{tuslaproc}. Moreover, by considering a special case of $F$, a more practical stepsize restriction $\tilde{\lambda}_{\max}$ is provided in \eqref{relaxedlamax}.
\begin{lemma}\label{2ndpthmmt} Let Assumption \ref{AI}, \ref{AG}, \ref{AF}, and \ref{AC} hold. Then, one obtains the following:
\begin{enumerate}[leftmargin=*]
\item For any $0<\lambda\leq \lambda_{1, \max}$ with $\lambda_{1, \max}$ given in \eqref{stepsizemax}, $n \in \N_0$, and $t \in (n, n+1]$,
\[
\E\left[ |\bar{\theta}^{\lambda}_t|^2  \right]  \leq  (1-\lambda(t-n)a_F\kappa)(1-a_F\kappa \lambda)^n \E\left[|\theta_0|^2\right]  +c_0(1+1/(a_F\kappa)),
\]
where the constants $c_0, \kappa$ are given explicitly in \eqref{constc0}. In particular, the above inequality implies $\sup_{t\geq 0}\E\left[|\bar{\theta}^{\lambda}_t|^2\right]  \leq  \E\left[|\theta_0|^2\right] +c_0(1+1/(a_F\kappa))<\infty$.
\label{2ndpthmmti}
\item  For any $p\in [2, \infty)\cap {\N}$, $0<\lambda\leq \lambda_{p,\max}$ with $\lambda_{p,\max}$ given in \eqref{stepsizemax}, $n \in \N_0$, and $t \in (n, n+1]$,
\begin{equation}\label{pthmmt}
\E\left[ |\bar{\theta}^{\lambda}_t|^{2p} \right]  \leq (1-\lambda (t-n) a_F\kappa^\sharp_2/2) (1-\lambda a_F\kappa^\sharp_2/2)^n\E\left[|\theta_0|^{2p}\right]  + c^\sharp_p(1+2/(a_F\kappa^\sharp_p )),
\end{equation}
where $\kappa^\sharp_p := \min\{\bar{\kappa}(p), \tilde{\kappa}(p)\}$, $c^\sharp_p := \max\{\bar{c}_0(p), \tilde{c}_0(p)\}$ with the constants $\bar{c}_0(p), \bar{\kappa}(p)$ and $\tilde{c}_0(p), \tilde{\kappa}(p)$ given explicitly in \eqref{constcp} and \eqref{constcpsp}, respectively. In particular, the above estimate implies $\sup_{t\geq 0}\E\left[|\bar{\theta}^{\lambda}_t|^{2p} \right]  \leq  \E\left[|\theta_0|^{2p}\right]+c^\sharp_p(1+2/(a_F\kappa^\sharp_p))<\infty$. \label{2ndpthmmtii}

\item \label{2ndpthmmtiii} If $F$ is a function only of $\theta$, i.e. for any $\theta \in \R^d$, $F(\theta, x) = F(\theta)$ for all $x \in \R^m$, then for any $p\in [2, \infty)\cap {\N}$, $n \in \N_0$, and $t \in (n, n+1]$, \eqref{pthmmt} holds with  $0<\lambda\leq \tilde{\lambda}_{\max} $ where
\begin{equation}\label{relaxedlamax}
\tilde{\lambda}_{\max} := \min\left\{1,\frac{a_F^2}{16K_F^4}, \frac{1}{a_F}, \frac{1}{4a_F^2}\right\}.
\end{equation}
\end{enumerate}
\end{lemma}
\begin{proof} See Appendix \ref{proof2ndpthmmt}.
\end{proof}
A drift condition is obtained for the Lyapunov function $V_p$, which is one of the key assumptions in \cite[Theorem 2.2]{eberle2019quantitative}. The statement is provided below.
\begin{lemma}
\label{driftcon} Let Assumption \ref{AI}, \ref{AG}, \ref{AF}, and \ref{AC} hold. Then, for any $p\in [2, \infty)\cap {\N}$, $\theta \in \R^d$, one obtains
\[
\Delta V_p(\theta)/\beta - \langle \nabla V_p(\theta), h(\theta) \rangle \leq -c_{V,1}(p) V_p(\theta) +c_{V,2}(p),
\]
where $c_{V,1}(p) := a_hp/4$, $c_{V,2}(p) := (3/4)a_hp\mathrm{v}_p(M_V(p))$ with $M_V(p) := (1/3+4b_h/(3a_h)+4d/(3a_h\beta)+4(p-2)/(3a_h\beta))^{1/2}$.
\end{lemma}
\begin{proof} See \cite[Lemma 3.5]{nonconvex}.
\end{proof}
By using Lemma \ref{2ndpthmmt} and Lemma \ref{driftcon}, one can obtain the moment estimates for the process $(\bar{\zeta}^{\lambda, n}_t)_{t \geq nT}$ defined in Definition \ref{auxzeta}. The following lemma provides the second and the fourth moment bound of the aforementioned process.
\begin{lemma}\label{zetaprocme} Let Assumption \ref{AI}, \ref{AG}, \ref{AF}, and \ref{AC} hold. Then, one obtains the following:
\begin{enumerate}[leftmargin=*]
\item For any $0<\lambda\leq \lambda_{1,\max}$ with $\lambda_{1,\max}$ given in \eqref{stepsizemax}, $n \in \N_0$, and $t\geq nT$,
\[
\E[V_2(\bar{\zeta}^{\lambda, n}_t)] \leq e^{-\min\{a_F\kappa, a_h /2\}\lambda t}\E[V_2(\theta_0)]+c_0(1+1/(a_F\kappa))+3\mathrm{v}_2(M_V(2))+1,
\]
where $c_0, \kappa$ are given in \eqref{constc0} (see also Lemma \ref{2ndpthmmt}) and $M_V(2)$ is given in Lemma \ref{driftcon}. \item For any $0<\lambda\leq \lambda_{2,\max}$ with $\lambda_{2,\max}$ given in \eqref{stepsizemax}, $n \in \N_0$, and $t\geq nT$,
\[
\E[V_4(\bar{\zeta}^{\lambda, n}_t)] \leq 2e^{- \min\{a_F \kappa^\sharp_2/2, a_h\}\lambda t}\E[V_4(\theta_0)]+2c^\sharp_2(1+2/(a_F\kappa^\sharp_2 ))+3\mathrm{v}_4(M_V(4))+2,
\]
where $c^\sharp_2, \kappa^\sharp_2$ are given in Lemma \ref{2ndpthmmt} and $M_V(4)$ is given in Lemma \ref{driftcon}.
\end{enumerate}
\end{lemma}
\begin{proof} See Appendix \ref{proofzetaprocme}.
\end{proof}
\subsection{Proof of the main theorems}\label{ca}
We first present the key steps and results in proving Theorem \ref{mainw1}. To obtain a non-asymptotic estimate in Wasserstein-1 distance between $\mathcal{L}(\theta^{\lambda}_n)$ and $\pi_\beta$, we consider the following splitting using the continuous-time interpolation of the TUSLA algorithm \eqref{tusla} given in \eqref{tuslaproc}: for any $n \in \N_0$, and $t \in (nT, (n+1)T]$,
\begin{equation}\label{w1convermain}
W_1(\mathcal{L}(\bar{\theta}^{\lambda}_t),\pi_{\beta})\leq W_1(\mathcal{L}(\bar{\theta}^{\lambda}_t),\mathcal{L}(Z_t^{\lambda}))+W_1(\mathcal{L}(Z_t^{\lambda}),\pi_{\beta}).
\end{equation}
Moreover, the first term on the RHS of the above inequality can be further split as follows by using the auxiliary process $\bar{\zeta}^{\lambda, n}_t$ given in Definition \ref{auxzeta}: for any $n \in \N_0$, and $t \in (nT, (n+1)T]$,
\begin{equation}\label{w1conver}
W_1(\mathcal{L}(\bar{\theta}^{\lambda}_t),\mathcal{L}(Z_t^{\lambda})) \leq W_1(\mathcal{L}(\bar{\theta}^{\lambda}_t),\mathcal{L}(\bar{\zeta}^{\lambda, n}_t))+W_1(\mathcal{L}(\bar{\zeta}^{\lambda, n}_t),\mathcal{L}(Z_t^{\lambda})).
\end{equation}
In the following lemma, we provide an upper estimate for the first term on the RHS of \eqref{w1conver}.
\begin{lemma}\label{w1converp1} Let Assumption \ref{AI}, \ref{AG}, \ref{AF}, and \ref{AC} hold.  Then, for any $0<\lambda\leq \lambda_{\max}$ with $\lambda_{\max}$ given in \eqref{stepsizemax}, $n \in \N_0$, and $t \in (nT, (n+1)T]$, one obtains
\[
W_2(\mathcal{L}(\bar{\theta}^{\lambda}_t),\mathcal{L}(\bar{\zeta}^{\lambda, n}_t)) \leq \sqrt{\lambda}  \left(e^{-a_F \kappa^\sharp_2 n/4}\bar{C}_0\E\left[V_{4(2r+1)}(\theta_0)\right]+\bar{C}_1\right)^{1/2},
\]
where $\kappa^\sharp_2, \bar{C}_0, \bar{C}_1$ are given explicitly in \eqref{w1converp1const}.
\end{lemma}
\begin{proof} See Appendix \ref{proofw1converp1}.
\end{proof}
To obtain an upper bound for the second term on the RHS of \eqref{w1conver}, we consider the following functional: for any $p\geq 1$, $ \mu,\nu \in \mathcal{P}_{V_p}(\R^d)$, let
\begin{equation}\label{w1p}
w_{1,p}(\mu,\nu):=\inf_{\zeta\in\mathcal{C}(\mu,\nu)}\int_{\R^d}\int_{\R^d} [1\wedge |\theta-\theta'|](1+V_p(\theta)+V_p(\theta'))\zeta(\rmd\theta \rmd\theta').
\end{equation}
The case $p=2$, i.e. $w_{1,2}$, is used throughout the paper. For any $ \mu,\nu \in \mathcal{P}_{V_2}(\R^d)$, the following inequalities hold:
\[
W_1(\mu,\nu)\leq w_{1,2}(\mu,\nu), \quad W_2(\mu,\nu)\leq \sqrt{2w_{1,2}(\mu,\nu)}.
\]
One may refer to Lemma \ref{wassw1p} for the proof of these inequalities.

By using \cite[Theorem 2.2]{eberle2019quantitative}, one can derive a contraction property in $w_{1,2}$ with explicit constants as presented below.

\begin{proposition}\label{contr} Let Assumption \ref{AI}, \ref{AG}, \ref{AF}, and \ref{AC} hold. Let $Z_t'$, $ t\geq 0$, be the solution of \eqref{sde} with initial condition $Z'_0 := \theta'_0$ which is independent of $\mathcal{F}_\infty$ and satisfies $\theta_0' \in L^2$. Then,
\begin{equation}\label{w12contraction}
w_{1,2}(\mathcal{L}(Z_t),\mathcal{L}(Z'_t)) \leq \hat{c} e^{-\dot{c} t} w_{1,2}(\mathcal{L}(\theta_0),\mathcal{L}(\theta_0')),
\end{equation}
where the explicit expressions for $\dot{c}, \hat{c}$ are given below.

The contraction constant $\dot{c}$ is given by:
\begin{equation}\label{expcontra1}
\dot{c}:=\min\{\bar{\phi}, c_{V,1}(2), 4c_{V,2}(2) \epsilon c_{V,1}(2)\}/2,
\end{equation}
where $c_{V,1}(2) := a_h/2$ and $c_{V,2}(2) := (3/2)a_h\mathrm{v}_2(M_V(2))$ with $M_V(2)$ given in Lemma \ref{driftcon}, the constant $\bar{\phi} $ is given by
\begin{equation}\label{expcontra2}
\bar{\phi} := \left(\sqrt{8\beta\pi/ L_R} \dot{c}_0  \exp \left( \left(\dot{c}_0 \sqrt{\beta L_R/8} + \sqrt{8/(\beta L_R)} \right)^2 \right) \right)^{-1},
\end{equation}
and $\epsilon >0$ is chosen such that the following inequality is satisfied
\begin{equation}\label{expcontra3}
\epsilon  \leq 1 \wedge    \left(4 c_{V,2}(2) \sqrt{2 \beta\pi/  L_R }\int_0^{\dot{c}_1}\exp  \left( \left(s \sqrt{\beta L_R/8}+\sqrt{8/(\beta L_R)}\right)^2 \right) \,\rmd s \right)^{-1}
\end{equation}
with $\dot{c}_0 := 2(4c_{V,2}(2)(1+c_{V,1}(2))/c_{V,1}(2)-1)^{1/2}$ and $\dot{c}_1:=2(2 c_{V,2}(2)/c_{V,1}(2)-1)^{1/2}$.

Moreover, the constant $\hat{c}$ is given by:
\begin{equation}\label{expcontrabd}
\hat{c}: =2(1+ \dot{c}_0)\exp(\beta L_R \dot{c}_0^2/8+2\dot{c}_0)/\epsilon.
\end{equation}
\end{proposition}
\begin{proof} See Appendix \ref{proofcontr}.
\end{proof}
The following result provides an upper estimate for the second term on the RHS of \eqref{w1conver} based on the contraction property in $w_{1,2}$.
\begin{lemma}\label{w1converp2} Let Assumption \ref{AI}, \ref{AG}, \ref{AF}, and \ref{AC} hold.  Then, for any $0<\lambda\leq \lambda_{\max}$ with $\lambda_{\max}$ given in \eqref{stepsizemax}, $n \in \N_0$, and $t \in (nT, (n+1)T]$, one obtains
\[
W_1(\mathcal{L}(\bar{\zeta}_t^{\lambda,n}),\mathcal{L}(Z_t^\lambda)) \leq \sqrt{\lambda}\left(e^{-\min\{\dot{c}, a_F \kappa^\sharp_2/2, a_h\}n/4}\bar{C}_2\E\left[V_{4(2r+1)}(\theta_0)\right]+\bar{C}_3\right),
\]
where $\kappa^\sharp_2, \bar{C}_2, \bar{C}_3$ are given explicitly in \eqref{w1converp2const}.
\end{lemma}
\begin{proof} See Appendix \ref{proofw1converp2}.
\end{proof}
Finally, observe that by Proposition \ref{contr} and Lemma \ref{wassw1p},
\begin{equation}\label{w1converp3}
W_1(\mathcal{L}(Z_t^{\lambda}),\pi_{\beta}) \leq \hat{c} e^{-\dot{c}\lambda t} w_{1,2}(\theta_0, \pi_{\beta}),
\end{equation}
where the above inequality holds due to the fact that $\pi_{\beta}$ is the invariant measure of SDE \eqref{tcsde}, and where $\dot{c}$ and $\hat{c}$ are defined in \eqref{expcontra1}-\eqref{expcontra3} and \eqref{expcontrabd}, respectively. Combining \eqref{w1convermain}, \eqref{w1conver} together with the Lemma \ref{w1converp1}, \ref{w1converp2} and \eqref{w1converp3} yields the desired upper bound in Theorem \ref{mainw1}.
\begin{proof}[\textbf{Proof of Theorem \ref{mainw1}}] Recall the definition of $w_{1,p}(\mu, \nu)$ with $p\geq 1$, $ \mu,\nu \in \mathcal{P}_{V_p}(\R^d)$ given in \eqref{w1p}. By applying Proposition \ref{contr}, and by using the results in Lemma \ref{w1converp1} and Lemma \ref{w1converp2}, for any $t \in (nT, (n+1)T]$, $n \in \N_0$, $0<\lambda\leq \lambda_{\max}$ with $\lambda_{\max}$ given in \eqref{stepsizemax}, one obtains
\begin{align*}
W_1(\mathcal{L}(\bar{\theta}^{\lambda}_t),\pi_{\beta})
&\leq W_1(\mathcal{L}(\bar{\theta}^{\lambda}_t),\mathcal{L}(Z_t^{\lambda}))+W_1(\mathcal{L}(Z_t^{\lambda}),\pi_{\beta})\\
& \leq W_2(\mathcal{L}(\bar{\theta}^{\lambda}_t),\mathcal{L}(\bar{\zeta}^{\lambda, n}_t))+W_1(\mathcal{L}(\bar{\zeta}^{\lambda, n}_t),\mathcal{L}(Z_t^{\lambda}))+ \hat{c} e^{-\dot{c}\lambda t} w_{1,2}(\theta_0, \pi_{\beta})\\
&\leq \sqrt{\lambda}  \left(e^{-a_F \kappa^\sharp_2 n/4}\bar{C}_0\E\left[V_{4(2r+1)}(\theta_0)\right]+\bar{C}_1\right)^{1/2}\\
& \quad +\sqrt{\lambda}\left(e^{-\min\{\dot{c}, a_F \kappa^\sharp_2/2, a_h\}n/4}\bar{C}_2\E\left[V_{4(2r+1)}(\theta_0)\right]+\bar{C}_3\right)\\
&\quad +\hat{c} e^{-\dot{c}\lambda t}\left[1+ \mathbb{E}[V_2(\theta_0)]+\int_{\R^d}V_2(\theta)\pi_{\beta}(d\theta)\right]\\
& \leq C_1 e^{-C_0 (n+1)}(\E[|\theta_0|^{4(2r+1)}]+1) +C_2\sqrt{\lambda},
\end{align*}
where
\begin{align}\label{mainw1const}
\begin{split}
C_0
& := \min\{\dot{c}, a_F \kappa^\sharp_2/2, a_h\}/4,\\
C_1
& := 2^{2(2r+1)-1}e^{\min\{\dot{c}, a_F \kappa^\sharp_2/2, a_h\}/4}\left[\bar{C}_0^{1/2}+\bar{C}_2+\hat{c}  \left(2+\int_{\R^d}V_2(\theta)\pi_{\beta}(d\theta)\right)\right], \\
C_2
&:=  \bar{C}_1^{1/2}+\bar{C}_3
\end{split}
\end{align}
with  $\dot{c}, \hat{c}$ given in Proposition \ref{contr}, $\kappa^\sharp_2$ given in Lemma \ref{2ndpthmmt}, $\bar{C}_0, \bar{C}_1$ given in \eqref{w1converp1const} (Lemma \ref{w1converp1}), and $\bar{C}_2, \bar{C}_3$ given in \eqref{w1converp2const} (Lemma \ref{w1converp2}). The above result implies that for any $n \in \N_0$,
\[
W_1(\mathcal{L}(\bar{\theta}^{\lambda}_{nT}),\pi_{\beta}) \leq C_1 e^{-C_0 n}(\E[|\theta_0|^{4(2r+1)}]+1) +C_2\sqrt{\lambda}.
\]
However, we aim to obtain a non-asymptotic upper bound for the TUSLA algorithm $(\theta^{\lambda}_n)_{n \in \N_0}$. To achieve this, we set $nT$ to $n$ on the LHS of the above inequality, while $n$ on the RHS of the above inequality is set to $n/T$ with $T\equiv T(\lambda) : = \lfrf{1/\lambda}$. Finally,  for any $n \in \N_0$, $0<\lambda\leq \lambda_{\max}$ with $\lambda_{\max}$ given in \eqref{stepsizemax}, by noticing that $n \lambda \leq n/T$, one obtains
\[
W_1(\mathcal{L}(\theta^{\lambda}_n),\pi_{\beta}) \leq C_1 e^{-C_0 \lambda n}(\E[|\theta_0|^{4(2r+1)}]+1) +C_2\sqrt{\lambda},
\]
which completes the proof.
\end{proof}

Next, one can apply similar arguments to obtain the non-asymptotic error bound in Corollary \ref{mainw2}.
\begin{proof}[\textbf{Proof of Corollary \ref{mainw2}}]
We consider the following splitting: for any $t \in (nT, (n+1)T], n \in \N_0$
\begin{equation}\label{w2conversplit}
W_2(\mathcal{L}(\bar{\theta}^{\lambda}_t),\pi_{\beta}) \leq W_2(\mathcal{L}(\bar{\theta}^{\lambda}_t),\mathcal{L}(\bar{\zeta}^{\lambda, n}_t))+W_2(\mathcal{L}(\bar{\zeta}^{\lambda, n}_t),\mathcal{L}(Z_t^{\lambda}))+ W_2(\mathcal{L}(Z_t^{\lambda}),\pi_{\beta}).
\end{equation}
An upper estimate for the first term on the RHS of \eqref{w2conversplit} is provided in Lemma \ref{w1converp1}. Moreover, by using that $W_2 \leq \sqrt{2w_{1,2}}$ as presented in Lemma \ref{wassw1p}, one can obtain non-asymptotic upper bounds for the last two terms in the above inequality. In particular, for any $t \in (nT, (n+1)T], n \in \N_0$, an upper bound with explicit constants for the second term on the RHS of \eqref{w2conversplit}, i.e., $W_2(\mathcal{L}(\bar{\zeta}^{\lambda, n}_t),\mathcal{L}(Z_t^{\lambda}))$, is given as follows:
\begin{equation}\label{w2converp2}
W_2(\mathcal{L}(\bar{\zeta}_t^{\lambda,n}),\mathcal{L}(Z_t^\lambda)) \leq \lambda^{1/4}\left(e^{-\min\{\dot{c}, a_F \kappa^\sharp_2/2, a_h\}n/8}\bar{C}_4(\E\left[V_{4(2r+1)}(\theta_0)\right])^{1/2}+\bar{C}_5\right),
\end{equation}
where
\begin{align}\label{w2converp2const}
\begin{split}
\kappa^\sharp_2
&:=\min\{\bar{\kappa}(2), \tilde{\kappa}(2)\},\\
\bar{C}_4
&: =e^{\min\{\dot{c}, a_F \kappa^\sharp_2/2, a_h\}/8}\sqrt{\hat{c}}\left(1+ \frac{8}{\min\{\dot{c}, a_F \kappa^\sharp_2/2, a_h\}}\right)(\bar{C}_0^{1/2}+2\sqrt{2}),\\
\bar{C}_5
&: = 
4(\sqrt{\hat{c}}/\dot{c})e^{\dot{c}/4}(\bar{C}_1^{1/2}+2\sqrt{2}(c^\sharp_2(1+2/(a_F\kappa^\sharp_2 )))^{1/2}+(3\mathrm{v}_4(M_V(4)))^{1/2}+3\sqrt{2})
\end{split}
\end{align}
with $\dot{c}, \hat{c}$ given in Proposition \ref{contr}, $\bar{\kappa}(2), \tilde{\kappa}(2)$ given in \eqref{constcp} (Lemma \ref{2ndpthmmt}), $\bar{C}_0, \bar{C}_1$ given in \eqref{w1converp1const} (Lemma \ref{w1converp1}), $c^\sharp_2$ given in Lemma \ref{2ndpthmmt} and $M_V(4)$ given in Lemma \ref{zetaprocme}. The details of the proof of \eqref{w2converp2} are omitted here as the arguments follow the same lines as in the proof of Lemma \ref{w1converp2}.

Recall the definition of $w_{1,2}(\mu, \nu)$, $ \mu,\nu \in \mathcal{P}_{V_2}(\R^d)$ given in \eqref{w1p}. By using \eqref{w2conversplit}, \eqref{w2converp2}, Lemma \ref{w1converp1}, Lemma \ref{wassw1p}, and Proposition \ref{contr}, one obtains,  for any $t \in (nT, (n+1)T], n \in \N_0$,
\begin{align*}
W_2(\mathcal{L}(\bar{\theta}^{\lambda}_t),\pi_{\beta})
&\leq \sqrt{\lambda}  \left(e^{-a_F \kappa^\sharp_2 n/4}\bar{C}_0\E\left[V_{4(2r+1)}(\theta_0)\right]+\bar{C}_1\right)^{1/2}\\
& \quad +\lambda^{1/4}\left(e^{-\min\{\dot{c}, a_F \kappa^\sharp_2/2, a_h\}n/8}\bar{C}_4(\E\left[V_{4(2r+1)}(\theta_0)\right])^{1/2}+\bar{C}_5\right)\\
&\quad + \sqrt{2\hat{c}} e^{-\dot{c}\lambda t/2} (w_{1,2}(\theta_0, \pi_{\beta}))^{1/2}\\
&\leq \sqrt{\lambda}  \left(e^{-a_F \kappa^\sharp_2 n/4}\bar{C}_0\E\left[V_{4(2r+1)}(\theta_0)\right]+\bar{C}_1\right)^{1/2}\\
& \quad +\lambda^{1/4}\left(e^{-\min\{\dot{c}, a_F \kappa^\sharp_2/2, a_h\}n/8}\bar{C}_4(\E\left[V_{4(2r+1)}(\theta_0)\right])^{1/2}+\bar{C}_5\right)\\
&\quad + \sqrt{2\hat{c}}  e^{-\dot{c}\lambda t/2}\left[1+ \mathbb{E}[V_2(\theta_0)]+\int_{\R^d}V_2(\theta)\pi_{\beta}(d\theta)\right]^{1/2}\\
& \leq C_4 e^{-C_3 (n+1)}(\E[|\theta_0|^{4(2r+1)}]+1)^{1/2} +C_5\lambda^{1/4},
\end{align*}
where
\begin{align}\label{mainw2const}
\begin{split}
C_3
& :=\min\{\dot{c}, a_F \kappa^\sharp_2/2, a_h\}/8,\\
C_4
& := 2^{2r+1}e^{\min\{\dot{c}, a_F \kappa^\sharp_2/2, a_h\}/8}\left[\bar{C}_0^{1/2}+\bar{C}_4+\sqrt{\hat{c}}  \left(2+\int_{\R^d}V_2(\theta)\pi_{\beta}(d\theta)\right)^{1/2}\right], \\
C_5
& :=  \bar{C}_1^{1/2}+\bar{C}_5
\end{split}
\end{align}
with  $\dot{c}, \hat{c}$ given in Proposition \ref{contr}, $\kappa^\sharp_2$ given in Lemma \ref{2ndpthmmt}, $\bar{C}_0, \bar{C}_1$ given in \eqref{w1converp1const} (Lemma \ref{w1converp1}), and $\bar{C}_4, \bar{C}_5$ given in \eqref{w2converp2const}. Finally,  for any $n \in \N_0$, $0<\lambda\leq \lambda_{\max}$ with $\lambda_{\max}$ given in \eqref{stepsizemax}, one notes that $n \lambda \leq n/T$, hence, it holds that,
\[
W_2(\mathcal{L}(\theta^{\lambda}_n),\pi_{\beta}) \leq C_4 e^{-C_3 \lambda n}(\E[|\theta_0|^{4(2r+1)}]+1)^{1/2} +C_5\lambda^{1/4},
\]
which completes the proof.
\end{proof}

Recall that $\pi_{\beta}$ is defined in \eqref{pibetaexp}. We denote by $Z_{\infty}$ an $\R^d$-valued random variable with $\mathcal{L}(Z_{\infty}) = \pi_{\beta}$, and denote by $u^*:= \inf_{\theta \in \R^d} u(\theta)$, where $u$ is defined in \eqref{opproblem}. Then, to prove Theorem \ref{mainop}, we consider the following splitting for the expected excess risk:
\begin{equation}\label{eersplitting}
\E[u( \theta_n^{\lambda})] - u^* = \E[u( \theta_n^{\lambda})] - \E[u(Z_{\infty})]+\E[u(Z_{\infty})]- u^*.
\end{equation}

The result below provides an upper bound for the first term on the RHS of \eqref{eersplitting}.
\begin{lemma}\label{eerp1} Let Assumption \ref{AI}, \ref{AG}, \ref{AF}, and \ref{AC} hold.  Then, for any $0<\lambda\leq \lambda_{\max}$ with $\lambda_{\max}$ given in \eqref{stepsizemax} and $n \in \N_0$, one obtains
\[
\E[u( \theta_n^{\lambda})] - \E[u(Z_{\infty})] \leq C_7 e^{-C_6 \lambda n} +C_8\lambda^{1/4},
\]
where $C_6, C_7, C_8$ are given explicitly in \eqref{eerp1const}.
\end{lemma}
\begin{proof} See Appendix \ref{proofeerp1}.
\end{proof}
By applying similar arguments as in \cite[Lemma 3.2]{lovas2020taming} and \cite[Proposition 3.4]{raginsky}, one can obtain an upper estimate of the second term on the RHS of \eqref{eersplitting}. The result with explicit constants is given below.
\begin{lemma}\label{eerp2} Let Assumption \ref{AI}, \ref{AG}, \ref{AF}, and \ref{AC} hold.  Then, 
one obtains
\[
\E[u(Z_{\infty})]- u^* \leq C_9/\beta,
\]
where $C_9$ is given explicitly in \eqref{eerp2const}.
\end{lemma}
\begin{proof} See Appendix \ref{proofeerp2}.
\end{proof}
\begin{proof}[\textbf{Proof of Theorem \ref{mainop}}] Combining the upper bounds in Lemma \ref{eerp1} and \ref{eerp2} yields the desired result in Theorem \ref{mainop}.
\end{proof}
\section{Proof of results in Section \ref{app}}\label{appproof}
\begin{proof}[\textbf{Proof of Proposition \ref{propfixed}}]\label{proofpropfixed}
First, we obtain that the objective function defined in \eqref{nnexp} - \eqref{probfixed} is continuously differentiable with
\begin{align}
\partial_{W_1^{IJ}}u(\theta) &= -2\E\left[\left(Y^I - \mathfrak{N}^I(\theta, Z)\right)\sigma_1\left(\langle c^{J\cdot}, z \rangle+b_0^J\right)\right]+\eta W_1^{IJ}|\theta|^{2r}, \label{fixedderivativeuw1}\\
\partial_{b_0^J}u(\theta) &=-2\E\left[\sum_{i = 1}^{m_2}\left(Y^i - \mathfrak{N}^i(\theta, Z)\right)W_1^{iJ}\1_{A_J}(Z)\right]+\eta b_0^J|\theta|^{2r} \label{fixedderivativeub0}
\end{align}
for all $\theta \in \R^d$, and for all $I = 1, \dots, m_2, J = 1, \dots, d_1$. One notes that \eqref{fixedderivativeuw1} follows directly by the definition of $u$ in \eqref{nnexp} - \eqref{probfixed} and the chain rule. To show that \eqref{fixedderivativeub0} holds, we provide a proof for the case $m_1 = m_2 =d_1 =  1$ for the ease of notation (the same arguments can be applied for general $m_1, m_2, d_1 \in \N$). To that end, one observes that by using \eqref{probfixed}, for any $\theta \in \R^2$,
\begin{align*}
u(\theta)  &=  \E[(Y - \mathfrak{N} (\theta, Z))^2]+\frac{\eta}{2(r+1)}|\theta|^{2(r+1)} \nonumber\\
& = \E\left[\left(Y - W_1\sigma_1(cZ+b_0)\right)^2\right]+\frac{\eta}{2(r+1)}|\theta|^{2(r+1)}\nonumber\\
& =  \E\left[\left(Y^2 - 2YW_1\sigma_1(cZ+b_0)+W_1^2\sigma_1^2(cZ+b_0)\right) \right]+\frac{\eta}{2(r+1)}|\theta|^{2(r+1)}.
\end{align*}
Then, one obtains
\begin{equation}\label{fixedderivativeub01}
\partial_{b_0}u(\theta) = T_{0,1}(\theta)+T_{0,2}(\theta)+\eta b_0|\theta|^{2r},
\end{equation}
where for any $\theta \in \R^2$,
\begin{align*}
 T_{0,1}(\theta)&: =\partial_{b_0}\left(-2W_1\int_{-\infty}^{\infty}\int_{-\frac{b_0}{c}}^{\infty}y(cz+b_0)f_{Y,Z}(y,z)\,\rmd z \rmd y\right), \\
 T_{0,2}(\theta)&: = \partial_{b_0}\left(W_1^2\int_{-\frac{b_0}{c}}^{\infty}(c^2z^2+2czb_0+b_0^2)f_Z(z)\,\rmd z\right)
\end{align*}
with $f_{Y,Z}$ denoting the joint density of $Y, Z$ and $f_Z$ denoting the density function of $Z$. One notes that, for each $y \in \R$,
\begin{align*}
\partial_{b_0}\left(\int_{-\tfrac{b_0}{c}}^{\infty}y(cz+b_0)f_{Z|Y}(z|y)\,\rmd z\right)
& = yb_0f_{Z|Y}\left( \left.-\tfrac{b_0}{c}\right|y\right) - yb_0f_{Z|Y}\left( \left.-\tfrac{b_0}{c}\right|y\right) +\int_{-\frac{b_0}{c}}^{\infty}yf_{Z|Y}(z|y)\,\rmd z\\
& = \int_{-\frac{b_0}{c}}^{\infty}yf_{Z|Y}(z|y)\,\rmd z,
\end{align*}
where $f_{Z|Y}$ denotes the conditional density of $Z$ given $Y$. 
Note that for any $b_0 \in \R$, it holds that
\[
\inf_{\delta \in(0,\infty)}\int_{-\infty}^{\infty}\sup_{\gamma \in [-\delta, \delta]}\left|\int_{-\frac{b_0+\gamma}{c}}^{\infty}yf_{Z|Y}(z|y)\,\rmd z\right|f_Y(y)\,\rmd y \leq \int_{-\infty}^{\infty} |y|f_Y(y)\,\rmd y<\infty,
\]
where $f_Y$ denotes the density function of $Y$. Therefore, one obtains, by applying \cite[Theorem A.5.2.]{durrett2010probability}, that
\begin{align}\label{fixedderivativeub02}
 T_{0,1}(\theta)& =-2W_1\int_{-\infty}^{\infty}\partial_{b_0}\left(\int_{-\frac{b_0}{c}}^{\infty}y(cz+b_0)f_{Z|Y}(z|y)\,\rmd z \right) f_Y(y) \rmd y = -2\E\left[W_1Y\1_{\left\{Z \geq -b_0/c\right\}}\right].
\end{align}Similarly, one obtains
\begin{align}\label{fixedderivativeub03}
 T_{0,2}(\theta)& = -W_1^2 b_0^2f_Z\left(-\tfrac{b_0}{c}\right)+W_1^2  \left( 2b_0^2f_Z\left(-\tfrac{b_0}{c}\right)+\int_{-\frac{b_0}{c}}^{\infty} 2cz f_Z(z)\,\rmd z\right)\nonumber\\
 &\quad +W_1^2  \left(  -b_0^2f_Z\left(-\tfrac{b_0}{c}\right)+2b_0\int_{-\frac{b_0}{c}}^{\infty}   f_Z(z)\,\rmd z\right)\nonumber\\
 & = W_1^2\int_{-\frac{b_0}{c}}^{\infty} 2(cz+b_0) f_Z(z)\,\rmd z\nonumber\\
 & =2 \E\left[W_1^2(cZ+b_0) \1_{\left\{Z \geq -b_0/c\right\}}\right]
\end{align}
Substituting \eqref{fixedderivativeub02}, \eqref{fixedderivativeub03} into \eqref{fixedderivativeub01} yields
\begin{align*}
\partial_{b_0}u(\theta) &= -2\E\left[W_1(Y-W_1(cZ+b_0) \1_{\left\{Z \geq -b_0/c\right\}})\1_{\left\{Z \geq -b_0/c\right\}}\right]+\eta b_0|\theta|^{2r}\\
& = -2\E\left[W_1(Y-\mathfrak{N} (\theta, Z))\1_{\left\{Z \geq -b_0/c\right\}}\right]+\eta b_0|\theta|^{2r},
\end{align*}
which implies \eqref{fixedderivativeub0} holds. Moreover, since $(X_n)_{n\in\N_0}$ is a sequence of i.i.d. random variables with probability law $\mathcal{L}(X)$, by the definitions of $F, G$ given in \eqref{fixedfgexp1}, \eqref{fixedfgexp2} and as $H: = F+G$, we observe that $h(\theta) : = \nabla u(\theta)= \E[H(\theta, X_0)]$, for all $\theta \in \R$. Thus, Assumption \ref{AI} holds.

Recall that we assume at least one element in each row of the fixed input matrix $c \in \R^{d_1 \times m_1}$ is nonzero. For each $J= 1, \dots, d_1$, denote by $\nu_J := \min\{K\in \{1, \dots, m_1\}| c^{JK}\neq 0\}$, then $c^{J\nu_J}$ denotes the first nonzero element in the $J$-th row of $c$. For $C_{Z^{\nu_J}}, \bar{C}_{Z^{\nu_J}}>0$ introduced in \eqref{fixedmmtbd}, we denote by $C_{Z, \max} := \max_J\{C_{Z^{\nu_J}}/c^{J\nu_J}\}$, $\bar{C}_{Z, \max} :=\max_J\{\bar{C}_{Z^{\nu_J}}/c^{J\nu_J}\}$. For each $I = 1, \dots, m_2, K = 1, \dots, m_1$, let $f_{Z^1, \dots, Z^{K-1}, Z^{K+1} ,\dots, Z^{m_1}, Y^I}(z^1, \dots,z^{K-1}, z^{K+1}, \dots, z^{m_1}, y^I)$ be the joint density function of $Z^1, \dots, Z^{K-1}, Z^{K+1} ,\dots, Z^{m_1}, Y^I$. Then, to shorten the notation, we denote by
\begin{align*}
&f_{Z^K|Z_{-K}, Y^I}(z^K|z_{-K}, y^I)
:=f_{Z^K|Z^1, \dots, Z^{K-1}, Z^{K+1} ,\dots, Z^{m_1}, Y^I}(z^K|z^1, \dots,z^{K-1}, z^{K+1}, \dots, z^{m_1}, y^I),\\
&f_{Z_{-K}, Y^I}(z_{-K}, y^I): = f_{Z^1, \dots, Z^{K-1}, Z^{K+1} ,\dots, Z^{m_1}, Y^I}(z^1, \dots,z^{K-1}, z^{K+1}, \dots, z^{m_1}, y^I).
\end{align*}
Moreover, for any $I = 1, \dots, m_2, K = 1, \dots, m_1$, $z \in \R^{m_1}$, $y \in \R^{m_2}$, denote by
\begin{align*}
z_{-K} := (z^1, \dots,z^{K-1}, z^{K+1}, \dots, z^{m_1})\in \R^{m_1-1},\quad y_{-I} := (y^1, \dots,y^{I-1}, y^{I+1}, \dots, y^{m_2})\in \R^{m_2-1}.
\end{align*}
To show that Assumption \ref{AG} holds, consider any $\theta = ([W_1], b_0) \in \R^d$, $\bar{\theta} = ([\bar{W}_1], \bar{b}_0) \in \R^d$. For each $J = 1, \dots, d_1$, denote by
\begin{equation}\label{barAJexp}
\bar{A}_J := \{z\in \R^{m_1}|\langle c^{J\cdot}, z \rangle+\bar{b}_0^J\geq 0\}.
\end{equation}
Assume without loss of generality that $c^{J\nu_J}>0$, and $b_0^J \leq \bar{b}_0^J$ (the other cases can be obtained analogously). Then, one obtains the following estimates:
\begin{enumerate}[leftmargin=*]
\item For any $J = 1, \dots, d_1$, we have
\begin{align*}
&\E\left[ \left|\1_{A_J}(Z) - \1_{\bar{A}_J}(Z)\right| \right]  \nonumber \\
& = \E\left[ \1_{\{(-\bar{b}_0^J-\sum_{k\neq \nu_J}c^{Jk}Z^k)/c^{J\nu_J}  \leq Z^{\nu_J} <(-b_0^J-\sum_{k\neq \nu_J}c^{Jk}Z^k)/c^{J\nu_J} \}} \right] \nonumber \\
& = \int_{\R}\int_{\R^{m_1-1}} \int_{\frac{-\bar{b}_0^J-\sum_{k\neq \nu_J}c^{Jk}z^k}{c^{J\nu_J} }}^{\frac{-b_0^J-\sum_{k\neq \nu_J}c^{Jk}z^k}{c^{J\nu_J} }} f_{Z^{\nu_J}|Z_{-\nu_J}, Y^I}(z^{\nu_J}|z_{-\nu_J}, y^I)\, \rmd z^{\nu_J}   f_{Z_{-\nu_J}, Y^I}(z_{-\nu_J}, y^I)\, \rmd z_{-\nu_J}\, \rmd y^I.
\end{align*}
This implies, by using \eqref{fixedmmtbd}, that
\begin{align}\label{inddiff}
\E\left[ \left|\1_{A_J}(Z) - \1_{\bar{A}_J}(Z)\right| \right] \leq \frac{C_{Z^{\nu_J}}}{c^{J\nu_J}}|\bar{b}_0^J - b_0^J|  \leq C_{Z, \max}|\theta - \bar{\theta}|.
\end{align}
\item For any $I = 1, \dots, m_2, J = 1, \dots, d_1$, it follows that
\begin{align*}
&\E\left[ |Y^I|^2\left|\1_{A_J}(Z) - \1_{\bar{A}_J}(Z)\right| \right] \nonumber \\
& = \E\left[ |Y^I|^2 \1_{\{(-\bar{b}_0^J-\sum_{k\neq \nu_J}c^{Jk}Z^k)/c^{J\nu_J}  \leq Z^{\nu_J} <(-b_0^J-\sum_{k\neq \nu_J}c^{Jk}Z^k)/c^{J\nu_J} \}} \right] \nonumber \\
& = \int_{\R}\int_{\R^{m_1-1}} \int_{\frac{-\bar{b}_0^J-\sum_{k\neq \nu_J}c^{Jk}z^k}{c^{J\nu_J} }}^{\frac{-b_0^J-\sum_{k\neq \nu_J}c^{Jk}z^k}{c^{J\nu_J} }} f_{Z^{\nu_J}|Z_{-\nu_J}, Y^I}(z^{\nu_J}|z_{-\nu_J}, y^I)\, \rmd z^{\nu_J} \\
&\quad \times |y^I|^2f_{Z_{-\nu_J}, Y^I}(z_{-\nu_J}, y^I)\, \rmd  z_{-\nu_J}\, \rmd y^I.
\end{align*}
By using \eqref{fixedmmtbd}, the above result implies
\begin{align}\label{yi2inddiff}
&\E\left[ |Y^I|^2\left|\1_{A_J}(Z) - \1_{\bar{A}_J}(Z)\right| \right] \nonumber \\
& \leq \frac{C_{Z^{\nu_J}}}{c^{J\nu_J}}|\bar{b}_0^J - b_0^J| \int_{\R}\int_{\R^{m_1-1}}|y^I|^2f_{Z_{-\nu_J}, Y^I}(z_{-\nu_J}, y^I)\, \rmd z_{-\nu_J}\, \rmd y^I \nonumber \\
& =  \frac{C_{Z^{\nu_J}}}{c^{J\nu_J}}\E\left[|Y^I|^2\right]|\bar{b}_0^J - b_0^J| \nonumber \\
& \leq C_{Z, \max}\E\left[|Y^I|^2\right]|\theta - \bar{\theta}|.
\end{align}
\item For any $J = 1, \dots, d_1$, one obtains
\begin{align*}
&\E\left[ |Z|^2\left|\1_{A_J}(Z) - \1_{\bar{A}_J}(Z)\right| \right] \nonumber \\
& = \E\left[ (|Z^{\nu_J}|^2+|Z_{-\nu_J}|^2) \1_{\{(-\bar{b}_0^J-\sum_{k\neq \nu_J}c^{Jk}Z^k)/c^{J\nu_J}  \leq Z^{\nu_J} <(-b_0^J-\sum_{k\neq \nu_J}c^{Jk}Z^k)/c^{J\nu_J} \}} \right] \nonumber \\
& =  \int_{\R}\int_{\R^{m_1-1}} \int_{\frac{-\bar{b}_0^J-\sum_{k\neq \nu_J}c^{Jk}z^k}{c^{J\nu_J} }}^{\frac{-b_0^J-\sum_{k\neq \nu_J}c^{Jk}z^k}{c^{J\nu_J} }} |Z^{\nu_J}|^2f_{Z^{\nu_J}|Z_{-\nu_J}, Y^I}(z^{\nu_J}|z_{-\nu_J}, y^I)\, \rmd z^{\nu_J}\\
&\quad \times  f_{Z_{-\nu_J}, Y^I}(z_{-\nu_J}, y^I)\, \rmd z_{-\nu_J} \, \rmd y^I \nonumber \\
& \quad +\int_{\R} \int_{\R^{m_1-1}} \int_{\frac{-\bar{b}_0^J-\sum_{k\neq \nu_J}c^{Jk}z^k}{c^{J\nu_J} }}^{\frac{-b_0^J-\sum_{k\neq \nu_J}c^{Jk}z^k}{c^{J\nu_J} }}  f_{Z^{\nu_J}|Z_{-\nu_J}, Y^I}(z^{\nu_J}|z_{-\nu_J}, y^I)\, \rmd z^{\nu_J} \\
&\quad \times |Z_{-\nu_J}|^2 f_{Z_{-\nu_J}, Y^I}(z_{-\nu_J}, y^I)\, \rmd z_{-\nu_J} \, \rmd y^I.
\end{align*}
This implies, by using \eqref{fixedmmtbd},
\begin{align}\label{z2inddiff}
 \E\left[ |Z|^2\left|\1_{A_J}(Z) - \1_{\bar{A}_J}(Z)\right| \right]  & \leq \frac{\bar{C}_{Z^{\nu_J}}}{c^{J\nu_J}}|\bar{b}_0^J - b_0^J|+ \frac{C_{Z^{\nu_J}}}{c^{J\nu_J}}|\bar{b}_0^J - b_0^J|\E\left[|Z_{-\nu_J}|^2\right] \nonumber \\
&\leq \bar{C}_{Z, \max}|\bar{b}_0^J - b_0^J|+  C_{Z, \max}|\bar{b}_0^J - b_0^J|\E\left[|Z_{-\nu_J}|^2\right] \nonumber \\
&\leq \left(C_{Z, \max}+ \bar{C}_{Z, \max}\right)\E\left[ (1+|Z|)^2\right]|\theta - \bar{\theta}|.
\end{align}
\end{enumerate}
One further obtains the following estimates for $\mathfrak{N}$ given in \eqref{nnexp}:
\begin{enumerate}[leftmargin=*]
\item For any $I = 1, \dots, m_2$, $\theta \in \R^d, z \in \R^{m_1}$, it holds that
\begin{align}\label{fueexp}
\left|\mathfrak{N}^I(\theta, z)\right|
 = \left|\sum_{j = 1}^{d_1}W_1^{Ij}\sigma_1\left(\langle c^{j\cdot}, z \rangle+b_0^j\right)\right|
& = \left|\sum_{j = 1}^{d_1}W_1^{Ij} \left(\langle c^{j\cdot}, z \rangle+b_0^j\right)\1_{A_j}(z)\right|\nonumber \\
&\leq d_1|\theta|(c_F|z|+|\theta|)\nonumber \\
& \leq d_1 (1+c_F)(1+|z|)(1+|\theta|)^2.
\end{align}
\item  For any $I = 1, \dots, m_2$, $\theta, \bar{\theta} \in \R^d$, one obtains
\begin{align}\label{fdiffexp1}
&\E\left[\left|\mathfrak{N}^I(\theta, Z) - \mathfrak{N}^I(\bar{\theta}, Z)\right|\right]\nonumber \\
& =\E\left[\left|\sum_{j = 1}^{d_1}W_1^{Ij} \left(\langle c^{j\cdot}, Z \rangle+b_0^j\right)\1_{A_j}(Z)  - \sum_{j = 1}^{d_1}\bar{W}_1^{Ij} \left(\langle c^{j\cdot}, Z \rangle+\bar{b}_0^j\right)\1_{\bar{A}_j}(Z)\right|\right]\nonumber \\
&\leq T_1(\theta, \bar{\theta})+T_2(\theta, \bar{\theta})+T_3(\theta, \bar{\theta}),
\end{align}
where
\begin{align*}
T_1(\theta, \bar{\theta}) & := \E\left[\left|\sum_{j = 1}^{d_1}W_1^{Ij} \left(\langle c^{j\cdot}, Z \rangle+b_0^j\right)\1_{A_j}(Z) \right.\right.  \left.\left. - \sum_{j = 1}^{d_1}\bar{W}_1^{Ij} \left(\langle c^{j\cdot}, Z \rangle+b_0^j\right)\1_{A_j}(Z)\right|\right],\nonumber \\
T_2(\theta, \bar{\theta})&:=  \E\left[\left|\sum_{j = 1}^{d_1}\bar{W}_1^{Ij} \left(\langle c^{j\cdot}, Z \rangle+b_0^j\right)\1_{A_j}(Z) \right.\right. \left.\left. - \sum_{j = 1}^{d_1}\bar{W}_1^{Ij} \left(\langle c^{j\cdot}, Z \rangle+\bar{b}_0^j\right)\1_{A_j}(Z)\right|\right],\nonumber \\
T_3(\theta, \bar{\theta})&:= \E\left[\left|\sum_{j = 1}^{d_1}\bar{W}_1^{Ij} \left(\langle c^{j\cdot}, Z \rangle+\bar{b}_0^j\right)\1_{A_j}(Z) \right.\right. \left.\left. - \sum_{j = 1}^{d_1}\bar{W}_1^{Ij} \left(\langle c^{j\cdot}, Z \rangle+\bar{b}_0^j\right)\1_{\bar{A}_j}(Z)\right|\right].
\end{align*}
Further calculations yield
\begin{align}\label{fdiffexp2}
T_1(\theta, \bar{\theta}) & = \E\left[\left|\sum_{j = 1}^{d_1}W_1^{Ij} \left(\langle c^{j\cdot}, Z \rangle+b_0^j\right)\1_{A_j}(Z) \right.\right.  \left.\left. - \sum_{j = 1}^{d_1}\bar{W}_1^{Ij} \left(\langle c^{j\cdot}, Z \rangle+b_0^j\right)\1_{A_j}(Z)\right|\right]\nonumber \\
&\leq \sum_{j = 1}^{d_1}\E\left[\left|\langle c^{j\cdot}, Z \rangle+b_0^j\right|\right]\left|W_1^{Ij}-\bar{W}_1^{Ij} \right|\nonumber \\
&\leq d_1(1+c_F)\E[(1+|Z|)](1+|\theta|+|\bar{\theta}|)|\theta - \bar{\theta}|.
\end{align}
Moreover, it follows that
\begin{align}\label{fdiffexp3}
T_2(\theta, \bar{\theta})&=  \E\left[\left|\sum_{j = 1}^{d_1}\bar{W}_1^{Ij} \left(\langle c^{j\cdot}, Z \rangle+b_0^j\right)\1_{A_j}(Z) \right.\right.  \left.\left. - \sum_{j = 1}^{d_1}\bar{W}_1^{Ij} \left(\langle c^{j\cdot}, Z \rangle+\bar{b}_0^j\right)\1_{A_j}(Z)\right|\right]\nonumber \\
&\leq  \sum_{j = 1}^{d_1}\left| \bar{W}_1^{Ij} \right|\left|b_0^j - \bar{b}_0^j \right|\nonumber \\
&\leq d_1(1+|\theta|+|\bar{\theta}|)|\theta - \bar{\theta}|.
\end{align}
In addition, one obtains
\begin{align*}
T_3(\theta, \bar{\theta})&= \E\left[\left|\sum_{j = 1}^{d_1}\bar{W}_1^{Ij} \left(\langle c^{j\cdot}, Z \rangle+\bar{b}_0^j\right)\1_{A_j}(Z) \right.\right.  \left.\left. - \sum_{j = 1}^{d_1}\bar{W}_1^{Ij} \left(\langle c^{j\cdot}, Z \rangle+\bar{b}_0^j\right)\1_{\bar{A}_j}(Z)\right|\right]\nonumber\\
&\leq \sum_{j = 1}^{d_1}(1+c_F)\left(1+\left|\bar{b}_0^j\right|\right)\left|\bar{W}_1^{Ij} \right|\E\left[\left(1+|Z|\right)\left|\1_{A_j}(Z) -\1_{\bar{A}_j}(Z) \right|\right]\nonumber \\
&\leq 2\sum_{j = 1}^{d_1}(1+c_F)(1+|\theta|+|\bar{\theta}|)^2\E\left[\left(1+|Z|^2\right)\left|\1_{A_j}(Z) -\1_{\bar{A}_j}(Z) \right|\right].
\end{align*}
The above inequality implies, by using \eqref{inddiff}, \eqref{z2inddiff}, that
\begin{align}\label{fdiffexp4}
T_3(\theta, \bar{\theta})
&\leq 2\sum_{j = 1}^{d_1}(1+c_F)(1+|\theta|+|\bar{\theta}|)^2C_{Z, \max}|\theta - \bar{\theta}|\nonumber \\
&\quad + 2\sum_{j = 1}^{d_1}(1+c_F)(1+|\theta|+|\bar{\theta}|)^2\left(C_{Z, \max}+ \bar{C}_{Z, \max}\right)\E\left[ (1+|Z|)^2\right]|\theta - \bar{\theta}|\nonumber \\
&\leq 4d_1(1+c_F)\left(C_{Z, \max}+ \bar{C}_{Z, \max}\right)\E\left[ (1+|Z|)^2\right](1+|\theta|+|\bar{\theta}|)^2|\theta - \bar{\theta}|.
\end{align}
Substituting \eqref{fdiffexp2}, \eqref{fdiffexp3}, \eqref{fdiffexp4} into \eqref{fdiffexp1} yields
\begin{align}\label{fdiffexp5}
\begin{split}
&\E\left[\left|\mathfrak{N}^I(\theta, Z) - \mathfrak{N}^I(\bar{\theta}, Z)\right|\right]\\
& \leq 6d_1(1+c_F)\left(1+C_{Z, \max}+ \bar{C}_{Z, \max}\right) \E\left[ (1+|Z|)^2\right](1+|\theta|+|\bar{\theta}|)^2|\theta - \bar{\theta}|.
\end{split}
\end{align}
\item For any $I = 1, \dots, m_2$, $\theta, \bar{\theta} \in \R^d$, one obtains
\begin{align}\label{fdiffexp6}
&\E\left[\left(1+|Z|\right)\left|\mathfrak{N}^I(\theta, Z) - \mathfrak{N}^I(\bar{\theta}, Z)\right|\right]\\
& =\E\left[\left(1+|Z|\right)\left|\sum_{j = 1}^{d_1}W_1^{Ij} \left(\langle c^{j\cdot}, Z \rangle+b_0^j\right)\1_{A_j}(Z) \right.\right.  \left.\left. - \sum_{j = 1}^{d_1}\bar{W}_1^{Ij} \left(\langle c^{j\cdot}, Z \rangle+\bar{b}_0^j\right)\1_{\bar{A}_j}(Z)\right|\right]\nonumber \\
&\leq T_4(\theta, \bar{\theta})+T_5(\theta, \bar{\theta})+T_6(\theta, \bar{\theta}),
\end{align}
where
\begin{align*}
T_4(\theta, \bar{\theta}) & := \E\left[\left(1+|Z|\right)\left|\sum_{j = 1}^{d_1}W_1^{Ij} \left(\langle c^{j\cdot}, Z \rangle+b_0^j\right)\1_{A_j}(Z) \right.\right.\nonumber \\
&\qquad \left.\left. - \sum_{j = 1}^{d_1}\bar{W}_1^{Ij} \left(\langle c^{j\cdot}, Z \rangle+b_0^j\right)\1_{A_j}(Z)\right|\right],\nonumber \\
T_5(\theta, \bar{\theta})&:=  \E\left[\left(1+|Z|\right)\left|\sum_{j = 1}^{d_1}\bar{W}_1^{Ij} \left(\langle c^{j\cdot}, Z \rangle+b_0^j\right)\1_{A_j}(Z) \right.\right.\nonumber \\
&\qquad \left.\left. - \sum_{j = 1}^{d_1}\bar{W}_1^{Ij} \left(\langle c^{j\cdot}, Z \rangle+\bar{b}_0^j\right)\1_{A_j}(Z)\right|\right],\nonumber \\
T_6(\theta, \bar{\theta})&:= \E\left[\left(1+|Z|\right)\left|\sum_{j = 1}^{d_1}\bar{W}_1^{Ij} \left(\langle c^{j\cdot}, Z \rangle+\bar{b}_0^j\right)\1_{A_j}(Z) \right.\right.\nonumber \\
&\qquad \left.\left. - \sum_{j = 1}^{d_1}\bar{W}_1^{Ij} \left(\langle c^{j\cdot}, Z \rangle+\bar{b}_0^j\right)\1_{\bar{A}_j}(Z)\right|\right].
\end{align*}
By using similar arguments as in \eqref{fdiffexp2}, \eqref{fdiffexp3}, straightforward calculations show that
\begin{align}\label{fdiffexp7}
\begin{split}
T_4(\theta, \bar{\theta})  &\leq d_1(1+c_F)\E[(1+|Z|)^2](1+|\theta|+|\bar{\theta}|)|\theta - \bar{\theta}|,\\
T_5(\theta, \bar{\theta})  &\leq d_1\E[(1+|Z|)](1+|\theta|+|\bar{\theta}|)|\theta - \bar{\theta}|.
\end{split}
\end{align}
Furthermore, one obtains
\begin{align*}
T_6(\theta, \bar{\theta})&= \E\left[\left(1+|Z|\right)\left|\sum_{j = 1}^{d_1}\bar{W}_1^{Ij} \left(\langle c^{j\cdot}, Z \rangle+\bar{b}_0^j\right)\1_{A_j}(Z) \right.\right.\nonumber \\
&\qquad \left.\left. - \sum_{j = 1}^{d_1}\bar{W}_1^{Ij} \left(\langle c^{j\cdot}, Z \rangle+\bar{b}_0^j\right)\1_{\bar{A}_j}(Z)\right|\right]\nonumber\\
&\leq \sum_{j = 1}^{d_1}(1+c_F)\left(1+\left|\bar{b}_0^j\right|\right)\left|\bar{W}_1^{Ij} \right|\E\left[\left(1+|Z|\right)^2\left|\1_{A_j}(Z) -\1_{\bar{A}_j}(Z) \right|\right]\nonumber \\
&\leq 2\sum_{j = 1}^{d_1}(1+c_F)(1+|\theta|+|\bar{\theta}|)^2\E\left[\left(1+|Z|^2\right)\left|\1_{A_j}(Z) -\1_{\bar{A}_j}(Z) \right|\right],
\end{align*}
which, following the arguments in \eqref{fdiffexp4}, implies
\begin{align}\label{fdiffexp8}
T_6(\theta, \bar{\theta})
&\leq 4d_1(1+c_F)\left(C_{Z, \max}+ \bar{C}_{Z, \max}\right)\E\left[ (1+|Z|)^2\right](1+|\theta|+|\bar{\theta}|)^2|\theta - \bar{\theta}|.
\end{align}
Substituting \eqref{fdiffexp7}, \eqref{fdiffexp8} into \eqref{fdiffexp6} yields
\begin{align}\label{fdiffexp9}
\begin{split}
&\E\left[\left(1+|Z|\right)\left|\mathfrak{N}^I(\theta, Z) - \mathfrak{N}^I(\bar{\theta}, Z)\right|\right]\\
& \leq 6d_1(1+c_F)\left(1+C_{Z, \max}+ \bar{C}_{Z, \max}\right) \E\left[ (1+|Z|)^2\right](1+|\theta|+|\bar{\theta}|)^2|\theta - \bar{\theta}|.
\end{split}
\end{align}
\end{enumerate}
For any $I = 1, \dots, m_2, J = 1, \dots, d_1$, 
one obtains the following estimates for $G_{W_1^{IJ}}, G_{b_0^J}$ defined in \eqref{fixedfgexp2}:
\begin{enumerate}[leftmargin=*]
\item For any $\theta, \bar{\theta} \in \R^d$, by $\mathcal{L}(X) = \mathcal{L}(X_0) $, it holds that
\begin{align}\label{gw1diffexp1}
 \E\left[\left|G_{W_1^{IJ}}(\theta,X_0) -G_{W_1^{IJ}}(\bar{\theta},X_0)  \right|\right]
&=\E\left[\left|G_{W_1^{IJ}}(\theta,X ) -G_{W_1^{IJ}}(\bar{\theta},X )  \right|\right]\nonumber\\
& = \E\left[\left|-2(Y^I - \mathfrak{N}^I(\theta, Z)) \left(\langle c^{J\cdot}, Z \rangle+b_0^J\right) \1_{A_J}(Z)  \right.\right. \nonumber\\
&\qquad \left.\left.+2(Y^I - \mathfrak{N}^I(\bar{\theta}, Z)) \left(\langle c^{J\cdot}, Z \rangle+\bar{b}_0^J\right)\1_{\bar{A}_J}(Z)  \right|\right]\nonumber\\
&\leq T_7(\theta, \bar{\theta})+T_8(\theta, \bar{\theta})+T_9(\theta, \bar{\theta}),
\end{align}
where
\begin{align*}
T_7(\theta, \bar{\theta})&:=  \E\left[\left|-2(Y^I - \mathfrak{N}^I(\theta, Z)) \left(\langle c^{J\cdot}, Z \rangle+b_0^J\right) \1_{A_J}(Z)  \right.\right. \nonumber\\
&\qquad \left.\left.+2(Y^I - \mathfrak{N}^I(\bar{\theta}, Z)) \left(\langle c^{J\cdot}, Z \rangle+b_0^J\right)\1_{A_J}(Z)  \right|\right],\nonumber\\
T_8(\theta, \bar{\theta})&:= \E\left[\left|-2(Y^I - \mathfrak{N}^I(\bar{\theta}, Z)) \left(\langle c^{J\cdot}, Z \rangle+b_0^J\right) \1_{A_J}(Z)  \right.\right. \nonumber\\
&\qquad \left.\left.+2(Y^I - \mathfrak{N}^I(\bar{\theta}, Z)) \left(\langle c^{J\cdot}, Z \rangle+\bar{b}_0^J\right)\1_{A_J}(Z)  \right|\right],\nonumber\\
T_9(\theta, \bar{\theta})&:= \E\left[\left|-2(Y^I - \mathfrak{N}^I(\bar{\theta}, Z)) \left(\langle c^{J\cdot}, Z \rangle+\bar{b}_0^J\right) \1_{A_J}(Z)  \right.\right. \nonumber\\
&\qquad \left.\left.+2(Y^I - \mathfrak{N}^I(\bar{\theta}, Z)) \left(\langle c^{J\cdot}, Z \rangle+\bar{b}_0^J\right)\1_{\bar{A}_J}(Z)  \right|\right].
\end{align*}
By using \eqref{fdiffexp9}, one obtains
\begin{align}\label{gw1diffexp2}
&T_7(\theta, \bar{\theta})   =   \E\left[\left|-2(Y^I - \mathfrak{N}^I(\theta, Z)) \left(\langle c^{J\cdot}, Z \rangle+b_0^J\right) \1_{A_J}(Z)  \right.\right. \nonumber\\
&\hspace{5em} \left.\left.+2(Y^I - \mathfrak{N}^I(\bar{\theta}, Z)) \left(\langle c^{J\cdot}, Z \rangle+b_0^J\right)\1_{A_J}(Z)  \right|\right]\nonumber\\
&\leq 2 (1+c_F)\left(1+\left|b_0^J\right|\right)\E\left[(1+|Z|)\left|\mathfrak{N}^I(\theta, Z) - \mathfrak{N}^I(\bar{\theta}, Z)\right|\right]\nonumber\\
\begin{split}
&\leq 12d_1(1+c_F)^2\left(1+C_{Z, \max}+ \bar{C}_{Z, \max}\right) 
\E\left[ (1+|Z|)^2\right](1+|\theta|+|\bar{\theta}|)^3|\theta - \bar{\theta}|.
\end{split}
\end{align}
Moreover, by \eqref{fueexp}, and the fact that $X = (Y,Z)$, it follows that
\begin{align}\label{gw1diffexp3}
T_8(\theta, \bar{\theta})&= \E\left[\left|-2(Y^I - \mathfrak{N}^I(\bar{\theta}, Z)) \left(\langle c^{J\cdot}, Z \rangle+b_0^J\right) \1_{A_J}(Z)  \right.\right. \nonumber\\
&\qquad \left.\left.+2(Y^I - \mathfrak{N}^I(\bar{\theta}, Z)) \left(\langle c^{J\cdot}, Z \rangle+\bar{b}_0^J\right)\1_{A_J}(Z)  \right|\right],\nonumber\\
&\leq 2 \E\left[\left(\left|Y^I \right|+\left| \mathfrak{N}^I(\bar{\theta}, Z)\right|\right)\right]\left|\bar{b}_0^J-b_0^J\right|\nonumber\\
&\leq 2 \E\left[\left(\left|X \right|+d_1 (1+c_F)(1+|X|)(1+|\bar{\theta}|)^2\right)\right]\left|\theta-\bar{\theta}\right|\nonumber\\
&\leq 4d_1 (1+c_F)\E\left[ (1+|X |)\right](1+|\theta|+|\bar{\theta}|)^2|\theta - \bar{\theta}|.
\end{align}
Furthermore, by using \eqref{fueexp}, and by applying Young's inequality that $2ab \leq a^2+b^2$ for $a, b \geq 0$, one obtains,
\begin{align*}
&T_9(\theta, \bar{\theta})= \E\left[\left|-2(Y^I - \mathfrak{N}^I(\bar{\theta}, Z)) \left(\langle c^{J\cdot}, Z \rangle+\bar{b}_0^J\right) \1_{A_J}(Z)  \right.\right. \nonumber\\
&\hspace{5em} \left.\left.+2(Y^I - \mathfrak{N}^I(\bar{\theta}, Z)) \left(\langle c^{J\cdot}, Z \rangle+\bar{b}_0^J\right)\1_{\bar{A}_J}(Z)  \right|\right]\nonumber\\
&\leq  2(1+c_F)\left(1+\left|\bar{b}_0^J\right|\right)\E\left[(1+|Z|)\left(\left|Y^I \right|+\left| \mathfrak{N}^I(\bar{\theta}, Z)\right|\right)\left|\1_{\bar{A}_J}(Z) -\1_{A_J}(Z)  \right|\right]\nonumber\\
&\leq (1+c_F)\left(1+\left|\bar{b}_0^J\right|\right)\E\left[(1+|Z|)^2 \left|\1_{\bar{A}_J}(Z) -\1_{A_J}(Z)  \right|\right]\nonumber\\
&\quad + (1+c_F)\left(1+\left|\bar{b}_0^J\right|\right)\E\left[ \left|Y^I \right|^2\left|\1_{\bar{A}_J}(Z) -\1_{A_J}(Z)  \right|\right]\nonumber\\
&\quad +2(1+c_F)\left(1+\left|\bar{b}_0^J\right|\right)\E\left[(1+|Z|) \left| \mathfrak{N}^I(\bar{\theta}, Z)\right| \left|\1_{\bar{A}_J}(Z) -\1_{A_J}(Z)  \right|\right]\nonumber\\
&\leq 2 (1+c_F)\left(1+\left|\bar{b}_0^J\right|\right)\E\left[ \left|\1_{\bar{A}_J}(Z) -\1_{A_J}(Z)  \right|\right]\nonumber\\
&\quad + 2(1+c_F)\left(1+\left|\bar{b}_0^J\right|\right)\E\left[|Z|^2 \left|\1_{\bar{A}_J}(Z) -\1_{A_J}(Z)  \right|\right]\nonumber\\
&\quad +(1+c_F)\left(1+\left|\bar{b}_0^J\right|\right)\E\left[ \left|Y^I \right|^2\left|\1_{\bar{A}_J}(Z) -\1_{A_J}(Z)  \right|\right]\nonumber\\
&\quad +2d_1(1+c_F)^2\left(1+\left|\bar{b}_0^J\right|\right)(1+|\theta|+|\bar{\theta}|)^2
\E\left[ (1+|Z|)^2 \left|\1_{\bar{A}_J}(Z) -\1_{A_J}(Z)  \right|\right].
\end{align*}
Applying \eqref{inddiff}, \eqref{yi2inddiff}, \eqref{z2inddiff}, and the fact that $X = (Y,Z)$, yields
\begin{align}\label{gw1diffexp4}
T_9(\theta, \bar{\theta})
& \leq  2 (1+c_F)(1+|\bar{\theta}|)C_{Z, \max}|\theta - \bar{\theta}|\nonumber\\
&\quad +2 (1+c_F)(1+|\bar{\theta}|)\left(C_{Z, \max}+ \bar{C}_{Z, \max}\right)\E\left[ (1+|Z|)^2\right]|\theta - \bar{\theta}|\nonumber\\
&\quad + (1+c_F)(1+|\bar{\theta}|)C_{Z, \max}\E\left[|Y^I|^2\right]|\theta - \bar{\theta}|\nonumber\\
&\quad +4d_1(1+c_F)^2(1+|\theta|+|\bar{\theta}|)^3\E\left[  \left|\1_{\bar{A}_J}(Z) -\1_{A_J}(Z)  \right|\right]\nonumber\\
&\quad +4d_1(1+c_F)^2(1+|\theta|+|\bar{\theta}|)^3\E\left[ |Z|^2 \left|\1_{\bar{A}_J}(Z) -\1_{A_J}(Z)  \right|\right]\nonumber\\
&\hspace{-3em}\leq 5(1+c_F)(1+|\theta|+|\bar{\theta}|)\left(1+C_{Z, \max}+ \bar{C}_{Z, \max}\right)\E\left[ (1+|X|)^2\right]|\theta - \bar{\theta}|\nonumber\\
&\quad +4d_1(1+c_F)^2(1+|\theta|+|\bar{\theta}|)^3C_{Z, \max}|\theta - \bar{\theta}|\nonumber\\
&\quad +4d_1(1+c_F)^2(1+|\theta|+|\bar{\theta}|)^3\left(C_{Z, \max}+ \bar{C}_{Z, \max}\right)\E\left[ (1+|Z|)^2\right]|\theta - \bar{\theta}|\nonumber\\
\begin{split}
&\hspace{-3em}\leq 13d_1(1+c_F)^2(1+|\theta|+|\bar{\theta}|)^3\left(1+C_{Z, \max}+ \bar{C}_{Z, \max}\right)
\E\left[ (1+|X|)^2\right]|\theta - \bar{\theta}|.
\end{split}
\end{align}
Substituting \eqref{gw1diffexp2}, \eqref{gw1diffexp3}, \eqref{gw1diffexp4} into \eqref{gw1diffexp1} and using $\mathcal{L}(X) = \mathcal{L}(X_0) $ yield, for any $I = 1, \dots, m_2, J = 1, \dots, d_1$, $\theta, \bar{\theta} \in \R^d$,
\begin{align}\label{gw1diffexp5}
&\E\left[\left|G_{W_1^{IJ}}(\theta,X_0) -G_{W_1^{IJ}}(\bar{\theta},X_0)  \right|\right]\nonumber\\
\begin{split}
&\leq 29d_1(1+c_F)^2(1+|\theta|+|\bar{\theta}|)^3 \left(1+C_{Z, \max}+ \bar{C}_{Z, \max}\right)  \E\left[ (1+|X_0|)^2\right]|\theta - \bar{\theta}|.
\end{split}
\end{align}
\item For any $\theta, \bar{\theta} \in \R^d$, by $\mathcal{L}(X) = \mathcal{L}(X_0) $, it follows that
\begin{align}\label{gw1diffexp6}
 &\E\left[\left|G_{b_0^J}(\theta,X_0)  -G_{b_0^J}(\bar{\theta},X_0)  \right|\right]\nonumber\\
 &=\E\left[\left|G_{b_0^J}(\theta,X)  -G_{b_0^J}(\bar{\theta},X)  \right|\right]\nonumber\\
& = \E\left[\left|-2\sum_{i = 1}^{m_2}(Y^i - \mathfrak{N}^i(\theta, Z)) W_1^{iJ}\1_{A_J} (Z)   +2\sum_{i = 1}^{m_2}(Y^i - \mathfrak{N}^i(\bar{\theta}, Z)) \bar{W}_1^{iJ}\1_{\bar{A}_J} (Z) \right|\right]\nonumber\\
&\leq T_{10}(\theta, \bar{\theta})+T_{11}(\theta, \bar{\theta})+T_{12}(\theta, \bar{\theta}),
\end{align}
where
\begin{align*}
T_{10}(\theta, \bar{\theta})&:= \E\left[\left|-2\sum_{i = 1}^{m_2}(Y^i - \mathfrak{N}^i(\theta, Z)) W_1^{iJ}\1_{A_J} (Z)  \right.\right. \nonumber\\
&\qquad \left.\left.+2\sum_{i = 1}^{m_2}(Y^i - \mathfrak{N}^i(\bar{\theta}, Z)) W_1^{iJ}\1_{A_J} (Z) \right|\right],\nonumber\\
T_{11}(\theta, \bar{\theta})&:=  \E\left[\left|-2\sum_{i = 1}^{m_2}(Y^i - \mathfrak{N}^i(\bar{\theta}, Z)) W_1^{iJ}\1_{A_J} (Z)  \right.\right. \nonumber\\
&\qquad \left.\left.+2\sum_{i = 1}^{m_2}(Y^i - \mathfrak{N}^i(\bar{\theta}, Z)) \bar{W}_1^{iJ}\1_{A_J} (Z) \right|\right],\nonumber\\
T_{12}(\theta, \bar{\theta})&:=\E\left[\left|-2\sum_{i = 1}^{m_2}(Y^i - \mathfrak{N}^i(\bar{\theta}, Z)) \bar{W}_1^{iJ}\1_{A_J} (Z)  \right.\right. \nonumber\\
&\qquad \left.\left.+2\sum_{i = 1}^{m_2}(Y^i - \mathfrak{N}^i(\bar{\theta}, Z)) \bar{W}_1^{iJ}\1_{\bar{A}_J} (Z) \right|\right].\nonumber
\end{align*}
By using \eqref{fdiffexp5}, one obtains
\begin{align}\label{gw1diffexp7}
&T_{10}(\theta, \bar{\theta})
= \E\left[\left|-2\sum_{i = 1}^{m_2}(Y^i - \mathfrak{N}^i(\theta, Z)) W_1^{iJ}\1_{A_J} (Z)  \right.\right. \nonumber\\
&\hspace{5em} \left.\left.+2\sum_{i = 1}^{m_2}(Y^i - \mathfrak{N}^i(\bar{\theta}, Z)) W_1^{iJ}\1_{A_J} (Z) \right|\right]\nonumber\\
&\leq  2\sum_{i = 1}^{m_2}\left| W_1^{iJ}\right| \E\left[\left| \mathfrak{N}^i(\theta, Z) -  \mathfrak{N}^i(\bar{\theta}, Z)\right|\right]\nonumber\\
\begin{split}
&\leq 12m_2 d_1(1+c_F)\left(1+C_{Z, \max}+ \bar{C}_{Z, \max}\right)
\E\left[ (1+|Z|)^2\right](1+|\theta|+|\bar{\theta}|)^3|\theta - \bar{\theta}|.
\end{split}
\end{align}
Moreover, by using \eqref{fueexp}, and that $X = (Y,Z)$, we have
\begin{align}\label{gw1diffexp8}
T_{11}(\theta, \bar{\theta})
&= \E\left[\left|-2\sum_{i = 1}^{m_2}(Y^i - \mathfrak{N}^i(\bar{\theta}, Z)) W_1^{iJ}\1_{A_J} (Z)  \right.\right. \nonumber\\
&\qquad \left.\left.+2\sum_{i = 1}^{m_2}(Y^i - \mathfrak{N}^i(\bar{\theta}, Z)) \bar{W}_1^{iJ}\1_{A_J} (Z) \right|\right],\nonumber\\
&\leq  2\sum_{i = 1}^{m_2} \E\left[\left(\left|Y^i \right|+\left| \mathfrak{N}^i(\bar{\theta}, Z)\right|\right)\right]\left| \bar{W}_1^{iJ} - W_1^{iJ}\right|\nonumber\\
&\leq 2\sum_{i = 1}^{m_2}\E\left[\left(\left|X \right|+d_1 (1+c_F)(1+|X|)(1+|\bar{\theta}|)^2\right)\right]\left|\theta-\bar{\theta}\right|\nonumber\\
&\leq 4m_2d_1 (1+c_F)\E\left[(1+|X|)\right](1+|\theta|+|\bar{\theta}|)^2|\theta - \bar{\theta}|.
\end{align}
In addition, one obtains, by using \eqref{fueexp},
\begin{align*}
T_{12}(\theta, \bar{\theta})
&= \E\left[\left|-2\sum_{i = 1}^{m_2}(Y^i - \mathfrak{N}^i(\bar{\theta}, Z)) \bar{W}_1^{iJ}\1_{A_J} (Z)  \right.\right. \nonumber\\
&\qquad \left.\left.+2\sum_{i = 1}^{m_2}(Y^i - \mathfrak{N}^i(\bar{\theta}, Z)) \bar{W}_1^{iJ}\1_{\bar{A}_J} (Z) \right|\right]\nonumber\\
&\leq  2\sum_{i = 1}^{m_2}\left| \bar{W}_1^{iJ} \right| \E\left[\left(\left|Y^i \right|+\left| \mathfrak{N}^i(\bar{\theta}, Z)\right|\right)\left|1_{\bar{A}_J} (Z) - 1_{A_J} (Z) \right|\right]\nonumber\\
&\leq 2\sum_{i = 1}^{m_2}\left| \bar{W}_1^{iJ} \right| \E\left[\left(1+\left|Y^i \right|^2\right)\left|1_{\bar{A}_J} (Z) - 1_{A_J} (Z) \right|\right]\nonumber\\
&\quad +2\sum_{i = 1}^{m_2}\left| \bar{W}_1^{iJ} \right| \E\left[\left| \mathfrak{N}^i(\bar{\theta}, Z)\right|\left|1_{\bar{A}_J} (Z) - 1_{A_J} (Z) \right|\right]\nonumber\\
&\leq  2\sum_{i = 1}^{m_2}\left| \bar{W}_1^{iJ} \right| \E\left[ \left|1_{\bar{A}_J} (Z) - 1_{A_J} (Z) \right|\right]\nonumber\\
&\quad + 2\sum_{i = 1}^{m_2}\left| \bar{W}_1^{iJ} \right| \E\left[ \left|Y^i \right|^2 \left|1_{\bar{A}_J} (Z) - 1_{A_J} (Z) \right|\right]\nonumber\\
&\quad +2d_1 (1+c_F)(1+|\bar{\theta}|)^2\sum_{i = 1}^{m_2}\left| \bar{W}_1^{iJ} \right| \E\left[(1+|Z|)\left|1_{\bar{A}_J} (Z) - 1_{A_J} (Z) \right|\right]\nonumber\\
&\leq  2\sum_{i = 1}^{m_2}\left| \bar{W}_1^{iJ} \right| \E\left[ \left|1_{\bar{A}_J} (Z) - 1_{A_J} (Z) \right|\right]\nonumber\\
&\quad + 2\sum_{i = 1}^{m_2}\left| \bar{W}_1^{iJ} \right| \E\left[ \left|Y^i \right|^2 \left|1_{\bar{A}_J} (Z) - 1_{A_J} (Z) \right|\right]\nonumber\\
&\quad +4d_1 (1+c_F)(1+|\bar{\theta}|)^2\sum_{i = 1}^{m_2}\left| \bar{W}_1^{iJ} \right| \E\left[ \left|1_{\bar{A}_J} (Z) - 1_{A_J} (Z) \right|\right]\nonumber\\
&\quad +2d_1 (1+c_F)(1+|\bar{\theta}|)^2\sum_{i = 1}^{m_2}\left| \bar{W}_1^{iJ} \right| \E\left[|Z|^2\left|1_{\bar{A}_J} (Z) - 1_{A_J} (Z) \right|\right].
\end{align*}
This yields, by applying \eqref{inddiff}, \eqref{yi2inddiff}, \eqref{z2inddiff},
\begin{align}\label{gw1diffexp9}
T_{12}(\theta, \bar{\theta})
&\leq 2m_2(1+|\theta|+|\bar{\theta}|)C_{Z, \max}|\theta - \bar{\theta}|\nonumber\\
&\quad + 2m_2(1+|\theta|+|\bar{\theta}|)C_{Z, \max}\E\left[|Y^I|^2\right]|\theta - \bar{\theta}|\nonumber\\
&\quad +4m_2d_1 (1+c_F)(1+|\theta|+|\bar{\theta}|)^3C_{Z, \max}|\theta - \bar{\theta}|\nonumber\\
&\quad +2m_2d_1 (1+c_F)(1+|\theta|+|\bar{\theta}|)^3\left(C_{Z, \max}+ \bar{C}_{Z, \max}\right)\E\left[ (1+|Z|)^2\right]|\theta - \bar{\theta}|\nonumber\\
\begin{split}
&\hspace{-3em}\leq 10m_2d_1 (1+c_F)(1+|\theta|+|\bar{\theta}|)^3
\left(1+C_{Z, \max}+ \bar{C}_{Z, \max}\right)\E\left[ (1+|X|)^2\right]|\theta - \bar{\theta}| .
\end{split}
\end{align}
Substituting \eqref{gw1diffexp7}, \eqref{gw1diffexp8}, \eqref{gw1diffexp9} into \eqref{gw1diffexp6} and using $\mathcal{L}(X) = \mathcal{L}(X_0) $ yield, for any $J= 1, \dots, d_1$,
\begin{align}\label{gw1diffexp10}
&\E\left[\left|G_{b_0^J}(\theta,X_0)  -G_{b_0^J}(\bar{\theta},X_0)  \right|\right]\nonumber\\
\begin{split}
&\leq 26m_2d_1 (1+c_F)(1+|\theta|+|\bar{\theta}|)^3 \left(1+C_{Z, \max}+ \bar{C}_{Z, \max}\right)\E\left[ (1+|X_0|)^2\right]|\theta - \bar{\theta}| .
\end{split}
\end{align}
\end{enumerate}

By using \eqref{gw1diffexp5}, \eqref{gw1diffexp10}, one obtains, for any $\theta, \bar{\theta} \in \R^d$, that
\begin{align}\label{fixedsga2lg}
&\E\left[\left|G(\theta,X_0)  - G(\bar{\theta},X_0) \right|\right] \nonumber\\
& = \E\left[\left(\sum_{i=1}^{m_2}\sum_{j=1}^{d_1}\left|G_{W_1^{ij}}(\theta,X_0) -G_{W_1^{ij}}(\bar{\theta},X_0)  \right|^2 \right.\right.  \left.\left.+\sum_{j=1}^{d_1}\left| G_{b_0^J}(\theta,X_0)  -G_{b_0^J}(\bar{\theta},X_0)  \right|^2\right)^{1/2}\right] \nonumber\\
&\leq 29m_2d_1^2(1+c_F)^2(1+|\theta|+|\bar{\theta}|)^3 \left(1+C_{Z, \max}+ \bar{C}_{Z, \max}\right)  \E\left[ (1+|X_0|)^2\right]|\theta - \bar{\theta}|\nonumber\\
&\quad + 26m_2d_1^2 (1+c_F)(1+|\theta|+|\bar{\theta}|)^3  \left(1+C_{Z, \max}+ \bar{C}_{Z, \max}\right)\E\left[ (1+|X_0|)^2\right]|\theta - \bar{\theta}|\nonumber\\
\begin{split}
&\leq 55m_2d_1^2(1+c_F)^2\left(1+C_{Z, \max}+ \bar{C}_{Z, \max}\right)   \E\left[ (1+|X_0|)^2\right](1+|\theta|+|\bar{\theta}|)^3|\theta - \bar{\theta}|.
\end{split}
\end{align}
Moreover, for any $I = 1, \dots, m_2, J= 1, \dots, d_1$, the following upper bounds can be obtained by using the definition of $G_{W_1^{IJ}}, G_{b_0^J}$ given in \eqref{fixedfgexp2}:
\begin{enumerate}[leftmargin=*]
\item For any $\theta \in \R^d, x \in \R^m$, one obtains, by using \eqref{fueexp},
\begin{align}\label{fixedsga2ue1}
\left|G_{W_1^{IJ}}(\theta, x)\right|
&= \left| -2(y^I - \mathfrak{N}^I(\theta, z))\sigma_1\left(\langle c^{J\cdot}, z \rangle+b_0^J\right)\right|\nonumber\\
&\leq 2\left(\left|y^I \right|+\left|\mathfrak{N}^I(\theta, z))\right|\right)(1+c_F)(1+|z|)\left(1+\left|b_0^J\right|\right)\nonumber\\
&\leq  2\left(\left|y^I \right|+d_1 (1+c_F)(1+|x|)(1+|\theta|)^2\right)(1+c_F)(1+|z|)(1+|\theta|)\nonumber\\
&\leq 4d_1 (1+c_F)^2(1+|x|)^2(1+|\theta|)^3.
\end{align}
\item Similarly, for any $\theta \in \R^d, x \in \R^m$ with $x = (y, z) \in \R^{m_2} \times \R^{m_1}$, one obtains, by using \eqref{fueexp},
\begin{align}\label{fixedsga2ue2}
\left|G_{b_0^J}(\theta, x)\right|
&= \left|-2\sum_{i = 1}^{m_2}(y^i - \mathfrak{N}^i(\theta, z)) W_1^{iJ}\1_{A_J} (z)\right|\nonumber\\
&\leq 2\sum_{i = 1}^{m_2}\left(\left|y^i \right|+\left|\mathfrak{N}^i(\theta, z))\right|\right) \left(1+\left|W_1^{iJ}\right|\right)\nonumber\\
&\leq  2\sum_{i = 1}^{m_2}\left(\left|y^i \right|+d_1 (1+c_F)(1+|x|)(1+|\theta|)^2\right)(1+|\theta|)\nonumber\\
&\leq 4m_2d_1 (1+c_F) (1+|x|) (1+|\theta|)^3.
\end{align}
\end{enumerate}
The above results \eqref{fixedsga2ue1}, \eqref{fixedsga2ue2} imply that, for any $\theta \in \R^d, x \in \R^m$,
\begin{align}\label{fixedsga2kg}
\left|G(\theta, x)\right| & = \left(\sum_{i=1}^{m_2}\sum_{j=1}^{d_1}\left|G_{W_1^{ij}}(\theta, x)\right|^2+\sum_{j=1}^{d_1}\left|G_{b_0^J}(\theta, x)\right|^2\right)^{1/2}\nonumber\\
&\leq 4m_2d_1^2 (1+c_F)^2(1+|x|)^2(1+|\theta|)^3  +4m_2d_1^2 (1+c_F) (1+|x|) (1+|\theta|)^3\nonumber\\
&\leq 8m_2d_1^2 (1+c_F)^2(1+|x|)^2(1+|\theta|)^3.
\end{align}
Therefore, for fixed $q = 4$, $\rho =2$, by \eqref{fixedsga2lg}, \eqref{fixedsga2kg}, one observes that Assumption \ref{AG} holds with
\begin{align*}
L_G  = 55m_2d_1^2(1+c_F)^2  \left(1+C_{Z, \max}+ \bar{C}_{Z, \max}\right)  \E\left[ (1+|X_0|)^2\right],\quad K_G  =  8m_2d_1^2 (1+c_F)^2.
\end{align*}

To show that Assumption \ref{AF} is satisfied, we apply the same arguments as in \eqref{ftc1} - \eqref{ftc3} to obtain, for any $\theta, \bar{\theta} \in \R^d, x, \bar{x} \in \R^m$,
\[
\left|F(\theta, x) - F(\bar{\theta}, \bar{x})\right|  = \left|\eta \theta |\theta|^{2r} -\eta \bar{\theta} |\bar{\theta}|^{2r}\right|\leq \eta(1+2r)\left(1+|\theta|+|\bar{\theta}|\right)^{2r}|\theta - \bar{\theta}|.
\]
Furthermore, one notes that, for any $\theta \in \R^d, x \in \R^m$, $\left| F(\theta, x)\right|\leq \eta\left(1+|\theta|^{2r+1}\right)$. Thus, for fixed $r = 2, \rho = 2$, Assumption \ref{AF} holds with $L_F :=  5\eta, K_F := \eta$.

Finally, by using the same proof as in Appendix \ref{proofACexample}, one notes that Assumption \ref{AC} holds with $A(x) = \eta I_d/2, B(x) = 0, \bar{r} = 0, a =\eta/2 , b =0$.
\end{proof}
\begin{proof}[\textbf{Proof of Corollary \ref{corofixed}}]\label{proofremarkfixed}
Let $\theta = ([W_1], b_0) \in \R^d, \quad \bar{\theta} = ([\bar{W}_1], \bar{b}_0) \in \R^d$, and let $\bar{A}_J$ be defined in \eqref{barAJexp}. One notes that under the assumptions in Corollary \ref{corofixed}, the following result can be obtained. For any $I = 1, \dots, m_2, J = 1, \dots, d_1$, we have
\begin{align*}
&\E\left[ |Y^I|^2\left|\1_{A_J}(Z) - \1_{\bar{A}_J}(Z)\right| \right] \nonumber \\
& = \E\left[ |y^I(Z)|^2 \1_{\{(-\bar{b}_0^J-\sum_{k\neq \nu_J}c^{Jk}Z^k)/c^{J\nu_J}  \leq Z^{\nu_J} <(-b_0^J-\sum_{k\neq \nu_J}c^{Jk}Z^k)/c^{J\nu_J} \}} \right] \nonumber \\
& =  \int_{\R^{m_1-1}} \int_{\frac{-\bar{b}_0^J-\sum_{k\neq \nu_J}c^{Jk}z^k}{c^{J\nu_J} }}^{\frac{-b_0^J-\sum_{k\neq \nu_J}c^{Jk}z^k}{c^{J\nu_J} }}|y^I(z)|^2 f_{Z^{\nu_J}|Z_{-\nu_J} }(z^{\nu_J}|z_{-\nu_J} )\, \rmd z^{\nu_J}  f_{Z_{-\nu_J} }(z_{-\nu_J} )\, \rmd  z_{-\nu_J} \nonumber\\
& \leq  \int_{\R^{m_1-1}} \int_{\frac{-\bar{b}_0^J-\sum_{k\neq \nu_J}c^{Jk}z^k}{c^{J\nu_J} }}^{\frac{-b_0^J-\sum_{k\neq \nu_J}c^{Jk}z^k}{c^{J\nu_J} }}c_y^2(1+|z|^{q_y})^2 f_{Z^{\nu_J}|Z_{-\nu_J} }(z^{\nu_J}|z_{-\nu_J} )\, \rmd z^{\nu_J}   f_{Z_{-\nu_J} }(z_{-\nu_J} )\, \rmd  z_{-\nu_J}  \nonumber\\
& \leq 2c_y^2 \int_{\R^{m_1-1}} \int_{\frac{-\bar{b}_0^J-\sum_{k\neq \nu_J}c^{Jk}z^k}{c^{J\nu_J} }}^{\frac{-b_0^J-\sum_{k\neq \nu_J}c^{Jk}z^k}{c^{J\nu_J} }}  f_{Z^{\nu_J}|Z_{-\nu_J} }(z^{\nu_J}|z_{-\nu_J} )\, \rmd z^{\nu_J}   f_{Z_{-\nu_J} }(z_{-\nu_J} )\, \rmd  z_{-\nu_J}  \nonumber\\
& \quad +2c_y^2 \int_{\R^{m_1-1}} \int_{\frac{-\bar{b}_0^J-\sum_{k\neq \nu_J}c^{Jk}z^k}{c^{J\nu_J} }}^{\frac{-b_0^J-\sum_{k\neq \nu_J}c^{Jk}z^k}{c^{J\nu_J} }} |z|^{2q_y}  f_{Z^{\nu_J}|Z_{-\nu_J} }(z^{\nu_J}|z_{-\nu_J} )\, \rmd z^{\nu_J} f_{Z_{-\nu_J} }(z_{-\nu_J} )\, \rmd  z_{-\nu_J}.
\end{align*}
Since $X = (Y, Z)$, by using \eqref{remarkfixedmmtbd}, the above result implies
\begin{align}
&\E\left[ |Y^I|^2\left|\1_{A_J}(Z) - \1_{\bar{A}_J}(Z)\right| \right] \nonumber \\
&\leq \frac{2c_y^2 C_{Z^{\nu_J}}}{c^{J\nu_J}}|\bar{b}_0^J - b_0^J| \nonumber \\
& \quad +2^{q_y}c_y^2 \int_{\R^{m_1-1}} \int_{\frac{-\bar{b}_0^J-\sum_{k\neq \nu_J}c^{Jk}z^k}{c^{J\nu_J} }}^{\frac{-b_0^J-\sum_{k\neq \nu_J}c^{Jk}z^k}{c^{J\nu_J} }} |z^{\nu_J}|^{2q_y}  f_{Z^{\nu_J}|Z_{-\nu_J} }(z^{\nu_J}|z_{-\nu_J} )\, \rmd z^{\nu_J}   f_{Z_{-\nu_J} }(z_{-\nu_J} )\, \rmd  z_{-\nu_J}\nonumber\\
& \quad +2^{q_y}c_y^2 \int_{\R^{m_1-1}} \int_{\frac{-\bar{b}_0^J-\sum_{k\neq \nu_J}c^{Jk}z^k}{c^{J\nu_J} }}^{\frac{-b_0^J-\sum_{k\neq \nu_J}c^{Jk}z^k}{c^{J\nu_J} }}  f_{Z^{\nu_J}|Z_{-\nu_J} }(z^{\nu_J}|z_{-\nu_J} )\, \rmd z^{\nu_J} |z_{-\nu_J}|^{2q_y}   f_{Z_{-\nu_J} }(z_{-\nu_J} )\, \rmd  z_{-\nu_J}\nonumber\\
&\leq \frac{2c_y^2 C_{Z^{\nu_J}}}{c^{J\nu_J}}|\bar{b}_0^J - b_0^J| +\frac{2^{q_y}c_y^2 \bar{C}_{Z^{\nu_J}}}{c^{J\nu_J}}|\bar{b}_0^J - b_0^J|+\frac{2^{q_y}c_y^2 C_{Z^{\nu_J}}}{c^{J\nu_J}}\E\left[|Z_{-\nu_J}|^{2q_y}\right]|\bar{b}_0^J - b_0^J|\nonumber \\
& \leq  \left(\frac{2c_y^2C_{Z^{\nu_J}}}{c^{J\nu_J}}+\frac{2^{q_y}c_y^2 \bar{C}_{Z^{\nu_J}}}{c^{J\nu_J}}+\frac{2^{q_y}c_y^2 C_{Z^{\nu_J}}}{c^{J\nu_J}}\right)\E\left[(1+|X|)^{2q_y}\right]|\bar{b}_0^J - b_0^J| \nonumber \\
& \leq C_{Z, \max}\E\left[(1+|X|)^{2q_y}\right]|\theta - \bar{\theta}|,\nonumber
\end{align}
where
\begin{equation}\label{remarkczmax}
C_{Z, \max} := \max_{J\in\{ 1, \dots, d_1\}}\left\{c_y^2\left(2C_{Z^{\nu_J}}+2^{q_y}  \bar{C}_{Z^{\nu_J}}+2^{q_y}  C_{Z^{\nu_J}}\right)/c^{J\nu_J}\right\}.
\end{equation}
The rest of the proof follows the same lines as in the proof of Proposition \ref{propfixed} (see Appendix \ref{proofpropfixed}), however, one notes that, in this case, $\rho = 2q_y$, and $C_{Z, \max}$ is given by \eqref{remarkczmax}.
\end{proof}

\begin{proof}[\textbf{Proof of Proposition \ref{propat}}]\label{proofpropat}
First, since for any $a_1, a_2\in \R$,
\[
a_3= 2a_1, \quad a_4 = -a_1, \quad a_5 = 2a_2,  \quad a_6 = -a_2,
\]
one notes that the objective function $u$ defined in \eqref{atobj} - \eqref{atobjas} is given by
\[
u(\theta) =
\begin{cases}
\theta^{30}+a_2\theta^2 + (a_1-a_2)\theta^2\E[\1_{\{X \leq  \theta\}}], \quad |\theta| \leq 1,\\
\theta^{30}+a_5|\theta|+a_6+ ((a_3-a_5)|\theta|+(a_4-a_6))\E[\1_{\{X \leq  \theta\}}], \quad |\theta| > 1,
\end{cases}
\]
which is continuously differentiable with
\[
u'(\theta)=
\begin{cases}
30\theta^{29}+2a_2\theta +2 (a_1-a_2)  \theta\E[\1_{\{X \leq  \theta\}}] + (a_1-a_2)\theta^2f_{X}(\theta), &\quad |\theta| \leq 1,\\
30\theta^{29}+(a_5+ (a_3-a_5)\E[\1_{\{X \leq  \theta\}}])(\1_{\{\theta>1\}}-\1_{\{\theta<-1\}}) &\\
+ ((a_3-a_5)|\theta|+(a_4-a_6))f_{X}(\theta), &\quad |\theta| > 1.
\end{cases}
\]
Since $(X_n)_{n\in\N_0}$ is a sequence of i.i.d. random variables with probability law $\mathcal{L}(X)$, by the definitions of $F, G$ given in \eqref{atsgexp} and as $H: = F+G$, we observe that $h(\theta) : = u'(\theta):= \E[H(\theta, X_0)]$, for all $\theta \in \R$. Thus, Assumption \ref{AI} holds.

To show that Assumption \ref{AG} holds, one considers the following cases:
\begin{enumerate}[leftmargin=*]
\item For $|\theta|, |\bar{\theta}|<1$, one obtains
\begin{align}
 \E[|G(\theta,X_0)-G(\bar{\theta}, X_0)|]
&=\E[|G(\theta,X)-G(\bar{\theta}, X)|]\nonumber\\
 & = \E \left[\left|2a_2\theta +2 (a_1-a_2)  \theta \1_{\{X \leq  \theta\}} + (a_1-a_2)\theta^2f_{X}(\theta)\right. \right. \nonumber\\
 &\quad\left. \left.- \left(2a_2\bar{\theta} +2 (a_1-a_2)  \bar{\theta} \1_{\{X \leq  \bar{\theta}\}} + (a_1-a_2)\bar{\theta}^2f_{X}(\bar{\theta})\right)\right|\right] \nonumber\\
\begin{split}\label{atsga2eq1}
& \leq 2|a_2||\theta-\bar{\theta}|+2|a_1-a_2|\E\left[| \theta \1_{\{X \leq  \theta\}} - \bar{\theta} \1_{\{X \leq  \bar{\theta}\}}| \right]\\
& \quad +|a_1-a_2||\theta^2f_{X}(\theta)-\bar{\theta}^2f_{X}(\bar{\theta})|.
\end{split}
\end{align}
By assuming without loss of generality that $\theta \leq \bar{\theta}$, one obtains
\begin{align}\label{atsga2eq2}
\E\left[| \theta \1_{\{X \leq  \theta\}} - \bar{\theta} \1_{\{X \leq  \bar{\theta}\}}| \right]
&\leq \E\left[| \theta \1_{\{X \leq  \theta\}} - \bar{\theta} \1_{\{X \leq  \theta\}}| \right] +\E\left[| \bar{\theta}  \1_{\{X \leq  \theta\}} - \bar{\theta} \1_{\{X \leq  \bar{\theta}\}}| \right] \nonumber\\
&\leq |\theta - \bar{\theta}| +|\bar{\theta}|\E[\1_{\{\theta \leq X \leq \bar{\theta}\}}]\nonumber\\
& = |\theta - \bar{\theta}| +|\bar{\theta}|\int_{\theta}^{\bar{\theta}}f_{X}(x)\, \rmd x \nonumber\\
&\leq |\theta - \bar{\theta}| +c_{X}|\bar{\theta}| |\theta - \bar{\theta}|\nonumber\\
& \leq (1+c_{X})(1+|\theta|+|\bar{\theta}|)|\theta - \bar{\theta}|.
\end{align}
Moreover, it holds that
\begin{align}\label{atsga2eq3}
|\theta^2f_{X}(\theta)-\bar{\theta}^2f_{X}(\bar{\theta})|
& \leq |\theta^2f_{X}(\theta)-\bar{\theta}^2f_{X}(\theta)|   +|\bar{\theta}^2f_{X}(\theta)-\bar{\theta}^2f_{X}(\bar{\theta})|\nonumber \\
&\leq c_{X}(|\theta|+|\bar{\theta}|)|\theta - \bar{\theta}| +L_{X}|\bar{\theta}|^2|\theta - \bar{\theta}|\nonumber\\
&\leq (c_{X}+L_{X})(1+|\theta|+|\bar{\theta}|)^2|\theta - \bar{\theta}|.
\end{align}
Substituting \eqref{atsga2eq2} and \eqref{atsga2eq3} into \eqref{atsga2eq1} yields
\begin{align}\label{atsga2resl1}
 \E[|G(\theta,X_0)-G(\bar{\theta}, X_0)|] 
&\leq 2|a_2||\theta-\bar{\theta}|+2(|a_1|+|a_2|)(1+c_{X})(1+|\theta|+|\bar{\theta}|)|\theta - \bar{\theta}|\nonumber \\
& \quad +(|a_1|+|a_2|)(c_{X}+L_{X})(1+|\theta|+|\bar{\theta}|)^2|\theta - \bar{\theta}|\nonumber \\
&\leq (4+3c_{X}+L_{X})(1+|a_1|+|a_2|)(1+|\theta|+|\bar{\theta}|)^2|\theta - \bar{\theta}|.
\end{align}
\item For $|\theta|, |\bar{\theta}|>1$, one obtains
\begin{align}
&\E[|G(\theta,X_0)-G(\bar{\theta}, X_0)|]\nonumber\\
&=\E[|G(\theta,X)-G(\bar{\theta}, X)|]\nonumber\\
& = \E \left[\left|2(a_2+ (a_1-a_2) \1_{\{X \leq  \theta\}})(\1_{\{\theta>1\}}-\1_{\{\theta<-1\}}) + (a_1-a_2) (2|\theta|-1)f_{X}(\theta)\right. \right. \nonumber\\
&\quad\left. \left.- \left(2(a_2+ (a_1-a_2) \1_{\{X \leq  \bar{\theta}\}})(\1_{\{\bar{\theta}>1\}}-\1_{\{\bar{\theta}<-1\}}) + (a_1-a_2) (2|\bar{\theta}|-1)f_{X}(\bar{\theta})\right)\right|\right] \nonumber\\
\begin{split}\label{atsga2eq4}
& \leq \E \left[\left|2(a_2+ (a_1-a_2) \1_{\{X \leq  \theta\}})(\1_{\{\theta>1\}}-\1_{\{\theta<-1\}}) \right. \right. \\
&\quad\left. \left.- 2(a_2+ (a_1-a_2) \1_{\{X \leq  \bar{\theta}\}})(\1_{\{\bar{\theta}>1\}}-\1_{\{\bar{\theta}<-1\}})\right|\right] \\
&\quad + \E \left[\left| (a_1-a_2) (2|\theta|-1)f_{X}(\theta) - (a_1-a_2) (2|\bar{\theta}|-1)f_{X}(\bar{\theta})\right| \right].
\end{split}
\end{align}
One notes that as $|\theta|, |\bar{\theta}|>1$,
\begin{equation}\label{atsga2eq5}
\left|\1_{\{\theta>1\}}-\1_{\{\theta<-1\}} - (\1_{\{\bar{\theta}>1\}}-\1_{\{\bar{\theta}<-1\}}) \right| \leq |\theta- \bar{\theta}|.
\end{equation}
Then, by using \eqref{atsga2eq5}, and by assuming without loss of generality that $\theta \leq \bar{\theta}$, one obtains
\begin{align}\label{atsga2eq6}
&\E \left[\left|2(a_2+ (a_1-a_2) \1_{\{X \leq  \theta\}})(\1_{\{\theta>1\}}-\1_{\{\theta<-1\}}) \right. \right. \nonumber\\
&\quad\left. \left.- 2(a_2+ (a_1-a_2) \1_{\{X \leq  \bar{\theta}\}})(\1_{\{\bar{\theta}>1\}}-\1_{\{\bar{\theta}<-1\}})\right|\right]\nonumber\\
& \leq \E \left[\left|2(a_2+ (a_1-a_2) \1_{\{X \leq  \theta\}})(\1_{\{\theta>1\}}-\1_{\{\theta<-1\}}) \right. \right. \nonumber\\
&\quad\left. \left.- 2(a_2+ (a_1-a_2) \1_{\{X \leq  \theta\}})(\1_{\{\bar{\theta}>1\}}-\1_{\{\bar{\theta}<-1\}})\right|\right]\nonumber\\
&\quad +\E \left[\left|2(a_2+ (a_1-a_2) \1_{\{X \leq  \theta\}})(\1_{\{\bar{\theta}>1\}}-\1_{\{\bar{\theta}<-1\}}) \right. \right. \nonumber\\
&\quad\left. \left.- 2(a_2+ (a_1-a_2) \1_{\{X \leq  \bar{\theta}\}})(\1_{\{\bar{\theta}>1\}}-\1_{\{\bar{\theta}<-1\}})\right|\right]\nonumber\\
&\leq 2(|a_1|+2|a_2|)|\theta- \bar{\theta}|+2(|a_1|+ |a_2|)\E[|\1_{\{X \leq  \theta\}} - \1_{\{X \leq  \bar{\theta}\}}|]\nonumber\\
&\leq 2(|a_1|+2|a_2|)|\theta- \bar{\theta}|+2(|a_1|+ |a_2|)\E[\1_{\{\theta \leq X \leq \bar{\theta}\}}]\nonumber\\
&= 2(|a_1|+2|a_2|)|\theta- \bar{\theta}|+2(|a_1|+ |a_2|) \int_{\theta}^{\bar{\theta}}f_{X}(x)\, \rmd x \nonumber\\
&\leq 4(|a_1|+ |a_2|)|\theta- \bar{\theta}| +2c_{X}(|a_1|+ |a_2|) |\theta- \bar{\theta}|\nonumber\\
&\leq (4+2c_{X})(1+|a_1|+ |a_2|) |\theta- \bar{\theta}|.
\end{align}
In addition, one has that
\begin{align}\label{atsga2eq7}
&\E \left[\left| (a_1-a_2) (2|\theta|-1)f_{X}(\theta) - (a_1-a_2) (2|\bar{\theta}|-1)f_{X}(\bar{\theta})\right| \right]\nonumber\\
& \leq\E \left[\left| (a_1-a_2) (2|\theta|-1)f_{X}(\theta) - (a_1-a_2) (2|\bar{\theta}|-1)f_{X}(\theta)\right| \right]\nonumber\\
&\quad +\E \left[\left| (a_1-a_2) (2|\bar{\theta}|-1)f_{X}(\theta) - (a_1-a_2) (2|\bar{\theta}|-1)f_{X}(\bar{\theta})\right| \right]\nonumber\\
&\leq 2 c_{X}(|a_1|+ |a_2|) |\theta- \bar{\theta}|+2L_{X}(|a_1|+ |a_2|)(1+|\bar{\theta}|) |\theta- \bar{\theta}|\nonumber\\
&\leq 2(c_{X}+L_{X})(1+|a_1|+ |a_2|)(1+|\bar{\theta}|) |\theta- \bar{\theta}|.
\end{align}
Substituting \eqref{atsga2eq6} and \eqref{atsga2eq7} into \eqref{atsga2eq4} yields
\begin{align}\label{atsga2resl2}
 &\E[|G(\theta,X_0)-G(\bar{\theta}, X_0)|]\nonumber\\
& \leq (4+2c_{X})(1+|a_1|+ |a_2|) |\theta- \bar{\theta}| +2(c_{X}+L_{X})(1+|a_1|+ |a_2|)(1+|\bar{\theta}|) |\theta- \bar{\theta}|\nonumber\\
&\leq (4+4c_{X}+2L_{X})(1+|a_1|+ |a_2|)(1+|\theta|+|\bar{\theta}|) |\theta- \bar{\theta}|.
\end{align}
\item For $|\theta|\leq1$, $|\bar{\theta}|>1$, one obtains
\begin{align}
&\E[|G(\theta,X_0)-G(\bar{\theta}, X_0)|]\nonumber\\
&=\E[|G(\theta,X)-G(\bar{\theta}, X)|]\nonumber\\
& = \E \Big[\left|2a_2\theta +2 (a_1-a_2)  \theta \1_{\{X \leq  \theta\}} + (a_1-a_2)\theta^2f_{X}(\theta)\right. \Big. \nonumber\\
&\quad\Big. \left.- \left(2(a_2+ (a_1-a_2) \1_{\{X \leq  \bar{\theta}\}})(\1_{\{\bar{\theta}>1\}}-\1_{\{\bar{\theta}<-1\}}) + (a_1-a_2) (2|\bar{\theta}|-1)f_{X}(\bar{\theta})\right)\right|\Big] \nonumber\\
\begin{split}\label{atsga2eq8}
& \leq \E \left[\left|2a_2\theta +2 (a_1-a_2)  \theta \1_{\{X \leq  \theta\}}  -2(a_2+ (a_1-a_2) \1_{\{X \leq  \bar{\theta}\}})(\1_{\{\bar{\theta}>1\}}-\1_{\{\bar{\theta}<-1\}})\right|\right] \\
&\quad + \E \left[\left|(a_1-a_2)\theta^2f_{X}(\theta)- (a_1-a_2) (2|\bar{\theta}|-1)f_{X}(\bar{\theta})\right| \right].
\end{split}
\end{align}
One observes that as $|\theta|\leq1$, $|\bar{\theta}|>1$,
\begin{equation}\label{atsga2eq9}
\left|\theta- (\1_{\{\bar{\theta}>1\}}-\1_{\{\bar{\theta}<-1\}}) \right| \leq |\theta- \bar{\theta}|.
\end{equation}
Then, by using \eqref{atsga2eq9}, and by assuming without loss of generality that $\theta \leq \bar{\theta}$, one obtains
\begin{align}\label{atsga2eq10}
&\E \left[\left|2a_2\theta +2 (a_1-a_2)  \theta \1_{\{X \leq  \theta\}}  -2(a_2+ (a_1-a_2) \1_{\{X \leq  \bar{\theta}\}})(\1_{\{\bar{\theta}>1\}}-\1_{\{\bar{\theta}<-1\}})\right|\right] \nonumber\\
& \leq\E \left[\left|2(a_2 +(a_1-a_2)  \1_{\{X \leq  \theta\}} )\theta   -2(a_2+ (a_1-a_2) \1_{\{X \leq \theta\}})(\1_{\{\bar{\theta}>1\}}-\1_{\{\bar{\theta}<-1\}})\right|\right] \nonumber\\
&\quad +\E \left[\left|2(a_2+ (a_1-a_2) \1_{\{X \leq \theta\}})(\1_{\{\bar{\theta}>1\}}-\1_{\{\bar{\theta}<-1\}}) \right.\right.\nonumber\\
&\qquad \left.\left. -2(a_2+ (a_1-a_2) \1_{\{X \leq  \bar{\theta}\}})(\1_{\{\bar{\theta}>1\}}-\1_{\{\bar{\theta}<-1\}})\right|\right] \nonumber\\
&\leq 2(|a_1|+2|a_2|) |\theta- \bar{\theta}|+2(|a_1|+ |a_2|)\E[\1_{\{\theta \leq X \leq \bar{\theta}\}}]\nonumber\\
&\leq (4+2c_{X})(1+|a_1|+ |a_2|) |\theta- \bar{\theta}|,
\end{align}
where the last inequality follows by using the same arguments as in \eqref{atsga2eq6}. Furthermore, one notes that
\begin{align}\label{atsga2eq11}
\left|\theta^2 - (2|\bar{\theta}|-1)\right|
& = \left||\theta|^2- 2|\theta|+1+2|\theta|-  2|\bar{\theta}| \right|\nonumber\\
&\leq \left|(|\theta| - 1)^2\right|+2\left||\theta|-   |\bar{\theta}|\right|\nonumber\\
&\leq \left|(|\theta| - 1)^2\right|+2|\theta- \bar{\theta}|\nonumber\\
&\leq ||\theta| - 1|+2|\theta- \bar{\theta}|\nonumber\\
&\leq \left||\theta| - |\bar{\theta}|\right|+2|\theta- \bar{\theta}|\nonumber\\
&\leq 3|\theta- \bar{\theta}|,
\end{align}
where the third inequality holds due to $|\theta|\leq1$, while the fourth inequality holds due to $|\bar{\theta}|>1$. Thus, by using \eqref{atsga2eq11}, one obtains
\begin{align}\label{atsga2eq12}
&\E \left[\left|(a_1-a_2)\theta^2f_{X}(\theta)- (a_1-a_2) (2|\bar{\theta}|-1)f_{X}(\bar{\theta})\right| \right]\nonumber\\
& \leq \E \left[\left|(a_1-a_2)\theta^2f_{X}(\theta)- (a_1-a_2) (2|\bar{\theta}|-1)f_{X}(\theta)\right| \right]\nonumber\\
&\quad + \E \left[\left|(a_1-a_2) (2|\bar{\theta}|-1)f_{X}(\theta)- (a_1-a_2) (2|\bar{\theta}|-1)f_{X}(\bar{\theta})\right| \right]\nonumber\\
&\leq 3 c_{X}(|a_1|+ |a_2|) |\theta- \bar{\theta}|+2L_{X}(|a_1|+ |a_2|)(1+|\bar{\theta}|) |\theta- \bar{\theta}|\nonumber\\
&\leq  (3c_{X}+2L_{X})(1+|a_1|+ |a_2|)(1+|\bar{\theta}|) |\theta- \bar{\theta}|.
\end{align}
Substituting \eqref{atsga2eq10} and \eqref{atsga2eq12} into \eqref{atsga2eq8} yields
\begin{align}\label{atsga2resl3}
 &\E[|G(\theta,X_0)-G(\bar{\theta}, X_0)|] \nonumber\\
& \leq (4+2c_{X})(1+|a_1|+ |a_2|) |\theta- \bar{\theta}|  + (3c_{X}+2L_{X})(1+|a_1|+ |a_2|)(1+|\bar{\theta}|) |\theta- \bar{\theta}|\nonumber\\
&\leq (4+5c_{X}+2L_{X})(1+|a_1|+ |a_2|)(1+|\theta|+|\bar{\theta}|) |\theta- \bar{\theta}|.
\end{align}
\end{enumerate}
By \eqref{atsga2resl1}, \eqref{atsga2resl2}, and \eqref{atsga2resl3}, one concludes that, for any $\theta, \bar{\theta} \in \R$,
\[
\E[|G(\theta,X_0)-G(\bar{\theta}, X_0)|]\leq  (4+5c_{X}+2L_{X})(1+|a_1|+|a_2|) (1+|\theta|+|\bar{\theta}|)^2|\theta- \bar{\theta}|.
\]
Moreover, by the definition of $G$ in \eqref{atsgexp}, one can obtain the following estimates:
\begin{enumerate}[leftmargin=*]
\item For $|\theta|\leq 1$, we have, for any $x \in \R$, that
\begin{align}\label{atsga2resl4}
|G(\theta, x)|& = |2a_2\theta +2 (a_1-a_2)  \theta \1_{\{x \leq  \theta\}} + (a_1-a_2)\theta^2f_{X}(\theta)|\nonumber\\
&\leq 2(|a_1|+2|a_2|)|\theta|+c_{X}(|a_1|+ |a_2|)|\theta|^2\nonumber\\
&\leq (4+c_{X})(1+|a_1|+ |a_2|)(1+|\theta|)^2.
\end{align}
\item For $|\theta|> 1$, we have, for any $x \in \R$, that
\begin{align}\label{atsga2resl5}
|G(\theta, x)|& = |2(a_2+ (a_1-a_2) \1_{\{x \leq  \theta\}})(\1_{\{\theta>1\}}-\1_{\{\theta<-1\}}) + (a_1-a_2) (2|\theta|-1)f_{X}(\theta)|\nonumber\\
&\leq 2(|a_1|+2|a_2|) +2c_{X}(|a_1|+ |a_2|)(1+|\theta|)\nonumber\\
&\leq  (4+2c_{X})(1+|a_1|+ |a_2|)(1+|\theta|) .
\end{align}
\end{enumerate}
By \eqref{atsga2resl4}, \eqref{atsga2resl5}, we obtain, for every $\theta, x \in \R$, that
\[
|G(\theta, x)| \leq   (4+2c_{X})(1+|a_1|+|a_2|)(1+|\theta|)^2.
\]
Thus, since $q = 3$ and $\rho = 1$, Assumption \ref{AG} is satisfied with
\begin{equation}\label{atsga2consts}
L_G = (4+5c_{X}+2L_{X})(1+|a_1|+|a_2|)  , \quad K_G = (4+2c_{X})(1+|a_1|+ |a_2|).
\end{equation}

Next, we show that Assumption \ref{AF} is satisfied. For any $\theta, \bar{\theta} \in \R$, define
\begin{equation}\label{ftc1}
f(t) := (t\theta+(1-t)\bar{\theta})^{2r+1}, \quad t \in [0,1],
\end{equation}
which implies that
\begin{equation}\label{ftc2}
f'(t)  = (2r+1)(t\theta+(1-t)\bar{\theta})^{2r}(\theta- \bar{\theta}).
\end{equation}
Then, one obtains
\begin{align}\label{ftc3}
\begin{split}
\left|f(1) - f(0)\right|  = \left|\theta^{2r+1} - \bar{\theta}^{2r+1}\right|
&= \left|\int_0^1  (2r+1)(t\theta+(1-t)\bar{\theta})^{2r}(\theta- \bar{\theta})\,\rmd t\right|\\
&\leq (2r+1)\left(|\theta|+|\bar{\theta}|\right)^{2r}\left|\theta-\bar{\theta}\right|.
\end{split}
\end{align}
One notes that $r = 14$. Then, by using the above inequality, one obtains the following: for any $\theta, \bar{\theta} \in \R$, and any $x \in \R$, $|F(\theta, x) - F(\bar{\theta}, x)|= |30\theta^{29}-30\bar{\theta}^{29}| \leq 870(1+|\theta|+|\bar{\theta}|)^{28}|\theta- \bar{\theta}|$. Moreover, it holds that $|F(\theta, x)| \leq 30(1+|\theta|^{29})$. Thus, Assumption \ref{AF} is satisfied with $L_F = 870$, $K_F = 30$.

Finally, one can show that Assumption \ref{AC} is satisfied with $A(x) =15I_d, B(x) = 0, \bar{r} = 0, a = 15, b = 0$ by using similar arguments as in the proof of Remark \ref{ACexample} (see Appendix \ref{proofACexample}).
\end{proof}

\newpage
\appendix
\section{Proof of auxiliary results}
\subsection{Proof of auxiliary results in Section \ref{assumption}}
\begin{proof}[\textbf{Proof of statement in Remark \ref{ACexample}}] \label{proofACexample}
For any $\theta \in \R^d, x \in \R^m$, consider $F(\theta, x) := \eta \theta|\theta|^{2l}$ with $\eta \in (0,1), l \geq q/2$. Then, one obtains,
\begin{align*}
 \langle  \theta- \theta' , F(\theta,x)- F(\theta',x)\rangle
& = \frac{\eta}{2}\left(2\sum_{i = 1}^d (\theta_i- \theta'_i)( \theta_i|\theta|^{2l} -  \theta'_i|\theta'|^{2l})\right)\\
& = \frac{\eta}{2}\sum_{i = 1}^d \left((\theta_i- \theta'_i)^2|\theta|^{2l} +(\theta_i- \theta'_i)\theta'_i(|\theta|^{2l}  - |\theta'|^{2l} ) \right)\\
&\quad +\frac{\eta}{2}\sum_{i = 1}^d \left((\theta_i- \theta'_i)\theta_i(|\theta|^{2l}  - |\theta'|^{2l} )+(\theta_i- \theta'_i)^2|\theta'|^{2l}  \right)\\
& = \frac{\eta}{2}\sum_{i = 1}^d \left((\theta_i- \theta'_i)^2(|\theta|^{2l} +|\theta'|^{2l} )+(|\theta_i|^2- |\theta'_i|^2)(|\theta|^{2l}  - |\theta'|^{2l} ) \right)\\
& = \frac{\eta}{2}|\theta- \theta'|^2(|\theta|^{2l} +|\theta'|^{2l} )+ \frac{\eta}{2}(|\theta|^2- |\theta'|^2)(|\theta|^{2l}  - |\theta'|^{2l} ) \\
&\geq \langle\theta- \theta', A(x) (\theta- \theta')\rangle(|\theta|^{2l} +|\theta'|^{2l}),
\end{align*}
where $A(x)  = \eta I_d/2, B(x) = 0$ for all $x \in \R^m$, and moreover, $r=l, \bar{r} = 0, a= \eta/2, b  =0 $.
\end{proof}
\begin{proof}[\textbf{Proof of statement in Remark \ref{ADFh}}]  \label{proofADFh}
We first prove inequality \eqref{sldissipative}. By Assumption \ref{AI} and \ref{AC}, one obtains, for any $\theta \in \R^d$,
\begin{align*}
\langle \theta, \E[F(\theta, X_0)] \rangle
& \geq \langle \theta, \E[A( X_0)] \theta\rangle|\theta|^{2r} -\langle \theta, \E[B( X_0)] \theta\rangle|\theta|^{\bar{r}}+\langle \theta, \E[F(0, X_0)] \rangle \\
&\geq   a|\theta|^{2r+2} -b|\theta|^{\bar{r}+2}+\langle \theta, \E[F(0, X_0)] \rangle.
\end{align*}
Therefore, by applying Young's inequality, Assumption \ref{AF}, we obtain
\begin{equation}\label{ADFes1}
\langle \theta, \E[F(\theta, X_0)] \rangle \geq a|\theta|^{2r+2} -b|\theta|^{\bar{r}+2} - \frac{a}{2}|\theta|^2 - \frac{1}{2a}K_F^2\E[(1+|X_0|)^{2\rho}].
\end{equation}
By using $0\leq \bar{r} <2r$, $r\geq q/2 \geq 1/2$, it follows that, for $\theta \in \R^d$,
\[
\frac{a}{4}|\theta|^{2r+2} -b|\theta|^{\bar{r}+2} > 0 \quad \Leftrightarrow  \quad |\theta| > \left(\frac{4b}{a}\right)^{1/(2r-\bar{r})},
\]
and moreover,
\[
\frac{a}{4}|\theta|^{2r+2}  - \frac{a}{2}|\theta|^2 > 0 \quad \Leftrightarrow \quad |\theta| > 2^{1/(2r)}.
\]
Denote by $R_F := \max\{(4b/a)^{1/(2r-\bar{r})}, 2^{1/(2r)}\}>1$. For $|\theta| >  R_F$, \eqref{ADFes1} hence becomes
\begin{equation}\label{ADFes2}
\langle \theta, \E[F(\theta, X_0)] \rangle  > \frac{a}{2}|\theta|^{2r+2}  - \frac{1}{2a}K_F^2\E[(1+|X_0|)^{2\rho}],
\end{equation}
while for $|\theta| \leq R_F$, it follows that
\begin{align}\label{ADFes3}
\langle \theta, \E[F(\theta, X_0)] \rangle
& \geq  -b|\theta|^{\bar{r}+2} - \frac{a}{2}|\theta|^2 - \frac{1}{2a}K_F^2\E[(1+|X_0|)^{2\rho}] \nonumber\\
&\geq -b R_F^{\bar{r}+2} - \frac{a}{2} R_F^2 - \frac{1}{2a}K_F^2\E[(1+|X_0|)^{2\rho}] \nonumber\\
& \geq -\left(b+ \frac{a}{2} \right)R_F^{\bar{r}+2}- \frac{1}{2a}K_F^2\E[(1+|X_0|)^{2\rho}].
\end{align}
Finally, by using the estimates in \eqref{ADFes2} and \eqref{ADFes3}, one obtains
\begin{equation}\label{ADF}
\langle \theta, \E[F(\theta, X_0)] \rangle  \geq a_F|\theta|^{2r+2} - b_F,
\end{equation}
where $a_F := a/2$ and $b_F := (a/2+b)R_F^{\bar{r}+2}+K_F^2\E[(1+|X_0|)^{2\rho}]/(2a)$ with
\[
R_F := \max\{(4b/a)^{1/(2r-\bar{r})}, 2^{1/(2r)}\}.
\]

Recall the expression of $H$ in \eqref{expressionH}. Then, Assumption \ref{AI}, \ref{AG}, \eqref{ADF} and Cauchy-Schwarz inequality imply, for any $\theta \in \R^d$,
\begin{align}\label{dissiine}
\langle \theta, h(\theta) \rangle
&= \langle \theta, \E[G(\theta, X_0)] \rangle +\langle \theta, \E[F(\theta, X_0)] \rangle \nonumber \\
&\geq a_F|\theta|^{2r+2} - b_F-2^qK_G\E[(1+|X_0|)^{\rho}](1+|\theta|^{q+1}).
\end{align}
To prove \eqref{sdedissipative}, it suffices to show that
\[
a_F|\theta|^{2r+2} - b_F - 2^qK_G\E[(1+|X_0|)^{\rho}](1+|\theta|^{q+1}) \geq a_h|\theta|^2 -b_h
\]
for some $a_h, b_h>0$. Set
\[
a_h := 2^qK_G\E[(1+|X_0|)^{\rho}], \quad b_h := 3(2^{q+1}K_G\E[(1+|X_0|)^{\rho}]/\min{\{1, a_F\}})^{q+2}+b_F.
\]
Then, one observes that for $|\theta| \geq 2^{q+1}K_G\E[(1+|X_0|)^{\rho}]/\min{\{1, a_F\}}\geq 1$, $r\geq q/2 \geq 1/2$,
\begin{align}\label{dissies1}
2|\theta|^{2r+2} +2b_h/a_F
& \geq  |\theta|^3 +|\theta|^{q+2}+2b_h/a_F \nonumber\\
& \geq \big(2^{q+1}K_G\E[(1+|X_0|)^{\rho}]/\min{\{1, a_F\}}\big) |\theta|^2 \nonumber\\
&\quad + \big(2^{q+1}K_G\E[(1+|X_0|)^{\rho}]/\min{\{1, a_F\}}\big) |\theta|^{q+1} \nonumber\\
&\quad +2\big(3(2^{q+1}K_G\E[(1+|X_0|)^{\rho}]/\min{\{1, a_F\}})^{q+2}+b_F\big)/a_F \nonumber\\
& \geq 2\big(2^qK_G\E[(1+|X_0|)^{\rho}]\big) |\theta|^2/a_F +  \big(2^{q+1}K_G\E[(1+|X_0|)^{\rho}] \big) |\theta|^{q+1}/a_F \nonumber\\
&\quad +2^{q+1}K_G\E[(1+|X_0|)^{\rho}]/a_F + 2b_F/a_F \nonumber\\
& = 2a_h|\theta|^2/a_F  +2^{q+1}K_G\E[(1+|X_0|)^{\rho}](1+|\theta|^{q+1})/a_F+2b_F/a_F.
\end{align}
Similarly, for $|\theta| \leq 2^{q+1}K_G\E[(1+|X_0|)^{\rho}]/\min{\{1, a_F\}} $, it follows that
\begin{align}\label{dissies2}
2|\theta|^{2r+2} +2b_h/a_F
& \geq  2\left(3\left(2^{q+1}K_G\E[(1+|X_0|)^{\rho}]/\min{\{1, a_F\}}\right)^{q+2}+b_F\right)/a_F \nonumber\\
& \geq 2(2^{q+1}K_G\E[(1+|X_0|)^{\rho}]/\min{\{1, a_F\}})^3/a_F \nonumber\\
&\quad +(2^{q+1}K_G\E[(1+|X_0|)^{\rho}]/\min{\{1, a_F\}})^{q+2}/a_F  +2b_F/a_F \nonumber\\
&\geq (2^{q+1}K_G\E[(1+|X_0|)^{\rho}]/\min{\{1, a_F\}})|\theta|^2/a_F  +2^{q+1}K_G\E[(1+|X_0|)^{\rho}]/a_F \nonumber\\
&\quad +(2^{q+1}K_G\E[(1+|X_0|)^{\rho}]/\min{\{1, a_F\}} )|\theta|^{q+1}/a_F  +2b_F/a_F \nonumber\\
&\geq (2^{q+1}K_G\E[(1+|X_0|)^{\rho}])|\theta|^2/a_F   +2^{q+1}K_G\E[(1+|X_0|)^{\rho}]/a_F \nonumber\\
&\quad +2^{q+1}K_G\E[(1+|X_0|)^{\rho}]|\theta|^{q+1}/a_F+2b_F/a_F \nonumber\\
& =  2a_h|\theta|^2/a_F  +2^{q+1}K_G\E[(1+|X_0|)^{\rho}](1+|\theta|^{q+1})/a_F+2b_F/a_F,
\end{align}
where the first inequality holds due to $2|\theta|^{2r+2}  \geq0$, and the second inequality holds as $q \geq 1$. Thus, by using \eqref{dissies1} and \eqref{dissies2}, one concludes that, for any $\theta \in \R^d$,
\[
2|\theta|^{2r+2} +2b_h/a_F \geq 2a_h|\theta|^2/a_F  +2^{q+1}K_G\E[(1+|X_0|)^{\rho}](1+|\theta|^{q+1})/a_F+2b_F/a_F,
\]
which implies, by multiplying $a_F/2$ on both sides of the inequality, and by rearranging the terms
\[
a_F|\theta|^{2r+2} - b_F - 2^qK_G\E[(1+|X_0|)^{\rho}](1+|\theta|^{q+1}) \geq a_h|\theta|^2 -b_h.
\]
Finally, combining \eqref{dissiine} with the inequality above yields the desired result.
\end{proof}
\begin{proof}[\textbf{Proof of statement in Remark \ref{ACh}}] \label{proofACh}
By using Assumption \ref{AI}, \ref{AG}, \ref{AC}, the expression of $H$ in \eqref{expressionH} and Cauchy-Schwarz inequality, one obtains, for any $\theta, \theta' \in \R^d$,
\begin{align}\label{hoslosR}
&\langle \theta-\theta',h(\theta)-h(\theta')\rangle  \nonumber\\
& \geq \langle\theta- \theta', \E[A(X_0)] (\theta- \theta')\rangle(|\theta|^{2r} +|\theta'|^{2r})   -\langle\theta- \theta', \E[B(X_0)] (\theta- \theta')\rangle(|\theta|^{\bar{r}} +|\theta'|^{\bar{r}}) \nonumber \\
&\quad + \langle \theta-\theta',\E[G(\theta, X_0)-G(\theta', X_0)]\rangle \nonumber \\
&\geq a(|\theta|^{2r} +|\theta'|^{2r})|\theta-\theta'|^2 - b(|\theta|^{\bar{r}} +|\theta'|^{\bar{r}})|\theta-\theta'|^2    - L_G(1+|\theta|+|\theta'|)^{q-1}|\theta-\theta'|^2 \nonumber \\
\begin{split}
&\geq  \frac{a}{2}(|\theta|^{2r} +|\theta'|^{2r})|\theta-\theta'|^2 - b(|\theta|^{\bar{r}} +|\theta'|^{\bar{r}})|\theta-\theta'|^2  \\
&\quad  +\frac{a}{6}(|\theta|^{2r} +|\theta'|^{2r})|\theta-\theta'|^2- 3^{q-2}L_G |\theta-\theta'|^2  \\
&\quad  +\frac{a}{3}(|\theta|^{2r} +|\theta'|^{2r})|\theta-\theta'|^2- 3^{q-2}L_G( |\theta|^{q-1}+|\theta'|^{q-1})|\theta-\theta'|^2.
\end{split}
\end{align}
For $\theta, \theta' \in \R^d$, $0\leq \bar{r} <2r$, one observes the following:
\begin{align*}
\frac{a}{2}|\theta|^{2r}|\theta-\theta'|^2 - b|\theta|^{\bar{r}} |\theta-\theta'|^2  > 0  \quad
& \Leftrightarrow  \quad |\theta| > \left(\frac{2b}{a}\right)^{1/(2r - \bar{r})}, \\
\frac{a}{2}|\theta'|^{2r}|\theta-\theta'|^2 - b|\theta'|^{\bar{r}} |\theta-\theta'|^2  > 0  \quad
& \Leftrightarrow  \quad |\theta'| > \left(\frac{2b}{a}\right)^{1/(2r - \bar{r})} , \\
\frac{a}{6} |\theta|^{2r}  |\theta-\theta'|^2- \frac{3^{q-2}L_G}{2} |\theta-\theta'|^2 > 0 \quad
& \Leftrightarrow  \quad |\theta| > \left(\frac{3^{q-1}L_G}{a}\right)^{1/(2r )} , \\
\frac{a}{6} |\theta'|^{2r}  |\theta-\theta'|^2- \frac{3^{q-2}L_G}{2} |\theta-\theta'|^2 > 0 \quad
& \Leftrightarrow  \quad |\theta'| > \left(\frac{3^{q-1}L_G}{a}\right)^{1/(2r )} , \\
\frac{a}{3} |\theta|^{2r}  |\theta-\theta'|^2- 3^{q-2}L_G  |\theta|^{q-1} |\theta-\theta'|^2 > 0 \quad
& \Leftrightarrow  \quad |\theta | > \left(\frac{3^{q-1}L_G}{a}\right)^{1/(2r-q+1)}, \\
\frac{a}{3} |\theta'|^{2r}  |\theta-\theta'|^2- 3^{q-2}L_G  |\theta'|^{q-1} |\theta-\theta'|^2 > 0 \quad
& \Leftrightarrow  \quad |\theta' | > \left(\frac{3^{q-1}L_G}{a}\right)^{1/(2r-q+1)}.
\end{align*}
One notes that $ (3^{q-1}L_G/a)^{1/(2r )} \leq \max\{1, (3^{q-1}L_G/a)^{1/(2r-q+1)} \}$ due to the fact that $2r \geq q \geq 1$. Define $R :=\max\{1, (3^{q-1}L_G/a)^{1/(2r-q+1)}, (2b/a)^{1/(2r - \bar{r})} \}$, and denote by $\mathrm{\bar{B}}(0,R)$ the closed ball with radius $R$ centred at the zero vector in $\R^d$. Then, for $\theta, \theta' \notin \mathrm{\bar{B}}(0, R)$, one obtains, by using \eqref{hoslosR},
\begin{align}\label{oslosball}
\langle \theta-\theta',h(\theta)-h(\theta')\rangle
&> 0,
\end{align}
while for $\theta, \theta' \in \mathrm{\bar{B}}(0, R)$, it follows, by applying Remark \ref{growthHh},
\begin{equation}\label{oslball}
-\langle \theta-\theta',h(\theta)-h(\theta')\rangle \leq |\theta - \theta'||h(\theta)-h(\theta')| \leq L_R|\theta - \theta'|^2,
\end{equation}
where $L_R := L_h(1+2R)^{2r}$. For the case $\theta \in \mathrm{\bar{B}}(0, R)$, $\theta' \notin \mathrm{\bar{B}}(0, R)$, i.e. $|\theta|\leq R, |\theta'| > R$, it is straightforward to see that the inequality \eqref{oslosball} holds when $|\theta|=R, |\theta'| > R$; moreover, for $|\theta|<R, |\theta'| >R$, there exists a unique $\bar{\theta} \in \R^d$ which lies at the intersection of the boundary of the ball $\mathrm{\bar{B}}(0, R)$ and the line segment between $\theta, \theta'$, such that $\theta - \bar{\theta} = c_{\theta, \theta'}(\theta - \theta')$ and $ \bar{\theta}-\theta' = (1-c_{\theta, \theta'})(\theta - \theta')$, where $c_{\theta, \theta'} \in (0,1)$. Then, one obtains, for $|\theta|<R, |\theta'| >R$
\begin{align}\label{oslint}
&\langle \theta-\theta',h(\theta)-h(\theta')\rangle  \nonumber \\
& \geq \langle \theta-\bar{\theta},h(\theta)-h(\bar{\theta})\rangle + \langle \theta-\bar{\theta},h(\bar{\theta}) - h(\theta')\rangle  +\langle \bar{\theta}-\theta',h(\theta)-h(\bar{\theta})\rangle +\langle \bar{\theta}-\theta',h(\bar{\theta})-h(\theta')\rangle \nonumber \\
&= \langle \theta-\theta',h(\theta)-h(\bar{\theta})\rangle  +\left(\frac{c_{\theta, \theta'}}{1-c_{\theta, \theta'}}+1\right) \langle \bar{\theta}-\theta',h(\bar{\theta})-h(\theta')\rangle,
\end{align}
where the equality above is obtained by using $\theta - \bar{\theta} = c_{\theta, \theta'}( \bar{\theta}-\theta')/(1-c_{\theta, \theta'})$. One notes that the second term on the RHS of \eqref{oslint} is greater or equal to zero due to \eqref{oslosball}. Thus, it follows that
\begin{align*}
\langle \theta-\theta',h(\theta)-h(\theta')\rangle
& \geq \langle \theta-\theta',h(\theta)-h(\bar{\theta})\rangle \\
& = \frac{1}{c_{\theta, \theta'}}\langle \theta-\bar{\theta},h(\theta)-h(\bar{\theta})\rangle \\
&\geq -\frac{L_R}{c_{\theta, \theta'}}|\theta-\bar{\theta}|^2\\
& =  - c_{\theta, \theta'}L_R|\theta-\theta'|^2 \\
&\geq -L_R|\theta - \theta'|^2,
\end{align*}
where the second inequality holds due to \eqref{oslball}, and the last inequality holds due to $c_{\theta, \theta'} \in (0,1)$. Applying the same arguments to the case $\theta \notin \mathrm{\bar{B}}(0, R)$, $\theta' \in \mathrm{\bar{B}}(0, R)$ completes the proof.
\end{proof}
\begin{proof}[\textbf{Proof of statement in Remark \ref{lambdapmax}}] \label{prooflambdapmax}
For any $p\in {\N}$, recall the definition of $\lambda_{p,\max}$ given in \eqref{stepsizemax}. It suffices to show
\begin{equation}\label{stepsizemaxpterm}
\frac{\min\{(a_F/K_F)^2, (a_F/K_F)^{2/(2p-1)}\}}{9\binom{2p}{p}^2K_F^2(\E\left[(1+|X_0|)^{2p\rho  }\right])^2}
\end{equation}
decreases as $p$ increases. To this end, one first observes that the denominator of the fraction in \eqref{stepsizemaxpterm} is positive and it increases as $p$ increases. Then, for $0<a_F/K_F \leq 1$, one obtains that, for any $p\in {\N}$,
\[
\min\{(a_F/K_F)^2, (a_F/K_F)^{2/(2p-1)}\} = (a_F/K_F)^2>0;
\]
whereas for $ a_F/K_F \geq 1$, we have, for any $p\in {\N}$,
\[
\min\{(a_F/K_F)^2, (a_F/K_F)^{2/(2p-1)}\} = (a_F/K_F)^{2/(2p-1)}>0,
\]
which decreases as $p$ increases. Thus, one concludes that \eqref{stepsizemaxpterm} decreases as $p$ increases.
\end{proof}
\subsection{Proof of auxiliary results in Section \ref{po}}\label{proofpo}
\begin{lemma}\label{tcsdemmtbd} Let Assumption \ref{AI}, \ref{AG}, \ref{AF}, and \ref{AC} hold. Then, one obtains, for any $p \in {\N}, t \geq 0$,
\[
\E[|Z^{\lambda}_t|^{2p}] \leq e^{- \lambda pa_h t }\E[|\theta_0|^{2p}]+2(b_h+\beta^{-1}(d+2(p-1))) M_5^{2p-2}/a_h,
\]
where $M_5 :=(2(b_h+\beta^{-1}(d+2(p-1)))/a_h)^{1/2}$. In particular, the above result further implies
\[
\sup_{t\geq 0} \E[|Z^{\lambda}_t|^{2p}]\leq \E[|\theta_0|^{2p}]+2(b_h+\beta^{-1}(d+2(p-1))) M_5^{2p-2}/a_h.
\]
\end{lemma}
\begin{proof}
Consider the function $f(z) := |z|^{2p}$, $z \in \R^d$, $p \in \N$. Denote by $\nabla f$ and $\nabla^2 f$ the gradient and the Hessian of $f$, respectively. One observes that, for any $z \in \R^d$, $p \in \N$, $\nabla f(z) = 2pz|z|^{2p-2}$ and $\nabla^2 f(z)=2p |z|^{2p-2}I_d+4p(p-1)|z|^{2p-4}zz^{\mathsf{T}}$ with $I_d$ denoting the identity matrix and $z^{\mathsf{T}}$ denoting the transpose of $z$. Recall the definition of $(Z^{\lambda}_t)_{ t\geq 0}$ given in \eqref{tcsde}. For any $t\geq 0$, by applying It\^o's formula to $f(Z^{\lambda}_t)= |Z^{\lambda}_t|^{2p}$, one obtains, almost surely
\begin{align*}
\rmd  |Z^{\lambda}_t|^{2p}
& = \left[-\langle \nabla f(Z_t), \lambda h(Z^{\lambda}_t)\rangle   +\frac{1}{2}\Tr\left((\sqrt{2\lambda\beta^{-1}}I_d)^{\mathsf{T}}\nabla^2 f(Z_t) (\sqrt{2\lambda\beta^{-1}}I_d)\right)\right]dt\\
&\quad +\langle \nabla f(Z_t), \sqrt{2\lambda\beta^{-1}}\,\rmd B^{\lambda}_t\rangle\\
& =  - 2p\lambda   \langle Z^{\lambda}_t, h(Z^{\lambda}_t) \rangle |Z^{\lambda}_t|^{2p-2} \rmd t + 2p \langle Z^{\lambda}_t, \sqrt{2\lambda \beta^{-1}}d B^{\lambda}_t \rangle |Z^{\lambda}_t|^{2p-2}\\
&\quad +2p\lambda\beta^{-1}  (d+2(p-1))|Z^{\lambda}_t|^{2p-2} \rmd t,
\end{align*}
where $\Tr(A)$ and $A^{\mathsf{T}}$ denote the trace and the transpose of a given matrix $A$, respectively. Then, integrating both sides and taking expectation yield
\begin{align*}
\E[|Z^{\lambda}_t|^{2p} ]
& = \E[|\theta_0|^{2p} ]  - 2p\lambda  \int_0^t   \E[ \langle Z^{\lambda}_s, h(Z^{\lambda}_s) \rangle |Z^{\lambda}_s|^{2p-2} ] \, \rmd s \\
&\quad + 2p\lambda\beta^{-1} (d+2(p-1)) \int_0^t  \E[|Z^{\lambda}_s|^{2p-2}] \, \rmd s,
\end{align*}
where the expectation of the stochastic integral is zero by applying standard stopping time arguments (see, e.g., the proof of Lemma \ref{zetaprocme}). This further implies by differentiating both sides and by using Remark \ref{ADFh},
\begin{align}\label{z2pdiffub1}
\frac{\rmd}{\rmd t}\E[|Z^{\lambda}_t|^{2p} ]
& =  - 2p\lambda  \E[ \langle Z^{\lambda}_t, h(Z^{\lambda}_t) \rangle |Z^{\lambda}_t|^{2p-2} ]   +2p \lambda\beta^{-1} (d+2(p-1))    \E[|Z^{\lambda}_t|^{2p-2}] \nonumber\\
&\leq  -2pa_h \lambda   \E[|Z^{\lambda}_t|^{2p}] + 2p\lambda (b_h+\beta^{-1}(d+2(p-1)))  \E[|Z^{\lambda}_t|^{2p-2}].
\end{align}
For any $\theta \in \R^d$, one notes that
\begin{align}\label{sde2pub}
\begin{split}
& -  pa_h\lambda  |\theta|^{2p} + 2p\lambda  (b_h+\beta^{-1}(d+2(p-1)))|\theta|^{2p-2} < 0 \\
&\Leftrightarrow \quad |\theta| >\left(\frac{2(b_h+\beta^{-1}(d+2(p-1)))}{a_h}\right)^{1/2}.
\end{split}
\end{align}
Denote by $M_5 :=(2(b_h+\beta^{-1}(d+2(p-1)))/a_h)^{1/2}$ and $\mathsf{S}_{t,M_5} := \{\omega \in \Omega: |Z^{\lambda}_t(\omega)| >M_5\}$. Then, by \eqref{z2pdiffub1} and \eqref{sde2pub}, it holds that
\begin{align*}
\frac{\rmd}{\rmd t}\E[|Z^{\lambda}_t|^{2p} ]
&\leq -2pa_h \lambda   \E[|Z^{\lambda}_t|^{2p}\1_{\mathsf{S}_{t,M_5} }] + 2p\lambda (b_h+\beta^{-1}(d+2(p-1)))  \E[|Z^{\lambda}_t|^{2p-2}\1_{\mathsf{S}_{t,M_5} }]\\
&\quad -2pa_h \lambda   \E[|Z^{\lambda}_t|^{2p}\1_{\mathsf{S}_{t,M_5}^{\mathsf{c}} }] + 2p\lambda (b_h+\beta^{-1}(d+2(p-1)))  \E[|Z^{\lambda}_t|^{2p-2}\1_{\mathsf{S}_{t,M_5}^{\mathsf{c}} }]\\
&\leq   - pa_h\lambda   \E[|Z^{\lambda}_t|^{2p} \1_{\mathsf{S}_{t,M_5} }]  -2p a_h\lambda  \E[|Z^{\lambda}_t|^{2p} \1_{\mathsf{S}_{t,M_5}^{\mathsf{c}} }] \\
&\quad+2p \lambda  (b_h+\beta^{-1}(d+2(p-1))) M_5^{2p-2}\\
&\leq  -pa_h\lambda  \E[|Z^{\lambda}_t|^{2p}]+2p \lambda (b_h+\beta^{-1}(d+2(p-1))) M_5^{2p-2}.
\end{align*}
This implies, by multiplying $e^{pa_h\lambda t}$ and by integrating on both sides of the above inequality, that
\[
\E[|Z^{\lambda}_t|^{2p}] \leq e^{- \lambda pa_h t }\E[|\theta_0|^{2p}]+2(b_h+\beta^{-1}(d+2(p-1))) M_5^{2p-2}/a_h.
\]
Finally, it follows that
\[
\sup_{t\geq 0} \E[|Z^{\lambda}_t|^{2p}] \leq  \E[|\theta_0|^{2p}]+2(b_h+\beta^{-1}(d+2(p-1))) M_5^{2p-2}/a_h,
\]
which completes the proof.
\end{proof}
\begin{proof}[\textbf{Proof of Lemma \ref{2ndpthmmt}-\ref{2ndpthmmti}}] \label{proof2ndpthmmt}
For any $0<\lambda\leq \lambda_{1, \max}$, $t\in (n, n+1]$, $n \in \N_0$, define
\begin{align}\label{delxinotation}
\begin{split}
\Delta_{n,t}^{\lambda}
 := \bar{\theta}^{\lambda}_n - \lambda H_{\lambda}(\bar{\theta}^{\lambda}_n, X_{n+1})(t-n),\quad
\Xi_{n,t}^{\lambda}
 := \sqrt{2\lambda \beta^{-1}}(B_t^{\lambda} - B_n^{\lambda}).
\end{split}
\end{align}
By using \eqref{tuslaproc}, one obtains,
\begin{equation}\label{2ndmmtexp}
\E\left[\left.|\bar{\theta}^{\lambda}_t|^2\right|\bar{\theta}^{\lambda}_n \right]  = \E\left[\left.|\Delta_{n,t}^{\lambda}|^2\right|\bar{\theta}^{\lambda}_n \right] +2\lambda (t-n)d/\beta.
\end{equation}
Moreover, by using Assumption \ref{AC} and Remark \ref{ADFh}, one further calculates, for any $0<\lambda <\lambda_{1,\max}$,
\begin{align*}
\E\left[\left.|\Delta_{n,t}^{\lambda}|^2\right|\bar{\theta}^{\lambda}_n \right]
& = | \bar{\theta}^{\lambda}_n|^2-2\lambda(t-n) \E\left[\left.\left\langle  \bar{\theta}^{\lambda}_n, \frac{G( \bar{\theta}^{\lambda}_n, X_{n+1})+ F( \bar{\theta}^{\lambda}_n, X_{n+1})}{1+\sqrt{\lambda}| \bar{\theta}^{\lambda}_n|^{2r}} \right\rangle\right|\bar{\theta}^{\lambda}_n \right] \\
&\quad + \lambda^2(t-n)^2\E\left[\left.\left|\frac{G( \bar{\theta}^{\lambda}_n, X_{n+1})+ F( \bar{\theta}^{\lambda}_n, X_{n+1})}{1+\sqrt{\lambda}| \bar{\theta}^{\lambda}_n|^{2r}}\right|^2\right|\bar{\theta}^{\lambda}_n \right] \\
&\leq  | \bar{\theta}^{\lambda}_n|^2 -\frac{2\lambda(t-n) (a_F|\bar{\theta}^{\lambda}_n|^{2r+2} - b_F)}{1+\sqrt{\lambda}| \bar{\theta}^{\lambda}_n|^{2r}}  +\frac{2\lambda(t-n)|\bar{\theta}^{\lambda}_n|\E\left[\left.|G( \bar{\theta}^{\lambda}_n, X_{n+1})|\right|\bar{\theta}^{\lambda}_n \right] }{1+\sqrt{\lambda}| \bar{\theta}^{\lambda}_n|^{2r}}\\
&\quad +2\lambda^2(t-n)^2\left(\frac{\E\left[\left.|G( \bar{\theta}^{\lambda}_n, X_{n+1})|^2\right|\bar{\theta}^{\lambda}_n \right] }{(1+\sqrt{\lambda}| \bar{\theta}^{\lambda}_n|^{2r})^2} +\frac{\E\left[\left.|F( \bar{\theta}^{\lambda}_n, X_{n+1})|^2\right|\bar{\theta}^{\lambda}_n \right] }{(1+\sqrt{\lambda}| \bar{\theta}^{\lambda}_n|^{2r})^2} \right).
\end{align*}
The above estimate further yields, by using Assumption \ref{AG}, \ref{AF}, the following bound:
\begin{align}
\begin{split}\label{2ndestineq1}
\E\left[\left.|\Delta_{n,t}^{\lambda}|^2\right|\bar{\theta}^{\lambda}_n \right]
& \leq | \bar{\theta}^{\lambda}_n|^2- \lambda(t-n)\frac{2 a_F|\bar{\theta}^{\lambda}_n|^{2r+2}}{1+\sqrt{\lambda}| \bar{\theta}^{\lambda}_n|^{2r}} + 2\lambda(t-n) b_F \\
&\quad +\lambda(t-n)\frac{2^{q+1}K_G\E\left[ (1+|X_0|)^{\rho}\right](1+|\bar{\theta}^{\lambda}_n|^{q+1}) }{1+\sqrt{\lambda}| \bar{\theta}^{\lambda}_n|^{2r}}  \\
&\quad +\lambda^2(t-n)^2\frac{2^{2q}K_G^2\E\left[(1+|X_0|)^{2\rho}\right](1+|\bar{\theta}^{\lambda}_n|^{2q}) }{(1+\sqrt{\lambda}| \bar{\theta}^{\lambda}_n|^{2r})^2} \\
&\quad +\lambda^2(t-n)^2\frac{4K_F^2\E\left[(1+|X_0|)^{2\rho}\right] (1+|\bar{\theta}^{\lambda}_n|^{4r+2})}{(1+\sqrt{\lambda}| \bar{\theta}^{\lambda}_n|^{2r})^2}.
\end{split}
\end{align}
Moreover, one notes that, the fifth term on the RHS of \eqref{2ndestineq1} can be upper bounded using the following inequality: for any $\theta \in \R^d, r \geq q/2$, $0<\lambda <\lambda_{1,\max}<1$,
\[
\frac{\lambda(1+|\theta|^{2q})}{(1+\sqrt{\lambda}|\theta|^{2r})^2} \leq \frac{\lambda(1+|\theta|^{2q})}{ 1+\lambda|\theta|^{4r} } \leq  \frac{\lambda+\lambda(1+|\theta|^{4r})}{ 1+\lambda|\theta|^{4r} }\leq \frac{2+\lambda|\theta|^{4r}}{ 1+\lambda|\theta|^{4r} } \leq 2.
\]
This and \eqref{2ndestineq1} imply that
\begin{align}
\E\left[\left.|\Delta_{n,t}^{\lambda}|^2\right|\bar{\theta}^{\lambda}_n \right]
& \leq | \bar{\theta}^{\lambda}_n|^2- \lambda(t-n)\frac{2 a_F|\bar{\theta}^{\lambda}_n|^{2r+2}}{1+\sqrt{\lambda}| \bar{\theta}^{\lambda}_n|^{2r}} + 2\lambda(t-n) b_F \nonumber \\
&\quad +\lambda(t-n)\frac{2^{q+1}K_G\E\left[ (1+|X_0|)^{\rho}\right] |\bar{\theta}^{\lambda}_n|^{q+1}  }{1+\sqrt{\lambda}| \bar{\theta}^{\lambda}_n|^{2r}}  + 2^{q+1}\lambda(t-n)K_G \E\left[ (1+|X_0|)^{\rho}\right]  \nonumber\\
&\quad + 2^{2q+1}\lambda (t-n)^2K_G^2\E\left[(1+|X_0|)^{2\rho}\right]   \nonumber\\
&\quad +\lambda^2(t-n)^2\frac{4K_F^2\E\left[(1+|X_0|)^{2\rho}\right]  |\bar{\theta}^{\lambda}_n|^{4r+2}}{(1+\sqrt{\lambda}| \bar{\theta}^{\lambda}_n|^{2r})^2}   +4\lambda^2(t-n)^2K_F^2\E\left[(1+|X_0|)^{2\rho}\right]  \nonumber\\
\begin{split}\label{2ndestint}
&= | \bar{\theta}^{\lambda}_n|^2 - \lambda(t-n)|\bar{\theta}^{\lambda}_n|^2T_1^{\lambda}(\bar{\theta}^{\lambda}_n) -\lambda(t-n) T_2^{\lambda}(\bar{\theta}^{\lambda}_n)\\
\end{split}\\
&\quad + 2\lambda(t-n) b_F + 2^{q+1}\lambda(t-n)K_G\E\left[ (1+|X_0|)^{\rho}\right] \nonumber\\
&\quad +2^{2q+1}\lambda (t-n)^2K_G^2\E\left[(1+|X_0|)^{2\rho}\right]  +4\lambda^2(t-n)^2K_F^2\E\left[(1+|X_0|)^{2\rho}\right] \nonumber,
\end{align}
where for all $\theta \in \R^d \setminus \{\underbrace{(0,\dots, 0)}_\text{$d$}\}$,
\begin{equation}\label{defT1}
T_1^{\lambda}(\theta) :=\frac{1}{|\theta|^2}\left( \frac{ a_F|\theta|^{2r+2}}{1+\sqrt{\lambda}|\theta|^{2r}} - \frac{2^{q+1}K_G\E\left[ (1+|X_0|)^{\rho}\right]|\theta|^{q+1}}{1+\sqrt{\lambda}|\theta|^{2r}}\right),
\end{equation}
and moreover, for all $\theta \in \R^d$
\[
T_2^{\lambda}(\theta) :=  \frac{ a_F|\theta|^{2r+2}}{1+\sqrt{\lambda}|\theta|^{2r}} - \frac{4\lambda(t-n)K_F^2\E\left[(1+|X_0|)^{2\rho}\right] |\theta|^{4r+2}}{(1+\sqrt{\lambda}|\theta|^{2r})^2}.
\]
Then, for all $\theta \in \R^d$
\begin{align*}
&a_F|\theta|^{2r+2} - 2^{q+1}K_G\E\left[ (1+|X_0|)^{\rho}\right]|\theta|^{q+1} > \frac{a_F}{2}|\theta|^{2r+2}  \\
& \Leftrightarrow \quad |\theta| > \left(\frac{2^{q+2}K_G\E\left[ (1+|X_0|)^{\rho}\right]}{a_F}\right)^{1/(2r-q+1)}.
\end{align*}
Denote by $M_0 = (2^{q+2}K_G\E\left[(1+|X_0|)^{2\rho}\right]/\min\{1,a_F\} )^{1/(2r-q+1)}$. Then, for all $|\theta|>M_0$, by using the inequalities above, one obtains
\begin{equation}\label{lbT1}
T_1^{\lambda}(\theta) >  \frac{a_F|\theta|^{2r}}{2(1+\sqrt{\lambda}|\theta|^{2r})} \geq \frac{a_FM_0^{2r}}{2(1+M_0^{2r})},
\end{equation}
where the last inequality holds due to $0<\lambda \leq \lambda_{1,\max}\leq 1$ and the fact that $f(s) := s/(1+\sqrt{\lambda}s)$ is non-decreasing for all $s\geq 0$. Furthermore, one observes that for all $\theta \in \R^d$ and for all $\lambda \leq \lambda_{1,\max} \leq a_F^2/(16K_F^4(\E\left[(1+|X_0|)^{2\rho}\right])^2)$,
\begin{align}\label{lbT2}
T_2^{\lambda}(\theta)
&= \frac{a_F|\theta|^{2r+2} + \sqrt{\lambda}a_F|\theta|^{4r+2} - 4\lambda(t-n)K_F^2\E\left[(1+|X_0|)^{2\rho}\right] |\theta|^{4r+2}}{(1+\sqrt{\lambda}|\theta|^{2r})^2} \nonumber\\
& \geq  \frac{  \sqrt{\lambda}a_F|\theta|^{4r+2} - 4\lambda K_F^2\E\left[(1+|X_0|)^{2\rho}\right] |\theta|^{4r+2}}{(1+\sqrt{\lambda}|\theta|^{2r})^2}\nonumber\\
&\geq 0.
\end{align}
Denote by $\kappa := M_0^{2r}/(2(1+M_0^{2r}))$, and $\mathsf{S}_{n,M_0} := \{\omega \in \Omega: |\bar{\theta}^{\lambda}_n(\omega)| >M_0\}$. Inserting \eqref{lbT1}, \eqref{lbT2} into \eqref{2ndestint} yields, for $0<\lambda \leq \lambda_{1,\max} $,
\begin{align*}
\E\left[\left.|\Delta_{n,t}^{\lambda}|^2 \1_{\mathsf{S}_{n,M_0} }\right|\bar{\theta}^{\lambda}_n \right]
& \leq (1-\lambda(t-n)a_F\kappa)|\bar{\theta}^{\lambda}_n|^2  \1_{\mathsf{S}_{n,M_0} } +\lambda(t-n)c_1  \1_{\mathsf{S}_{n,M_0} },
\end{align*}
where $c_1 := 2b_F +2^{q+1} K_G\E\left[ (1+|X_0|)^{\rho}\right]+2^{2q+1} K_G^2\E\left[(1+|X_0|)^{2\rho}\right] +4 K_F^2\E\left[(1+|X_0|)^{2\rho}\right] $. In addition, one obtains, by using the definition of $T_1^{\lambda}(\theta)$ given in \eqref{defT1},
\begin{align*}
\E\left[\left.|\Delta_{n,t}^{\lambda}|^2 \1_{\mathsf{S}_{n,M_0}^{\mathsf{c}} }\right|\bar{\theta}^{\lambda}_n \right]
& \leq (1-\lambda(t-n)a_F\kappa)|\bar{\theta}^{\lambda}_n|^2  \1_{S_{n,M_0}^{\mathsf{c}} } +\lambda(t-n)c_1  \1_{\mathsf{S}_{n,M_0}^{\mathsf{c}} }\\
&\quad +\lambda(t-n) (a_F\kappa M_0^2 + 2^{q+1} K_G\E\left[ (1+|X_0|)^{\rho}\right]M_0^{q+1})\1_{\mathsf{S}_{n,M_0}^{\mathsf{c}} }.
\end{align*}
Combining the two cases yields
\begin{equation}\label{delta2}
\E\left[\left.|\Delta_{n,t}^{\lambda}|^2 \right|\bar{\theta}^{\lambda}_n \right]  \leq (1-\lambda(t-n)a_F\kappa)|\bar{\theta}^{\lambda}_n|^2    +\lambda(t-n)c_2,
\end{equation}
where
\begin{equation}\label{constc2}
c_2 := c_1 + a_F\kappa M_0^2 + 2^{q+1} K_G\E\left[ (1+|X_0|)^{\rho}\right]M_0^{q+1}.
\end{equation}
Thus, one can conclude from \eqref{2ndmmtexp} that, for $t \in (n, n+1], n \in \N_0$, $0<\lambda \leq \lambda_{1,\max} $,
\begin{align*}
\E\left[\left.|\bar{\theta}^{\lambda}_t|^2\right|\bar{\theta}^{\lambda}_n \right]
&=\E\left[\left.|\Delta_{n,t}^{\lambda}|^2 \right|\bar{\theta}^{\lambda}_n \right]  +2\lambda (t-n)d/\beta \leq  (1-\lambda(t-n)a_F\kappa)|\bar{\theta}^{\lambda}_n|^2 + \lambda(t-n)c_0,
\end{align*}
where
\begin{align}\label{constc0}
\begin{split}
\kappa
&:= M_0^{2r}/(2(1+M_0^{2r})),\\
M_0
&:= (2^{q+2}K_G\E\left[(1+|X_0|)^{2\rho}\right]/\min\{1,a_F\})^{1/(2r-q+1)},\\
c_0
&:= 2d/\beta +a_F\kappa M_0^2 + 2^{q+1} K_G\E\left[ (1+|X_0|)^{\rho}\right]M_0^{q+1} + 2b_F \\
&\quad+2^{q+1} K_G\E\left[ (1+|X_0|)^{\rho}\right]+2^{2q+1} K_G^2\E\left[(1+|X_0|)^{2\rho}\right] +4 K_F^2\E\left[(1+|X_0|)^{2\rho}\right].
\end{split}
\end{align}
This further implies, for $t \in (n, n+1], n \in \N_0$, $0<\lambda \leq \lambda_{1,\max} \leq 1$, that
\begin{align}\label{2ndmmttp}
\begin{split}
\E\left[ |\bar{\theta}^{\lambda}_t|^2  \right]
& \leq (1-\lambda(t-n)a_F\kappa) \E\left[|\bar{\theta}^{\lambda}_n|^2 \right]   + \lambda(t-n)c_0 \\
& \leq (1-\lambda(t-n)a_F\kappa) (1-\lambda a_F\kappa) \E\left[|\bar{\theta}^{\lambda}_{n-1}|^2 \right] +\lambda_{1,\max} c_0 +\lambda c_0\\
& \leq (1-\lambda(t-n)a_F\kappa) (1-\lambda a_F\kappa)^2 \E\left[|\bar{\theta}^{\lambda}_{n-2}|^2 \right] +  c_0 +\lambda c_0(1+(1-\lambda a_F\kappa))\\
&\leq \dots \\
&\leq  (1-\lambda(t-n)a_F\kappa)(1-a_F\kappa \lambda)^n \E\left[|\theta_0|^2\right]  +c_0(1+1/(a_F\kappa)),
\end{split}
\end{align}
which completes the proof.
\end{proof}

\begin{proof}[\textbf{Proof of Lemma \ref{2ndpthmmt}-\ref{2ndpthmmtii}}]
For any $p\in [2, \infty)\cap {\N}$, $0<\lambda\leq \lambda_{p,\max}$, $t\in (n, n+1]$, $n \in \N_0$, recall the definition for $\Delta_{n,t}^{\lambda}$ and $\Xi_{n,t}^{\lambda}$ in \eqref{delxinotation}. To obtain the $2p$-th moment estimate (with $p\in [2, \infty)\cap {\N}$) of the TUSLA algorithm \eqref{tuslaproc}, one writes
\begin{align*}
\E\left[\left.|\bar{\theta}^{\lambda}_t|^{2p}\right|\bar{\theta}^{\lambda}_n \right]
& = \E\left[\left.\left(|\Delta_{n,t}^{\lambda}|^2 +2\langle \Delta_{n,t}^{\lambda}, \Xi_{n,t}^{\lambda} \rangle +|\Xi_{n,t}^{\lambda}|^2\right)^p\right|\bar{\theta}^{\lambda}_n \right]  \\
& = \E\left[\left.|\Delta_{n,t}^{\lambda}|^{2p}\right|\bar{\theta}^{\lambda}_n \right]   + 2p \E\left[\left.|\Delta_{n,t}^{\lambda}|^{2p-2} \langle \Delta_{n,t}^{\lambda}, \Xi_{n,t}^{\lambda} \rangle\right|\bar{\theta}^{\lambda}_n \right]  \\
&\quad + \E\left[\left.\sum_{\substack{k_1+k_2+k_3=p\\ \{k_1\neq p-1\} \cap \{k_2 \neq 1\}\\ \{k_1 \neq p\} }} \frac{p!}{k_1!k_2!k_3!}|\Delta_{n,t}^{\lambda}|^{2k_1}(2\langle \Delta_{n,t}^{\lambda}, \Xi_{n,t}^{\lambda} \rangle)^{k_2}|\Xi_{n,t}^{\lambda}|^{2k_3} \right|\bar{\theta}^{\lambda}_n \right]       \\
& \leq \E\left[\left.|\Delta_{n,t}^{\lambda}|^{2p}\right|\bar{\theta}^{\lambda}_n \right]  + 2p \E\left[\left.|\Delta_{n,t}^{\lambda}|^{2p-2} \langle \Delta_{n,t}^{\lambda}, \Xi_{n,t}^{\lambda} \rangle\right|\bar{\theta}^{\lambda}_n \right]  \\
&\quad + \sum_{k = 2}^{2p}\binom{2p}{k}\E\left[\left.|\Delta_{n,t}^{\lambda}|^{2p-k} |\Xi_{n,t}^{\lambda}|^k \right|\bar{\theta}^{\lambda}_n \right] \\
& =  \E\left[\left.|\Delta_{n,t}^{\lambda}|^{2p}\right|\bar{\theta}^{\lambda}_n \right]  +\sum_{k = 2}^{2p}\binom{2p}{k}\E\left[\left.|\Delta_{n,t}^{\lambda}|^{2p-k} |\Xi_{n,t}^{\lambda}|^k \right|\bar{\theta}^{\lambda}_n \right].
\end{align*}
where the inequality above holds due to \cite[Lemma A.3]{nonconvex}. This and the fact that $\Xi_{n,t}^{\lambda}$ is independent of $\Delta_{n,t}^{\lambda}$, and $\Xi_{n,t}^{\lambda}$ is independent of $\bar{\theta}^{\lambda}_n$, $t\in (n, n+1]$, $n \in \N_0$, yield
\begin{align}\label{2pthme}
\E\left[\left.|\bar{\theta}^{\lambda}_t|^{2p}\right|\bar{\theta}^{\lambda}_n \right]
& \leq \E\left[\left.|\Delta_{n,t}^{\lambda}|^{2p}\right|\bar{\theta}^{\lambda}_n \right] +\sum_{l = 0}^{2p-2}\binom{2p}{l+2}\E\left[\left.|\Delta_{n,t}^{\lambda}|^{2p-2-l} |\Xi_{n,t}^{\lambda}|^{l+2} \right|\bar{\theta}^{\lambda}_n \right] \nonumber\\
& =\E\left[\left.|\Delta_{n,t}^{\lambda}|^{2p}\right|\bar{\theta}^{\lambda}_n \right]  +\sum_{l = 0}^{2p-2} \frac{2p(2p-1)}{(l+2)(l+1)} \binom{2p-2}{l}\E\left[\left.\left(|\Delta_{n,t}^{\lambda}|^{2p-2-l} |\Xi_{n,t}^{\lambda}|^l\right)|\Xi_{n,t}^{\lambda}|^2 \right|\bar{\theta}^{\lambda}_n \right]  \nonumber\\
&\leq \E\left[\left.|\Delta_{n,t}^{\lambda}|^{2p}\right|\bar{\theta}^{\lambda}_n \right]  +p(2p-1)\E\left[\left.\left(|\Delta_{n,t}^{\lambda}|+|\Xi_{n,t}^{\lambda}|\right)^{2p-2} |\Xi_{n,t}^{\lambda}|^2 \right|\bar{\theta}^{\lambda}_n \right] \nonumber \\
&\leq \E\left[\left.|\Delta_{n,t}^{\lambda}|^{2p}\right|\bar{\theta}^{\lambda}_n \right]  +2^{2p-3}p(2p-1)\E\left[\left. |\Delta_{n,t}^{\lambda}|^{2p-2}   \right|\bar{\theta}^{\lambda}_n \right]\E\left[  |\Xi_{n,t}^{\lambda}|^2 \right] \nonumber \\
&\quad + 2^{2p-3}p(2p-1)\E\left[  |\Xi_{n,t}^{\lambda}|^{2p} \right] \nonumber \\
\begin{split}
&\leq \E\left[\left.|\Delta_{n,t}^{\lambda}|^{2p}\right|\bar{\theta}^{\lambda}_n \right]  +2^{2p-2}p(2p-1)\lambda (t-n)d \beta^{-1}\E\left[\left. |\Delta_{n,t}^{\lambda}|^{2p-2}   \right|\bar{\theta}^{\lambda}_n \right]  \\
&\quad +2^{2p-4}(2p(2p-1))^{p+1}(d\beta^{-1}\lambda (t-n) )^p,
\end{split}
\end{align}
where the last inequality holds due to \cite[Theorem 7.1]{mao2007stochastic}. The first term in \eqref{2pthme} can be upper bounded in the following way:
\begin{align*}
\E\left[\left.|\Delta_{n,t}^{\lambda}|^{2p}\right|\bar{\theta}^{\lambda}_n \right]
& = \E\left[\left.\left(|\bar{\theta}^{\lambda}_n|^2 -2\lambda(t-n)\langle \bar{\theta}^{\lambda}_n, H_{\lambda}(\bar{\theta}^{\lambda}_n, X_{n+1})\rangle + |\lambda H_{\lambda}(\bar{\theta}^{\lambda}_n, X_{n+1})(t-n)|^2\right)^p\right|\bar{\theta}^{\lambda}_n \right]  \\
&= |\bar{\theta}^{\lambda}_n|^{2p} -2p\lambda(t-n) |\bar{\theta}^{\lambda}_n|^{2p-2} \E\left[\left.\langle \bar{\theta}^{\lambda}_n, H_{\lambda}(\bar{\theta}^{\lambda}_n, X_{n+1})\rangle \right|\bar{\theta}^{\lambda}_n \right]  \\
&\quad +\sum_{\substack{k_1+k_2+k_3=p\\ \{k_1\neq p-1\} \cap \{k_2 \neq 1\}\\ \{k_1 \neq p\} }} \frac{p!}{k_1!k_2!k_3!} \E\left[\left.|\bar{\theta}^{\lambda}_n|^{2k_1}( -2\lambda(t-n)\langle \bar{\theta}^{\lambda}_n, H_{\lambda}(\bar{\theta}^{\lambda}_n, X_{n+1})\rangle)^{k_2}\right.\right.\\
&\qquad \times \left.\left.|\lambda H_{\lambda}(\bar{\theta}^{\lambda}_n, X_{n+1})(t-n)|^{2k_3} \right|\bar{\theta}^{\lambda}_n \right]       \\
& \leq  |\bar{\theta}^{\lambda}_n|^{2p} -2p\lambda(t-n) |\bar{\theta}^{\lambda}_n|^{2p-2} \E\left[\left.\langle \bar{\theta}^{\lambda}_n, H_{\lambda}(\bar{\theta}^{\lambda}_n, X_{n+1})\rangle \right|\bar{\theta}^{\lambda}_n \right]  \\
&\quad + \sum_{k = 2}^{2p}\binom{2p}{k}\E\left[\left.|\bar{\theta}^{\lambda}_n|^{2p-k} |\lambda H_{\lambda}(\bar{\theta}^{\lambda}_n, X_{n+1})(t-n)|^k \right|\bar{\theta}^{\lambda}_n \right] ,
\end{align*}
where we apply \cite[Lemma A.3]{nonconvex} to obtain the last inequality above. Moreover, by Assumption \ref{AG} and Remark \ref{ADFh}, the above estimate further yields
\begin{align}\label{binomialterm}
 \E\left[\left.|\Delta_{n,t}^{\lambda}|^{2p}\right|\bar{\theta}^{\lambda}_n \right] 
& \leq  |\bar{\theta}^{\lambda}_n|^{2p} -2p\lambda(t-n) |\bar{\theta}^{\lambda}_n|^{2p-2} \E\left[\left.\left\langle  \bar{\theta}^{\lambda}_n, \frac{G( \bar{\theta}^{\lambda}_n, X_{n+1})+ F( \bar{\theta}^{\lambda}_n, X_{n+1})}{1+\sqrt{\lambda}| \bar{\theta}^{\lambda}_n|^{2r}} \right\rangle\right|\bar{\theta}^{\lambda}_n \right]  \nonumber  \\
&\quad + \sum_{k = 2}^{2p}\binom{2p}{k}\E\left[\left.|\bar{\theta}^{\lambda}_n|^{2p-k} |\lambda H_{\lambda}(\bar{\theta}^{\lambda}_n, X_{n+1})(t-n)|^k \right|\bar{\theta}^{\lambda}_n \right] \nonumber \\
&\leq  |\bar{\theta}^{\lambda}_n|^{2p} -2p\lambda(t-n) |\bar{\theta}^{\lambda}_n|^{2p-2}\frac{(a_F|\bar{\theta}^{\lambda}_n|^{2r+2} - b_F)}{1+\sqrt{\lambda}| \bar{\theta}^{\lambda}_n|^{2r}} \nonumber  \\
&\quad +\frac{2^{q+1}p\lambda(t-n)K_G\E\left[ (1+|X_0|)^{\rho}\right]|\bar{\theta}^{\lambda}_n|^{2p-2}(1+|\bar{\theta}^{\lambda}_n|^{q+1}) }{1+\sqrt{\lambda}| \bar{\theta}^{\lambda}_n|^{2r}} \nonumber \\
&\quad + \sum_{k = 2}^{2p}\binom{2p}{k}\E\left[\left.|\bar{\theta}^{\lambda}_n|^{2p-k} |\lambda H_{\lambda}(\bar{\theta}^{\lambda}_n, X_{n+1})(t-n)|^k \right|\bar{\theta}^{\lambda}_n \right].
\end{align}
Next, for any $\theta \in \R^d, x \in \R^m$, denote by
\[
G_{\lambda}(\theta, x) := \frac{G(\theta, x)}{1+\sqrt{\lambda}|\theta|^{2r}}, \quad F_{\lambda}(\theta, x) := \frac{F(\theta, x)}{1+\sqrt{\lambda}|\theta|^{2r}}
\]
To obtain the optimal stepsize restriction $\lambda_{p,\max}$, we estimate the term in \eqref{binomialterm} using Assumption \ref{AG}, \ref{AF} as follows:
\begin{align}\label{2pmmtes1}
& \sum_{k = 2}^{2p}\binom{2p}{k}\E\left[\left.|\bar{\theta}^{\lambda}_n|^{2p-k} |\lambda H_{\lambda}(\bar{\theta}^{\lambda}_n, X_{n+1})(t-n)|^k \right|\bar{\theta}^{\lambda}_n \right]\nonumber\\
& \leq \sum_{k = 2}^{2p}\binom{2p}{k}|\bar{\theta}^{\lambda}_n|^{2p-k}\lambda^k(t-n)^k \E\left[\left.( |G_{\lambda}(\bar{\theta}^{\lambda}_n, X_{n+1})| +|F_{\lambda}(\bar{\theta}^{\lambda}_n, X_{n+1})|)^k \right|\bar{\theta}^{\lambda}_n \right]\nonumber\\
\begin{split}
& = \sum_{k = 2}^{2p}\binom{2p}{k}|\bar{\theta}^{\lambda}_n|^{2p-k}\lambda^k(t-n)^k \E\left[\left. \sum_{l=0}^{k-1}\binom{k}{l} |G_{\lambda}(\bar{\theta}^{\lambda}_n, X_{n+1}) |^{k-l}|F_{\lambda}(\bar{\theta}^{\lambda}_n, X_{n+1})|^l \right|\bar{\theta}^{\lambda}_n \right]\\
&\quad +\sum_{k = 2}^{2p}\binom{2p}{k}|\bar{\theta}^{\lambda}_n|^{2p-k}\lambda^k(t-n)^k \E\left[\left. |F_{\lambda}(\bar{\theta}^{\lambda}_n, X_{n+1})|^k \right|\bar{\theta}^{\lambda}_n \right].
\end{split}
\end{align}
By Assumption \ref{AG}, \ref{AF}, one notes that for any $\theta \in \R^d, m \in \R^m$, $2 \leq k \leq 2p$,
\begin{align}\label{2pmmtes2}
&\sum_{l=0}^{k-1}\binom{k}{l} |G_{\lambda}(\theta, x)|^{k-l}|F_{\lambda}(\theta, x)|^l \nonumber\\
&= \sum_{l=1}^{k-1}\binom{k}{l} |G_{\lambda}(\theta, x)|^{k-l}|F_{\lambda}(\theta, x)|^l +|G_{\lambda}(\theta, x)|^k \nonumber\\
& \leq  \sum_{l=1}^{k-1}\binom{k}{l} \frac{K_G^{k-l}(1+|x|)^{\rho(k-l)}(1+|\theta|)^{q(k-l)}}{(1+\sqrt{\lambda}|\theta|^{2r})^{k-l}}\times\frac{K_F^l(1+|x|)^{\rho l}(1+|\theta|^{2r+1})^l}{(1+\sqrt{\lambda}|\theta|^{2r})^l}\nonumber \\
& \quad +\frac{K_G^{k}(1+|x|)^{\rho k}(1+|\theta|)^{qk}}{(1+\sqrt{\lambda}|\theta|^{2r})^k}\nonumber\\
& \leq  \sum_{l=1}^{k-1}\binom{k}{l} 2^{qk-ql+l-2}K_G^{k-l}K_F^l (1+|x|)^{\rho k}\frac{1+|\theta|^{q(k-l)}}{(1+\sqrt{\lambda}|\theta|^{2r})^{k-l}}\times\frac{ 1+|\theta|^{(2r+1)l}}{(1+\sqrt{\lambda}|\theta|^{2r})^l}\nonumber \\
& \quad +2^{qk-1}K_G^{k}(1+|x|)^{\rho k}\frac{ 1+|\theta|^{qk}}{(1+\sqrt{\lambda}|\theta|^{2r})^k}\nonumber\\
& \leq  \sum_{l=1}^{k-1}\binom{k}{l} \lambda^{-k/2}2^{qk-ql+l-2}K_G^{k-l}K_F^l (1+|x|)^{\rho k}  \frac{1+\lambda^{(k-l)/2}|\theta|^{q(k-l)}}{1+\lambda^{(k-l)/2}|\theta|^{2r(k-l)} }\times\frac{ 1+\lambda^{l/2}|\theta|^{(2r+1)l}}{1+ \lambda^{l/2}|\theta|^{2rl}}\nonumber \\
& \quad +\lambda^{-k/2}2^{qk-1}K_G^{k}(1+|x|)^{\rho k}\frac{ 1+\lambda^{k/2}|\theta|^{qk}}{1+\lambda^{k/2}|\theta|^{2rk} }\nonumber\\
&\leq \sum_{l=1}^{k-1}\binom{k}{l} \lambda^{-k/2}2^{qk}K_G^{k-l}K_F^l (1+|x|)^{\rho k}(1+|\theta|)^l+\lambda^{-k/2}2^{qk}K_G^{k}(1+|x|)^{\rho k}\nonumber\\
& =  \sum_{l=0}^{k-1}\binom{k}{l} \lambda^{-k/2}2^{qk}K_G^{k-l}K_F^l (1+|x|)^{\rho k}(1+|\theta|)^l,
\end{align}
where the third inequality holds due to $(u+v)^s \geq u^s+v^s$ for $u, v \geq 0, s \geq 1$ and $\lambda \leq \lambda_{p,\max} \leq 1$, while the last inequality holds due to the following inequalities, for $2r \geq q \geq 1$,
\[
\frac{1+\lambda^{(k-l)/2}|\theta|^{q(k-l)}}{1+\lambda^{(k-l)/2}|\theta|^{2r(k-l)} } \leq 2, \quad \frac{ 1+\lambda^{l/2}|\theta|^{(2r+1)l}}{1+ \lambda^{l/2}|\theta|^{2rl}} \leq 2(1+|\theta|)^l.
\]
Inserting \eqref{2pmmtes2} into \eqref{2pmmtes1} together with Assumption \ref{AG}, \ref{AF} yields
\begin{align*}
& \sum_{k = 2}^{2p}\binom{2p}{k}\E\left[\left.|\bar{\theta}^{\lambda}_n|^{2p-k} |\lambda H_{\lambda}(\bar{\theta}^{\lambda}_n, X_{n+1})(t-n)|^k \right|\bar{\theta}^{\lambda}_n \right]\\
&\leq \sum_{k = 2}^{2p}\binom{2p}{k}\lambda^{k/2}(t-n)^k\sum_{l=0}^{k-1}\binom{k}{l} 2^{qk}K_G^{k-l}K_F^l \E\left[(1+|X_0|)^{\rho k} \right] (1+|\bar{\theta}^{\lambda}_n|)^{2p-k+l}\\
&\quad +\sum_{k = 2}^{2p}\binom{2p}{k}|\bar{\theta}^{\lambda}_n|^{2p-k}\lambda^k(t-n)^k \frac{2^{k-1}K_F^{k}\E\left[(1+|X_0|)^{\rho k}\right] (1+|\bar{\theta}^{\lambda}_n|^{k(2r+1)})}{(1+\sqrt{\lambda}| \bar{\theta}^{\lambda}_n|^{2r})^k}\\
& \leq \binom{2p}{p}^2 2^{2p(q+1)-2} p(2p-1)\lambda(t-n) K_G^{2p}(1+K_F)^{2p}\E\left[(1+|X_0|)^{2p\rho } \right] (1+|\bar{\theta}^{\lambda}_n|^{2p-1})\\
&\quad +\binom{2p}{p} 2^{2p-1}(2p-1) \lambda(t-n) (1+K_F)^{2p}\E\left[(1+|X_0|)^{2p\rho } \right](1+|\bar{\theta}^{\lambda}_n|^{2p-1})\\
&\quad + \sum_{k = 2}^{2p}\binom{2p}{k}\lambda^k(t-n)^k \frac{2^{k-1}K_F^{k}\E\left[(1+|X_0|)^{\rho k}\right] |\bar{\theta}^{\lambda}_n|^{2rk+2p}}{(1+\sqrt{\lambda}| \bar{\theta}^{\lambda}_n|^{2r})^k}\\
&\leq c_3(p)\lambda(t-n) |\bar{\theta}^{\lambda}_n|^{2p-1} + c_3(p)\lambda(t-n) \\
&\quad + \sum_{k = 2}^{2p}\binom{2p}{k}\lambda^k(t-n)^k \frac{2^{k-1}K_F^{k}\E\left[(1+|X_0|)^{\rho k}\right] |\bar{\theta}^{\lambda}_n|^{2rk+2p}}{(1+\sqrt{\lambda}| \bar{\theta}^{\lambda}_n|^{2r})^k},
\end{align*}
where $c_3(p) :=\binom{2p}{p}^2 2^{2p(q+1)} p(2p-1) K_G^{2p}(1+K_F)^{2p}\E\left[(1+|X_0|)^{2p\rho } \right] $. By substituting the above estimate back into \eqref{binomialterm} and by applying Young's inequality, one obtains
\begin{align} \label{Jterms}
\E\left[\left.|\Delta_{n,t}^{\lambda}|^{2p}\right|\bar{\theta}^{\lambda}_n \right]
&\leq  |\bar{\theta}^{\lambda}_n|^{2p} -2p\lambda(t-n) |\bar{\theta}^{\lambda}_n|^{2p-2}\frac{ a_F|\bar{\theta}^{\lambda}_n|^{2r+2} }{1+\sqrt{\lambda}| \bar{\theta}^{\lambda}_n|^{2r}}  \nonumber\\
&\quad +\frac{2^{q+1}p\lambda(t-n)K_G\E\left[ (1+|X_0|)^{\rho}\right]|\bar{\theta}^{\lambda}_n|^{2p+q-1}}{1+\sqrt{\lambda}| \bar{\theta}^{\lambda}_n|^{2r}}\nonumber \\
&\quad + \sum_{k = 2}^{2p}\binom{2p}{k}\lambda^k(t-n)^k \frac{2^{k-1}K_F^{k}\E\left[(1+|X_0|)^{\rho k}\right] |\bar{\theta}^{\lambda}_n|^{2rk+2p}}{(1+\sqrt{\lambda}| \bar{\theta}^{\lambda}_n|^{2r})^k}\nonumber\\
&\quad +\lambda(t-n) (2pb_F +2^{q+1}p K_G\E\left[ (1+|X_0|)^{\rho}\right])(1+|\bar{\theta}^{\lambda}_n|^{2p-1})\nonumber \\
&\quad +  c_3(p)\lambda(t-n) |\bar{\theta}^{\lambda}_n|^{2p-1} + c_3(p)\lambda(t-n) \nonumber\\
\begin{split}
& = |\bar{\theta}^{\lambda}_n|^{2p} -\lambda(t-n) |\bar{\theta}^{\lambda}_n|^{2p}J_1^{\lambda}(\bar{\theta}^{\lambda}_n) - \lambda(t-n)(J_2^{\lambda}(\bar{\theta}^{\lambda}_n) +J_3^{\lambda}(\bar{\theta}^{\lambda}_n) )\\
&\quad +\lambda(t-n) (2pb_F +2^{q+1}p K_G\E\left[ (1+|X_0|)^{\rho}\right]+c_3(p)),
\end{split}
\end{align}
where for all $\theta \in \R^d$, $p\in [2, \infty) \cap {\N}$,
\begin{align*}
J_1^{\lambda}(\theta)
& := \frac{ a_F|\theta|^{2r}}{1+\sqrt{\lambda}|\theta|^{2r}} - \frac{2^{q+1}pK_G\E\left[ (1+|X_0|)^{\rho}\right]|\theta|^{q-1}}{1+\sqrt{\lambda}|\theta|^{2r}},\\
J_2^{\lambda}(\theta)
& := \frac{ a_F|\theta|^{2p+2r}- (1+\sqrt{\lambda}|\theta|^{2r})(2pb_F +2^{q+1}p K_G\E\left[ (1+|X_0|)^{\rho}\right]+c_3(p))|\theta|^{2p-1}}{1+\sqrt{\lambda}|\theta|^{2r}} ,\\
J_3^{\lambda}(\theta)
& := \frac{(2p-2) a_F|\theta|^{2p+2r}}{1+\sqrt{\lambda}|\theta|^{2r}} -  \sum_{k = 2}^{2p}\binom{2p}{k}\lambda^{k-1}(t-n)^{k-1} \frac{2^{k-1}K_F^{k}\E\left[(1+|X_0|)^{\rho k}\right] |\theta|^{2rk+2p}}{(1+\sqrt{\lambda}| \theta|^{2r})^k}.
\end{align*}
Next, we aim to choose a large enough constant $M_1(p)>0$, such that for $|\theta|>M_1(p)$, $J_1^{\lambda}(\theta)$ and $J_2^{\lambda}(\theta)$ are nonnegative. To obtain an explicit form of such a constant, one notes that, for all $\theta \in \R^d$,
\begin{align}\label{2pmmtesJ1}
\begin{split}
&a_F|\theta|^{2r} - 2^{q+1}pK_G\E\left[ (1+|X_0|)^{\rho}\right]|\theta|^{q-1} > \frac{a_F}{2}|\theta|^{2r}  \\
& \Leftrightarrow \quad |\theta|> M_{1,0}(p):=\left(\frac{2^{q+2}pK_G\E\left[ (1+|X_0|)^{\rho}\right]}{a_F}\right)^{1/(2r-q+1)}.
\end{split}
\end{align}
Moreover, one observes that, for all $\theta \in \R^d$,
\begin{align}\label{2pmmtesJ2p1}
\begin{split}
&\frac{a_F}{2}|\theta|^{2p+2r} -    (2pb_F +2^{q+1}p K_G\E\left[ (1+|X_0|)^{\rho}\right]+c_3(p))|\theta|^{2p+2r-1}  >  0 \\
& \Leftrightarrow \quad |\theta| > M_{1,1}(p):=\frac{4pb_F +2^{q+2}p K_G\E\left[ (1+|X_0|)^{\rho}\right]+2c_3(p)}{a_F},
\end{split}
\end{align}
and
\begin{align}\label{2pmmtesJ2p2}
\begin{split}
&\frac{a_F}{2}|\theta|^{2p+2r} -  (2pb_F +2^{q+1}p K_G\E\left[ (1+|X_0|)^{\rho}\right]+c_3(p))|\theta|^{2p-1}  >  0 \\
& \Leftrightarrow \quad |\theta| > M_{1,2}(p):=\left(\frac{4pb_F +2^{q+2}p K_G\E\left[ (1+|X_0|)^{\rho}\right]+2c_3(p)}{a_F}\right)^{1/(2r+1)}.
\end{split}
\end{align}
Denote by $M_1(p) := (4pb_F +2^{q+2}p K_G\E\left[ (1+|X_0|)^{\rho}\right]+2c_3(p))/\min\{1, a_F\}\geq 1$. It is straightforward to see that, for $2r\geq q \geq 1$, $p\in [2, \infty) \cap {\N}$,
\begin{equation}\label{M1consts}
M_1(p) \geq \max\{M_{1,0}(p), M_{1,1}(p), M_{1,2}(p)\}.
\end{equation}
Thus, for all $\theta \in \R^d, |\theta| >M_1(p)$, by \eqref{2pmmtesJ1}, \eqref{M1consts}, it holds that
\begin{align}\label{2pmmtesJ1lb}
J_1^{\lambda}(\theta)
& = \frac{ a_F|\theta|^{2r} - 2^{q+1}pK_G\E\left[ (1+|X_0|)^{\rho}\right]|\theta|^{q-1}  }{1+\sqrt{\lambda}|\theta|^{2r}}\nonumber\\
& > \frac{ a_F|\theta|^{2r}}{2(1+\sqrt{\lambda}|\theta|^{2r})}\nonumber \\
&\geq \frac{ a_F(M_1(p)) ^{2r}}{2(1+(M_1(p))^{2r})}\nonumber\\
& = a_F\bar{\kappa}(p),
\end{align}
where $\bar{\kappa}(p) := (M_1(p) )^{2r}/(2(1+(M_1(p) )^{2r}))$, and the last inequality holds due to $0<\lambda \leq \lambda_{p,\max} \leq 1$ and the fact that $f(s) := s/(1+\sqrt{\lambda}s)$ is non-decreasing for all $s\geq 0$.
Similarly, for all $|\theta| >M_1(p) $, it follows that $J_2^{\lambda}(\theta) >0$. Indeed, for all $\theta \in \R^d, |\theta| >M_1(p)$, $\lambda \leq \lambda_{p,\max}  \leq 1$, by using \eqref{2pmmtesJ2p1}, \eqref{2pmmtesJ2p2}, \eqref{M1consts}, one obtains
\begin{align}\label{2pmmtesJ2lb}
J_2^{\lambda}(\theta)
& = \frac{ a_F|\theta|^{2p+2r}- (1+\sqrt{\lambda}|\theta|^{2r})(2pb_F +2^{q+1}p K_G\E\left[ (1+|X_0|)^{\rho}\right]+c_3(p))|\theta|^{2p-1}}{1+\sqrt{\lambda}|\theta|^{2r}} \nonumber \\
&\geq \frac{ a_F|\theta|^{2p+2r}- (1+|\theta|^{2r})(2pb_F +2^{q+1}p K_G\E\left[ (1+|X_0|)^{\rho}\right]+c_3(p))|\theta|^{2p-1}}{1+\sqrt{\lambda}|\theta|^{2r}} \nonumber \\
& = \frac{ a_F|\theta|^{2p+2r}/2-  (2pb_F +2^{q+1}p K_G\E\left[ (1+|X_0|)^{\rho}\right]+c_3(p))|\theta|^{2p-1}}{1+\sqrt{\lambda}|\theta|^{2r}} \nonumber \\
&\quad +\frac{ a_F|\theta|^{2p+2r}/2-(2pb_F +2^{q+1}p K_G\E\left[ (1+|X_0|)^{\rho}\right]+c_3(p))|\theta|^{2p+2r-1}}{1+\sqrt{\lambda}|\theta|^{2r}} \nonumber \\
&\geq 0.
\end{align}
Furthermore, it follows that, for any $\theta \in \R^d$,
\begin{align*}
J_3^{\lambda}(\theta)
&  =  \sum_{k = 2}^{2p}\bigg(\frac{(2p-2) a_F|\theta|^{2p+2r}(1+\sqrt{\lambda}|\theta|^{2r})^{k-1}}{(2p-1)(1+\sqrt{\lambda}|\theta|^{2r})^k}\bigg.\\
&\quad \bigg. - \frac{\binom{2p}{k}\lambda^{k-1}(t-n)^{k-1} 2^{k-1}K_F^{k}\E\left[(1+|X_0|)^{\rho k}\right] |\theta|^{2rk+2p}}{(1+\sqrt{\lambda}| \theta|^{2r})^k}\bigg)\\
&  \geq  \sum_{k = 2}^{2p}\bigg(\frac{(2p-2) a_F|\theta|^{2p+2r}(1+\lambda^{(k-1)/2}|\theta|^{2r(k-1)})}{(2p-1)(1+\sqrt{\lambda}|\theta|^{2r})^k}\bigg.\\
&\quad \bigg. - \frac{\binom{2p}{k}\lambda^{k-1}(t-n)^{k-1} 2^{k-1}K_F^{k}\E\left[(1+|X_0|)^{\rho k}\right] |\theta|^{2rk+2p}}{(1+\sqrt{\lambda}| \theta|^{2r})^k}\bigg)\\
&  \geq  \sum_{k = 2}^{2p}\bigg(\frac{(2p-2) a_F\lambda^{(k-1)/2}|\theta|^{2p+2rk}}{(2p-1)(1+\sqrt{\lambda}|\theta|^{2r})^k}\bigg.\\
&\quad \bigg. - \frac{\binom{2p}{k}\lambda^{k-1}(t-n)^{k-1} 2^{k-1}K_F^{k}\E\left[(1+|X_0|)^{\rho k}\right] |\theta|^{2rk+2p}}{(1+\sqrt{\lambda}| \theta|^{2r})^k}\bigg).
\end{align*}
Then, direct calculations yield that $J_3^{\lambda}(\theta) \geq 0$ when
\[
\lambda \leq \lambda_k(p) := \frac{(a_F/K_F)^{2/(k-1)}}{9\binom{2p}{k}^{2/(k-1)}K_F^2(\E\left[(1+|X_0|)^{\rho k}\right])^{2/(k-1)}}
\]
for each $2 \leq k \leq 2p$. The above inequality further implies, a possible choice of the stepsize restriction (independent of $k$) would be:
\[
\lambda \leq \lambda_{p,\max} \leq  \lambda(p) :=\frac{\min\{(a_F/K_F)^2, (a_F/K_F)^{2/(2p-1)}\}}{9\binom{2p}{p}^2K_F^2(\E\left[(1+|X_0|)^{2p\rho  }\right])^2},
\]
which is a lower bound of $ \lambda_k(p)$, i.e. $\lambda(p) \leq \lambda_k(p)$, for all $2 \leq k \leq 2p$. Thus, for all $\theta \in \R^d$,
\begin{equation}\label{2pmmtesJ3lb}
0<\lambda \leq \lambda_{p,\max} \leq \lambda(p) \implies J_3^{\lambda}(\theta) \geq 0.
\end{equation}
Denote by $\mathsf{S}_{n,M_1(p) } := \{\omega \in \Omega: |\bar{\theta}^{\lambda}_n(\omega)| >M_1(p) \}$. Substituting the \eqref{2pmmtesJ1lb}, \eqref{2pmmtesJ2lb}, \eqref{2pmmtesJ3lb} into \eqref{Jterms} yields, for any $0<\lambda \leq \lambda_{p,\max} $
\begin{align*}
\E\left[\left.|\Delta_{n,t}^{\lambda}|^{2p} \1_{\mathsf{S}_{n,M_1(p) } }\right|\bar{\theta}^{\lambda}_n \right]
& \leq (1-\lambda(t-n)a_F\bar{\kappa}(p))|\bar{\theta}^{\lambda}_n|^{2p}  \1_{\mathsf{S}_{n,M_1(p) } }  \\
&\quad +\lambda(t-n) (2pb_F +2^{q+1}p K_G\E\left[ (1+|X_0|)^{\rho}\right]+c_3(p))\1_{\mathsf{S}_{n,M_1(p) } },
\end{align*}
and moreover, by \eqref{2pmmtesJ3lb}, one obtains, for any $0<\lambda \leq \lambda_{p,\max} $,
\begin{align*}
&\E\left[\left.|\Delta_{n,t}^{\lambda}|^{2p} \1_{\mathsf{S}_{n,M_1(p) }^{\mathsf{c}} }\right|\bar{\theta}^{\lambda}_n \right] \\
& \leq (1-\lambda(t-n)a_F\bar{\kappa}(p))|\bar{\theta}^{\lambda}_n|^{2p}  \1_{\mathsf{S}_{n,M_1(p) } } \\
&\quad + \lambda(t-n) ( a_F\bar{\kappa}(p)(M_1(p) )^{2p}+2^{q+1}p K_G\E\left[ (1+|X_0|)^{\rho}\right](M_1(p) )^{2p+q-1}) \1_{\mathsf{S}_{n,M_1(p)}^{\mathsf{c}} }\\
&\quad +\lambda(t-n) (2pb_F +2^{q+1}p K_G\E\left[ (1+|X_0|)^{\rho}\right]+c_3(p)) (1+(M_1(p) )^{2p-1})\1_{\mathsf{S}_{n,M_1(p)}^{\mathsf{c}} }.
\end{align*}
Thus, one obtains, for $p\in [2, \infty)\cap {\N}$,
\begin{equation}\label{delta2p}
\E\left[\left.|\Delta_{n,t}^{\lambda}|^{2p}\right|\bar{\theta}^{\lambda}_n \right]  \leq (1-\lambda(t-n)a_F\bar{\kappa}(p))|\bar{\theta}^{\lambda}_n|^{2p}    +\lambda(t-n) c_4(p),
\end{equation}
where
\begin{align*}
c_4(p)
& := a_F\bar{\kappa}(p)(M_1(p) )^{2p}+2^{q+1}p K_G\E\left[ (1+|X_0|)^{\rho}\right](M_1(p) )^{2p+q-1}\\
&\quad + (2pb_F +2^{q+1}p K_G\E\left[ (1+|X_0|)^{\rho}\right]+c_3(p)) (1+(M_1(p) )^{2p-1}).
\end{align*}
Moreover, one notes that, $c_2 \leq c_4(1)$ with $c_2$ given in \eqref{constc2}. Thus, by using \eqref{delta2p} and \eqref{delta2}, it holds that, for any $p\in [2, \infty)\cap {\N}$,
\begin{equation}\label{delta2p-2}
\E\left[\left.|\Delta_{n,t}^{\lambda}|^{2p-2}\right|\bar{\theta}^{\lambda}_n \right]  \leq |\bar{\theta}^{\lambda}_n|^{2p-2}  +\lambda(t-n) c_4(p-1).
\end{equation}
Substituting the upper bounds in \eqref{delta2p} and \eqref{delta2p-2} into \eqref{2pthme} therefore yields
\begin{align*}
\E\left[\left.|\bar{\theta}^{\lambda}_t|^{2p}\right|\bar{\theta}^{\lambda}_n \right]
&\leq  (1-\lambda(t-n)a_F\bar{\kappa}(p))|\bar{\theta}^{\lambda}_n|^{2p} + 2^{2p-2}p(2p-1)\lambda (t-n)d \beta^{-1}|\bar{\theta}^{\lambda}_n|^{2p-2}   \\
&\quad +\lambda(t-n) c_4(p) +2^{2p-2}p(2p-1)\lambda^2 (t-n)^2d \beta^{-1}c_4(p-1)\\
&\quad +2^{2p-4}(2p(2p-1))^{p+1}(d\beta^{-1}\lambda (t-n) )^p.
\end{align*}
One notes that for any $\theta \in \R^d$
\begin{align}\label{2pdeltaesp1}
\begin{split}
&  (1-\lambda(t-n)a_F\bar{\kappa}(p))|\theta|^{2p} + 2^{2p-2}p(2p-1) \lambda(t-n)d \beta^{-1}|\theta|^{2p-2}   < (1-\lambda(t-n) a_F\bar{\kappa}(p)/2)|\theta|^{2p}  \\
& \Leftrightarrow \quad |\theta| > \left(\frac{2^{2p-1}p(2p-1) d \beta^{-1}}{a_F \bar{\kappa}(p) }\right)^{1/2}.
\end{split}
\end{align}
Denote by $M_2(p) := (2^{2p-1}p(2p-1)d \beta^{-1}/(a_F\bar{\kappa}(p)))^{1/2}$ and $\mathsf{S}_{n,M_2(p) } := \{\omega \in \Omega: |\bar{\theta}^{\lambda}_n(\omega)| >M_2(p) \}$. Then, by \eqref{2pdeltaesp1}, it follows that
\[
\E\left[\left.|\bar{\theta}^{\lambda}_t|^{2p} \1_{\mathsf{S}_{n,M_2(p) }} \right|\bar{\theta}^{\lambda}_n \right]
\leq  (1-\lambda(t-n)a_F\bar{\kappa}(p)/2)|\bar{\theta}^{\lambda}_n|^{2p}\1_{\mathsf{S}_{n,M_2(p) }} +\lambda (t-n) c_5(p)\1_{\mathsf{S}_{n,M_2(p) }} ,
\]
where $c_5(p) := c_4(p) +2^{2p-2}p(2p-1) d \beta^{-1}c_4(p-1) +2^{2p-4}(2p(2p-1))^{p+1}(d\beta^{-1})^p$. Furthermore,
\begin{align*}
\E\left[\left.|\bar{\theta}^{\lambda}_t|^{2p} \1_{\mathsf{S}_{n,M_2(p) }^{\mathsf{c}} }\right|\bar{\theta}^{\lambda}_n \right]
&\leq  (1-\lambda(t-n)a_F\bar{\kappa}(p))|\bar{\theta}^{\lambda}_n|^{2p}\1_{\mathsf{S}_{n,M_2(p) }^{\mathsf{c}}} +\lambda (t-n) c_5(p)\1_{\mathsf{S}_{n,M_2(p) }^{\mathsf{c}}}\\
& \quad + \lambda (t-n) 2^{2p-2}p(2p-1) d \beta^{-1}(M_2(p))^{2p-2}\1_{\mathsf{S}_{n,M_2(p) }^{\mathsf{c}}}.
\end{align*}
By combining the two cases, one obtains, for $t \in (n, n+1], n \in \N_0$, $0 <\lambda \leq \lambda_{p,\max} $,
\[
\E\left[\left.|\bar{\theta}^{\lambda}_t|^{2p}\right|\bar{\theta}^{\lambda}_n \right]  \leq  (1-\lambda(t-n)a_F\bar{\kappa}(p)/2)|\bar{\theta}^{\lambda}_n|^{2p} +\lambda (t-n) \bar{c}_0(p),
\]
where
\begin{align}\label{constcp}
\begin{split}
\bar{\kappa}(p)
&:= (M_1(p))^{2r}/(2(1+(M_1(p))^{2r})),\\
M_1(p)
&:=(4pb_F +2^{q+2}p K_G\E\left[ (1+|X_0|)^{\rho}\right]+2c_3(p))/\min\{1, a_F\},\\
\bar{c}_0(p)
&:= c_5(p)+ 2^{2p-2}p(2p-1) d \beta^{-1}(M_2(p))^{2p-2}, \\
M_2 (p)
& := (2^{2p-1}p(2p-1)d \beta^{-1}/(a_F\bar{\kappa}(p)))^{1/2}, \\
c_5(p)
& :=  c_4(p) +2^{2p-2}p(2p-1) d \beta^{-1}c_4(p-1) +2^{2p-4}(2p(2p-1))^{p+1}(d\beta^{-1})^p, \\
c_4(p)
&: = a_F\bar{\kappa}(p)(M_1(p))^{2p}+2^{q+1}p K_G\E\left[ (1+|X_0|)^{\rho}\right](M_1(p))^{2p+q-1}\\
&\quad + (2pb_F +2^{q+1}p K_G\E\left[ (1+|X_0|)^{\rho}\right]+c_3(p)) (1+(M_1(p))^{2p-1}), \\
c_3(p)
&:=\binom{2p}{p}^2 2^{2p(q+1)} p(2p-1) K_G^{2p}(1+K_F)^{2p}\E\left[(1+|X_0|)^{2p\rho } \right].
\end{split}
\end{align}
Therefore, by noticing $\bar{\kappa}(p) \geq \bar{\kappa}(2) $, for any $p \geq 2$, and by using similar arguments as in \eqref{2ndmmttp}, one can conclude that, for $t \in (n, n+1], n \in \N_0$, $0 <\lambda \leq \lambda_{p,\max} $,
\begin{align}\label{2pme}
\begin{split}
\E\left[ |\bar{\theta}^{\lambda}_t|^{2p}  \right]
&\leq  (1-\lambda(t-n)a_F\bar{\kappa}(2)/2)\left(1- \lambda a_F\bar{\kappa}(2)/2 \right)^n \E\left[|\theta_0|^{2p}\right] + \bar{c}_0(p)(1+2/(a_F\bar{\kappa}(p)) ).
\end{split}
\end{align}
Finally, for any $p\in [2, \infty)\cap {\N}$, denote by $\kappa^\sharp_p := \min\{\bar{\kappa}(p), \tilde{\kappa}(p)\}$ and $c^\sharp_p := \max\{\bar{c}_0(p), \tilde{c}_0(p)\}$, where $\tilde{\kappa}(p), \tilde{c}_0(p)$ are given in \eqref{constcpsp}. The above inequality further implies, for $t \in (n, n+1], n \in \N_0$, $0<\lambda \leq \lambda_{p,\max}$,
\[
\E\left[ |\bar{\theta}^{\lambda}_t|^{2p} \right]  \leq (1-\lambda (t-n) a_F\kappa^\sharp_2/2) (1-\lambda a_F\kappa^\sharp_2/2)^n\E\left[|\theta_0|^{2p}\right]  + c^\sharp_p(1+2/(a_F\kappa^\sharp_p )),
\]
which completes the proof.
\end{proof}
\begin{proof}[\textbf{Proof of Lemma \ref{2ndpthmmt}-\ref{2ndpthmmtiii}}]
We have established an upper estimate for the $2p$-th moment (with $p\in [2, \infty)\cap {\N}$) of the TUSLA algorithm \eqref{tusla} under the condition that  $0<\lambda \leq \lambda_{p,\max}$ with $\lambda_{p,\max}$ given in \eqref{stepsizemax}. One may notice that $\lambda_{p,\max}$ is quite restrictive for practical implementations when $p$ is large. Thus, in this subsection, we will show that, in some special case of $F$, the $2p$-th moment of the TUSLA algorithm \eqref{tusla} can be obtained under a relaxed stepsize restriction.

We assume in this subsection that for any $\theta \in \R^d$, $F(\theta, x) = F(\theta)$ for all $x \in \R^m$. One notes that for $F$ satisfying Assumption \ref{AF}, it further satisfies the following growth condition: for all $\theta \in \R^d$,
\begin{equation}\label{2pspeFgrwoth}
|F(\theta) | \leq K_F (1+|\theta|^{2r+1}).
\end{equation}
Moreover, for $F$ satisfying Assumption \ref{AC}, we have that by Remark \ref{ADFh}, for all $\theta \in \R^d$,
\begin{equation}\label{2pspeFdissi}
\langle \theta, F(\theta)\rangle \geq a_F|\theta|^{2r+2} - b_F.
\end{equation}

Denote by
\[
\tilde{\lambda}_{\max} := \min\left\{1,\frac{a_F^2}{16K_F^4}, \frac{1}{a_F}, \frac{1}{4a_F^2}\right\}
\]
as presented in \eqref{relaxedlamax}. To establish the $2p$-th moment estimate (with $p\in [2, \infty)\cap {\N}$) under the condition that $0<\lambda <\tilde{\lambda}_{\max}$, we apply the same arguments as in the proof of Lemma \ref{2ndpthmmt}-\ref{2ndpthmmtii} up to \eqref{2pthme}, then, we adopt a different method to obtain an upper bound of $\E\left[\left.|\Delta_{n,t}^{\lambda}|^{2p}\right|\bar{\theta}^{\lambda}_n \right]  $. For any $0<\lambda <\tilde{\lambda}_{\max}$ with $\tilde{\lambda}_{\max}$ given in \eqref{relaxedlamax}, $t\in (n, n+1]$, $n \in \N_0$, recall the definition of $\Delta_{n,t}^{\lambda}$ given in \eqref{delxinotation}. One notes that, for $0<\lambda \leq \tilde{\lambda}_{\max}$, by using Assumption \ref{AG}, \eqref{2pspeFgrwoth}, \eqref{2pspeFdissi},
\begin{align}
|\Delta_{n,t}^{\lambda}|^2
& \leq | \bar{\theta}^{\lambda}_n|^2- \lambda(t-n)\frac{2 a_F|\bar{\theta}^{\lambda}_n|^{2r+2}}{1+\sqrt{\lambda}| \bar{\theta}^{\lambda}_n|^{2r}} + 2\lambda(t-n) b_F \nonumber\\
&\quad +\lambda(t-n)\frac{2^{q+1}K_G (1+|X_{n+1}|)^{\rho}(1+|\bar{\theta}^{\lambda}_n|^{q+1}) }{1+\sqrt{\lambda}| \bar{\theta}^{\lambda}_n|^{2r}} \nonumber\\
&\quad +\lambda^2(t-n)^2\frac{2^{2q}K_G^2 (1+|X_{n+1}|)^{2\rho} (1+|\bar{\theta}^{\lambda}_n|^{2q}) }{(1+\sqrt{\lambda}| \bar{\theta}^{\lambda}_n|^{2r})^2} \nonumber\\
&\quad +\lambda^2(t-n)^2\frac{4K_F^2 (1+|\bar{\theta}^{\lambda}_n|^{4r+2})}{(1+\sqrt{\lambda}| \bar{\theta}^{\lambda}_n|^{2r})^2} \nonumber\\
\begin{split}\label{stepsizere1}
&\leq | \bar{\theta}^{\lambda}_n|^2- \lambda(t-n)\frac{2a_F|\bar{\theta}^{\lambda}_n|^{2r+2}}{1+\sqrt{\lambda}| \bar{\theta}^{\lambda}_n|^{2r}}   + 2\lambda(t-n)  b_F  \\
&\quad +\lambda(t-n)2^{q+1}K_G (1+|X_{n+1}|)^{\rho}  +\lambda(t-n)\frac{2^{q+1}K_G (1+|X_{n+1}|)^{\rho} |\bar{\theta}^{\lambda}_n|^{q+1} }{1+\sqrt{\lambda}| \bar{\theta}^{\lambda}_n|^{2r}} \\
&\quad +\lambda (t-n) 2^{2q+1}K_G^2 (1+|X_{n+1}|)^{2\rho}  +4\lambda^2(t-n)^2K_F^2+\lambda^2(t-n)^2\frac{4K_F^2  |\bar{\theta}^{\lambda}_n|^{4r+2}}{(1+\sqrt{\lambda}| \bar{\theta}^{\lambda}_n|^{2r})^2}  \\
\end{split}
\\
\begin{split}\label{stepsizere2}
&\leq | \bar{\theta}^{\lambda}_n|^2- \lambda(t-n)\frac{ a_F|\bar{\theta}^{\lambda}_n|^{2r+2}}{1+\sqrt{\lambda}| \bar{\theta}^{\lambda}_n|^{2r}}    + \lambda(t-n)(2 b_F+4K_F^2)\\
&\quad +\lambda (t-n) 2^{2q+2}K_G^2 (1+|X_{n+1}|)^{2\rho}    +\lambda(t-n)\frac{2^{q+1}K_G (1+|X_{n+1}|)^{\rho} |\bar{\theta}^{\lambda}_n|^{q+1} }{1+\sqrt{\lambda}| \bar{\theta}^{\lambda}_n|^{2r}} \\
\end{split}
\\
\begin{split}\label{delta2pJ4J5}
& = J_{4,n,t}^{\lambda}(\bar{\theta}^{\lambda}_n) +  J_{5,n,t}^{\lambda}(\bar{\theta}^{\lambda}_n, X_{n+1}),
\end{split}
\end{align}
where for any $\theta \in \R^d, x \in \R^m$,
\begin{align*}
J_{4,n,t}^{\lambda}(\theta) & :=  \left(1- \lambda(t-n)\frac{ a_F|\theta|^{2r}}{1+\sqrt{\lambda}| \theta|^{2r}} \right)| \theta|^2 \\
J_{5,n,t}^{\lambda}(\theta, x)
&:= \lambda(t-n) (2b_F+4K_F^2) +\lambda (t-n)2^{2q+2} K_G^2 (1+|x|)^{2\rho}\\
&\quad +\lambda(t-n)\frac{2^{q+1}K_G (1+|x|)^{\rho} |\theta|^{q+1} }{1+\sqrt{\lambda}| \theta|^{2r}},
\end{align*}
where the inequality \eqref{stepsizere1} holds due to the following: for $0<\lambda \leq \tilde{\lambda}_{\max} \leq 1$, $2r \geq q \geq 1$,
\begin{align*}
 \lambda^2(t-n)^2\frac{2^{2q}K_G^2 (1+|X_{n+1}|)^{2\rho} (1+|\bar{\theta}^{\lambda}_n|^{2q}) }{(1+\sqrt{\lambda}| \bar{\theta}^{\lambda}_n|^{2r})^2}
& \leq \lambda (t-n) \frac{2^{2q}K_G^2 (1+|X_{n+1}|)^{2\rho} (1+\lambda|\bar{\theta}^{\lambda}_n|^{2q}) }{1+\lambda| \bar{\theta}^{\lambda}_n|^{4r}}\\
& \leq \lambda (t-n)2^{2q+1}K_G^2 (1+|X_{n+1}|)^{2\rho} ,
\end{align*}
while \eqref{stepsizere2} holds due to the fact that for $0<\lambda \leq \tilde{\lambda}_{\max} \leq a_F^2/(16K_F^4)$,
\[
\frac{ a_F|\theta|^{2r+2}}{1+\sqrt{\lambda}| \theta|^{2r}} - \frac{ \lambda (t-n)4K_F^2|\theta|^{4r+2}}{(1+\sqrt{\lambda}| \theta|^{2r})^2} \geq 0.
\]
Moreover, one notes that for any nonzero $\theta \in \R^d$, $t\in (n, n+1]$, $n \in \N_0$, the stepsize restriction $0<\lambda \leq \tilde{\lambda}_{\max}  \leq 1/(4a_F^2)$ is chosen such that the following inequalities hold:
\[
0<1- \lambda(t-n)\frac{ 2a_F|\theta|^{2r}}{1+\sqrt{\lambda}| \theta|^{2r}} <1.
\]
Then, for $0<\lambda <\tilde{\lambda}_{\max}$, by using \eqref{delta2pJ4J5}, one obtains
\begin{align}\label{delta2pes1}
\E\left[\left.|\Delta_{n,t}^{\lambda}|^{2p}\right|\bar{\theta}^{\lambda}_n \right]
& = \sum_{k = 0}^p \binom{p}{k}(J_{4,n,t}^{\lambda}(\bar{\theta}^{\lambda}_n))^{p-k} \E\left[\left. (  J_{5,n,t}^{\lambda}(\bar{\theta}^{\lambda}_n, X_{n+1}))^k\right|\bar{\theta}^{\lambda}_n \right]  \nonumber\\
&=\left(1- \lambda(t-n)\frac{ a_F|\bar{\theta}^{\lambda}_n|^{2r}}{1+\sqrt{\lambda}| \bar{\theta}^{\lambda}_n|^{2r}} \right)^p| \bar{\theta}^{\lambda}_n|^{2p} \nonumber\\
& \quad + \sum_{k = 1}^p \binom{p}{k} (J_{4,n,t}^{\lambda}(\bar{\theta}^{\lambda}_n))^{p-k}\E\left[\left. (  J_{5,n,t}^{\lambda}(\bar{\theta}^{\lambda}_n, X_{n+1}))^k\right|\bar{\theta}^{\lambda}_n \right]  \nonumber\\
&\leq  \left(1- \lambda(t-n)\frac{  a_F|\bar{\theta}^{\lambda}_n|^{2r}}{ 1+\sqrt{\lambda}| \bar{\theta}^{\lambda}_n|^{2r}} \right)| \bar{\theta}^{\lambda}_n|^{2p} \nonumber\\
&\quad + \sum_{k = 1}^p \binom{p}{k}  \lambda^k(t-n)^k  3^{k-1}| \bar{\theta}^{\lambda}_n|^{2p-2k} (2b_F+4K_F^2)^k \nonumber\\
&\quad + \sum_{k = 1}^p \binom{p}{k}  \lambda^k(t-n)^k  3^{k-1}2^{k(2q+2)} K_G^{2k}| \bar{\theta}^{\lambda}_n|^{2p-2k}\E\left[  (1+|X_0|)^{2k\rho} \right]  \nonumber\\
&\quad + \sum_{k = 1}^p \binom{p}{k}  \lambda^k(t-n)^k  3^{k-1} \frac{2^{k(q+1)}K_G^k \E\left[ (1+|X_0|)^{k\rho} \right]  |\bar{\theta}^{\lambda}_n|^{k(q-1)+2p} }{(1+\sqrt{\lambda}| \bar{\theta}^{\lambda}_n|^{2r})^k}   \nonumber\\
& = \left(1- \lambda(t-n)\frac{  a_F|\bar{\theta}^{\lambda}_n|^{2r}}{2(1+\sqrt{\lambda}| \bar{\theta}^{\lambda}_n|^{2r})} \right)| \bar{\theta}^{\lambda}_n|^{2p}-J_{6,n,t}^{\lambda}(\bar{\theta}^{\lambda}_n) -J_{7,n,t}^{\lambda}(\bar{\theta}^{\lambda}_n),
\end{align}
where for any $\theta \in \R^d$,
\begin{align*}
J_{6,n,t}^{\lambda}(\theta )
& :=    \lambda(t-n)\frac{ a_F|\theta|^{2r+2p}}{4(1+\sqrt{\lambda}| \theta|^{2r})} \\
& \quad -\sum_{k = 1}^p \binom{p}{k}\lambda^k (t-n)^k3^{k-1}\frac{2^{k(q+1)}K_G^k \E\left[(1+|X_0|)^{k\rho}\right]  | \theta|^{k(q-1)+2p} }{(1+\sqrt{\lambda}| \theta|^{2r})^k}\\
J_{7,n,t}^{\lambda}(\theta )
&:=     \lambda(t-n)\frac{ a_F|\theta|^{2r+2p}}{4(1+\sqrt{\lambda}| \theta|^{2r})}   -\sum_{k = 1}^p \binom{p}{k}\lambda^k (t-n)^k3^{k-1}(2b_F+4K_F^2)^k | \theta|^{2(p-k)}\\
&\quad -\sum_{k = 1}^p \binom{p}{k}\lambda^k (t-n)^k3^{k-1}2^{k(2q+2)} K_G^{2k} \E\left[(1+|X_0|)^{2k\rho}\right] | \theta|^{2(p-k)}.
\end{align*}
One notes that, for any $\theta \in \R^d$, $1 \leq k \leq p$, $2r \geq q \geq 1$,
\begin{align*}
&   (t-n) \frac{ a_F\lambda^{(k+1)/2}|\theta|^{2rk+2p}}{9p(1+\sqrt{\lambda}| \theta|^{2r})^k} \\
& \quad -   \binom{p}{\lcrc{p/2}} (t-n) 3^{p-1}2^{p(q+1)}K_G^p \E\left[(1+|X_0|)^{p\rho}\right]   \frac{\lambda^{(k+1)/2} | \theta|^{k(q-1)+2p} }{(1+\sqrt{\lambda}| \theta|^{2r})^k}>0\\
&\Leftrightarrow \quad |\theta|> M_{3,0,k}(p):=\left(\frac{p\binom{p}{\lcrc{p/2}} 3^{p+1} 2^{p(q+1)}K_G^p \E\left[(1+|X_0|)^{p\rho}\right]}{a_F} \right)^{1/(k(2r-q+1))}.
\end{align*}
This implies for all $\theta \in \R^d$, $|\theta| >M_{3,0}(p) := \max_{1\leq k \leq p}\{M_{3,0,k}(p)\}$,
\begin{align}\label{2pmmtspeJ6}
J_{6,n,t}^{\lambda}(\theta )
& \geq    \lambda(t-n)\sum_{k = 1}^p\frac{ a_F|\theta|^{2r+2p}(1+\sqrt{\lambda}| \theta|^{2r})^{k-1}}{4p(1+\sqrt{\lambda}| \theta|^{2r})^k} \nonumber\\
& \quad - \binom{p}{\lcrc{p/2}}(t-n) 3^{p-1} 2^{p(q+1)}K_G^p\E\left[(1+|X_0|)^{p\rho}\right] \sum_{k = 1}^p  \lambda^k \frac{ | \theta|^{k(q-1)+2p} }{(1+\sqrt{\lambda}| \theta|^{2r})^k} \nonumber\\
& \geq    \lambda(t-n)\sum_{k = 1}^p\frac{ a_F (1+\lambda^{(k-1)/2}| \theta|^{2rk+2p}) }{4p(1+\sqrt{\lambda}| \theta|^{2r})^k}\nonumber\\
& \quad - \binom{p}{\lcrc{p/2}}(t-n) 3^{p-1} 2^{p(q+1)}K_G^p\E\left[(1+|X_0|)^{p\rho}\right] \sum_{k = 1}^p   \frac{\lambda^{(k+1)/2} | \theta|^{k(q-1)+2p} }{(1+\sqrt{\lambda}| \theta|^{2r})^k}\nonumber\\
& \geq   \sum_{k = 1}^p\Bigg((t-n)\frac{ a_F  \lambda^{(k+1)/2}| \theta|^{2rk+2p} }{9p(1+\sqrt{\lambda}| \theta|^{2r})^k}\Bigg.\nonumber\\
& \quad \Bigg. - \binom{p}{\lcrc{p/2}}(t-n) 3^{p-1} 2^{p(q+1)}K_G^p\E\left[(1+|X_0|)^{p\rho}\right]   \frac{\lambda^{(k+1)/2} | \theta|^{k(q-1)+2p} }{(1+\sqrt{\lambda}| \theta|^{2r})^k}\Bigg)\nonumber\\
&>0.
\end{align}
In addition,  for $1 \leq k \leq p$, $2r\geq q \geq 1$, it folllows that
\begin{align*}
& \lambda(t-n)\frac{ a_F|\theta|^{2r+2p}}{27p}   -\binom{p}{\lcrc{p/2}}\lambda  (t-n) 3^{p-1}(2b_F+4K_F^2)^p | \theta|^{2(p-k)}>0\\
&\Leftrightarrow \quad |\theta|>M_{3,1,k}(p):=  \left(\frac{p\binom{p}{\lcrc{p/2}}3^{p+2}(2b_F+4K_F^2)^p}{a_F}\right)^{1/(2r+2k)},\\
& \lambda(t-n)\frac{ a_F|\theta|^{2r+2p}}{27p}  - \binom{p}{\lcrc{p/2}} \lambda  (t-n) 3^{p-1}2^{p(2q+2)} K_G^{2p} \E\left[(1+|X_0|)^{2p\rho}\right] | \theta|^{2(p-k)}>0\\
&\Leftrightarrow \quad |\theta|>M_{3,2,k}(p):=\left(\frac{p\binom{p}{\lcrc{p/2}} 3^{p+2} 2^{2p(q+1)}K_G^{2p} \E\left[(1+|X_0|)^{2p\rho}\right]}{a_F} \right)^{1/(2r+2k)},\\
& \lambda(t-n)\frac{ a_F|\theta|^{2r+2p}}{27p}   -\binom{p}{\lcrc{p/2}}\lambda  (t-n) 3^{p-1}(2b_F+4K_F^2)^p | \theta|^{2r+2(p-k)}>0\\
&\Leftrightarrow \quad |\theta|>M_{3,3,k}(p):=  \left(\frac{p\binom{p}{\lcrc{p/2}}3^{p+2}(2b_F+4K_F^2)^p}{a_F}\right)^{1/(2k)},\\
& \lambda(t-n)\frac{ a_F|\theta|^{2r+2p}}{27p}  - \binom{p}{\lcrc{p/2}} \lambda  (t-n) 3^{p-1}2^{p(2q+2)} K_G^{2p} \E\left[(1+|X_0|)^{2p\rho}\right] | \theta|^{2r+2(p-k)}>0\\
&\Leftrightarrow \quad |\theta|>M_{3,4,k}(p):=\left(\frac{p\binom{p}{\lcrc{p/2}} 3^{p+2} 2^{2p(q+1)}K_G^{2p} \E\left[(1+|X_0|)^{2p\rho}\right]}{a_F} \right)^{1/(2k)}.
\end{align*}
For all $\theta \in \R^d$, $|\theta| >M_{3,1}(p)$ with
\[
M_{3,1}(p) := \max_{1\leq k \leq p}\{M_{3,1,k}(p), M_{3,2,k}(p), M_{3,3,k}(p), M_{3,4,k}(p)\},
\]
the above inequalities hence imply,
\begin{align}\label{2pmmtspeJ7}
 J_{7,n,t}^{\lambda}(\theta )
&\geq    \lambda(t-n)\sum_{k = 1}^p \frac{ a_F|\theta|^{2r+2p}}{4p(1+\sqrt{\lambda}| \theta|^{2r})} \nonumber \\
&\quad -\binom{p}{\lcrc{p/2}}\lambda  (t-n) 3^{p-1}(2b_F+4K_F^2)^p  \sum_{k = 1}^p   \frac{ | \theta|^{2(p-k)} (1+\sqrt{\lambda}| \theta|^{2r})}{ (1+\sqrt{\lambda}| \theta|^{2r})} \nonumber \\
&\quad -\binom{p}{\lcrc{p/2}} \lambda  (t-n) 3^{p-1}2^{p(2q+2)} K_G^{2p} \E\left[(1+|X_0|)^{2p\rho}\right] \sum_{k = 1}^p   \frac{ | \theta|^{2(p-k)} (1+\sqrt{\lambda}| \theta|^{2r})}{ (1+\sqrt{\lambda}| \theta|^{2r})} \nonumber \\
&\geq \sum_{k = 1}^p   \Bigg(\lambda(t-n)\frac{ a_F|\theta|^{2r+2p}}{27p(1+\sqrt{\lambda}| \theta|^{2r})}   -\binom{p}{\lcrc{p/2}}\frac{\lambda  (t-n) 3^{p-1}(2b_F+4K_F^2)^p | \theta|^{2(p-k)}}{1+\sqrt{\lambda}| \theta|^{2r}}\Bigg) \nonumber \\
&\quad + \sum_{k = 1}^p   \Bigg(\lambda(t-n)\frac{ a_F|\theta|^{2r+2p}}{27p(1+\sqrt{\lambda}| \theta|^{2r})}  \Bigg. \nonumber \\
&\quad \Bigg. - \binom{p}{\lcrc{p/2}} \frac{\lambda  (t-n) 3^{p-1}2^{p(2q+2)} K_G^{2p} \E\left[(1+|X_0|)^{2p\rho}\right] | \theta|^{2(p-k)}}{ 1+\sqrt{\lambda}| \theta|^{2r} }\Bigg) \nonumber \\
&\quad+ \sum_{k = 1}^p   \Bigg(\lambda(t-n)\frac{ a_F|\theta|^{2r+2p}}{27p(1+\sqrt{\lambda}| \theta|^{2r})}   -\binom{p}{\lcrc{p/2}}\frac{\lambda  (t-n) 3^{p-1}(2b_F+4K_F^2)^p | \theta|^{2r+2(p-k)}}{ 1+\sqrt{\lambda}| \theta|^{2r}}\Bigg) \nonumber \\
&\quad + \sum_{k = 1}^p   \Bigg( \lambda(t-n)\frac{ a_F|\theta|^{2r+2p}}{27p(1+\sqrt{\lambda}| \theta|^{2r})} \Bigg.  \nonumber \\
&\quad \Bigg. - \binom{p}{\lcrc{p/2}}\frac{ \lambda  (t-n) 3^{p-1}2^{p(2q+2)} K_G^{2p} \E\left[(1+|X_0|)^{2p\rho}\right] | \theta|^{2r+2(p-k)}}{1+\sqrt{\lambda}| \theta|^{2r}}\Bigg) \nonumber \\
&>0.
\end{align}
Denote by
\[
M_3(p) :=  \frac{\binom{p}{\lcrc{p/2}}3^{p+2}p\left( 2^{2p(q+1)}K_G^{2p} \E\left[(1+|X_0|)^{2p\rho}\right]+(1+2b_F+4K_F^2)^p\right)}{\min\{1, a_F\}},
\]
$\tilde{\kappa}(p) := (M_3(p))^{2r}/(2(1+(M_3(p))^{2r}))$ and $\mathsf{S}_{n,M_3(p) } := \{\omega \in \Omega: |\bar{\theta}^{\lambda}_n(\omega)| >M_3(p) \}$. One notes that, for $2r\geq q \geq 1$, $p\in [2, \infty) \cap {\N}$,
\begin{equation}\label{M3consts}
M_3(p) \geq \max\{M_{3,0}(p), M_{3,1}(p)\}.
\end{equation}
Thus,  for $|\theta| > M_3(p)$, by using \eqref{M3consts}, and by substituting \eqref{2pmmtspeJ6}, \eqref{2pmmtspeJ7} into \eqref{delta2pes1}, one obtains,
\begin{align*}
\E\left[\left.|\Delta_{n,t}^{\lambda}|^{2p}\1_{\mathsf{S}_{n,M_3(p) }}\right|\bar{\theta}^{\lambda}_n \right]
&\leq \left(1- \lambda(t-n)\frac{  a_F|\bar{\theta}^{\lambda}_n|^{2r}}{2(1+\sqrt{\lambda}| \bar{\theta}^{\lambda}_n|^{2r})} \right)| \bar{\theta}^{\lambda}_n|^{2p}\1_{\mathsf{S}_{n,M_3(p) }}\\
& \leq \left(1- \lambda(t-n)a_F \tilde{\kappa}(p)\right)| \bar{\theta}^{\lambda}_n|^{2p}\1_{\mathsf{S}_{n,M_3(p) }},
\end{align*}
where the last inequality holds due to the fact that, for any fixed $0<\lambda <\tilde{\lambda}_{\max}$, the function $f(s) := s/(1+\sqrt{\lambda}s)$ is non-decreasing for all $s\geq 0$. In addition, one obtains
\begin{align*}
\E\left[\left.|\Delta_{n,t}^{\lambda}|^{2p}\1_{\mathsf{S}_{n,M_3(p) }^{\mathsf{c}}}\right|\bar{\theta}^{\lambda}_n \right]
& \leq \left(1- \lambda(t-n)a_F \tilde{\kappa}(p)\right)| \bar{\theta}^{\lambda}_n|^{2p}\1_{\mathsf{S}_{n,M_3(p) }^{\mathsf{c}}} +\lambda(t-n) c_6(p)\1_{\mathsf{S}_{n,M_3(p) }^{\mathsf{c}}},
\end{align*}
where
\begin{align*}
c_6(p)
&:= a_F \tilde{\kappa}(p)(M_3(p) )^{2p}\\
&\quad+\sum_{k = 1}^p \binom{p}{k}3^k \left(2^{k(2q+2)} K_G^{2k} \E\left[(1+|X_0|)^{2k\rho}\right] +(2b_F+4K_F^2)^k \right)(M_3(p) )^{2p+k(q-1)}.
\end{align*}
Thus, it follows that for any $p\in [2, \infty)\cap {\N}$,
\begin{equation}\label{Fdelta2p}
\E\left[\left.|\Delta_{n,t}^{\lambda}|^{2p}\right|\bar{\theta}^{\lambda}_n \right]   \leq  \left(1- \lambda(t-n)a_F \tilde{\kappa}(p)\right)| \bar{\theta}^{\lambda}_n|^{2p} +\lambda(t-n) c_6(p).
\end{equation}
Furthermore, one observes that $c_2 \leq c_6(1)$ with $c_2$ given in \eqref{constc2}, hence, by using \eqref{Fdelta2p} and \eqref{delta2}, one obtains for $p\in [2, \infty)\cap {\N}$,
\begin{equation}\label{Fdelta2p-2}
\E\left[\left.|\Delta_{n,t}^{\lambda}|^{2(p-1)}\right|\bar{\theta}^{\lambda}_n \right]   \leq | \bar{\theta}^{\lambda}_n|^{2p-2} +\lambda(t-n) c_6(p-1).
\end{equation}
Substituting \eqref{Fdelta2p} and \eqref{Fdelta2p-2} into \eqref{2pthme} yields
\begin{align*}
\E\left[\left.|\bar{\theta}^{\lambda}_t|^{2p}\right|\bar{\theta}^{\lambda}_n \right]
&\leq \left(1- \lambda(t-n)a_F \tilde{\kappa}(p)\right)| \bar{\theta}^{\lambda}_n|^{2p} +2^{2p-2}p(2p-1)\lambda (t-n)d \beta^{-1}| \bar{\theta}^{\lambda}_n|^{2p-2}  \\
&\quad +\lambda(t-n) c_6(p) +2^{2p-2}p(2p-1)\lambda^2 (t-n)^2d \beta^{-1} c_6(p-1)\\
&\quad +2^{2p-4}(2p(2p-1))^{p+1}(d\beta^{-1}\lambda (t-n) )^p.
\end{align*}
One notes that for any $\theta \in \R^d$
\begin{align}\label{2pdeltaesp2}
\begin{split}
& a_F \tilde{\kappa}(p) |\theta|^{2p} - 2^{2p-2}p(2p-1) d \beta^{-1}|\theta|^{2p-2}  > \frac{a_F\tilde{\kappa}(p) }{2} |\theta|^{2p}  \\
& \Leftrightarrow \quad |\theta| > \left(\frac{2^{2p-1}p(2p-1) d \beta^{-1}}{a_F \tilde{\kappa}(p) }\right)^{1/2}.
\end{split}
\end{align}
Denote by $M_4(p) := (2^{2p-1}p(2p-1)d \beta^{-1}/(a_F\tilde{\kappa}(p)))^{1/2}$ and $\mathsf{S}_{n,M_4(p) } := \{\omega \in \Omega: |\bar{\theta}^{\lambda}_n(\omega)| >M_4(p) \}$. Then, by \eqref{2pdeltaesp2}, it follows that
\[
\E\left[\left.|\bar{\theta}^{\lambda}_t|^{2p} \1_{\mathsf{S}_{n,M_4(p) }} \right|\bar{\theta}^{\lambda}_n \right]
\leq  (1-\lambda(t-n)a_F\tilde{\kappa}(p)/2)|\bar{\theta}^{\lambda}_n|^{2p}\1_{\mathsf{S}_{n,M_4(p) }} +\lambda (t-n) c_7(p)\1_{\mathsf{S}_{n,M_4(p) }} ,
\]
where $c_7(p) := c_6(p) +2^{2p-2}p(2p-1) d \beta^{-1}c_6(p-1) +2^{2p-4}(2p(2p-1))^{p+1}(d\beta^{-1})^p$, and furthermore,
\begin{align*}
\E\left[\left.|\bar{\theta}^{\lambda}_t|^{2p} \1_{\mathsf{S}_{n,M_4(p) }^{\mathsf{c}} }\right|\bar{\theta}^{\lambda}_n \right]
&\leq  (1-\lambda(t-n)a_F\tilde{\kappa}(p))|\bar{\theta}^{\lambda}_n|^{2p}\1_{\mathsf{S}_{n,M_4(p) }^{\mathsf{c}}} +\lambda (t-n) c_7(p)\1_{\mathsf{S}_{n,M_4(p) }^{\mathsf{c}}}\\
& \quad + \lambda (t-n) 2^{2p-2}p(2p-1) d \beta^{-1}(M_4(p))^{2p-2}\1_{\mathsf{S}_{n,M_4(p) }^{\mathsf{c}}}.
\end{align*}
By combining the two cases, one obtains, for $p\in [2, \infty)\cap {\N}$, $0<\lambda \leq \tilde{\lambda}_{\max}$,
\[
\E\left[\left.|\bar{\theta}^{\lambda}_t|^{2p}\right|\bar{\theta}^{\lambda}_n \right]  \leq (1-\lambda(t-n)a_F\tilde{\kappa}(p)/2)|\bar{\theta}^{\lambda}_n|^{2p} +\lambda (t-n)\tilde{c}_0(p),
\]
where
\begin{align}\label{constcpsp}
\begin{split}
\tilde{\kappa}(p)
& := (M_3(p))^{2r}/(2(1+(M_3(p))^{2r})),\\
M_3(p)
&:=  \frac{\binom{p}{\lcrc{p/2}}3^{p+2}p\left( 2^{2p(q+1)}K_G^{2p} \E\left[(1+|X_0|)^{2p\rho}\right]+(1+2b_F+4K_F^2)^p\right)}{\min\{1, a_F\}}, \\
\tilde{c}_0(p)
& :=  c_7(p) +2^{2p-2}p(2p-1) d \beta^{-1}(M_4(p))^{2p-2},\\
M_4(p)
&:= (2^{2p-1}p(2p-1)d \beta^{-1}/(a_F\tilde{\kappa}(p)))^{1/2},\\
c_7(p)
& := c_6(p) +2^{2p-2}p(2p-1) d \beta^{-1}c_6(p-1) +2^{2p-4}(2p(2p-1))^{p+1}(d\beta^{-1})^p,\\
c_6(p)
&:= a_F \tilde{\kappa}(p)(M_3(p) )^{2p}\\
&\quad+\sum_{k = 1}^p \binom{p}{k}3^k \left(2^{k(2q+2)} K_G^{2k} \E\left[(1+|X_0|)^{2k\rho}\right] +(2b_F+4K_F^2)^k \right)(M_3(p) )^{2p+k(q-1)}.
\end{split}
\end{align}
This further implies, by noticing $\tilde{\kappa}(p) \geq \tilde{\kappa}(2)$, for any $p\in [2, \infty)\cap {\N}$, and by using similar arguments as in \eqref{2ndmmttp}, that
\begin{equation}\label{2pmesp}
\E\left[ |\bar{\theta}^{\lambda}_t|^{2p} \right]  \leq (1-\lambda(t-n)a_F\tilde{\kappa}(2)/2)(1-\lambda a_F\tilde{\kappa}(2)/2)^n\E\left[|\theta_0|^{2p}\right]  + \tilde{c}_0(p)(1+2/(a_F\tilde{\kappa}(p))).
\end{equation}
For any $p\in [2, \infty)\cap {\N}$, denote by $\kappa^\sharp_p := \min\{\bar{\kappa}(p), \tilde{\kappa}(p)\}$ and $c^\sharp_p := \max\{\bar{c}_0(p), \tilde{c}_0(p)\}$. Then, by using \eqref{2pme}, \eqref{2pmesp}, one can conclude that, for $t \in (n, n+1], n \in \N_0$, $0<\lambda \leq \tilde{\lambda}_{\max}$,
\[
\E\left[ |\bar{\theta}^{\lambda}_t|^{2p} \right]  \leq (1-\lambda (t-n) a_F\kappa^\sharp_2/2) (1-\lambda a_F\kappa^\sharp_2/2)^n\E\left[|\theta_0|^{2p}\right]  + c^\sharp_p(1+2/(a_F\kappa^\sharp_p )),
\]
which completes the proof.
\end{proof}
\begin{proof}[\textbf{Proof of Lemma \ref{zetaprocme}}]\label{proofzetaprocme}
For any $p\in [2, \infty)\cap {\N}$, $0<\lambda\leq \lambda_{\lcrc{p/2},\max}$ with $\lambda_{\lcrc{p/2},\max}$ given in \eqref{stepsizemax}, $t \in (nT, (n+1)T], n \in \N_0$, one obtains, by applying It\^o's formula, that
\begin{align}\label{zetaprocito}
\E[V_p(\bar{\zeta}_t^{\lambda,n})]
&= \E[V_p( \bar{\theta}^{\lambda}_{nT})] + \int_{nT}^t \E\left[\lambda \Delta V_p(\bar{\zeta}_s^{\lambda,n})/\beta - \lambda \langle h(\bar{\zeta}_s^{\lambda,n}), \nabla V_p(\bar{\zeta}_s^{\lambda,n}) \rangle\right] \rmd s \nonumber\\
&\quad +\E\left[\int_{nT}^t \left\langle   \nabla V_p(\bar{\zeta}_s^{\lambda,n}),\sqrt{2\lambda\beta^{-1}}\,\rmd B^{\lambda}_s \right\rangle \right] \nonumber\\
& = \E[V_p( \bar{\theta}^{\lambda}_{nT})] + \int_{nT}^t \E\left[\lambda \Delta V_p(\bar{\zeta}_s^{\lambda,n})/\beta - \lambda \langle h(\bar{\zeta}_s^{\lambda,n}), \nabla V_p(\bar{\zeta}_s^{\lambda,n}) \rangle\right] \rmd s.
\end{align}

To see that \eqref{zetaprocito} holds, it suffices to show that $\E\left[\int_{nT}^t |\nabla V_p(\bar{\zeta}_s^{\lambda,n})|^2\, \rmd s\right] <\infty$. To this end, define $\tau_k : = \inf\{s\geq nT:  |\bar{\zeta}_s^{\lambda,n} |> k\}$. Applying It\^o's formula to the stopped process $V_{2p}(\bar{\zeta}_{t\wedge \tau_k}^{\lambda,n})$ yields
\begin{align*}
\E[V_{2p}(\bar{\zeta}_{t\wedge \tau_k}^{\lambda,n})]
&= \E[V_{2p}( \bar{\theta}^{\lambda}_{nT})] + \int_{nT}^t \E\left[\lambda \Delta V_{2p}(\bar{\zeta}_{s\wedge \tau_k}^{\lambda,n})/\beta - \lambda \langle h(\bar{\zeta}_{s\wedge \tau_k}^{\lambda,n}), \nabla V_{2p}(\bar{\zeta}_{s\wedge \tau_k}^{\lambda,n}) \rangle\right] \rmd s\\
& \leq C^*_0+ \int_{nT}^t \left(-\lambda c_{V,1}(2p) \E[V_{2p}(\bar{\zeta}_{s\wedge \tau_k}^{\lambda,n})] +\lambda c_{V,2}(2p)\right)\,\rmd s\\
&\leq C^*_1+C^*_2\int_{nT}^t \E[V_{2p}(\bar{\zeta}_{s\wedge \tau_k}^{\lambda,n})] \,\rmd s\\
&\leq C^*_1e^{C^*_2(t-nT)}<\infty,
\end{align*}
for some constants $C^*_0, C^*_1, C^*_2>0$ which are independent of $\tau_k$, where the first inequality holds due to Lemma \ref{driftcon}, while the last inequality holds due to Gr\"{o}nwall's lemma. By applying Fatou's lemma, one obtains, for any $t \geq nT$,
\[
\E[V_{2p}(\bar{\zeta}_t^{\lambda,n})] \leq \liminf_{k\rightarrow \infty} \E[V_{2p}(\bar{\zeta}_{t\wedge \tau_k}^{\lambda,n})] \leq C^*_1e^{C^*_2(t-nT)}<\infty.
\]
Since $|\nabla V_p(\theta)|^2 \leq p^2|V_p(\theta)|^2=p^2V_{2p}(\theta)$ for all $\theta \in \R^d$, we have that
\[
\E\left[\int_{nT}^t |\nabla V_p(\bar{\zeta}_s^{\lambda,n})|^2\, \rmd s\right]  \leq p^2\E\left[\int_{nT}^t  V_{2p}(\bar{\zeta}_s^{\lambda,n}) \, \rmd s\right] <\infty
\]
as desired.

Then, differentiating both sides of \eqref{zetaprocito} and applying Lemma \ref{driftcon} yield
\begin{align*}
\frac{\rmd}{\rmd t} \E[V_p(\bar{\zeta}_t^{\lambda,n})]
&= \E\left[ \lambda \Delta V_p(\bar{\zeta}_t^{\lambda,n})/\beta - \lambda \langle h(\bar{\zeta}_t^{\lambda,n}), \nabla V_p(\bar{\zeta}_t^{\lambda,n}) \rangle\right] \leq -\lambda c_{V,1}(p) \E[V_p(\bar{\zeta}_t^{\lambda,n})] +\lambda c_{V,2}(p).
\end{align*}
The above inequality further implies
\[
\E[V_p(\bar{\zeta}_t^{\lambda,n})] \leq  e^{-\lambda  c_{V,1}(p) (t-nT)}\E[V_p( \bar{\theta}^{\lambda}_{nT})] +\frac{c_{V,2}(p)}{c_{V,1}(p)}\left(1- e^{-\lambda c_{V,1}(p) (t-nT)}\right).
\]
By setting $p = 2$ and by using Lemma \ref{2ndpthmmt}, one obtains, for any $0<\lambda\leq \lambda_{1,\max}$,
\begin{align*}
\E[V_2(\bar{\zeta}_t^{\lambda,n})]
&\leq  e^{-\lambda   c_{V,1}(2) (t-nT)}(1-a_F\kappa \lambda)^{nT} \E\left[|\theta_0|^2\right]  +c_0(1+1/(a_F\kappa))+\frac{c_{V,2}(2)}{c_{V,1}(2)}+1\\
&\leq e^{- \min\{a_F\kappa, a_h/2\} \lambda t}\E\left[V_2(\theta_0)\right] +c_0(1+1/(a_F\kappa))+3\mathrm{v}_2(M_V(2))+1,
\end{align*}
where the last inequality holds due to $1-\nu \leq e^{-\nu}$ for any $\nu \in \R$, moreover, $\kappa, c_0$ are given in \eqref{constc0} (Lemma \ref{2ndpthmmt}), and $M_V(2)$ is given in Lemma \ref{driftcon}. By using the same arguments, one can obtain an upper bound for $\E[V_4(\bar{\zeta}_t^{\lambda,n})] $.
\end{proof}
\begin{lemma}\label{onesteperror}  Let Assumption \ref{AI}, \ref{AG}, \ref{AF}, and \ref{AC} hold. Then, for any $0<\lambda\leq \lambda_{\max}$ with $\lambda_{\max}$ given in \eqref{stepsizemax}, $ t\geq 0$, one obtains
\[
\E\left[|\bar{\theta}^{\lambda}_t  - \bar{\theta}^{\lambda}_{\lfrf{t}}|^4\right] \leq \lambda^2\left(e^{-\lambda a_F \kappa^\sharp_2\lfrf{t}/2}\bar{C}_{0,1}\E\left[|\theta_0|^{4(2r+1)}\right]+\bar{C}_{1,1}\right),
\]
where $\bar{C}_{0,1}, \bar{C}_{1,1}$ are given explicitly in \eqref{onesteperrorconst}.
\end{lemma}
\begin{proof} For any $ t\geq 0$, by using the continuous-time interpolation of the TUSLA algorithm given in \eqref{tuslaproc}, one obtains,
\begin{align*}
\E\left[|\bar{\theta}^{\lambda}_t - \bar{\theta}^{\lambda}_{\lfrf{t}}|^4\right]
&= \E\left[\left|-\lambda \int_{\lfrf{t}}^t  H_{\lambda}(\bar{\theta}^{\lambda}_{\lfrf{s}},X_{\lcrc{s}})\, \rmd s
+ \sqrt{\frac{2\lambda}{\beta}} \int_{\lfrf{t}}^t  \rmd B^{\lambda}_s\right|^4\right] \\
& \leq 8\lambda^4\E\left[|H(\bar{\theta}^{\lambda}_{\lfrf{t}},X_{\lcrc{t}})|^4\right]+32d(d+2)\lambda^2\beta^{-2},
\end{align*}
where one notices that for any $s \in [{\lfrf{t}}, t]$, $t \geq 0$, it hods that $\lfrf{s} = \lfrf{t}$ and $\lcrc{s} = \lfrf{s}+1 = \lcrc{t} $. Then, applying Remark \ref{growthHh} and Lemma \ref{2ndpthmmt} yield, for $0<\lambda\leq \lambda_{\max}$,
\begin{align*}
\E\left[|\bar{\theta}^{\lambda}_t - \bar{\theta}^{\lambda}_{\lfrf{t}}|^4\right]
& \leq 64\lambda^4K_H^4\E\left[(1+|X_0|)^{4\rho}\right] \E\left[|\bar{\theta}^{\lambda}_{\lfrf{t}}|^{4(2r+1)}\right] \\
&\quad  +64\lambda^4K_H^4\E\left[(1+|X_0|)^{4\rho}\right]+32d(d+2)\lambda^2\beta^{-2}\\
&\leq 64\lambda^4K_H^4\E\left[(1+|X_0|)^{4\rho}\right]  \left(e^{-\lambda a_F\kappa^\sharp_2 \lfrf{t} /2}\E\left[|\theta_0|^{8r+4}\right] +c^\sharp_{4r+2}(1+2/(a_F\kappa^\sharp_{4r+2} )) \right) \\
&\quad  +64\lambda^4K_H^4\E\left[(1+|X_0|)^{4\rho}\right]+32d(d+2)\lambda^2\beta^{-2}\\
&\leq \lambda^2\left(e^{-\lambda a_F \kappa^\sharp_2\lfrf{t}/2}\bar{C}_{0,1}\E\left[|\theta_0|^{4(2r+1)}\right]+\bar{C}_{1,1}\right),
\end{align*}
where the second inequality holds due to $1-\nu \leq e^{-\nu}$ for any $\nu \in \R$, and where
\begin{align}\label{onesteperrorconst}
\begin{split}
\bar{C}_{0,1}
& :=  64 K_H^4\E\left[(1+|X_0|)^{4\rho}\right],\\
\bar{C}_{1,1}
&:= 64 K_H^4\E\left[(1+|X_0|)^{4\rho}\right](1+c^\sharp_{4r+2}(1+2/(a_F\kappa^\sharp_{4r+2} ))) +32d(d+2) \beta^{-2},
\end{split}
\end{align}
with $K_H$ given in Remark \ref{growthHh}, and $c^\sharp_{4r+2}, \kappa^\sharp_{4r+2}$ given in Lemma \ref{2ndpthmmt}.
\end{proof}
\begin{proof}[\textbf{Proof of Lemma \ref{w1converp1}}]\label{proofw1converp1}
Recall the continuous-time interpolation of the TUSLA algorithm $(\bar{\theta}^{\lambda}_t)_{t \geq 0}$ given in \eqref{tuslaproc} and the definition of the auxiliary process $(\bar{\zeta}^{\lambda, n}_t)_{t \geq nT}, n \in \N_0$, $T : = \lfrf{1/\lambda}$ given in Definition \ref{auxzeta}.
By It\^o's formula, one obtains for any $t \in (nT, (n+1)T], n \in \N_0$,
\begin{align}\label{w2essplitting}
W_2^2(\mathcal{L}(\bar{\theta}^{\lambda}_t),\mathcal{L}(\bar{\zeta}^{\lambda, n}_t))
&\leq \E\left[|\bar{\zeta}^{\lambda, n}_t- \bar{\theta}^{\lambda}_t  |^2\right]\nonumber\\
& = -2\lambda \E\left[\int_{nT}^t  \left\langle \bar{\zeta}^{\lambda, n}_s -\bar{\theta}^{\lambda}_s ,  h(\bar{\zeta}^{\lambda, n}_s)- H_{\lambda}(\bar{\theta}^{\lambda}_{\lfrf{s}},X_{\lcrc{s}})\right\rangle \, \rmd s \right]\nonumber\\
\begin{split}
& = -2\lambda \int_{nT}^t\E\left[ \left \langle \bar{\zeta}^{\lambda, n}_s -\bar{\theta}^{\lambda}_s ,  h(\bar{\zeta}^{\lambda, n}_s)- h(\bar{\theta}^{\lambda}_s)\right\rangle \right]\, \rmd s \\
&\quad -2\lambda \int_{nT}^t\E\left[  \left\langle \bar{\zeta}^{\lambda, n}_s -\bar{\theta}^{\lambda}_s ,  h(\bar{\theta}^{\lambda}_s)- h(\bar{\theta}^{\lambda}_{\lfrf{s}}) \right\rangle \right]\, \rmd s \\
&\quad -2\lambda \int_{nT}^t\E\left[  \left\langle \bar{\zeta}^{\lambda, n}_s -\bar{\theta}^{\lambda}_s ,  h(\bar{\theta}^{\lambda}_{\lfrf{s}})- H(\bar{\theta}^{\lambda}_{\lfrf{s}},X_{\lcrc{s}})\right\rangle \right]\, \rmd s \\
&\quad -2\lambda \int_{nT}^t\E\left[  \left\langle \bar{\zeta}^{\lambda, n}_s -\bar{\theta}^{\lambda}_s ,  H(\bar{\theta}^{\lambda}_{\lfrf{s}},X_{\lcrc{s}})- H_{\lambda}(\bar{\theta}^{\lambda}_{\lfrf{s}},X_{\lcrc{s}})\right\rangle \right]\, \rmd s.
\end{split}
\end{align}
This implies, by applying Remark \ref{ACh} to the first term on the RHS of \eqref{w2essplitting}, and by applying Young's inequality, i.e., $2uv \leq \varepsilon u^2+v^2/\varepsilon$ with $\varepsilon = L_R$ for $u, v \geq 0$, to the second and the last term on the RHS of \eqref{w2essplitting},
\begin{align}\label{w2es1}
\begin{split}
\E\left[|\bar{\zeta}^{\lambda, n}_t- \bar{\theta}^{\lambda}_t  |^2\right]
& \leq 4\lambda L_R \int_{nT}^t\E\left[|\bar{\zeta}^{\lambda, n}_s -\bar{\theta}^{\lambda}_s|^2 \right]\, \rmd s   + \lambda L_R^{-1} \int_{nT}^t\E\left[  |h(\bar{\theta}^{\lambda}_s)- h(\bar{\theta}^{\lambda}_{\lfrf{s}})|^2 \right]\, \rmd s \\
&\quad + \lambda L_R^{-1} \int_{nT}^t\E\left[ | H(\bar{\theta}^{\lambda}_{\lfrf{s}},X_{\lcrc{s}})- H_{\lambda}(\bar{\theta}^{\lambda}_{\lfrf{s}},X_{\lcrc{s}})|^2\right]\, \rmd s\\
&\quad -2\lambda \int_{nT}^t\E\left[  \left\langle \bar{\zeta}^{\lambda, n}_s -\bar{\theta}^{\lambda}_{\lfrf{s}} ,  h(\bar{\theta}^{\lambda}_{\lfrf{s}})- H(\bar{\theta}^{\lambda}_{\lfrf{s}},X_{\lcrc{s}})\right\rangle \right]\, \rmd s\\
&\quad -2\lambda \int_{nT}^t\E\left[  \left\langle \bar{\theta}^{\lambda}_{\lfrf{s}} -\bar{\theta}^{\lambda}_s ,  h(\bar{\theta}^{\lambda}_{\lfrf{s}})- H(\bar{\theta}^{\lambda}_{\lfrf{s}},X_{\lcrc{s}})\right\rangle \right]\, \rmd s.
\end{split}
\end{align}
By using Remark \ref{growthHh} and Cauchy–Schwarz inequality, one obtains, for any $s \in (nT, (n+1)T], n \in \N_0$,
\begin{align}\label{w2es2}
 \E\left[  |h(\bar{\theta}^{\lambda}_s)- h(\bar{\theta}^{\lambda}_{\lfrf{s}})|^2 \right]
& \leq L_h^2\E\left[ (1+|\bar{\theta}^{\lambda}_s|+|\bar{\theta}^{\lambda}_{\lfrf{s}}|)^{4r}|\bar{\theta}^{\lambda}_s-\bar{\theta}^{\lambda}_{\lfrf{s}}|^2\right]\nonumber\\
&\leq  L_h^2  \left(\E\left[( 1+ | \bar{\theta}^{\lambda}_s| +|\bar{\theta}^{\lambda}_{\lfrf{s}}|)^{8r} \right]\right)^{1/2}\left(\E\left[  | \bar{\theta}^{\lambda}_s -\bar{\theta}^{\lambda}_{\lfrf{s}}|^4 \right]\right)^{1/2}\nonumber\\
&\leq  3^{4r-(1/2)}L_h^2  \left(\E\left[ 1+ | \bar{\theta}^{\lambda}_s|^{8r}  +|\bar{\theta}^{\lambda}_{\lfrf{s}}|^{8r} \right]\right)^{1/2}\left(\E\left[  | \bar{\theta}^{\lambda}_s -\bar{\theta}^{\lambda}_{\lfrf{s}}|^4 \right]\right)^{1/2}.
\end{align}
Moreover, for any $s \in (nT, (n+1)T], n \in \N_0$, by using Remark \ref{growthHh} and \eqref{tamedH}, the following estimate can be obtained:
\begin{align}\label{w2es3}
 \E\left[ | H(\bar{\theta}^{\lambda}_{\lfrf{s}},X_{\lcrc{s}})- H_{\lambda}(\bar{\theta}^{\lambda}_{\lfrf{s}},X_{\lcrc{s}})|^2\right]
& = \E\left[  \left|\frac{(1+\sqrt{\lambda}|\bar{\theta}^{\lambda}_{\lfrf{s}}|^{2r})H(\bar{\theta}^{\lambda}_{\lfrf{s}},X_{\lcrc{s}})- H (\bar{\theta}^{\lambda}_{\lfrf{s}},X_{\lcrc{s}})}{1+\sqrt{\lambda}|\bar{\theta}^{\lambda}_{\lfrf{s}}|^{2r}} \right|^2\right] \nonumber\\
&\leq \lambda\E\left[|\bar{\theta}^{\lambda}_{\lfrf{s}}|^{4r} | H(\bar{\theta}^{\lambda}_{\lfrf{s}},X_{\lcrc{s}})|^2\right]\nonumber\\
&\leq \lambda2K_H^2\E\left[(1+|X_{\lcrc{s}}|)^{2\rho}(1+|\bar{\theta}^{\lambda}_{\lfrf{s}}|^{4r+2})|\bar{\theta}^{\lambda}_{\lfrf{s}}|^{4r} \right]\nonumber\\
&\leq \lambda2K_H^2\E\left[(1+|X_{\lcrc{s}}|)^{2\rho}(1+|\bar{\theta}^{\lambda}_{\lfrf{s}}|^{4r+2})^2 \right]\nonumber\\
&\leq \lambda4K_H^2\E\left[(1+|X_{\lcrc{s}}|)^{2\rho}(1+|\bar{\theta}^{\lambda}_{\lfrf{s}}|^{8r+4}) \right].
\end{align}
Define a continuous-time filtration $(\mathcal{H}_t)_{t \geq 0}$ by $\mathcal{H}_t := \mathcal{F}^{\lambda}_{\infty} \vee \mathcal{G}_{\lfrf{t}} \vee \sigma(\theta_0), t \geq 0$. Substituting \eqref{w2es2}, \eqref{w2es3} into \eqref{w2es1}, and using the definition of $(\bar{\theta}^{\lambda}_t)_{ t\geq 0}$ given in \eqref{tuslaproc} yield
\begin{align}\label{w2es4}
 \begin{split}
 \E\left[|\bar{\zeta}^{\lambda, n}_t- \bar{\theta}^{\lambda}_t  |^2\right]
& \leq 4\lambda L_R \int_{nT}^t\E\left[|\bar{\zeta}^{\lambda, n}_s -\bar{\theta}^{\lambda}_s|^2 \right]\, \rmd s \\
&\quad + 3^{4r-(1/2)}\lambda L_h^2L_R^{-1} \int_{nT}^t\left(\E\left[ 1+ | \bar{\theta}^{\lambda}_s|^{8r}+|\bar{\theta}^{\lambda}_{\lfrf{s}}|^{8r} \right]\right)^{1/2}\left(\E\left[  | \bar{\theta}^{\lambda}_s -\bar{\theta}^{\lambda}_{\lfrf{s}}|^4 \right]\right)^{1/2}\, \rmd s \\
&\quad + 4\lambda^2 K_H^2 L_R^{-1} \int_{nT}^t\E\left[ (1+|X_{\lcrc{s}}|)^{2\rho}(1+|\bar{\theta}^{\lambda}_{\lfrf{s}}|^{8r+4})\right]\, \rmd s\\
&\quad -2\lambda \int_{nT}^t\E\left[ \E\left[ \left.\left\langle \bar{\zeta}^{\lambda, n}_s -\bar{\theta}^{\lambda}_{\lfrf{s}} ,  h(\bar{\theta}^{\lambda}_{\lfrf{s}})- H(\bar{\theta}^{\lambda}_{\lfrf{s}},X_{\lcrc{s}})\right\rangle\right| \mathcal{H}_s \right]\right]\, \rmd s\\
&\quad -2\lambda^2 \int_{nT}^t\E\left[  \left\langle  \int_{\lfrf{s}}^s H_{\lambda}(\bar{\theta}^{\lambda}_{\lfrf{r}},X_{\lcrc{r}}) \, \rmd r ,  h(\bar{\theta}^{\lambda}_{\lfrf{s}})- H(\bar{\theta}^{\lambda}_{\lfrf{s}},X_{\lcrc{s}})\right\rangle \right]\, \rmd s\\
&\quad +2\lambda\sqrt{2\lambda\beta^{-1}} \int_{nT}^t\E\left[  \left\langle  \int_{\lfrf{s}}^s \rmd B^{\lambda}_r ,  h(\bar{\theta}^{\lambda}_{\lfrf{s}})- H(\bar{\theta}^{\lambda}_{\lfrf{s}},X_{\lcrc{s}})\right\rangle \right]\, \rmd s.
\end{split}
\end{align}
For any $r \in [{\lfrf{s}}, s]$, $s \geq 0$, we have $\lfrf{r} = \lfrf{s}$, and $\lcrc{r} = \lfrf{r}+1 = \lcrc{s}$. By using Remark \ref{growthHh} and Young's inequality, for any $s\in (nT, (n+1)T], n \in \N_0$, the fifth term on the RHS of \eqref{w2es4} can be estimated as follows:
\begin{align}
&-\E\left[  \left\langle  \int_{\lfrf{s}}^s H_{\lambda}(\bar{\theta}^{\lambda}_{\lfrf{r}},X_{\lcrc{r}}) \, \rmd r ,  h(\bar{\theta}^{\lambda}_{\lfrf{s}})- H(\bar{\theta}^{\lambda}_{\lfrf{s}},X_{\lcrc{s}})\right\rangle \right]\nonumber\\
& \leq \E\left[   | H_{\lambda}(\bar{\theta}^{\lambda}_{\lfrf{s}},X_{\lcrc{s}})||h(\bar{\theta}^{\lambda}_{\lfrf{s}})- H(\bar{\theta}^{\lambda}_{\lfrf{s}},X_{\lcrc{s}}) | \right]\nonumber\\
&\leq \E\left[|H(\bar{\theta}^{\lambda}_{\lfrf{s}},X_{\lcrc{s}}) |\left(|h(\bar{\theta}^{\lambda}_{\lfrf{s}})|+|H(\bar{\theta}^{\lambda}_{\lfrf{s}},X_{\lcrc{s}}) | \right)\right]\nonumber\\
&\leq \E\left[|h(\bar{\theta}^{\lambda}_{\lfrf{s}})|^2\right] + 2\E\left[|H(\bar{\theta}^{\lambda}_{\lfrf{s}},X_{\lcrc{s}}) |^2\right]\nonumber\\
&\leq  \E\left[\left(L_h(1+|\bar{\theta}^{\lambda}_{\lfrf{s}}|)^{2r+1}+|h(0)|\right)^2\right]  + 4K_H^2\E\left[(1+|X_{\lcrc{s}}|)^{2\rho}(1+|\bar{\theta}^{\lambda}_{\lfrf{s}}|^{4r+2})  \right]\nonumber\\
&\leq 2(1+L_F+L_G)^2\E\left[ (1+2|X_0|)^{2\rho}\right] \E\left[ (1+|\bar{\theta}^{\lambda}_{\lfrf{s}}|)^{4r+2} \right] +2 |h(0)|^2\nonumber\\
&\quad + 8K_H^2\E\left[(1+|X_{\lcrc{s}}|)^{2\rho}(1+|\bar{\theta}^{\lambda}_{\lfrf{s}}|^{8r+4})  \right]\nonumber\\
\begin{split}\label{w2es5}
& \leq 2^{4r+3+2\rho}(1+L_F+L_G)^2\E\left[ (1+ |X_0|)^{2\rho}\right] \E\left[ (1+|\bar{\theta}^{\lambda}_{\lfrf{s}}|^{8r+4} )\right] +2 |h(0)|^2\\
&\quad + 8K_H^2\E\left[(1+|X_{\lcrc{s}}|)^{2\rho}(1+|\bar{\theta}^{\lambda}_{\lfrf{s}}|^{8r+4})  \right].
\end{split}
\end{align}
Furthermore, one notes that the fourth and the last term on the RHS of \eqref{w2es4} is zero, and $X_{\lcrc{s}}$ is independent of $\bar{\theta}^{\lambda}_{\lfrf{s}}$ for any $s \geq 0$. Therefore, it follows that, by substituting \eqref{w2es5} into \eqref{w2es4},
 \begin{align*}
&\E\left[|\bar{\zeta}^{\lambda, n}_t- \bar{\theta}^{\lambda}_t  |^2\right] \\
& \leq 4\lambda L_R \int_{nT}^t\E\left[|\bar{\zeta}^{\lambda, n}_s -\bar{\theta}^{\lambda}_s|^2 \right]\, \rmd s \\
&\quad + 3^{4r-1/2}\lambda L_h^2L_R^{-1} \int_{nT}^t\left(\E\left[ 1+ | \bar{\theta}^{\lambda}_s|^{8r}+|\bar{\theta}^{\lambda}_{\lfrf{s}}|^{8r} \right]\right)^{1/2}\left(\E\left[  | \bar{\theta}^{\lambda}_s -\bar{\theta}^{\lambda}_{\lfrf{s}}|^4 \right]\right)^{1/2}\, \rmd s \\
&\quad + 4\lambda^2 K_H^2 L_R^{-1} \int_{nT}^t\E\left[ (1+|X_0|)^{2\rho}\right]\left(1+\E\left[ |\bar{\theta}^{\lambda}_{\lfrf{s}}|^{8r+4}\right]\right)\, \rmd s\\
&\quad + 2^{4r+4+2\rho}\lambda^2(1+L_F+L_G)^2\int_{nT}^t\E\left[ (1+ |X_0|)^{2\rho}\right]\left(1+ \E\left[ |\bar{\theta}^{\lambda}_{\lfrf{s}}|^{8r+4} \right]\right)\, \rmd s +4\lambda |h(0)|^2\\
&\quad +16\lambda^2K_H^2 \int_{nT}^t\E\left[ (1+|X_0|)^{2\rho}\right]\left(1+\E\left[ |\bar{\theta}^{\lambda}_{\lfrf{s}}|^{8r+4}\right]\right)\, \rmd s.
\end{align*}
This yields, by applying Lemma \ref{2ndpthmmt}, Lemma \ref{onesteperror}, and by using the fact that $1-\nu \leq e^{-\nu}$ for any $\nu \in \R$, that
\begin{align*}
&\E\left[|\bar{\zeta}^{\lambda, n}_t- \bar{\theta}^{\lambda}_t  |^2\right] \\
&\leq 4\lambda L_R \int_{nT}^t\E\left[|\bar{\zeta}^{\lambda, n}_s -\bar{\theta}^{\lambda}_s|^2 \right]\, \rmd s +4\lambda |h(0)|^2\\
&\quad + 3^{4r-1/2}\lambda^2 L_h^2L_R^{-1} \int_{nT}^t\left( 1+2e^{-\lambda a_F\kappa^\sharp_2 \lfrf{s} /2}\E\left[|\theta_0|^{8r}\right]  + 2c^\sharp_{4r}(1+2/(a_F\kappa^\sharp_{4r} )) \right)^{1/2}  \\
&\quad\times \left(e^{-\lambda a_F \kappa^\sharp_2\lfrf{s}/2}\bar{C}_{0,1}\E\left[|\theta_0|^{4(2r+1)}\right]+\bar{C}_{1,1}\right)^{1/2}\, \rmd s \\
&\quad + \lambda^2(16K_H^2(1+L_R^{-1})+2^{4r+4+2\rho} (1+L_F+L_G)^2)\E\left[ (1+ |X_0|)^{2\rho}\right]\\
&\quad \times  \int_{nT}^t \left(1+e^{-\lambda a_F\kappa^\sharp_2 \lfrf{s} /2}\E\left[|\theta_0|^{8r+4}\right] +c^\sharp_{4r+2}(1+2/(a_F\kappa^\sharp_{4r+2} )) \right)\, \rmd s  \\
&\leq 4\lambda L_R \int_{nT}^t\E\left[|\bar{\zeta}^{\lambda, n}_s -\bar{\theta}^{\lambda}_s|^2 \right]\, \rmd s  +4\lambda |h(0)|^2\\
&\quad + 3^{4r}\lambda^2 L_h^2L_R^{-1}  \int_{nT}^t  \left(e^{-\lambda a_F \kappa^\sharp_2\lfrf{s}/2}\bar{C}_{0,1}\E\left[V_{4(2r+1)}(\theta_0)\right]+\bar{C}_{1,1}+c^\sharp_{4r}(1+2/(a_F\kappa^\sharp_{4r} )) +1\right) \, \rmd s \\
&\quad + \lambda^2(16K_H^2(1+L_R^{-1})+2^{4r+4+2\rho}(1+L_F+L_G)^2)\E\left[ (1+ |X_0|)^{2\rho}\right]\\
&\quad \times  \int_{nT}^t \left(1+e^{-\lambda a_F\kappa^\sharp_2 \lfrf{s} /2}\E\left[|\theta_0|^{8r+4}\right] +c^\sharp_{4r+2}(1+2/(a_F\kappa^\sharp_{4r+2} )) \right)\, \rmd s  \\
& \leq 4\lambda L_R \int_{nT}^t\E\left[|\bar{\zeta}^{\lambda, n}_s -\bar{\theta}^{\lambda}_s|^2 \right]\, \rmd s +4\lambda |h(0)|^2\\
&\quad + 3^{4r}\lambda L_h^2L_R^{-1} \left(e^{- a_F \kappa^\sharp_2n/4}\bar{C}_{0,1}\E\left[V_{4(2r+1)}(\theta_0)\right]+\bar{C}_{1,1}+c^\sharp_{4r}(1+2/(a_F\kappa^\sharp_{4r}))+1\right) \\
&\quad + \lambda (16K_H^2(1+L_R^{-1})+2^{4r+4+2\rho}(1+L_F+L_G)^2)\E\left[ (1+ |X_0|)^{2\rho}\right]\\
&\qquad \times \left(e^{- a_F \kappa^\sharp_2n/4}\E\left[V_{4(2r+1)}(\theta_0)\right]+c^\sharp_{4r+2}(1+2/(a_F\kappa^\sharp_{4r+2}))+1\right)\\
&\leq 4\lambda L_R \int_{nT}^t\E\left[|\bar{\zeta}^{\lambda, n}_s -\bar{\theta}^{\lambda}_s|^2 \right]\, \rmd s +\lambda e^{-4L_R}\left(e^{-a_F \kappa^\sharp_2 n/4}\bar{C}_0\E\left[V_{4(2r+1)}(\theta_0)\right]+\bar{C}_1\right),
\end{align*}
where the third inequality holds due to $1/2 \leq \lambda T \leq 1$, and where
\begin{align}\label{w1converp1const}
\begin{split}
\kappa^\sharp_2
&:=\min\{\bar{\kappa}(2), \tilde{\kappa}(2)\},\\
\bar{C}_0
& := e^{4L_R}\Big(3^{4r} L_h^2L_R^{-1}\bar{C}_{0,1}\Big.  \Big.+(16K_H^2 (1+ L_R^{-1} )+2^{4r+4+2\rho}(1+L_F+L_G)^2) \E\left[ (1+|X_0|)^{2\rho}\right]\Big) ,\\
\bar{C}_1
& := e^{4L_R}\Big(3^{4r} L_h^2L_R^{-1} + 16K_H^2 (1+ L_R^{-1} ) +2^{4r+4+2\rho}(1+L_F+L_G)^2\Big)\\
&\quad \times \Big(\bar{C}_{1,1}+c^\sharp_{4r}(1+2/(a_F\kappa^\sharp_{4r}))+1\Big)+ 4  e^{4L_R} |h(0)|^2
\end{split}
\end{align}
with $\bar{\kappa}(2), \tilde{\kappa}(2)$ given explicitly in \eqref{constcp} and \eqref{constcpsp} (Lemma \ref{2ndpthmmt}), $\bar{C}_{0,1}, \bar{C}_{1,1}$  given explicitly in \eqref{onesteperrorconst} (Lemma \ref{onesteperror}) and $c^\sharp_{4r}, \kappa^\sharp_{4r}$ given explicitly in Lemma \ref{2ndpthmmt}. Finally, by Gr\"{o}nwall's lemma, one hence obtains
\[
\E\left[|\bar{\zeta}^{\lambda, n}_t- \bar{\theta}^{\lambda}_t  |^2\right] \leq \lambda  \left(e^{-a_F \kappa^\sharp_2 n/4}\bar{C}_0\E\left[V_{4(2r+1)}(\theta_0)\right]+\bar{C}_1\right),
\]
which completes the proof.
\end{proof}
\begin{lemma}\label{wassw1p} For any $ \mu,\nu \in \mathcal{P}_{V_p}(\R^d)$, the following inequalities hold for $w_{1,2}$ defined in \eqref{w1p}:
\begin{equation}\label{w1w2w12}
W_1(\mu,\nu)\leq w_{1,2}(\mu,\nu), \quad W_2(\mu,\nu)\leq \sqrt{2w_{1,2}(\mu,\nu)}.
\end{equation}
\end{lemma}
\begin{proof} Consider two probability measures $\mu, \nu \in \mathcal{P}_{V_p}(\R^d)$. Recall the definition of $w_{1,2}(\mu, \nu)$ given in \eqref{w1p}, and the definition of $W_p(\mu, \nu)$ given in \eqref{eq:definition-W-p}. We prove the first inequality in \eqref{w1w2w12}. For any $ \zeta\in \mathcal{C}(\mu, \nu)$, one has
\begin{align*}
&\int_{\R^d}\int_{\R^d}|\theta-\theta'|\zeta(\rmd \theta \rmd \theta')\\
& = \int_{\R^d}\int_{\R^d}|\theta-\theta'|\1_{\{|\theta-\theta'|\geq 1\}}\zeta(\rmd \theta \rmd \theta')  + \int_{\R^d}\int_{\R^d}|\theta-\theta'|\1_{\{|\theta-\theta'|< 1\}}\zeta(\rmd \theta \rmd \theta')\\
&\leq \int_{\R^d}\int_{\R^d}(|\theta|+|\theta'|)\1_{\{|\theta-\theta'|\geq 1\}}\zeta(\rmd \theta \rmd \theta')   + \int_{\R^d}\int_{\R^d} |\theta-\theta'| (1+V_2(\theta)+V_2(\theta'))\1_{\{|\theta-\theta'|< 1\}}\zeta(\rmd\theta \rmd\theta')\\
&\leq \int_{\R^d}\int_{\R^d}  (1+V_2(\theta)+V_2(\theta'))\1_{\{|\theta-\theta'|\geq 1\}}\zeta(\rmd\theta \rmd\theta')\\
&\quad + \int_{\R^d}\int_{\R^d} |\theta-\theta'| (1+V_2(\theta)+V_2(\theta'))\1_{\{|\theta-\theta'|< 1\}}\zeta(\rmd\theta \rmd\theta')\\
&= \int_{\R^d}\int_{\R^d}
[1\wedge |\theta-\theta'|](1+V_2(\theta)+V_2(\theta'))\1_{\{|\theta-\theta'|\geq 1\}}\zeta(\rmd\theta \rmd\theta')\\
&\quad + \int_{\R^d}\int_{\R^d}
[1\wedge |\theta-\theta'|] (1+V_2(\theta)+V_2(\theta'))\1_{\{|\theta-\theta'|< 1\}}\zeta(\rmd\theta \rmd\theta')\\
& =  \int_{\R^d}\int_{\R^d}
[1\wedge |\theta-\theta'|] (1+V_2(\theta)+V_2(\theta')) \zeta(\rmd\theta \rmd\theta').
\end{align*}
By taking infimum over $ \zeta\in \mathcal{C}(\mu, \nu)$, the above inequality yields $W_1(\mu,\nu)\leq w_{1,2}(\mu,\nu)$.

Moreover, the second inequality in \eqref{w1w2w12} can be obtained by applying similar arguments. For any $ \zeta\in \mathcal{C}(\mu, \nu)$, one obtains
\begin{align*}
&\int_{\R^d}\int_{\R^d}|\theta-\theta'|^2\zeta(\rmd \theta \rmd \theta')\\
& = \int_{\R^d}\int_{\R^d}|\theta-\theta'|^2\1_{\{|\theta-\theta'|\geq 1\}}\zeta(\rmd \theta \rmd \theta') + \int_{\R^d}\int_{\R^d} |\theta-\theta'|^2 \1_{\{|\theta-\theta'|< 1\}}\zeta(\rmd \theta \rmd \theta')\\
&\leq \int_{\R^d}\int_{\R^d} 2(|\theta|^2+|\theta'|^2 )\1_{\{|\theta-\theta'|\geq 1\}}\zeta(\rmd \theta \rmd \theta')  + \int_{\R^d}\int_{\R^d} |\theta-\theta'|(|\theta|+|\theta'|)\1_{\{|\theta-\theta'|< 1\}}\zeta(\rmd\theta \rmd\theta')\\
&\leq \int_{\R^d}\int_{\R^d}  2 (1+V_2(\theta)+V_2(\theta'))\1_{\{|\theta-\theta'|\geq 1\}}\zeta(\rmd\theta \rmd\theta')\\
&\quad + \int_{\R^d}\int_{\R^d}2 |\theta-\theta'| (1+V_2(\theta)+V_2(\theta'))\1_{\{|\theta-\theta'|< 1\}}\zeta(\rmd\theta \rmd\theta')\\
&=2 \int_{\R^d}\int_{\R^d}
[1\wedge |\theta-\theta'|](1+V_2(\theta)+V_2(\theta'))\1_{\{|\theta-\theta'|\geq 1\}}\zeta(\rmd\theta \rmd\theta')\\
&\quad + 2\int_{\R^d}\int_{\R^d}
[1\wedge |\theta-\theta'|] (1+V_2(\theta)+V_2(\theta'))\1_{\{|\theta-\theta'|< 1\}}\zeta(\rmd\theta \rmd\theta')\\
& =  2\int_{\R^d}\int_{\R^d}
[1\wedge |\theta-\theta'|] (1+V_2(\theta)+V_2(\theta')) \zeta(\rmd\theta \rmd\theta').
\end{align*}
By taking infimum over $ \zeta\in \mathcal{C}(\mu, \nu)$, the above inequality yields $W_2^2(\mu,\nu)\leq 2w_{1,2}(\mu,\nu)$.
\end{proof}

\begin{proof}[\textbf{Proof of Proposition \ref{contr}}]  \label{proofcontr}
One notes that \cite[Assumption 2.1]{eberle2019quantitative} holds with $\kappa = L_R$ due to Remark \ref{ACh}. \cite[Assumption 2.2]{eberle2019quantitative} holds with $V = V_2$ due to Lemma \ref{driftcon}. Moreover, \cite[Assumption 2.4 and 2.5]{eberle2019quantitative} hold due to \eqref{contrassumption}. Thus, \cite[Theorem 2.2, Corollary 2.3]{eberle2019quantitative} hold under Assumption \ref{AI}, \ref{AG}, \ref{AF}, and \ref{AC} . Then, \eqref{w12contraction} can be obtained by using the same argument as in the proof of \cite[Proposition~3.14]{nonconvex}.

To obtain the explicit expression of the contraction constant $\dot{c}$ for SDE \eqref{sde}, we apply the same arguments as in the proof of \cite[Theorem 2.2]{eberle2019quantitative} but replace $h(r)$ in \cite[Eqn. (5.14)]{eberle2019quantitative} with
\begin{equation}\label{cchr}
h(r): = \frac{\beta}{4}\int_0^r s\kappa \, \rmd s +2Q(\epsilon)r,
\end{equation}
where $ \kappa = L_R$ as explained above, $Q(\epsilon)$ is given in \cite[Eqn. (2.24)]{eberle2019quantitative}, and replace \cite[Eqn. (2.25)]{eberle2019quantitative} with
\[
 \left(4 c_{V,2}(2) \epsilon\right)^{-1}  \geq\frac{\beta}{2}\int_0^{R_1} \int_0^s\exp\left( \frac{\beta}{4}\int_r^s u\kappa\, \rmd u+2Q(\epsilon)(s-r)\right)\, \rmd r\,\rmd s.
\]
Then, one can derive an explicit expression for $\dot{c}$, which is given by
\[
\dot{c}=\min\{\phi, c_{V,1}(2), 4c_{V,2}(2) \epsilon c_{V,1}(2)\}/2,
\]
where $c_{V,1}(2) := a_h/2$, $c_{V,2}(2) := (3/2)a_h\mathrm{v}_2(M_V(2))$ with $M_V(2)$ given in Lemma \ref{driftcon}, and where $\phi$ is given by
\[
\phi^{-1} =\beta \int_0^{R_2} \int_0^s\exp\left(\frac{\beta}{4}\int_r^s u \kappa \, \rmd u+2Q(\epsilon) (s-r)\right)\,\rmd r\, \rmd s,
\]
where $R_2$ is given in \cite[Eqn. (2.29)]{eberle2019quantitative}. Furthermore, $ \epsilon \in (0,1]$ is required to satisfy
\[
\epsilon^{-1} \geq 2\beta c_{V,2}(2) \int_0^{R_1} \int_0^s\exp\left(\frac{\beta}{4}\int_r^s u \kappa \, \rmd u+2Q(\epsilon) (s-r)\right)\,\rmd r\, \rmd s,
\]
where $R_1$ is given in \cite[Eqn. (2.29)]{eberle2019quantitative}. To simplify the expressions for $\phi$ and $\epsilon$, we follow the proof of \cite[Lemma~3.24]{nonconvex}, and thus \eqref{expcontra1}, \eqref{expcontra2}, \eqref{expcontra3} can be obtained.

To obtain an explicit expression for $\hat{c}$, one first notes that, by using \eqref{cchr}, \cite[Eqn. (5.4)]{eberle2019quantitative} becomes: for any $r \in [0,R_2]$,
\[
r\exp(-\beta \kappa R_2^2/8-2Q(\epsilon) R_2) \leq \Phi(r) \leq 2f(r) \leq 2 \Phi(r) \leq 2r.
\]
Then, in view of \cite[Eqn. (60)]{nonconvex}, and by applying the same arguments as in the proof of \cite[Lemma~3.24]{nonconvex}, one obtains
\[
C_9 = C_{11}/C_{10} \leq \hat{c}: = 2(1+\overline{R}_2)\exp(\beta K_1\overline{R}_2^2/8+2\overline{R}_2)/\epsilon,
\]
where $\overline{R}_2 :=\dot{c}_0:=2\sqrt{4c_{V,2}(2)(1+c_{V,1}(2))/c_{V,1}(2)-1}$, $K_1 =  L_R$, and $\epsilon$ is given in \eqref{expcontra3}.
\end{proof}
\begin{proof}[\textbf{Proof of Lemma \ref{w1converp2}}] \label{proofw1converp2} The proof follows the same idea as in the proof of \cite[Lemma 3.18]{nonconvex}, the details are provided for the explicit constants. By using Definition \ref{auxzeta}, Lemma \ref{wassw1p}, Proposition \ref{contr}, one obtains, for any $t \in (nT, (n+1)T], n \in \N_0$,
\begin{align*}
&W_1(\mathcal{L}(\bar{\zeta}_t^{\lambda,n}),\mathcal{L}(Z_t^\lambda))\\
&\leq \sum_{k=1}^n W_1(\mathcal{L}(\bar{\zeta}_t^{\lambda,k}),\mathcal{L}(\bar{\zeta}_t^{\lambda,k-1}))  \\
&\leq \sum_{k=1}^n w_{1,2}(\mathcal{L}(\zeta^{kT,\bar{\theta}^{\lambda}_{kT}, \lambda}_t ),\mathcal{L}(\zeta^{kT,\bar{\zeta}_{kT}^{\lambda,k-1}, \lambda}_t))  \\
&\leq \hat{c} \sum_{k=1}^n e^{-\dot{c} (n-k)/2} w_{1,2}(\mathcal{L}(\bar{\theta}^{\lambda}_{kT}),\mathcal{L}(\bar{\zeta}_{kT}^{\lambda,k-1})) \\
&\leq \hat{c} \sum_{k=1}^n e^{-\dot{c} (n-k)/2}  W_2(\mathcal{L}(\bar{\theta}^{\lambda}_{kT}),\mathcal{L}(\bar{\zeta}_{kT}^{\lambda,k-1})) \left[1 + \left\lbrace \E[V_4(\bar{\theta}^{\lambda}_{kT})]\right\rbrace^{1/2}+ \left\lbrace\E[V_4(\bar{\zeta}_{kT}^{\lambda,k-1})] \right\rbrace^{1/2}\right],
\end{align*}
where the last inequality is obtained by using \eqref{w1p}, Cauchy-Schwarz inequality, and Minkowski inequality. This further implies, due to Young's inequality,  Lemma \ref{w1converp1}, \ref{2ndpthmmt}, and \ref{zetaprocme},
\begin{align*}
 W_1(\mathcal{L}(\bar{\zeta}_t^{\lambda,n}),\mathcal{L}(Z_t^\lambda))
&\leq (\sqrt{\lambda})^{-1}\hat{c} \sum_{k=1}^n e^{-\dot{c} (n-k)/2}  W^2_2(\mathcal{L}(\bar{\theta}^{\lambda}_{kT}),\mathcal{L}(\bar{\zeta}_{kT}^{\lambda,k-1}))  \\
&\quad +3\sqrt{\lambda}\hat{c} \sum_{k=1}^n e^{-\dot{c} (n-k)/2}  \left[1 + \E[V_4(\bar{\theta}^{\lambda}_{kT})] +\E[V_4(\bar{\zeta}_{kT}^{\lambda,k-1})] \right]  \\
&\leq \sqrt{\lambda}\hat{c} \sum_{k=1}^n e^{-\dot{c} (n-k)/2} e^{-(k-1)\min\{a_F \kappa^\sharp_2/2, a_h\}/2}(\bar{C}_0+12)\E[V_{4(2r+1)}(\theta_0)]  \\
&\quad +\sqrt{\lambda}\frac{\hat{c} }{1 - e^{-\dot{c}/2}}(\bar{C}_1+12c^\sharp_2(1+2/(a_F\kappa^\sharp_2 ))+9\mathrm{v}_4(M_V(4))+15)\\
& \leq \sqrt{\lambda}\hat{c}n e^{-(n-1)\min\{\dot{c}, a_F \kappa^\sharp_2/2, a_h\}/2}(\bar{C}_0+12)\E[V_{4(2r+1)}(\theta_0)] \\
&\quad +2\sqrt{\lambda}(\hat{c}/\dot{c})e^{ \dot{c}/2} (\bar{C}_1+12c^\sharp_2(1+2/(a_F\kappa^\sharp_2 ))+9\mathrm{v}_4(M_V(4))+15),
\end{align*}
where the last inequality holds due to $1-e^{-s} \geq s e^{-s}, s \in \R$. Finally, one notes that $e^{-\alpha y}(y+1) \leq 1+\alpha^{-1}$, for any $\alpha >0, y \geq 0$, hence, by using the inequality with $\alpha = \min\{\dot{c}, a_F \kappa^\sharp_2/2, a_h\}/4$ and $y = n-1$, one obtains
\[
W_1(\mathcal{L}(\bar{\zeta}_t^{\lambda,n}),\mathcal{L}(Z_t^\lambda)) \leq \sqrt{\lambda}\left(e^{-\min\{\dot{c}, a_F \kappa^\sharp_2/2, a_h\}n/4}\bar{C}_2\E\left[V_{4(2r+1)}(\theta_0)\right]+\bar{C}_3\right),
\]
where
\begin{align}\label{w1converp2const}
\begin{split}
\kappa^\sharp_2
& :=\min\{\bar{\kappa}(2), \tilde{\kappa}(2)\},\\
\bar{C}_2
& :=e^{\min\{\dot{c}, a_F \kappa^\sharp_2/2, a_h\}/4}\hat{c}\left(1+ \frac{4}{\min\{\dot{c}, a_F \kappa^\sharp_2/2, a_h\}}\right)(\bar{C}_0+12),\\
\bar{C}_3
& := 2(\hat{c}/\dot{c})e^{ \dot{c}/2}(\bar{C}_1+12c^\sharp_2(1+2/(a_F\kappa^\sharp_2 ))+9\mathrm{v}_4(M_V(4))+15)
\end{split}
\end{align}
with $\dot{c}, \hat{c}$ given in Proposition \ref{contr}, $\bar{\kappa}(2), \tilde{\kappa}(2)$ given in \eqref{constcp} (Lemma \ref{2ndpthmmt}), $\bar{C}_0, \bar{C}_1$ given in \eqref{w1converp1const} (Lemma \ref{w1converp1}), $c^\sharp_2$ given in Lemma \ref{2ndpthmmt} and $M_V(4)$ given in Lemma \ref{zetaprocme}.
\end{proof}


\begin{proof}[\textbf{Proof of Lemma \ref{eerp1}}]\label{proofeerp1}
We follow a similar approach as in \cite[Lemma 3.5]{raginsky}. Recall that $h := \nabla u$. By Remark \ref{growthHh}, for any $\theta \in \R^d$, it follows that
\[
|h(\theta)| \leq K_H\E[(1+|X_0|)^{\rho}](|\theta|^{2r+1}+1).
\]
Denote by $\bar{C}_6 := K_H\E[(1+|X_0|)^{\rho}]$, then, one obtains, for any $\theta, \theta' \in \R^d$,
\begin{align}\label{eerp1ineq}
u(\theta) - u(\theta')
&= \int_0^1 \langle h(t\theta +(1-t)\theta'), \theta-\theta' \rangle\, \rmd t \nonumber\\
& \leq \left(\int_0^1  (2^{2r}\bar{C}_6(t^{2r+1}|\theta|^{2r+1} + (1-t)^{2r+1}|\theta'|^{2r+1})+\bar{C}_6)\, \rmd t \right)|\theta-\theta'| \nonumber\\
&\leq \left(\frac{2^{2r}\bar{C}_6}{2r+2} |\theta|^{2r+1}+\frac{2^{2r}\bar{C}_6}{2r+2} |\theta'|^{2r+1}+\bar{C}_6\right)|\theta-\theta'|.
\end{align}
Recall $Z_{\infty} \sim \pi_{\beta}$ with $\pi_{\beta}(\theta) \propto e^{-\beta u(\theta)}$, $\theta \in \R^d$. We consider the coupling $\mathbf{P} \in \mathcal{C}(\mathcal{L}(\theta^{\lambda}_n), \mathcal{L}(Z_{\infty}))$ such that
\[
W_2^2(\mathcal{L}(\theta^{\lambda}_n),\mathcal{L}(Z_{\infty})) = \E_{\mathbf{P} }[|\theta^{\lambda}_n - Z_{\infty}|^2].
\]
Then, by using \eqref{eerp1ineq} and Cauchy-Schwarz inequality, it follows that
\begin{align*}
 \E[u( \theta_n^{\lambda})] - \E[u(Z_{\infty})]
 &= \E_{\mathbf{P}} [u( \theta_n^{\lambda}) - u(Z_{\infty})]\\
 &\leq \left(\frac{2^{2r}\bar{C}_6}{2r+2}(\E[|\theta^{\lambda}_n|^{4r+2}])^{1/2}+\frac{2^{2r}\bar{C}_6}{2r+2}(\E[|Z_{\infty}|^{4r+2}])^{1/2}+\bar{C}_6\right)  W_2(\mathcal{L}(\theta^{\lambda}_n),\mathcal{L}(Z_{\infty})).
\end{align*}
Finally, applying Lemma \ref{2ndpthmmt} and Corollary \ref{mainw2} yield
\begin{align*}
  \E[u( \theta_n^{\lambda})] - \E[u(Z_{\infty})]
 &\leq \left(\frac{2^{2r}\bar{C}_6}{2r+2}\left(\left(\E[|\theta^{\lambda}_0|^{4r+2}]+c^\sharp_{2r+1}(1+2/(a_F\kappa^\sharp_{2r+1} ))\right)^{1/2}+c_{Z_{\infty}, 4r+2}^{1/2}\right)+\bar{C}_6\right) \\
 &\qquad \times \left[C_4 e^{-C_3 \lambda n}(\E[|\theta_0|^{4(2r+1)}]+1)^{1/2} +C_5\lambda^{1/4}\right]\\
 &\leq C_7 e^{-C_6 \lambda n} +C_8\lambda^{1/4},
\end{align*}
where
\begin{align}\label{eerp1const}
\begin{split}
C_6
& := C_3, \\
C_7
& := C_4 \left(\frac{2^{2r}\bar{C}_6}{2r+2}\left(1+(c^\sharp_{2r+1}(1+2/(a_F\kappa^\sharp_{2r+1} )))^{1/2}+c_{Z_{\infty}, 4r+2}^{1/2}\right)+\bar{C}_6\right) (\E[|\theta_0|^{4(2r+1)}]+1), \\
C_8
& := C_5\left(\frac{2^{2r}\bar{C}_6}{2r+2}\left((\E[|\theta^{\lambda}_0|^{4r+2}]+c^\sharp_{2r+1}(1+2/(a_F\kappa^\sharp_{2r+1} )))^{1/2}+c_{Z_{\infty}, 4r+2}^{1/2}\right)+\bar{C}_6\right) ,\\
\bar{C}_6
&:= K_H\E[(1+|X_0|)^{\rho}],
\end{split}
\end{align}
with $C_3, C_4, C_5$ given in \eqref{mainw2const} (Corollary \ref{mainw2}), $c^\sharp_{2r+1}, \kappa^\sharp_{2r+1}$ given in Lemma \ref{2ndpthmmt}, $c_{Z_{\infty}, 4r+2}$ denoting the $(4r+2)$-th moment of $\pi_{\beta}$.
\end{proof}
\begin{proof}[\textbf{Proof of Lemma \ref{eerp2}}] \label{proofeerp2}
By using \cite[Equation (3.18), (3.20)]{raginsky}, one obtains
\begin{align}\label{eerp2t1}
\E[u(Z_{\infty})]- u^*
&= \frac{1}{\beta}\left(-\int_{\R^d}   \frac{e^{-\beta u(\theta)}}{\bar{C}_{\pi_{\beta}}}\log \frac{e^{-\beta u(\theta)}}{\bar{C}_{\pi_{\beta}}}   \, \rmd \theta - \log \bar{C}_{\pi_{\beta}} \right)-u^* \nonumber\\
& \leq \frac{d}{2\beta}\log\left(\frac{2\pi e(b_h+d/\beta)}{a_hd}\right)- \frac{\log \bar{C}_{\pi_{\beta}}}{\beta} -u^*,
\end{align}
where $\bar{C}_{\pi_{\beta}} := \int_{\R^d} e^{-\beta u(\theta)}\, \rmd \theta$ is the normalizing constant. Then, to obtain an upper bound for the term $\log \bar{C}_{\pi_{\beta}}/\beta$, we follow the arguments in \cite[Lemma 3.2]{lovas2020taming}. Denote by $\theta^* \in \R^d$ a minimizer of $u$. By Remark \ref{ADFh}, we have that
\[
0=  \langle \theta^*, h(\theta^*) \rangle \geq a_h|\theta^*|^2 - b_h,
\]
which implies that $|\theta^*| \leq \sqrt{b_h/a_h}$. Moreover, one observes that, for any $\theta, \theta' \in \R^d$,
\[
-\beta(u(\theta^*) - u(\theta))  \leq \beta \left|\int_0^1 \langle h(t\theta^*+(1-t)\theta)- h(\theta^*), \theta^*-\theta \rangle \, \rmd t\right|,
\]
which implies, by using Remark \ref{growthHh},
\begin{align*}
-\beta(u(\theta^*) - u(\theta))
& \leq \beta \left( \int_0^1 L_h(1-t)(1+|t\theta^*+(1-t)\theta|+|\theta^*|)^{2r} \, \rmd t \right)|\theta^*-\theta|^2\\
&\leq \beta L_h (1+2|\theta^*-\theta|+2|\theta^*|)^{2r}|\theta^*-\theta|^2/2.
\end{align*}
Denote by $R_{\theta^*}: = \max\{ \sqrt{b_h/a_h},\sqrt{2d/(\beta L_h)} \}$, and $\bar{\mathrm{B}}(\theta^*,R_{\theta^*})$ the closed ball with radius $R_{\theta^*}$ centred at $\theta^*\in \R^d$. By noticing $u^* = u(\theta^*)$, further calculations hence yield
\begin{align}\label{eerp2t2}
\frac{\log \bar{C}_{\pi_{\beta}}}{\beta}
& = -u^* +\frac{1}{\beta}\log \int_{\R^d} e^{\beta (u^*-u(\theta))}\, \rmd \theta \nonumber\\
& \geq  -u^* +\frac{1}{\beta}\log \int_{\bar{\mathrm{B}}(\theta^*,R_{\theta^*})} e^{-\beta L_h (1+4R_{\theta^*})^{2r}|\theta^*-\theta|^2/2}\, \rmd \theta\nonumber\\
& = -u^* + \frac{1}{\beta}\log \left(\left(\frac{2\pi}{\beta\bar{C}_7}\right)^{d/2} \int_{\bar{\mathrm{B}}(\theta^*,R_{\theta^*})} f_{\Theta}(\theta)\, \rmd \theta \right),
\end{align}
where $\bar{C}_7 :=  L_h(1+4R_{\theta^*})^{2r}$ and $f_{\Theta}$ denotes the density function of a Gaussian random variable $\Theta$ with mean $\theta$ and covariance $(\beta\bar{C}_7)^{-1} I_d$. 
Therefore, applying Chebyshev’s inequality yields 
\begin{align}\label{eerp2t2cig}
\P(|\Theta - \theta^*| > R_{\theta^*})
&= \P\left(|\Theta - \theta^*| >  \sqrt{\frac{\beta\bar{C}_7R_{\theta^*}^2}{d}}\sqrt{\frac{d}{\beta\bar{C}_7}}\right) \leq \frac{d}{\beta R_{\theta^*}^2\bar{C}_7},
\end{align}
which, by substituting \eqref{eerp2t2cig} into \eqref{eerp2t2}, implies,
\begin{align}\label{eerp2t2ue}
\frac{\log \bar{C}_{\pi_{\beta}}}{\beta}
& \geq  -u^* +\frac{1}{\beta}\log \left(\left(\frac{2\pi}{\beta\bar{C}_7}\right)^{d/2} \left(1-\frac{d}{\beta R_{\theta^*}^2\bar{C}_7}\right) \right)  \geq-u^* +\frac{1}{\beta}\log \left(\frac{1}{2}\left(\frac{2\pi}{\beta\bar{C}_7}\right)^{d/2}  \right).
\end{align}
Combining the results in \eqref{eerp2t1} and \eqref{eerp2t2ue}, one obtains
\[
\E[u(Z_{\infty})]- u^* \leq C_9/\beta,
\]
where
\begin{align}\label{eerp2const}
\begin{split}
C_9 \equiv C_9(\beta)
& := \frac{d}{2}\log\left(\frac{\bar{C}_7  e}{a_h}\left(\frac{\beta b_h}{d}+1\right)\right) +\log 2 ,\\
\bar{C}_7
& :=  L_h(1+4R_{\theta^*})^{2r},\\
R_{\theta^*}
& : =  \max\{ \sqrt{b_h/a_h},\sqrt{2d/(\beta L_h)} \}.
\end{split}
\end{align}
In particular, we have that $\lim_{\beta \to \infty} C_9(\beta)/\beta =0$.
\end{proof}
\newpage



\begin{table}[H]
\renewcommand{\arraystretch}{2}
\centering
\caption{Analytic expressions of constants}
\label{table1fullexp}
\begin{threeparttable}
\scriptsize
\begin{tabular}{@{}ccll@{}} 
\toprule
\multicolumn{2}{c}{Constant}  &\multicolumn{1}{c}{Full expression} &\multicolumn{1}{c}{Dependence on $d, \beta$}\\ 
\midrule
	\multirow{3}{*}{Lemma \ref{driftcon}}					&$M_V(p) $			 					& $\sqrt{1/3+4b_h/(3a_h)+4d/(3a_h\beta)+4(p-2)/(3a_h\beta)}$ 							& $O(1+(d/\beta)^{1/2})$\\\cline{2-4}
    																	&$c_{V,1}(p) $							&$a_hp/4$  &$O(1)$ \\\cline{2-4}
      																	&$c_{V,2}(p) $							&$(3/4)a_hp\mathrm{v}_p(M_V(p))$& $O(1+(d/\beta)^{p/2})$\\ \hline
	\multirow{4}{*}{Lemma \ref{w1converp1}}			&\multirow{2}{*}{$\bar{C}_0$}		& $e^{4L_R}\Big(3^{4r} L_h^2L_R^{-1}64 K_H^4\E\left[(1+|X_0|)^{4\rho}\right]\Big.$										& \multirow{2}{*}{$O(1)$}\\
      																	&											&$ \Big. + \left(16K_H^2 (1+ L_R^{-1} )+2^{4r+4+2\rho}(1+L_F+L_G)^2\right) \E\left[ (1+|X_0|)^{2\rho}\right]\Big)$		& \\ \cline{2-4}
      																	&\multirow{3}{*}{$\bar{C}_1$}		 & $ e^{4L_R}\Big(3^{4r} L_h^2L_R^{-1} + 16K_H^2 (1+ L_R^{-1} )+2^{4r+4+2\rho}(1+L_F+L_G)^2 \Big)$					& \multirow{3}{*}{$O(1+(d/\beta)^{4r+2})$}\\
      																	&											&$\times \Big(64 K_H^4\E\left[(1+|X_0|)^{4\rho}\right] \times(1+c^\sharp_{4r+2}(1+2/(a_F\kappa^\sharp_{4r+2} ))) \Big.$\\
      																	&											&$ \Big. +32d(d+2) \beta^{-2}+c^\sharp_{4r}(1+2/(a_F\kappa^\sharp_{4r}))+1\Big)+4  e^{4L_R} |h(0)|^2$					&  \\\hline
	\multirow{2}{*}{Lemma \ref{w1converp2}}			& $\bar{C}_2$ 							&$e^{\min\{\dot{c}, a_F \kappa^\sharp_2/2, a_h\}/4}\hat{c}\left(1+ \frac{4}{\min\{\dot{c}, a_F \kappa^\sharp_2/2, a_h\}}\right)(\bar{C}_0+12)$& 
	$O\left( e^{C(1+\beta)(1+\frac{d}{\beta})}\right)$\tnote{*}\\ \cline{2-4}
         																&$\bar{C}_3$ 							&  $2(\hat{c}/\dot{c})e^{ \dot{c}/2} (\bar{C}_1+12c^\sharp_2(1+2/(a_F\kappa^\sharp_2 ))+9\mathrm{v}_4(M_V(4))+15)$ & 
         																$O\left( e^{C(1+\beta)(1+\frac{d}{\beta})}\right)$\tnote{*}\\ \hline
	\multirow{3}{*}{Theorem \ref{mainw1}}				&$C_0$									& $\min\{\dot{c}, a_F \kappa^\sharp_2/2, a_h\}/4$ & $O(1)$\\\cline{2-4}
																	&$ C_1$									&$2^{4r+1}e^{\min\{\dot{c}, a_F \kappa^\sharp_2/2, a_h\}/4}\left[\bar{C}_0^{1/2}+\bar{C}_2+\hat{c}  \left(2+\int_{\R^d}V_2(\theta)\pi_{\beta}(d\theta)\right)\right]$&
																	$O\left( e^{C(1+\beta)(1+\frac{d}{\beta})}\right)$\tnote{*}\\\cline{2-4}
																	&$ C_2$									&$\bar{C}_1^{1/2}+\bar{C}_3$				&
																	$O\left( e^{C(1+\beta)(1+\frac{d}{\beta})}\right)$\tnote{*}\\\hline
	\multirow{5}{*}{Corollary \ref{mainw2}}				&$\bar{C}_4$							&$e^{\min\{\dot{c}, a_F \kappa^\sharp_2/2, a_h\}/8}\sqrt{\hat{c}}\left(1+ \frac{8}{\min\{\dot{c}, a_F \kappa^\sharp_2/2, a_h\}}\right)(\bar{C}_0^{1/2}+2\sqrt{2})$  & 
	$O\left(
	e^{C(1+\beta)(1+\frac{d}{\beta})}\right)$\tnote{*} \\\cline{2-4}
																	&$\bar{C}_5$								&$4(\sqrt{\hat{c}}/\dot{c})e^{\dot{c}/4}(\bar{C}_1^{1/2}+2\sqrt{2}(c^\sharp_2(1+2/(a_F\kappa^\sharp_2 )))^{1/2}+(3\mathrm{v}_4(M_V(4)))^{1/2}+3\sqrt{2})$  & 
																	$O\left(
																	e^{C(1+\beta)(1+\frac{d}{\beta})}\right)$\tnote{*}\\\cline{2-4}	
																	&$C_3$									&$\min\{\dot{c}, a_F \kappa^\sharp_2/2, a_h\}/8$ & $O(1)$  \\\cline{2-4}
																	&$ C_4 $									&$2^{2r+1}e^{\min\{\dot{c}, a_F \kappa^\sharp_2/2, a_h\}/8}\left[\bar{C}_0^{1/2}+\bar{C}_4+\sqrt{\hat{c}}  \left(2+\int_{\R^d}V_2(\theta)\pi_{\beta}(d\theta)\right)^{1/2}\right]$& 
																	$O\left(
																	e^{C(1+\beta)(1+\frac{d}{\beta})}\right)$\tnote{*}\\\cline{2-4}
																	&$ C_5 $									&$ \bar{C}_1^{1/2}+\bar{C}_5$& 
																	$O\left(
																	e^{C(1+\beta)(1+\frac{d}{\beta})}\right)$\tnote{*}\\\hline
  \multirow{5}{*}{Lemma \ref{eerp1}}						&$\bar{C}_6$							&$ K_H\E[(1+|X_0|)^{\rho}]$& $O(1)$\\\cline{2-4}
 																	 &$C_6$									& $C_3$ &  $O(1)$\\\cline{2-4}
        																 &\multirow{2}{*}{$C_7$	}									&$ C_4 \left(\frac{2^{2r}\bar{C}_6}{2r+2}\left(1+(c^\sharp_{2r+1}(1+2/(a_F\kappa^\sharp_{2r+1} )))^{1/2}+c_{Z_{\infty}, 4r+2}^{1/2}\right)+\bar{C}_6\right) $ & \multirow{2}{*}{
        																 $O\left(
        																 e^{C(1+\beta)(1+\frac{d}{\beta})}\right)$\tnote{*}}\\
        																 &									&		$\times\left(\E[|\theta_0|^{4(2r+1)}]+1\right)$		&\\\cline{2-4}
        																 &$C_8$									&$ C_5\left(\frac{2^{2r}\bar{C}_6}{2r+2}\left((\E[|\theta^{\lambda}_0|^{4r+2}]+c^\sharp_{2r+1}(1+2/(a_F\kappa^\sharp_{2r+1} )))^{1/2}+c_{Z_{\infty}, 4r+2}^{1/2}\right)+\bar{C}_6\right)$& 
        																 $O\left(
        																 e^{C(1+\beta)(1+\frac{d}{\beta})}\right)$\tnote{*}\\ \hline
      	\multirow{2}{*}{Lemma \ref{eerp2}}     					&$R_{\theta^*}$					 	&$\max\{ \sqrt{b_h/a_h},\sqrt{2d/(\beta L_h)} \}$& $O(1+(d/\beta)^{1/2})$\\\cline{2-4}
 																	&$C_9$									&$ \frac{d}{2}\log\left(\frac{L_h(1+4R_{\theta^*})^{2r}  e}{a_h}\left(\frac{\beta b_h}{d}+1\right)\right) + \log 2$& $O\left(1+d\log\left(C(1+\frac{d}{\beta})^{r}(\frac{\beta}{d}+1)\right)\right)$\tnote{*}\\    \bottomrule
    \end{tabular}
        \begin{tablenotes}
  \item[*] $C>0$ \text{is} a constant that may take different values at different places, but it is always independent of $d$ and $\beta$.
  \end{tablenotes}
    \end{threeparttable}
\end{table}



\begin{table}
\begin{center}
\renewcommand{\arraystretch}{2}
\caption{Constants in Lemma \ref{2ndpthmmt} and Proposition \ref{contr}, and their dependency on key parameters}\label{table2keydep}
\scriptsize
\begin{threeparttable}
\begin{tabular}{@{}ccccc@{}}
    \toprule
     \multicolumn{2}{c}{\multirow{2}{*}{Constant}}&  \multicolumn{3}{c}{Key parameters}   \\\cline{3-5}
    															&										&$d$							&$\beta$							& Moments of $X_0$  \\\midrule
\multirow{4}{*}{Lemma \ref{2ndpthmmt}}     		&    $\kappa \in (1/4, 1/2)$ 		& --- 							& --- 								& --- 	\\\cline{2-5}
       														&  $c_0$ 								& $O(1+d/\beta)$ 			& $O(1+d/\beta)$				&$O (\E[(1+|X_0|)^{2\rho}])$  \\\cline{2-5}
															&$\kappa^\sharp_p \in (1/4, 1/2)$
															& --- 							& --- 								& --- 	\\\cline{2-5}
          														&$c^\sharp_p$
          																	&$O(1+(d/\beta)^p)$ 		& $O(1+(d/\beta)^p)$			&$O ((\E[(1+|X_0|)^{2p\rho}])^{p(q+1)+1})$  \\\hline
\multirow{2}{*}{Proposition \ref{contr}}   			& $\dot{c}$							& \multicolumn{3}{c}{$ \left(32\sqrt{\pi}(1+a^{-2}) \left(1+\frac{d}{\beta}\right)^{3/2}\sqrt{\frac{\beta}{L_R}}e^{8C^\star(a_h,b_h)(1+ \beta L_R)\left(1+\frac{d}{\beta}\right)+\frac{16}{\beta L_R}}\right)^{-1}$\tnote{$\dagger$}} \\\cline{2-5}
     															&  $\hat{c}$							& \multicolumn{3}{c}{$O \left(\sqrt{\frac{\beta}{L_R}}(1+\frac{d}{\beta})^2 e^{12C^\star(a_h,b_h)(1+ \beta L_R)\left(1+\frac{d}{\beta}\right)+\frac{16}{\beta L_R}}\right)$\tnote{$\dagger$}}  \\
       \bottomrule
    \end{tabular}
    \begin{tablenotes}
  \item[$\dagger$] $C^\star(a_h,b_h)  = (1+2/a_h) (1+a_h+b_h)$.
  \end{tablenotes}
    \end{threeparttable}
    \end{center}
\end{table}
\newpage

\bibliographystyle{plainnat}

\bibliography{references}

\end{document}